\documentclass[11pt]{amsart}
\usepackage[margin=1.0in]{geometry} 
\usepackage{esint,amsthm,amsmath,amssymb,amsfonts,graphicx,color,comment,enumerate,psfrag,caption,bbm,bm,hyperref,enumitem,mathrsfs}
\usepackage{mathtools}

\usepackage{tikz,pgf}
\usepackage{pgfplots}
\usetikzlibrary{calc}
\usetikzlibrary{patterns}
\usetikzlibrary{patterns.meta}

\allowdisplaybreaks

\DeclareMathOperator\supp{supp}

\newtheorem{lemma}{Lemma}[section]
\newtheorem{remark}{Remark}[section]
\numberwithin{equation}{section}
\newtheorem{theorem}{Theorem}[section]
\newtheorem{proposition}[theorem]{Proposition}

\newtheorem{corollary}[theorem]{Corollary}  
\newtheorem{conjecture}[theorem]{Conjecture}

\begin{document}
\title[Square function and its applications]{Stein's square function associated with the Bochner-Riesz means on M\'etivier groups and its applications}

\author[J. Singh]{Joydwip Singh} 

\address[J. Singh]{Department of Mathematical Sciences, Indian Institute of Science Education and Research Mohali, Mohali--140306, Punjab, India.}
\email{joydwipsingh@iisermohali.ac.in}



\begin{abstract}
In this paper, we study the $L^p$-boundedness of Stein's square function $\mathfrak{S}^{\alpha}(\mathcal{L})$ associated with the sub-Laplacian $\mathcal{L}$ on M\'etivier group $G$. A key aspect of our result is that the smoothness condition is expressed in terms of the topological dimension $d$ of the underlying M\'etivier group $G$. Consequently, we also present several applications of the $L^p$-boundedness of $\mathfrak{S}^{\alpha}(\mathcal{L})$.

First, we provide an alternate proof of the sharp $L^p$-boundedness result for spectral multipliers on M\'etivier groups, recently obtained by Niedorf [Niedorf, Studia Math., 2025]. Next we prove $L^p$-boundedness of maximal spectral multipliers and consequently establish sharp $L^p$-boundedness result for the maximal Bochner-Riesz operator on M\'etivier groups, which also yields pointwise almost everywhere convergence of Bochner-Riesz means with smoothness parameter given in terms of the topological dimension of $G$. In case of M\'etivier groups our result improves upon the existing works of Mauceri-Meda [Mauceri, Meda, Rev. Mat. Iberoam., 1990] and Horwich-Martini [Horwich, Martini, J. Lond. Math. Soc., 2021]. Our result further imply the mixed norm regularity estimates for the solution of fractional Schr\"odinger equation on M\'etivier groups, where the regularity index is again expressed in terms of the topological dimension of $G$.

Finally, we study the $L^{p_1}(G) \times L^{p_2}(G)$ to $L^p(G)$ boundedness of the bilinear Bochner-Riesz means and its maximal version, associated with the sub-Laplacian on M\'etivier group $G$. Our result improves upon the recent work of the author with Bagchi and Molla [Bagchi, Molla, Singh, J. Funct. Anal., 2026] in the range $2\leq p_1, p_2 <\infty$. In the same range, we also prove boundedness of the bilinear Bochner-Riesz square functions and maximal bilinear spectral multipliers on M\'etivier groups. In all of these results the smoothness parameter is achieved in terms of the topological dimension of $G$.

\end{abstract}

\maketitle

\tableofcontents

\section{Introduction and Main results}
\subsection{Stein's square function}
\label{Subsection: stein square function on Euclidean}
The convergence of Bochner–Riesz means is one of the most fundamental and extensively studied topics in harmonic analysis and spectral multiplier theory. Broadly speaking, the Bochner–Riesz problem focuses on understanding the convergence and summability properties of Fourier series and Fourier integrals in $L^p$-spaces. For $R>0$, the Bochner-Riesz means of order $\alpha \geq 0$ in $\mathbb{R}^n$ are defined by
\begin{align}
\label{Definition: Euclidean Bochner-Riesz means}
    S^{\alpha}_R f(x) = \int_{\mathbb{R}^n} \left(1-\frac{|\xi|^2}{R^2}\right)_{+}^{\alpha} \widehat{f}(\xi)\, e^{2 \pi i x \cdot \xi} \, d\xi,
\end{align}
where $(r)_{+} = \max\{r, 0\}$ for $r \in \mathbb{R}$, $f \in \mathcal{S}(\mathbb{R}^n)$, the space of all Schwartz class functions in $\mathbb{R}^n$, and $\widehat{f}$ is its Fourier transform. The study of Bochner-Riesz means lies at the heart of harmonic analysis, which connects many different problems in this area including restriction theory, estimates for oscillatory integrals, and the study of dispersive partial differential equations. A fundamental question is to determine the optimal values of the parameter $\alpha \geq 0$, for which the operator $S^{\alpha}_R$ extends to a bounded operator on $L^p$-spaces for $1\leq p \leq \infty$. When $\alpha=0$, a celebrated result of Fefferman \cite{Fefferman_Ball_multiplier_problem_1971} shows that $S_R^{0}$ is bounded on $L^p(\mathbb{R}^n)$ if and only if $p=2$ and $n\geq 2$. Consequently, for $\alpha>0$, the famous Bochner-Riesz conjecture asserts that for $1\leq p\leq \infty $ and $ p\neq 2$, the estimate 
\begin{align*}
    \|S^{\alpha}_R f\|_{L^p(\mathbb{R}^n)} \leq C \|f\|_{L^p(\mathbb{R}^n)} \quad \text{holds if and only if} \quad \alpha> \alpha_n(p) := \max\{n|\tfrac{1}{p}-\tfrac{1}{2}|-\tfrac{1}{2}, 0 \}.
\end{align*}

For $n=2$, this conjecture was settled by Carleson and Sj\"olin \cite{Carleson_Sjolin_Multiplier_Problem_on_Disc_1972}, but for $n\geq 3$ the conjecture still remains open. However, some partial progress has been made for $n\geq 3$, for developments and recent results, see \cite{Fefferman_Strongly_Singular_Convolution_Operator_1970}, \cite{Bourgain_Besicovitch_type_maximal_1991}, \cite{Tao_Vargas_Bilinear_approach_2000}, \cite{Sanghyuk_Lee_Improved_Bochner-Riesz_2004}, \cite{Bourgain_Guth_Oscillatory_Integral_2011}, \cite{Guth_Hickman_Marina_Bochner_Riesz_Latest_2019}, \cite{Wu_Bochner_Riesz_On_R3_case_2023}, \cite{Guo_Oh_Wang_Wu_Zhang_The_Bochner_Riesz_problem_2025} and references therein.

\medskip
The Stein's square function associated with the Bochner-Riesz means is defined by
\begin{align*}
    \mathfrak{S}^{\alpha}f(x) &= \left(\int_0^{\infty} \left|\frac{\partial}{\partial t}S^{\alpha}_t f(x) \right|^2 t\, dt \right)^{1/2} .
\end{align*}
In 1958, Stein \cite{Stein_Summability_Fourier-Series_1958} first introduce this square function in order to study the pointwise converges of Fourier series and studied the $L^2$-boundedness of $\mathfrak{S}^{\alpha}$. Since then, many works have been devoted to determine the sharp smoothness parameter $\alpha$, so that the $L^p$-boundedness of $\mathfrak{S}^{\alpha}$ holds for $1< p < \infty$. When $1<p\leq 2$, using Plancherel theorem, Calder\'on-Zygmund theory of vector valued singular integral and interpolation it was shown that $\mathfrak{S}^{\alpha}$ is bounded on $L^p(\mathbb{R}^n)$ whenever $\alpha>n(\frac{1}{p}-\frac{1}{2})+\frac{1}{2}$, see \cite{Stein_Summability_Fourier-Series_1958, Sunouchi_Littlewood_Paley_1967, Igari_Kuratsubo_Lp_multipliers_1971}. In fact, the condition on $\alpha$ is also known to be necessary in this case, see \cite{Lee_Roger_Seeger_Square_Function_Maximal_Fourier_Mult_2014}. On the other hand, for $p>2$ it was conjectured that
\begin{align*}
    \|\mathfrak{S}^{\alpha}f\|_{L^p(\mathbb{R}^n)} &\leq C \|f\|_{L^p(\mathbb{R}^n)} \quad  \text{holds if and only if} \quad \alpha>\max\{n(\tfrac{1}{2}-\tfrac{1}{p}), \tfrac{1}{2}\} = \alpha_n(p)+\tfrac{1}{2} .
\end{align*}

When $n=2$, the conjecture is known to be true, which is due to the work of Carbery \cite{Carbery_Maximal_Bochner_Riesz_R2_case_1983}. Although the condition on $\alpha$ is known to be necessary \cite{Lee_Roger_Seeger_Square_Function_Maximal_Fourier_Mult_2014}, but for $n \geq 3$, the conjecture is only verified for restrictive ranges. When $n\geq 3$, for $p\geq \frac{2(n+1)}{n-1}$, the conjecture was proved by Christ \cite{Christ_almost_Everywhere_Bochner_Riesz_1985} and independently by Seeger \cite{Seeger_Maximal_Fourier_Mult_1986}. The range of $p$ was further improved by Lee, Rogers and Seeger \cite{Lee_Roger_Seeger_Improved_Stein_Square_Function_2012} and Lee \cite{Lee_Square_Function_to_Bochner_Riesz_2018}. The best known result till now was recently obtained in \cite{Gan_Jing_Wu_Stein_square_on_R3_case} for $n=3$ and in \cite{Gan_Oh_Wu_New_Bound_higher_dimension_2025} for higher dimensions. In fact, they proved the following result.

\begin{theorem}\cite[Theorem 1.1]{Gan_Oh_Wu_New_Bound_higher_dimension_2025}
\label{Theorem: Boundedness of Euclidean square function}
Define
\begin{align*}
    p_n &= \min_{2\leq k\leq n-1} \max\{2+\frac{4}{2n-k}, 2+\frac{6}{2(n-1)+(k-1)\prod_{i=k}^{n-1} \frac{2i}{2i+1}} \}.
\end{align*}
Let $n \geq 3$ and $p\geq p_n$. Then whenever $\alpha>n(\frac{1}{2}-\frac{1}{p})$ we have
\begin{align*}
    \|\mathfrak{S}^{\alpha}f\|_{L^p(\mathbb{R}^n)} &\leq C \|f\|_{L^p(\mathbb{R}^n)} .
\end{align*}
\end{theorem}

\medskip
There are various applications of the $L^p$-boundedness of Stein's square function. For example, sharp $L^p$-estimates of $\mathfrak{S}^{\alpha}$ gives the sharp $L^p$-bound for maximal Bochner-Riesz means (see \cite{Lee_Square_Function_to_Bochner_Riesz_2018}), which again implies the pointwise convergence and $L^p$-boundedness of the Bochner-Riesz means. One can also deduce the sharp version of H\"ormander-Mikhlin multiplier type theorem for general radial Fourier multipliers \cite{Carbery_Gasper_Trebels_Radial_Fourier_Mult_1984} and $L^p$-boundedness of maximal radial Fourier multipliers \cite{Carbery_Maximal_Radial_Fourier_Mult_1985}. $L^p$-boundedness of $\mathfrak{S}^{\alpha}$ has also close connection in obtaining regularity estimates of the solution of Schr\"odinger and wave equation as well as spherical means \cite{Lee_Roger_Seeger_Improved_Stein_Square_Function_2012}. On the other hand, note that apart from the Euclidean setup, the $L^p$-boundedness of Stein's square function has been also studied on the space of homogeneous type, see \cite{Chen_Duong_Yan_Stein_Square_function_Homogeneous_2013}.

\medskip
In this paper, our aim is to study the $L^p$-boundedness of square functions associated with the sub-Laplacian on M\'etivier group $G$ and present several applications of this, where the smoothness parameter is expressed in terms of the topological dimension of the underlying M\'etivier group. Note that, M\'etivier groups are the generalization of Heisenberg or Heisenberg-type groups (see \cite{Niedorf_Metivier_group_2023}), which are two-step stratified Lie groups. On such groups there are notion of homogeneous and topological dimensions, and in general topological dimension is strictly smaller than the homogeneous dimension. M\"uller-Stein \cite{Muller_Stein_Spectral_Multiplier_Heisenberg_Related_groups_1994} and independently Hebisch \cite{Hebisch_Spectral_Multiplier_Heisenberg_1993} first studied the $L^p$-boundedness of spectral multipliers on Heisenberg-type groups with smoothness parameter are expressed in terms of the topological dimension of the group. This results are turn out to be sharp. Since the seminal work of M\"uller-Stein \cite{Muller_Stein_Spectral_Multiplier_Heisenberg_Related_groups_1994} and Hebisch \cite{Hebisch_Spectral_Multiplier_Heisenberg_1993}, spectral multiplier results with smoothness parameter described in terms of the topological dimension have attracted considerable attention and a plenty of research has been carried out in this direction, see for example \cite{Cowling_Klima_Sikora_Kohn_Laplacian_On_Sphere_2011}, \cite{Martini_Lie_groups_Polynomial_Growth_2012}, \cite{Martini_Muller_New_Class_Two_Step_Stratified_Groups_2014}, \cite{Gorges_Muller_Pointwise_Heisenberg_2002}, \cite{Bagchi_Molla_Singh_Bilinear_Metivier}, \cite{Bagchi_Molla_Singh_Bilinear_Bochner_Riesz_Grushin}, \cite{Molla_Singh_Bochner_Riesz_Commutators_Grushin} and references therein. Motivated by this, our main objective in this paper is to obtain boundedness of the Stein's square function on M\'etivier groups and give various applications, where the smoothness parameter is given in terms of the topological dimension of the underlying M\'etivier groups. Some of the results are also turn out to be sharp. In the below we summarize our key findings in this paper, which are listed in order as follows.
\begin{enumerate}
    \item \textbf{Stein's square function on M\'etivier groups:} We prove $L^p$-boundedness of the Stein's square function associated with the Bochner-Riesz means on M\'etivier groups. For $p \geq 2$, on the Stein-Tomas restriction range ($p \geq \mathfrak{p}_G$), our result (Theorem \ref{Theorem: Stein square estimate}) can be seen as an analogue of the Euclidean result of Christ \cite[Lemma 1]{Christ_almost_Everywhere_Bochner_Riesz_1985} and Seeger \cite{Seeger_Maximal_Fourier_Mult_1986}, while for $1<p<2$, our result (Theorem \ref{Theorem: Stein square function for p less that 2 case}) is an analogue of the result of Sunouchi \cite[Theorem 2]{Sunouchi_Littlewood_Paley_1967}, where the Euclidean dimension $n$ in the smoothness parameter is replaced by the topological dimension of the underlying M\'etivier group. 
    \medskip
    \item \textbf{Sharp spectral multipliers on M\'etivier groups:} As an applications of $L^p$-boundedness of the Stein's square function on M\'etivier groups, we prove a sharp \emph{$p$-specific spectral multiplier theorem} on M\'etivier groups (Theorem \ref{Theorem: Analogue of Niedorf theorem}), which also provides a different proof of the recent result of Niedorf \cite{Niedorf_Metivier_group_2023}. As a corollary we also obtain the sharp $L^p$-boundedness of Bochner-Riesz means on M\'etivier groups (Corollary \ref{Corollary: Lp boundedness of Bochner-Riesz on mativier}).
    \medskip
    \item \textbf{Maximal spectral multipliers and sharp maximal Bochner-Riesz mean on M\'etivier groups:} We show that maximal spectral multipliers can be pointwise dominated by Stein's square function on M\'etivier groups. This can be seen as an analogue of the Euclidean result by Carbery \cite{Carbery_Maximal_Radial_Fourier_Mult_1985}. Therefore as an application of the $L^p$-boundedness of Stein's square function on M\'etivier groups we obtain $L^p$-boundedness of the maximal spectral multipliers (Theorem \ref{Theorem: Maximal spectral multiplier}). In case of M\'etivier groups our result improves upon the existing result of Mauceri and Meda \cite[Theorem 2.6]{Mauceri_Meda_Multipliers_Stratified_groups_1990}. As a corollary we also obtain the $L^p$-boundedness of maximal Bochner-Riesz mean on M\'etivier groups. For $p \geq 2$, on the Stein-Tomas range ($p \geq \mathfrak{p}_G$), our result turn out to be sharp (Corollary \ref{Corollary: Boundedness of linear maximal Bochner-Riesz}), which can be seen as M\'etivier analogue of the Euclidean results of \cite{Christ_almost_Everywhere_Bochner_Riesz_1985, Seeger_Maximal_Fourier_Mult_1986}. Our result also improves the corresponding result of Horwich and Martini \cite[Theorem 1.2]{Horwich_Martini_Pointwise_Heisenberg_type_2021} in case of M\'etivier groups (see Figure \ref{Figure: Linear pointwise convergence}). As a byproduct we also obtain the pointwise almost everywhere convergence of the Bochner-Riesz means on M\'etivier groups with smoothness parameter expressed in terms of the topological dimension of the M\'etivier group (Corollary \ref{Corollary: Pointwise convergence of maximal Bochner_Riesz}). This result improves the result of \cite[Theorem 1.1]{Horwich_Martini_Pointwise_Heisenberg_type_2021} regarding pointwise convergence on Heisenberg-type groups for certain ranges of $p$ in two ways: first our result applies to M\'etivier gorups, strictly larger than the class of Heisenberg-type groups, and secondly for large values of $p$, range of the smoothness parameter is also improved (see Figure \ref{Figure: Linear pointwise convergence}).
    \medskip
    \item \textbf{Regularity estimates for the solution of fractional Schr\"odinger equation on M\'etivier groups:} On $\mathbb{R}^n$, the $L^p(\mathbb{R}^n; L^2(I))$ mixed norm regularity estimate for the solution of fractional Schr\"odinger equation was obtained by Lee, Roger and Seeger \cite[Proposition 5.1]{Lee_Roger_Seeger_Improved_Stein_Square_Function_2012}. In case of M\'etivier group, we have proved analogous result for the fractional Schr\"odinger equation associated with the sub-Laplacian on M\'etivier groups (Theorem \ref{Theorem: Lp to L2 local smoothing for Metivier}), again smoothness parameter described in terms of the topological dimension of the underlying group. To the best of authors knowledge, this type of estimates has not been considered before in case of M\'etivier or Heisenberg-type groups.
    \medskip
    \item \textbf{Bilinear Bochner-Riesz means and its maximal version on M\'etivier groups:} One of the other main contribution of this paper is that, we prove improved estimate for the bilinear Bochner-Riesz means on M\'etivier group (Corollary \ref{Corollary: Bochner-Riesz main theorem Metivier}) for $2 \leq p_1, p_2 < \infty$, over the recent result of Bagchi, Molla and the author \cite[Theorem 1.2]{Bagchi_Molla_Singh_Bilinear_Metivier}. We also show that boundedness of the maximal bilinear Bochner-Riesz mean (Theorem \ref{Theorem: Bilinear maximal on Metivier}) hold on the same range and with the same smoothness threshold as of the boundedness of bilinear Bochner-Riesz means on M\'etivier groups (Corollary \ref{Corollary: Bochner-Riesz main theorem Metivier}). Both of our results can be seen as an analogue of the Euclidean result by Jeong, Lee and Vargas \cite[Corollary 1.3]{Jeong_Lee_Vargas_Bilinear_Bochner_Riesz_2018} and Jotsaroop-Shrivastava \cite[Theorem 2.1]{Jotsaroop_Shrivastava_Maximal_Bochner_Riesz_2022} respectively in the context of M\'etivier groups.
    \medskip
    \item \textbf{Bilinear Bochner-Riesz square function on M\'etivier groups:} As an application of the $L^p$-boundedness of Stein's square function, we prove boundedness of the bilinear Bochner-Riesz square functions on M\'etivier groups (Theorem \ref{Theorem: Bilinear Stein square function estimate}), analogue of the Euclidean result \cite[Theorem 2.2]{Choudhary_Jotsaroop_Shrivastava_Shuin_Bilinear_square_function}. Again here also in the smoothness parameter, the Euclidean dimension is replaced by the topological dimension of the underlying M\'etivier groups.
    \medskip
    \item \textbf{Maximal bilinear spectral multipliers on M\'etivier groups:} Finally, we prove boundedness of a maximal bilinear spectral multiplier on M\'etvier groups (Theorem \ref{Theorem: Bilinear maximal spectral multiplier}), with smoothness parameter of the corresponding multiplier is expressed in terms of the topological dimension of the underlying groups.
\end{enumerate}

\medskip
Let us start our discussion with Stein's square function associated with the sub-Laplacian on M\'etivier groups. First we briefly recall the definition of M\'etivier groups and spectral theory of the sub-Laplacian, then we move to the corresponding Stein's square function defined on it.

\subsection{Stein's square function on M\'etivier groups}
\label{Subsection: Stein's square function on Metivier}
Let $G$ denote the connected, simply connected nilpotent Lie group with associated Lie algebra $\mathfrak{g}$, which admits the decomposition $\mathfrak{g} = \mathfrak{g}_1 \oplus \mathfrak{g}_2$ such that $[\mathfrak{g}, \mathfrak{g}_1] = \mathfrak{g}_2$ and $[\mathfrak{g}, \mathfrak{g}_2] = \{0\}$. $G$ is also called the two-step stratified Lie group. Let $\{X_1, \ldots, X_{d_1}\}$ and $\{T_1, \ldots, T_{d_2}\}$ be a basis of $\mathfrak{g}_1$ and $\mathfrak{g}_2$ respectively, where $d_i = \dim \mathfrak{g}_i$ for $i=1,2$. Denote $d:=d_1+d_2$ and $Q=d_1+2d_2$ to be the topological and homogeneous dimension of $G$ respectively. There exists an inner product $\langle \cdot, \cdot \rangle$ on $\mathfrak{g}$ such that the basis $\{X_1, \ldots, X_{d_1}, T_1, \ldots, T_{d_2}\}$ become an orthonormal basis on $\mathfrak{g}$. Also let $|\cdot|$ denote the norm on $\mathfrak{g}_2^*$, the dual of $\mathfrak{g}_2$, induced from the inner product $\langle \cdot, \cdot \rangle$. For every $\lambda \in \mathfrak{g}_2^*$, the map $\omega_{\lambda}: (x_1, x_2) \to \lambda([x_1, x_2])$ defines a skew symmetric bilinear form on $\mathfrak{g}_1 \times \mathfrak{g}_1$ and hence there exists a skew symmetric endomorphism $J_{\lambda}$ on $\mathfrak{g}_1$ such that
\begin{align*}
    \omega_{\lambda}(x_1, x_2) &= \langle J_{\lambda} x_1, x_2 \rangle , \quad \quad \text{for} \quad x_1, x_2 \in \mathfrak{g}_1 .
\end{align*}
We call $G$ to be a M\'etivier group if $J_{\lambda}$ is invertible for all $\lambda \in \mathfrak{g}_2^* \setminus \{0\}$, which is again equivalent to say that the bilinear form $\omega_{\lambda}$ is non-degenerate for all $\lambda \in \mathfrak{g}_2^* \setminus \{0\}$. The group $G$ is said to be the Heisenberg-type group if $J_{\lambda}^2 = -|\lambda|^2 \, \text{id}_{\mathfrak{g}_1}$. Note that class of Heisenberg-type groups is contained in the class of M\'etivier groups, in fact the containment is strict, see \cite{Muller_Seeger_Singular_Spherical_maximal_operator_Nilpotent_2004}. In this paper, we will always assume $G$ to be the M\'etivier group unless or otherwise specified.

\medskip
On M\'etivier groups the sub-Laplacian $\mathcal{L}$ is defined to be the negative of the sum of squares of the basis vectors $\{X_1, \ldots, X_{d_1}\}$ of $\mathfrak{g}_1$, that is
\begin{align*}
    \mathcal{L} &= -(X_1^2 + \cdots + X_{d_1}^2) .
\end{align*}
The sub-Laplacian $\mathcal{L}$ become positive and essentially self-adjoint operator on $L^2(G)$ on the domain dom\,$ \mathcal{L}= \{f \in L^2(G) : \mathcal{L}f \in L^2(G)\}$, where $G$ is endowed with a left-invariant Haar measure. Therefore using spectral theorem one can define the functional calculus for $\mathcal{L}$, that is for any bounded Borel function $F : \mathbb{R} \to \mathbb{C}$ we have
\begin{align}
\label{Defintion: Spectral resolution for subLaplacian}
    F(\mathcal{L}) &= \int_0^{\infty} F(\eta) \, dE(\eta) , \quad \text{on} \quad L^2(G) ,
\end{align}
where $dE(\eta)$ denotes the spectral measure of $\mathcal{L}$. 

The spectral decomposition of $\mathcal{L}$ is known explicitly (see \cite[Theorem 3.10]{Niedorf_Restriction_Stratified_group_2025}). Let $\Lambda \in \mathbb{N}$, $\mathbf{b}=(b_1, \ldots, b_{\Lambda}) \in (0, \infty)^\Lambda$, $\mathbf{r}=(r_1, \ldots, r_\Lambda) \in \mathbb{N}^\Lambda$, $\mathbf{k}=(k_1, \ldots, k_\Lambda) \in \mathbb{N}_0^\Lambda$. We define the $(\mathbf{b}, \mathbf{r})$-rescaled Laguerre functions $\varphi_{\mathbf{k}}^{\mathbf{b}, \mathbf{r}}$ by setting
\begin{align}
\label{Definition of rescaled Laguerre functions}
    \varphi_{\mathbf{k}}^{\mathbf{b}, \mathbf{r}} &= \varphi_{k_1}^{(b_1, r_1)} \otimes \cdots \otimes \varphi_{k_\Lambda}^{(b_\Lambda, r_\Lambda)},
\end{align}
where $\varphi_{k}^{(\mu, m)}(z) = \mu^m L^{m-1}_k(\tfrac{1}{2}\mu |z|^2) e^{-\frac{1}{2}\mu |z|^2}$ for $z \in \mathbb{R}^{2m}$, $\mu>0$ is the $\mu$-rescaled Laguerre function and $L^{m-1}_k$ denotes the $k$-th Laguerre polynomial of type $m-1$. 

For $f, g \in \mathcal{S}(\mathfrak{g}_1)$, let us define the $\lambda$-twisted convolution of $f$ and $g$ by
\begin{align}
\label{definition: twisted convolution}
    f \times_{\lambda} g(x) &= \int_{\mathfrak{g}_1} f(x') g(x-x') e^{\frac{i}{2} \omega_{\lambda}(x, x')} \, dx' , \quad \quad x \in \mathfrak{g}_1 .
\end{align}
For any bounded Borel function $F : \mathbb{R} \to \mathbb{C}$, the spectral multiplier associated to the sub-Laplacian $\mathcal{L}$ is given by
\begin{align}
\label{Definition: general spectral multipler on Metivier}
    F(\mathcal{L})f(x,u) &= \frac{1}{(2\pi)^{d_2}} \int_{\mathfrak{g}_{2,r}^{*}} \sum_{\mathbf{k} \in \mathbb{N}^\Lambda} F(\eta_{\mathbf{k}}^{\lambda}) \left[f^{\lambda} \times_{\lambda} \varphi_{\mathbf{k}}^{\mathbf{b}^{\lambda}, \mathbf{r}}(R_{\lambda}^{-1}\cdot) \right](x) \, e^{i \langle \lambda, u \rangle} \, d\lambda ,
\end{align}
where $\eta_{\mathbf{k}}^{\lambda} := \eta_{\mathbf{k}}^{\mathbf{b}^{\lambda}, \mathbf{r}} = \sum_{n=1}^{\Lambda} (2 k_n + r_n) b_n^{\lambda}$, for $\lambda \in \mathfrak{g}_{2}^{*}$ the function $f^{\lambda}$ denotes the Fourier transform of $f$ along $\mathfrak{g}_2$, that is, $f^{\lambda}(x) = \int_{\mathfrak{g}_2} f(x,u)\, e^{- i \langle \lambda, u \rangle} \, du$, the set $\mathfrak{g}_{2,r}^{*}$ is the Zariski open subset of $\mathfrak{g}_{2}^{*}$, the functions $\lambda \to b_n^{\lambda}$ are homogeneous of degree $1$ for $n=1, \ldots, \Lambda$ and the functions $\lambda \to R_{\lambda}$ are homogeneous of degree $0$ (see \cite[p. 5]{Bagchi_Molla_Singh_Bilinear_Metivier}).

\medskip
For $\alpha \geq 0$ and $R>0$, if we take $F(\eta) = (1-\frac{\eta}{R^2})_{+}^{\alpha}$ in \eqref{Definition: general spectral multipler on Metivier}, then the Bochner-Riesz means associated with the sub-Laplacian $\mathcal{L}$ on the M\'etivier groups is defined by
\begin{align}
\label{Definition: Bochner-Riesz means}
    S_R^{\alpha}(\mathcal{L})f(x,u) &= \frac{1}{(2\pi)^{d_2}} \int_{\mathfrak{g}_{2,r}^{*}} \sum_{\mathbf{k} \in \mathbb{N}^\Lambda} \left(1-\frac{\eta_{\mathbf{k}}^{\lambda}}{R^2} \right)_{+}^{\alpha} \left[f^{\lambda} \times_{\lambda} \varphi_{\mathbf{k}}^{\mathbf{b}^{\lambda}, \mathbf{r}}(R_{\lambda}^{-1}\cdot) \right](x) \ e^{i \langle \lambda, u \rangle} \, d\lambda .
\end{align}

Consequently, Stein's square function for the Bochner-Riesz means associated with the sub-Laplacian $\mathcal{L}$ on M\'etivier group is defined by
\begin{align}
\label{Stein square function for subLaplacian}
    \mathfrak{S}^{\alpha}(\mathcal{L})f(x, u) &= \left( \int_0^{\infty} \left| \frac{\partial}{\partial t} S^{\alpha}_t(\mathcal{L})f(x, u) \right|^2 t \, dt \right)^{1/2} .
\end{align}
We often call the Stein's square function for the Bochner-Riesz means associated with the sub-Laplacian $\mathcal{L}$ on M\'etivier group simply as the square function or Stein's square function $\mathfrak{S}^{\alpha}(\mathcal{L})$. Here we are interested in the $L^p$-boundedness of $\mathfrak{S}^{\alpha}(\mathcal{L})$ for $1<p<\infty$, where the range of the smoothness parameter $\alpha$ is possibly expressed in terms of the topological dimension $d$ of the M\'etivier groups. Analogous to the Euclidean results (see \cite{Sunouchi_Littlewood_Paley_1967, Christ_almost_Everywhere_Bochner_Riesz_1985, Seeger_Maximal_Fourier_Mult_1986}) as discussed above, regarding the $L^p$-boundedness of $\mathfrak{S}^{\alpha}(\mathcal{L})$, here also we get two different ranges of the smoothness parameter $\alpha$, depending on whether $p$ is bigger than or less than $2$. Before presenting our results, we first define the relevant exponents that will be used throughout the paper.

\medskip
For $m \in \mathbb{N}$ and $p \in [1, \infty]$, set 
\begin{align*}
    \alpha_m(p) := \max\{m|\tfrac{1}{2}-\tfrac{1}{p}|-\tfrac{1}{2}, \tfrac{1}{2}\} .
\end{align*}
Recall that the topological dimension of $G$ is denoted by $d=d_1+d_2$. Let us set $p_{d_2} := \frac{2(d_2+1)}{(d_2-1)}$, $P_{d_1, d_2} = p_{d_2}$ if $(d_1, d_2) \notin \{(4,3), (8,6), (8,7)\}$, $P_{4,3}=6$, $P_{8,6}=17/5$ and $P_{8,7}=14/3$. Further, we define
\begin{align}
\label{Definition of mathfrak pG}
    \mathfrak{p}_G := \left\{\begin{array}{ll}
        P_{d_1, d_2} & \quad \text{if} \quad G \ = \ \text{M\'etivier groups} \\
        p_{d_2} & \quad \text{if} \quad G \ = \ \text{Heisenberg-type groups} .
    \end{array} \right.
\end{align}
We call this $\mathfrak{p}_G$ to be the Stein-Tomas exponent on M\'etivier or Heisenberg-type groups.

\medskip
The following are our first main results of this paper.

\begin{theorem}
\label{Theorem: Stein square estimate}
Let $\mathfrak{p}_G \leq p < \infty$ or $p=2$. Whenever $\alpha> \max\{d(\frac{1}{2}-\frac{1}{p}), \frac{1}{2}\} = \alpha_d(p) +\frac{1}{2}$, the Stein's square function $\mathfrak{S}^{\alpha}(\mathcal{L})$ is bounded on $L^p(G)$, that is there exists a constant $C>0$ such that
\begin{align}
\label{Boundedness inequality for Stein square function}
    \|\mathfrak{S}^{\alpha}(\mathcal{L})f\|_{L^p(G)} &\leq C \|f\|_{L^p(G)} .
\end{align}
    
\end{theorem}

\begin{theorem}
\label{Theorem: Stein square function for p less that 2 case}
Let $1<p<2$. Whenever $\alpha>d(\frac{1}{p}-\frac{1}{2})+\frac{1}{2}$, the Stein's square function $\mathfrak{S}^{\alpha}(\mathcal{L})$ is bounded on $L^p(G)$, that is there exists a constant $C>0$ such that
\begin{align*}
    \|\mathfrak{S}^{\alpha}(\mathcal{L})f\|_{L^p(G)} &\leq C \|f\|_{L^p(G)} .
\end{align*}
Moreover, if $\alpha>\frac{d+1}{2}$, then the Stein's square function $\mathfrak{S}^{\alpha}(\mathcal{L})$ is also of weak-type $(1,1)$.    
\end{theorem}

\medskip
It is important to note that Theorem \ref{Theorem: Stein square estimate} is an analogue of the Euclidean result obtained independently by Christ \cite[Lemma 1]{Christ_almost_Everywhere_Bochner_Riesz_1985} and Seeger \cite{Seeger_Maximal_Fourier_Mult_1986}, while Theorem \ref{Theorem: Stein square function for p less that 2 case} is analogue to the result of Sunouchi \cite[Theorem 2]{Sunouchi_Littlewood_Paley_1967} in the context of M\'etivier groups, where the Euclidean dimension $n$ in the smoothness threshold $\alpha_n(p)$ is replaced by the topological dimension $d$ of the underlying M\'etivier group $G$. One can also interpolate \eqref{Boundedness inequality for Stein square function} between the two instances of $p=2$ and $p=\mathfrak{p}_G$, to obtain the boundedness of $\mathfrak{S}^{\alpha}(\mathcal{L})$ for the range $2 < p <\mathfrak{p}_G$.

\subsection{Sharp spectral multipliers on M\'etivier groups}
\label{Subsection: sharp spectral multiplier}
Recall that from \eqref{Defintion: Spectral resolution for subLaplacian} and spectral theorem, for any bounded Borel function $F: \mathbb{R} \to \mathbb{C}$ the spectral multiplier operator $F(\mathcal{L})$ is bounded on $L^2(G)$ if and only if the function $F$, also called the multiplier is $E$-essentially bounded. Therefore one can ask whether the operator $F(\mathcal{L})$ initially defined on $L^2(G) \cap L^p(G)$ also extends to a bounded operator on $L^p(G)$ for $1<p<\infty$ and $p\neq 2$ or may be with some additional assumption on $F$. Characterizing the multiplier $F$ so that the corresponding spectral multiplier $F(\mathcal{L})$ is bounded on $L^p$-spaces for $1\leq p \leq \infty$, is one of the most classical problem in harmonic analysis. In case of $\mathbb{R}^n$, if we take the Laplacian $\mathcal{L}=-\Delta$, then a sufficient condition regarding the $L^p$-boundedness of $F(-\Delta)$ is given by the celebrated Mikhlin-H\"ormander multiplier theorem.

\medskip
Let $\varphi : (0, \infty) \to \mathbb{R}$ be a nontrivial smooth function with compact support. Then for any bounded Borel function $F: \mathbb{R} \to \mathbb{C}$ and $s>0$, we define
\begin{align*}
    \|F\|_{L^{2}_{s, sloc}} &:= \sup_{t>0}\|\varphi F(t\cdot)\|_{L^2_s(\mathbb{R})} .
\end{align*}

\begin{theorem}\cite{Hormander_Translation_invariant_operator_1960}
The operator $F(-\Delta)$ is bounded on $L^p(\mathbb{R}^n)$ for $1<p<\infty$, whenever $\|F\|_{L^{2}_{s, sloc}} < \infty$ for some $s>n/2$.
    
\end{theorem}

Note that the smoothness threshold $n/2$ in the above theorem is sharp and can not be decreased further (see \cite{Christ_Spectral_multiplier_Nilpotent_groups_1991}, \cite{Sikora_Wright_Imaginary_Power_Laplace_2001}). On the other hand instead of boundedness of $F(-\Delta)$ on all $L^p$ spaces for $1<p<\infty$, one can ask, for a given $p$ lying between $1$ and $\infty$, what is minimum order of smoothness needed on the multiplier in order to ensure the boundedness of the spectral multipliers on that $L^p$-space. This problem is also known as $p$-specific spectral multiplier problem. In case of $\mathbb{R}^n$, the following sharp result is well known.

\begin{theorem}\cite[Theorem 1.4]{Lee_Roger_Seeger_Square_Function_Maximal_Fourier_Mult_2014}
Let $1<p<2(n+1)/(n+3)$. Then the operator $F(-\Delta)$ is bounded on $L^p(\mathbb{R}^n)$ provided $\|F\|_{L^{2}_{s, sloc}} < \infty$ for $s>\max\{n|1/p-1/2|,1/2\}$.
    
\end{theorem}

Mikhlin-H\"ormander multiplier theorem has been generalized to various elliptic and sub-elliptic operators in many different setup. For instance, on any two stratified Lie groups Christ \cite{Christ_Spectral_multiplier_Nilpotent_groups_1991} and independently Mauceri-Meda \cite{Mauceri_Meda_Multipliers_Stratified_groups_1990} (actually they proved for any steps) proved the $L^p$-boundedness of $F(\mathcal{L})$ for $1<p<\infty$, whenever $\|F\|_{L^{2}_{s, sloc}} < \infty$ for some $s>Q/2$, where $Q=d_1+2d_2$ is so called the homogeneous dimension of the underlying group $G$. In case of Heisenberg-type groups, M\"uller-Stein \cite{Muller_Stein_Spectral_Multiplier_Heisenberg_Related_groups_1994} and independently Hebisch \cite{Hebisch_Spectral_Multiplier_Heisenberg_1993} showed that the smoothness order $Q/2$ is not sharp, and it can be pushed down to $d/2$, which turns out to be sharp, where $d=d_1+d_2$ is the topological dimension of the underlying Heisenberg-type group. In case of M\'etivier groups, sharp $L^p$-boundedness of the spectral multipliers was proved in \cite{Martini_Lie_groups_Polynomial_Growth_2012} with smoothness parameter $s>d/2$.

On the other hand, $p$-specific type spectral multiplier problem has been recently studied in \cite{Niedorf_p_specific_Heisenberg_group_2024}, \cite{Niedorf_Metivier_group_2023} for Heisenberg-type and M\'etivier groups. As an application of $L^p$-boundedness of $\mathfrak{S}^{\alpha}(\mathcal{L})$ (see Theorem \ref{Theorem: Stein square estimate}), here we will present a alternate proof of the $p$-specific type of the spectral multiplier results associated with the sub-Laplacian on M\'etivier groups, compared to the one already proved in \cite[Theorem 1.1]{Niedorf_p_specific_Heisenberg_group_2024} and \cite[Theorem 1.1]{Niedorf_Metivier_group_2023}. The following theorem gives the sharp $p$-specific type spectral multiplier result on M\'etivier groups.

\begin{theorem}
\label{Theorem: Analogue of Niedorf theorem}
Let $p \in (1, \mathfrak{p}_G'] \cup [\mathfrak{p}_G, \infty)$ and $s>d|\frac{1}{p}-\frac{1}{2}|$. Then for any bounded Borel function $F: \mathbb{R} \to \mathbb{C}$ we have
\begin{align*}
    \|F(\mathcal{L})f\|_{L^p(G)} &\leq C \|F\|_{L^{2}_{s, sloc}} \|f\|_{L^p(G)} .
\end{align*} 
\end{theorem}

As an immediate corollary of the above Theorem \ref{Theorem: Analogue of Niedorf theorem} we also get the $p$-specific boundedness of the Bochner-Riesz multiplier on M\'etivier groups (see \eqref{Definition: Bochner-Riesz means}). For more details, see \cite{Christ_Spectral_multiplier_Nilpotent_groups_1991}, \cite{Dallara_Martini_Optimal-grushin_2022}.

\begin{corollary}
\label{Corollary: Lp boundedness of Bochner-Riesz on mativier}
Let $p \in (1, \mathfrak{p}_G'] \cup [\mathfrak{p}_G, \infty)$ and $\alpha>d|\frac{1}{p}-\frac{1}{2}|-\tfrac{1}{2}$. Then we have
\begin{align*}
    \|S_R^{\alpha}(\mathcal{L})f \|_{L^p(G)} &\leq C \|f\|_{L^p(G)} .
\end{align*}
    
\end{corollary}

It is important to note that, in both of the above two results Theorem \ref{Theorem: Analogue of Niedorf theorem} and Corollary \ref{Corollary: Lp boundedness of Bochner-Riesz on mativier}, the smoothness parameters $s$ and $\alpha$ respectively, are sharp and can not be decreased further, see \cite{Niedorf_Metivier_group_2023}, \cite{Martini_Muller_Golo_Spectral_Multiplier_Lower_Regularity_2023}.

\subsection{Maximal spectral multipliers and sharp maximal Bochner-Riesz means on M\'etivier groups}
Here we will discuss about the maximal version of the corresponding spectral multiplier theorem discussed in subsection \ref{Subsection: sharp spectral multiplier}. First let us start with the definition of some relevant norms of a multiplier. For $s>1/2$, let $L^2_s(\mathbb{R}^+)$ denote the completion of all compactly supported smooth functions on $(0, \infty)$ under the norm
\begin{align*}
    \|F\|_{L^2_s(\mathbb{R}^+)} &= \left( \int_0^{\infty} \left|\lambda^{s+1} \frac{d^{s}}{d\lambda^{s}}\left(\frac{F(\lambda)}{\lambda} \right) \right|^2 \, \frac{d\lambda}{\lambda} \right)^{1/2} ,
\end{align*}
where $\left(\frac{d^{s}h}{d\lambda^{s}}\right)^{\widehat{}}(\xi) = (-i \xi)^{s} \widehat{h}(\xi)$.

Using Mellin transform one can identify the space $L^2_s(\mathbb{R}^+)$ with the usual fractional Sobolev space $L^2_s(\mathbb{R})$ under the map $\lambda \mapsto \exp \lambda$ (\cite[p. 53]{Carbery_Maximal_Radial_Fourier_Mult_1985}), that is, we have
\begin{align*}
    \|F\|_{L^2_s(\mathbb{R}^+)} &= \|F \circ \exp\|_{L^2_{s}(\mathbb{R})} .
\end{align*}
In 1985, Carbery \cite{Carbery_Maximal_Radial_Fourier_Mult_1985} first studied the maximal operator associated with the radial Fourier multipliers on $\mathbb{R}^n$ and obtained the following pointwise inequality: Let $F: \mathbb{R} \to \mathbb{C}$ be a bounded Borel function, then for $s>1/2$
\begin{align}
\label{Pointwise inequality for maximal Euclidean}
    \sup_{R>0} |F(R^{-1}\Delta)f(x) | &\leq C \, \|F\|_{L^2_s(\mathbb{R}^+)} \, \mathfrak{S}^{\alpha}f(x) ,
\end{align}
where $\Delta$ denotes the Laplacian on $\mathbb{R}^n$. From the above pointwise estimate and the $L^p$-boundedness of $\mathfrak{S}^{\alpha}$ (see Theorem \ref{Theorem: Boundedness of Euclidean square function}, Theorem \ref{Theorem: Stein square function for p less that 2 case}) one can obtain the $L^p$-boundedness of the corresponding maximal operator of $F(R^{-1}\Delta)$. For other related results, see \cite{Dappa_Trebels_Maximal_Fourier_Mult_1985}, \cite{Rubio_Maximal_Fourier_1986}, \cite{Deleaval_Kriegler_Maximal_Hormander_2023}, \cite{Chen_Lin_Wu_Yan_Maximal_Hormander_multiplier_JFA}, \cite{Gan_Oh_Wu_New_Bound_higher_dimension_2025} and references therein.

Now let us consider the maximal operator associated to the spectral multiplier for the sub-Laplacian $\mathcal{L}$ on M\'etivier groups. For any bounded Borel function $F: \mathbb{R} \to \mathbb{C}$, we define the maximal spectral multiplier by
\begin{align}
\label{definition of maximal spectral multipliers}
    F^{*}(\mathcal{L})f(x,u) &= \sup_{R>0} \left|F(R^{-1}\mathcal{L})f(x,u) \right| .
\end{align}

On any stratified Lie groups Mauceri and Meda \cite{Mauceri_Meda_Multipliers_Stratified_groups_1990} proved the following result.

\begin{theorem}\cite[Theorem 2.6]{Mauceri_Meda_Multipliers_Stratified_groups_1990}
Let $G$ be any stratified Lie groups with $Q$ be the homogeneous dimension. Also let $F$ be a function defined on $(0, \infty)$ such that it satisfies
\begin{align}
\label{Conition for maximal spectral multiplier}
    \sum_{k \in \mathbb{Z}} \|\varphi F(2^k \cdot)\|_{L^2_{\alpha}}^2 = D < \infty \quad \quad \text{for some} \quad \alpha>0 ,
\end{align}
where $\varphi$ is a nontrivial smooth function with compact support. Then
\begin{align*}
    \|F^{*}(\mathcal{L})f\|_{L^p(G)} &\leq C  \, D \, \|f\|_{L^p(G)} ,
\end{align*}
whenever
\begin{enumerate}
    \item $\alpha> Q(\frac{1}{p}-\frac{1}{2})+\frac{1}{2}$ and $1<p\leq 2$.
    \item $\alpha>(Q-1)(\frac{1}{2}-\frac{1}{p})+\frac{1}{2}$ and $2 \leq p \leq \infty$.
\end{enumerate}

\end{theorem}

Before we are going to state our result of this section, let us mention that, (\cite[Lemma 3.2]{Chen_Lin_Wu_Yan_Maximal_Hormander_multiplier_JFA}) if a function $F$ defined on $(0, \infty)$ satisfies \eqref{Conition for maximal spectral multiplier}, then $F \in L^2_{\alpha}(\mathbb{R}^+)$ and
\begin{align*}
    \|F\|_{L^2_{\alpha}(\mathbb{R}^+)} &\leq C_{\alpha} \sum_{k \in \mathbb{Z}} \|\varphi F(2^k \cdot)\|_{L^2_{\alpha}}^2 .
\end{align*}
The following is our main result regarding the $L^p$-boundedness of the maximal spectral multipliers on M\'etivier groups, which improves the result of Mauceri and Meda \cite[Theorem 2.6]{Mauceri_Meda_Multipliers_Stratified_groups_1990} in this case by improving the range of the smoothness parameter.

\begin{theorem}
\label{Theorem: Maximal spectral multiplier}
Let $F: \mathbb{R} \to \mathbb{C}$ be a bounded Borel function. Then we have
\begin{align}
\label{Estimate for maximal spectral mutliplier}
    \|F^{*}(\mathcal{L})f\|_{L^p(G)} &\leq C \, \|F\|_{L^2_{\alpha}(\mathbb{R}^+)} \, \|f\|_{L^p(G)} ,
\end{align}
whenever
\begin{enumerate}
    \item $\alpha>\max\{d(\frac{1}{2}-\frac{1}{p}), \frac{1}{2}\}= \alpha_d(p) + \frac{1}{2}$ and $\mathfrak{p}_G \leq p<\infty$ or $p=2$ .
    \item $\alpha>d(\frac{1}{p}-\frac{1}{2}) + \frac{1}{2}$ and $1<p <2$.
\end{enumerate}
    
\end{theorem}

Actually, analogous to \eqref{Pointwise inequality for maximal Euclidean}, we have also obtained a pointwise inequality for the sub-Laplacian $\mathcal{L}$ on M\'etivier groups (see \eqref{After application of Holder final estimate}), in order to prove the above theorem. Note that interpolating between the two instances $p=2$ and $p=\mathfrak{p}_G$ in (1) of Theorem \ref{Theorem: Maximal spectral multiplier} one can also obtain the estimate \eqref{Estimate for maximal spectral mutliplier} for the range $2\leq p \leq \mathfrak{p}_G$.

\medskip
Next we will discuss about sharp maximal Bochner-Riesz means on M\'etivier groups, which we obtain as an application of the above Theorem \ref{Theorem: Maximal spectral multiplier}. Since $L^p$-boundedness of the maximal Bochner-Riesz mean on Euclidean setup has been attracted lot of attention over the past several years, let us first start with it. On $\mathbb{R}^n$, associated to the Bochner-Riesz operator $S_R^{\alpha}$ (see \eqref{Definition: Euclidean Bochner-Riesz means}), the maximal Bochner-Riesz mean is denoted by $S^{\alpha}_{*}$ and defined by
\begin{align*}
    S^{\alpha}_{*}f(x) &= \sup_{R>0}|S_R^{\alpha}f(x)| .
\end{align*}
As in the case of Bochner-Riesz means, many works have been devoted to determine the optimal range of $\alpha$ such that $S^{\alpha}_{*}$ is bounded on $L^p$-spaces. For $p>2$, it has been conjectured that $L^p$-boundedness of $S^{\alpha}_{*}$ holds on the same range as of $S_R^{\alpha}$, that is, for $p>2$ the following estimate
\begin{align}
\label{Maximal Bochner-Riesz conjecture}
    \|S^{\alpha}_{*}f\|_{L^p(\mathbb{R}^n)} \leq C \|f\|_{L^p(\mathbb{R}^n)} \quad \text{holds if and only if} \quad \alpha> \alpha_n(p) .
\end{align}
It is known that sharp bounds for $\mathfrak{S}^{\alpha}$ implies the sharp estimates of $S^{\alpha}_{*}$ \cite{Lee_Square_Function_to_Bochner_Riesz_2018}. For $n=2$, the conjecture was settled down by Carbery \cite{Carbery_Maximal_Bochner_Riesz_R2_case_1983}. For $n\geq 3$, partial results are known, for instance, in the range $p\geq \frac{2(n+1)}{n-1}$, one can see \cite{Christ_almost_Everywhere_Bochner_Riesz_1985, Seeger_Maximal_Fourier_Mult_1986} and consequently, this range was further extended by Lee in \cite{Sanghyuk_Lee_Improved_Bochner-Riesz_2004} and the latest result can be found in \cite{Lee_Square_Function_to_Bochner_Riesz_2018} and \cite{Gan_Oh_Wu_New_Bound_higher_dimension_2025}. 


\medskip
While for $1<p<2$, the range of $\alpha$ where $S^{\alpha}_{*}$ is bounded on $L^p$ is not same as $S_R^{\alpha}$, and additional restriction is needed and Tao \cite{Tao_Bochner_Riesz_P_less_2_case_1998} conjectured that for $1<p<2$
\begin{align*}
    \|S^{\alpha}_{*}f\|_{L^p(\mathbb{R}^n)} \leq C \|f\|_{L^p(\mathbb{R}^n)} \quad \text{ holds if and only if} \quad \alpha> \tfrac{2n-1}{2p}-\tfrac{n}{2} .
\end{align*}
In \cite{Tao_Bochner_Riesz_P_less_2_case_1998} Tao proves that the above condition is necessary. Note that $S^{\alpha}_{*}$ is bounded of $L^p(\mathbb{R}^n)$ for $1< p < 2$ whenever $\alpha>(n-1)(\frac{1}{p}-\frac{1}{2})$, which follows from the real interpolation (see \cite{Gan_Wu_Maximal_Bochner_Riesz_2025}). When $n=2$, in \cite{Tao_Maximal_Bochner_Riesz_Plane_2002}, Tao first proved some new estimates for $S^{\alpha}_{*}$ and later it was further improved by Li and Wu \cite{Li_Wu_Maximal_Bochner_Riesz_2020}. Recently, Gan and Wu \cite{Gan_Wu_Maximal_Bochner_Riesz_2025} again improved this result for the case $n=2$ and $3$.

\medskip
Instead of asking for the $L^p$-boundedness of $S^{\alpha}_{*}$ one can also ask for a slightly weaker property, that is, what is the optimal range of $\alpha$, so that $S_R^{\alpha}f \to f$ almost everywhere (a.e.) as $R \to \infty$ for $f \in L^p(\mathbb{R}^n)$. When $p \geq 2$, it has been conjectured that almost everywhere convergence of $S_R^{\alpha}$ holds on the same range in which the $L^p$-boundedness of $S_*^{\alpha}$ hold. Although the maximal Bochner-Riesz conjecture remains open for $n \geq 3$ (see \eqref{Maximal Bochner-Riesz conjecture}), almost everywhere convergence problem for $p>2$ and $\alpha>\alpha_n(p)$ was completely solved by Carbery, Rubio de Francia and Vega \cite{Carbery_Rubio_Vega_Almost_everywhere_1988}. However, for $1<p<2$, by Stein's maximal principle a.e. convergence of $S_R^{\alpha}$ for $f \in L^p(\mathbb{R}^n)$ is equivalent to the $L^p \to L^{p, \infty}$ estimate for the corresponding maximal operator $S_*^{\alpha}$, see \cite{Tao_Maximal_Bochner_Riesz_Plane_2002}, \cite{Li_Wu_Maximal_Bochner_Riesz_2020}, \cite{Gan_Wu_Maximal_Bochner_Riesz_2025}.

\medskip
For the sub-Laplacian $\mathcal{L}$ on M\'etivier group, corresponding to the Bochner-Riesz operator $S_R^{\alpha}(\mathcal{L})$ (see \eqref{Definition: Bochner-Riesz means}), the maximal Bochner-Riesz operator is defined by
\begin{align*}
    S^{\alpha}_{*}(\mathcal{L})f(x,u) &:= \sup_{R>0} |S_R^{\alpha}(\mathcal{L})f(x,u)| .
\end{align*}
In this paper our another aim is to prove the $L^p$-boundedness of $S^{\alpha}_{*}(\mathcal{L})$ with smoothness index $\alpha$ expressed in terms of the topological dimension $d$ of $G$. Outside the Euclidean setup, the maximal Bochner-Riesz for elliptic operators on space of homogeneous type also has been studied in \cite{Chen_Lee_Sikora_Yan_Bochner_Elliptic_2020} and on a stratified group in \cite[Corollary 2.8]{Mauceri_Meda_Multipliers_Stratified_groups_1990}. Since M\'etivier groups are fall into these classes, if one try to use that results for the sub-Laplacian $\mathcal{L}$ on M\'etivier group $G$, one may only get result in terms of the homogeneous dimension $Q$ of $G$. On the other hand, in \cite{Horwich_Martini_Pointwise_Heisenberg_type_2021}, Horwich and Martini proved the following result:

\begin{theorem}\cite[Theorem 1.2]{Horwich_Martini_Pointwise_Heisenberg_type_2021}, \cite{Martini_Lie_groups_Polynomial_Growth_2012}
\label{Theorem: Maximal Bochner-Riesz Horwich Martini}
    Let $G$ be a M\'etivier group and $2 \leq p \leq \infty$. Then whenever $\alpha>(d-1)(\tfrac{1}{2}-\tfrac{1}{p})$ we have
    \begin{align*}
        \|S^{\alpha}_{*}(\mathcal{L})f\|_{L^p(G)} &\leq C \|f\|_{L^p(G)} .
    \end{align*}
\end{theorem}

Actually in \cite{Horwich_Martini_Pointwise_Heisenberg_type_2021}, they proved the above result for sub-Laplacian on any stratified Lie groups and under the assumption the Mikhlin-H\"ormander theorem hold with threshold $\varsigma_{+}(\mathcal{L})$, see \cite{Martini_Muller_spectral_mult_two_step_2016}. On the other hand, for M\'etivier groups it is known that $\varsigma_{+}(\mathcal{L})=d/2$ \cite{Martini_Lie_groups_Polynomial_Growth_2012}. Hence, combining this two results one get the above Theorem \ref{Theorem: Maximal Bochner-Riesz Horwich Martini}.

\medskip
As a corollary of Theorem \ref{Theorem: Maximal spectral multiplier} (for more details, see \cite[proof of Corollary 2.8]{Mauceri_Meda_Multipliers_Stratified_groups_1990}), we have the following result and which can be seen as an analogue of the Euclidean result obtained by Christ \cite{Christ_almost_Everywhere_Bochner_Riesz_1985} and Seeger \cite{Seeger_Maximal_Fourier_Mult_1986}, where the Euclidean dimension $n$ in the smoothness threshold is replaced by the topological dimension $d$ of the M\'etivier group.

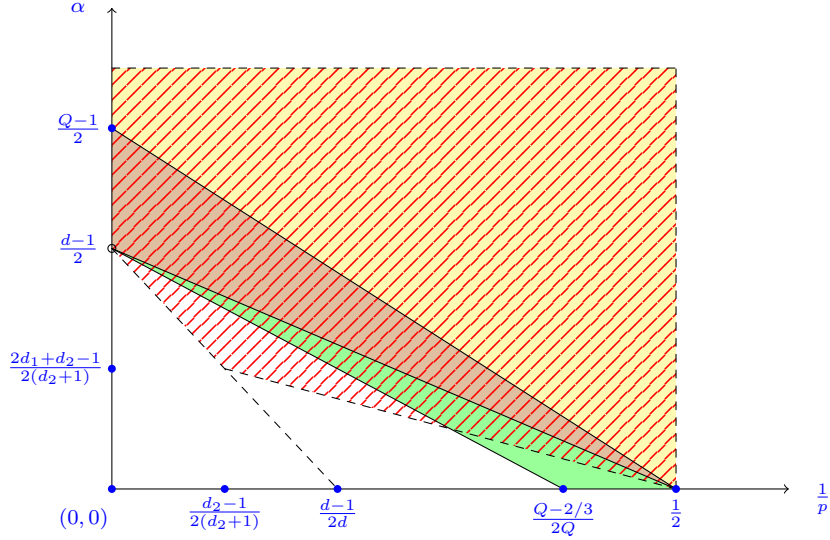
\begin{figure}[!ht]
\begin{centering}
\definecolor{qqqqff}{rgb}{0,0,1}
\begin{tikzpicture}[line cap=round,line join=round,x=1.5cm,y=0.8cm]
\fill[color=yellow,fill=yellow,fill opacity=0.3] (0,7) -- (5,7) -- (5,0) -- (0,6) -- cycle;
\fill[color=brown,fill=brown,fill opacity=0.5] (0,4) -- (5,0) -- (0,6) -- cycle;
\fill[color=green,fill=green,fill opacity=0.4] (0,4) -- (4,0) -- (5,0) -- cycle;
\fill[pattern={Lines[angle=45,distance=4pt,line width=0.6pt]},
    pattern color=red] (0,4) -- (1,2) -- (5,0) -- (5,7) -- (0,7) -- cycle;
\draw [dashed](0,4)-- (2,0);
\draw [dashed](1,2)-- (5,0);
\draw [-] (0,4) -- (4,0);
\draw [-] (0,6) -- (5,0);
\draw [-] (0,4) -- (5,0);
\draw [->] (0,0) -- (6,0);
\draw [->] (0,0) -- (0,8);
\draw [dashed](5,0)-- (5,7);
\draw [dashed](0,7)-- (5,7);
\begin{scriptsize}
\fill [color=qqqqff] (0,0) circle (1.5pt);
\fill [color=qqqqff] (0,6) circle (1.5pt);
\fill [color=qqqqff] (5,0) circle (1.5pt);
\fill [color=qqqqff] (4,0) circle (1.5pt);
\fill [color=qqqqff] (2,0) circle (1.5pt);
\fill [color=qqqqff] (1,0) circle (1.5pt);
\fill [color=qqqqff] (0,2) circle (1.5pt);
\draw [] (0,4) circle (1.5pt);
\draw[color=qqqqff] (-0.3,6) node {$\frac{Q-1}{2}$};
\draw[color=qqqqff] (-0.3,4) node {$\frac{d-1}{2}$};
\draw[color=qqqqff] (-0.5,2) node {$\frac{2d_1+d_2-1}{2(d_2+1)}$};
\draw[color=qqqqff] (5.0,-0.4) node {$\frac{1}{2}$};
\draw[color=qqqqff] (1,-0.4) node {$\frac{d_2-1}{2(d_2+1)}$};
\draw[color=qqqqff] (2,-0.4) node {$\frac{d-1}{2d}$};
\draw[color=qqqqff] (4,-0.5) node {$\frac{Q-2/3}{2Q}$};
\draw[color=qqqqff] (-0.25,-0.5) node {$(0,0)$};
\draw[color=qqqqff] (6.3,-0.16) node {$\frac{1}{p}$};
\draw[color=qqqqff] (-0.3,8) node {$\alpha$};
\end{scriptsize}
\end{tikzpicture}
        \caption{$S^{\alpha}_{*}(\mathcal{L})$ \textbf{boundedness}: Yellow region \cite[Corollary 2.8]{Mauceri_Meda_Multipliers_Stratified_groups_1990} (any stratified Lie groups), Brown region (Theorem \ref{Theorem: Maximal Bochner-Riesz Horwich Martini}) (\cite[Theorem 1.2]{Horwich_Martini_Pointwise_Heisenberg_type_2021} + \cite{Martini_Lie_groups_Polynomial_Growth_2012}) (M\'etivier groups), Red shaded region Corollary \ref{Corollary: Boundedness of linear maximal Bochner-Riesz} (Our result on M\'etivier groups); \\ \textbf{Pointwise convergence}: Yellow region \cite[Corollary 2.8]{Mauceri_Meda_Multipliers_Stratified_groups_1990} (any stratified Lie groups), Brown + Green region \cite[Theorem 1.1]{Horwich_Martini_Pointwise_Heisenberg_type_2021} (Heisenberg-type groups), Red shaded region Corollary \ref{Corollary: Pointwise convergence of maximal Bochner_Riesz} (Our result on M\'etivier groups).}
        \label{Figure: Linear pointwise convergence}
\end{centering}        
\end{figure}

\begin{corollary}
\label{Corollary: Boundedness of linear maximal Bochner-Riesz}
We have
\begin{align*}
    \|S^{\alpha}_{*}(\mathcal{L})f\|_{L^p(G)} &\leq C \|f\|_{L^p(G)} ,
\end{align*}
whenever
\begin{enumerate}
    \item $\alpha>\alpha_d(p)$ and $\mathfrak{p}_G \leq p<\infty$ or $p=2$.
    \item $\alpha>d(\frac{1}{p}-\frac{1}{2})$ and $1<p <2$.
\end{enumerate}
    
\end{corollary}

One can also interpolate between the two instances $p=2$ and $p=\mathfrak{p}_G$ of (1) in the above Corollary \ref{Corollary: Boundedness of linear maximal Bochner-Riesz}, to obtain the exact range of $\alpha$ when $p \in [2, \mathfrak{p}_G]$. Note that Corollary \ref{Corollary: Boundedness of linear maximal Bochner-Riesz} improves the result of \cite[Corollary 2.8]{Mauceri_Meda_Multipliers_Stratified_groups_1990} and part (1) of Corollary \ref{Corollary: Boundedness of linear maximal Bochner-Riesz} also improves Theorem \ref{Theorem: Maximal Bochner-Riesz Horwich Martini} (see Figure \ref{Figure: Linear pointwise convergence}). It is important to mention that for $\mathfrak{p}_G \leq p<\infty$ or $p=2$ our result is sharp. In fact, the $L^p$-boundedness of $S^{\alpha}_{*}(\mathcal{L})$ implies the $L^p$-boundedness of $S^{\alpha}_R(\mathcal{L})$, which was already known to be sharp, see \cite{Niedorf_Metivier_group_2023}, \cite{Martini_Muller_Golo_Spectral_Multiplier_Lower_Regularity_2023} (see also Corollary \ref{Corollary: Lp boundedness of Bochner-Riesz on mativier}). Therefore in view of the above Corollary \ref{Corollary: Boundedness of linear maximal Bochner-Riesz}, it is natural to make the following conjecture:

\begin{conjecture}
    Let $G$ be a M\'etivier group. Then for $2 \leq p\leq \infty$, whenever $\alpha>\alpha_d(p)$
    \begin{align*}
    \|S^{\alpha}_{*}(\mathcal{L})f\|_{L^p(G)} &\leq C \|f\|_{L^p(G)} .
\end{align*}
\end{conjecture}

\medskip
Next we discuss about to the pointwise almost everywhere convergence of $S_R^{\alpha}(\mathcal{L})$. This problem has been already studied on Heisenberg groups \cite{Gorges_Muller_Pointwise_Heisenberg_2002} and on Heisenberg-type groups \cite{Horwich_Martini_Pointwise_Heisenberg_type_2021}. As a consequence of the $L^p$-boundedness of the maximal Bochner-Riesz operator on M\'etivier groups we have the following pointwise almost everywhere convergence result for $S_R^{\alpha}(\mathcal{L})$.

\begin{corollary}
\label{Corollary: Pointwise convergence of maximal Bochner_Riesz}
Let $\mathfrak{p}_G \leq p<\infty$ or $p=2$ and $\alpha>\alpha_d(p)$. Then for $f \in L^p(G)$,
\begin{align*}
    \lim_{R \to \infty} S_R^{\alpha}(\mathcal{L})f(x,u) = f(x,u) \quad \quad \text{a.e.} \quad (x,u) \in G .
\end{align*}
    
\end{corollary}

Note that our result Corollary \ref{Corollary: Pointwise convergence of maximal Bochner_Riesz} improves the previously known result of \cite[Theorem 1.1]{Horwich_Martini_Pointwise_Heisenberg_type_2021} for the Heisenberg-type groups in two direction: firstly our results applies not only on Heisenberg-type groups but also on M\'etivier groups, which are more general class than Heisenberg-type groups and secondly the range of $\alpha$ where the pointwise almost everywhere convergence of $S_R^{\alpha}(\mathcal{L})$ hold for large $p$, (for example when $p\geq \tfrac{2(d_2+1)}{d_2-1}$) are also improved. For more details, see Figure \ref{Figure: Linear pointwise convergence}. It is also important to remark that, both of our results Corollary \ref{Corollary: Boundedness of linear maximal Bochner-Riesz} and Corollary \ref{Corollary: Pointwise convergence of maximal Bochner_Riesz} are not applicable for the Heisenberg groups, since in this case the dimension of the center $d_2=1$ forces $\mathfrak{p}_G = \infty$, which is not included in our results, see (2) of Remark \ref{Remark: Combinind remark} for more details.

\subsection{Regularity estimate for the solution of fractional Schr\"odinger equation on M\'etivier groups}
For $a \in (0, \infty)$, consider the fractional Schr\"odinger equation:
\begin{align*}
    i \partial_s u + (-\Delta)^{a/2}u = 0 \quad \text{with} \quad u(\cdot, 0) = f .
\end{align*}
Then the solution is given by $u(x,s) = e^{i s (-\Delta)^{a/2}}f(x)$.

Note that when $a=2$ and $a=1$, this corresponds to the Schr\"odinger equation and wave equation respectively.

There are various studies regarding the mixed norm estimates of $u$, see \cite{Roger_Seeger_Maximal_Smoothing_Schrodinger_2010, Lee_Roger_Seeger_Improved_Stein_Square_Function_2012, Lee_Roger_Seeger_On_space_time_Schrodinger_2013, Lee_Roger_Seeger_Square_Function_Maximal_Fourier_Mult_2014} The following result gives $L^p(\mathbb{R}^n ; L^2(I))$ mixed norm regularity estimates for the solution of the fractional Schr\"odinger equation.

\begin{theorem}\cite[Theorem 1.4]{Lee_Roger_Seeger_Improved_Stein_Square_Function_2012}
\label{Theorem: Euclidean Lp L2 local smoothing}
Let $n \geq 2$, $a \in (0, \infty)$ and $I$ be a compact interval. Then whenever $p>2+4/n$ and $\beta \geq a \, \alpha_n(p)$ we have
\begin{align*}
    \left\|\left(\int_I |e^{i s (-\Delta)^{a/2}}f|^2 \, ds \right)^{1/2} \right\|_{L^p(\mathbb{R}^n)} &\leq C \|f\|_{L^{p}_{\beta}(\mathbb{R}^n)} .
\end{align*}
    
\end{theorem}

It is important to note that the smoothness index $\beta \geq a \, \alpha_n(p)$ in the above theorem is sharp except for the case $a=1$. Recently, for $n \geq 3$ the range of $p$ in the above Theorem \ref{Theorem: Euclidean Lp L2 local smoothing} was further improved, see \cite[Corollary 1.4]{Gan_Oh_Wu_New_Bound_higher_dimension_2025}. As a consequence of the above $L^p(\mathbb{R}^n ; L^2(I))$ mixed norm estimates, one can also obtain $L^p(\mathbb{R}^n ; L^q(I))$ mixed norm estimates for $q \geq 2$ (see \cite{Lee_Roger_Seeger_Improved_Stein_Square_Function_2012}, \cite{Gan_Oh_Wu_New_Bound_higher_dimension_2025}). In particular, if one takes $p=q$ in the $L^p(\mathbb{R}^n ; L^q(I))$ mixed norm estimates, one gets the local smoothing estimate for the fractional Schr\"odinger equation.

Let $\chi$ be a compactly supported non-negative smooth function supported in $(\frac{5}{8}, \frac{15}{8})$ such that
\begin{align*}
    \sum_{k \in \mathbb{Z}} \chi(2^{-k} \eta) = 1 , \quad \quad \eta>0 .
\end{align*}

In \cite{Lee_Roger_Seeger_Improved_Stein_Square_Function_2012}, it was shown that the above Theorem \ref{Theorem: Euclidean Lp L2 local smoothing} can be deduced from the following proposition.

\begin{proposition}\cite[Proposition 5.1]{Lee_Roger_Seeger_Improved_Stein_Square_Function_2012}
\label{Proposition: Lp l2 localized local smothing}
Let $n \geq 2$, $p > 2+4/n$, $a \in (0, \infty)$ and $I$ be a compact interval. Then for $k \geq 1$ and $\alpha \geq \alpha_n(p)$ we have
\begin{align*}
    \left\|\left( \int_I |e^{i s (-\Delta)^{a/2}} \chi(2^{-k}\sqrt{\Delta}) f|^2 \, ds \right)^{1/2} \right\|_{L^p(\mathbb{R}^n)} &\leq C 2^{k a \alpha} \|f\|_{L^p(\mathbb{R}^n)} .
\end{align*}
    
\end{proposition}

\medskip

Now analogous to the Euclidean setup, here we consider the fractional Schr\"odinger equation associated with the sub-Laplacian $\mathcal{L}$ on M\'etivier groups, which is defined as
\begin{align*}
    i \partial_s v((x,u),s) + \mathcal{L}^{a/2} v((x,u),s) = 0 \quad \text{with} \quad v((x,u), 0) = f(x,u) .
\end{align*}
Then the solution $v : G \times \mathbb{R} \to \mathbb{C}$ is given by
\begin{align*}
    v((x,u),s) &= e^{i s \mathcal{L}^{a/2}}f(x,u) .
\end{align*}

In particular, it we take $a=1$ and $a=2$, then the operator $e^{i s \mathcal{L}^{a/2}}$ corresponds to the wave and Schr\"odinger operators on the M\'etivier groups. When $a=1$, finding sharp threshold $\beta$ so that the wave operator $e^{i s \sqrt{\mathcal{L}}}$ is bounded from $L^p_{\beta}$ to $L^p$, has been attracted lots of attention over past several years, see for example, in case of Heisenberg group \cite{Muller_Stein_Wave_Equation_1999}, Heisenberg-type groups \cite{Muller_Seeger_Wave_Equation_Heisenberg_Type}, M\'etivier groups \cite{Martini_Muller_Wave_Metivier}. On the the hand, for $a \neq 1$ the sharp threshold is different from that of $a=1$ case. In this case, see \cite{Bramati_Ciatti_Green_Wright_Oscillating_Multiplier_2022}, \cite{Bui_Hong_Hu_Oscillating_Multiplier_JGA} for more details.

\medskip
Here we are interested in the $L^p(G; L^2(I))$ regularity estimate of the solution of the fractional Schr\"odinger equation associated with the sub-Laplacian on M\'etivier groups. This can be seen as vector-valued analogue of the operator considered in \cite{Bramati_Ciatti_Green_Wright_Oscillating_Multiplier_2022}. In case of M\'etivier groups, similar to Proposition \ref{Proposition: Lp l2 localized local smothing}, the following is our main result in this direction. To the best of authors knowledge, this type of estimates on M\'etivier groups has not been considered before.

\begin{theorem}
\label{Theorem: Lp to L2 local smoothing for Metivier}
Let $\mathfrak{p}_G \leq p < \infty$, $a \in (0, \infty)$ and $I$ be a compact interval. Then for $k \geq 1$ and $\alpha>\alpha_d(p)$ we have
\begin{align*}
    \left\|\left( \int_I |e^{i s \mathcal{L}^{a/2}} \chi(2^{-k} \sqrt{\mathcal{L}}) f|^2 \, ds \right)^{1/2} \right\|_{L^p(G)} &\leq C 2^{k a \alpha} \|f\|_{L^p(G)} .
\end{align*}
    
\end{theorem}

\subsection{Bilinear Bochner-Riesz means and its maximal version on M\'etivier groups}
In this subsection, we discuss about the bilinear analogue of the Bochner-Riesz means and its maximal version on M\'etivier groups. Before going to the bilinear Bochner-Riesz means on M\'etivier groups, let us start our discussion with the bilinear Bochner-Riesz means on Euclidean setup. While the classical Bochner-Riesz means has been studied extensively, its bilinear counterpart have only attracted attention in recent years, due to its more complexity arises from mixed summability. For $\alpha \geq 0$, $R>0$ and suitable functions $f, g$, the bilinear Bochner-Riesz operator on $\mathbb{R}^n$ is defined by
\begin{align}
\label{Definition: Euclidean bilinear Bochner-Riesz means}
    B_R^{\alpha}(f,g)(x) &= \int_{\mathbb{R}^n} \int_{\mathbb{R}^n} \left(1-\frac{|\xi|^2+|\eta|^2}{R^2} \right)_{+}^{\alpha} \widehat{f}(\xi) \, \widehat{g}(\eta)\, e^{2 \pi i x \cdot (\xi+\eta)} \, d\xi \, d\eta .
\end{align}

A central problem in the study of bilinear Bochner-Riesz means is to determine the sharp range of the parameter $\alpha \geq 0$ such that $B^{\alpha}_R$ is bounded from $L^{p_1}(\mathbb{R}^n) \times L^{p_2}(\mathbb{R}^n) \to L^p(\mathbb{R}^n)$ for $(p_1, p_2, p)$ satisfying $1\leq p_1, p_2 \leq \infty$ and the H\"older's relation $1/p=1/p_1 +1/p_2$. Due to the transference theory, the problem is also related to the convergence of the product of two $n$-dimensional Fourier series (see \cite{Bernicot_Grafakos_Song_Yan_Bilinear_Bochner_Riesz_2015}). When $\alpha=0$, the operator $B^{\alpha}_R$ is called bilinear ball multiplier operator. In $n=1$, Grafakos and Li \cite{Grafakos_Li_Disc_Multiplier_2006} proved that for $\alpha=0$, $B^{\alpha}_R$ is bounded from $L^{p_1}(\mathbb{R}^n) \times L^{p_2}(\mathbb{R}^n) \to L^p(\mathbb{R}^n)$ whenever $2\leq p_1, p_2, p' <\infty$, where $p'$ denotes the conjugate exponent of $p$. On contrary for $n \geq 2$, Diestel and Grafakos \cite{Diestel_Grafakos_Ball_multiplier_problem_2007} showed that for $\alpha=0$, $B^{\alpha}_R$ is not bounded if exactly one of $p_1, p_2, p'$ is less than $2$. For $\alpha>0$ and $n=1$, this problem has been well studied, see for example \cite{Bernicot_Germain_Bilinear_multilier_narrow_support_2013}, \cite[Theorem 4.1]{Bernicot_Grafakos_Song_Yan_Bilinear_Bochner_Riesz_2015}, \cite{Jotsaroop_Shrivastava_Maximal_Bochner_Riesz_2022}. For $n \geq 2$, the problem was first studied by Bernicot, Grafakos, Song and Yan in \cite{Bernicot_Grafakos_Song_Yan_Bilinear_Bochner_Riesz_2015}. Actually they proved the boundedness result for more general bilinear multiplier, which are compactly supported, radial in both variable and satisfies some smoothness condition. From this they deduce the boundedness of $B^{\alpha}_R$ for $(p_1, p_2,p)$ satisfies the H\"older's relation. But their result does not give the sharp range of $\alpha$ except at $(p_1,p_2,p)=(1,1,1/2)$ and $(2,2,1)$. Later in 2020, Liu and Wang \cite{Liu_Wang_Bilinear_Bochner_Riesz_Non_Banach_2020} further improved the result of \cite{Bernicot_Grafakos_Song_Yan_Bilinear_Bochner_Riesz_2015} in the non-Banach region, that is when $p<1$. On the other hand, when $p\geq 1$, Jeong, Lee and Vargas in \cite{Jeong_Lee_Vargas_Bilinear_Bochner_Riesz_2018} significantly improved the result of Bernicot et.al. \cite{Bernicot_Grafakos_Song_Yan_Bilinear_Bochner_Riesz_2015} for the range $p_1, p_2 \geq 2$. In order to describe the result of \cite{Jeong_Lee_Vargas_Bilinear_Bochner_Riesz_2018} in details, first we need to define some notation.

\medskip

For $\Upsilon>0$, $m \in \mathbb{N}$, we define 
\begin{align}
\label{Bilinear smoothness index}
\alpha_{*}(p_1, p_2, m, \Upsilon) =
    \left\{ \begin{array}{ll}
        \medskip
        \alpha_m(p_1) + \alpha_m(p_2) , & \quad \Upsilon \leq p_1, p_2 \leq \infty \hspace{1.6cm} [\text{Region}\  R_1(p_s)] \\ 
        \medskip
        \frac{2-2 p_1^{-1}-2 p_2^{-1}}{1-2 \Upsilon^{-1}}\alpha_m(\Upsilon) , & \quad 2\leq p_1, p_2 \leq \Upsilon \hspace{1.8cm} [\text{Region}\  R_2(p_s)] \\
        \medskip
        \alpha_m(p_2) + \alpha_m(\Upsilon) \frac{1-2 p_1^{-1}}{1-2 \Upsilon^{-1}} , & \quad 2 \leq p_1 < \Upsilon \leq p_2 \leq \infty \hspace{0.55cm} [\text{Region}\  R_3(p_s)] \\
        \medskip
        \alpha_m(p_1) + \alpha_m(\Upsilon) \frac{1-2 p_2^{-1}}{1-2 \Upsilon^{-1}} , & \quad 2 \leq p_2 < \Upsilon \leq p_1 \leq \infty \hspace{0.55cm} [\text{Region}\  R_3(p_s)]
    \end{array} \right.
\end{align}

\begin{theorem}\cite[Corollary 1.3]{Jeong_Lee_Vargas_Bilinear_Bochner_Riesz_2018}
\label{Theorem: Jeong_Lee_Vargas_Bilinear}
For $n \geq 2$, let
\begin{align*}
    p_0(n) = 2+\frac{12}{4n-6-k}, \  n \equiv k (\hspace{-4mm}\mod 3), \ k=0,1,2 ,\quad \text{and} \quad p_s  = \min\{p_0(n), 2(n+2)/n \} .
\end{align*}
Also let $2 \leq p_1, p_2 \leq \infty$ with $1/p=1/p_1 +1/p_2$. Then $B_R^{\alpha}$ is bounded from $L^{p_1}(\mathbb{R}^n) \times L^{p_2}(\mathbb{R}^n)$ to $L^{p}(\mathbb{R}^n)$ whenever $\alpha> \alpha_*(p_1, p_2, n, p_s)$.
    
\end{theorem}

Now let us turn our attention to the maximal analogue of the bilinear Bochner-Riesz means. On $\mathbb{R}^n$, corresponding to the bilinear Bochner-Riesz operator $B_R^{\alpha}$ (see \eqref{Definition: Euclidean bilinear Bochner-Riesz means}), associated maximal bilinear Bochner-Riesz operator is defined by
\begin{align*}
    B^{\alpha}_{*}(f,g)(x) &= \sup_{R>0}|B_R^{\alpha}(f,g)(x)| .
\end{align*}
Similar to the bilinear Bochner-Riesz operator, there are various studies regarding finding the optimal range of $\alpha$ such that $B^{\alpha}_{*}$ is bounded from $L^{p_1}(\mathbb{R}^n) \times L^{p_2}(\mathbb{R}^n)$ to $L^{p}(\mathbb{R}^n)$, where the triple $(p_1,p_2,p)$ satisfies the H\"older's relation $1/p_1+1/p_2=1/p$ and $1\leq p_1, p_2 \leq \infty$. This problem was first studied by Grafakos, He and Honz\'ik \cite{Grafakos-He_Honzik_Maximal_Bilinear_mult_2021} and they proved that boundedness of $B^{\alpha}_{*}$ holds at $(p_1,p_2,p)=(2,2,1)$ whenever $\alpha>\frac{2n+3}{4}$. After that, Jeong and Lee \cite{Jeong_Lee_MaximaL_Bilinear_Bochner_Riesz_2020} significantly improved the range of the parameter $\alpha$ when $p_1, p_2 \geq 2$. Recently, Jotsaroop and Shrivastava \cite{Jotsaroop_Shrivastava_Maximal_Bochner_Riesz_2022} further improved the result of Jeong and Lee, in fact they proved that for $p_1, p_2 \geq 2$, the boundedness of $B^{\alpha}_{*}$ holds with the same range of the smoothness parameter as of $B^{\alpha}_R$ (see Theorem \ref{Theorem: Jeong_Lee_Vargas_Bilinear}, Figure \ref{Figure: Figure for Bilinear Bochner-Riesz Euclidean}).

\begin{theorem}\cite[Theorem 2.1]{Jotsaroop_Shrivastava_Maximal_Bochner_Riesz_2022}
\label{Theorem: Euclidean maximal bilinear Bochner-Riesz}
Let $2 \leq p_1, p_2 \leq \infty$ with $1/p=1/p_1 +1/p_2$. Then $B^{\alpha}_{*}$ is bounded from $L^{p_1}(\mathbb{R}^n) \times L^{p_2}(\mathbb{R}^n)$ to $L^{p}(\mathbb{R}^n)$ whenever $\alpha> \alpha_*(p_1, p_2, n, p_s)$.
    
\end{theorem}

See Figure \ref{Figure: Figure for Bilinear Bochner-Riesz Euclidean} for more details on the range of $\alpha$ in Theorem \ref{Theorem: Jeong_Lee_Vargas_Bilinear} and \ref{Theorem: Euclidean maximal bilinear Bochner-Riesz} depending on the region $R_i(p_s)$ for $i=1,2,3$. In this paper, one of our main aim is to prove analogue of Theorem \ref{Theorem: Jeong_Lee_Vargas_Bilinear} and \ref{Theorem: Euclidean maximal bilinear Bochner-Riesz} for the case of M\'etivier groups.

\begin{figure}[!ht]
\begin{centering}
\definecolor{qqqqff}{rgb}{0,0,1}
\begin{tikzpicture}[line cap=round,line join=round,x=0.8cm,y=0.8cm]
\draw (6,6)-- (0,6);
\draw (6,0)-- (6,6);
\draw (0,4.5)-- (6,4.5);
\draw (4.5,0)-- (4.5,6);
\draw [->] (0,0) -- (8,0);
\draw [->] (0,0) -- (0,7.5);
\begin{scriptsize}
\fill [color=qqqqff] (0,0) circle (1.5pt);
\draw[color=qqqqff] (-1.3,0.3) node {$\alpha>2\alpha_n(\infty)$};
\fill [color=qqqqff] (0,6) circle (1.5pt);
\draw[color=qqqqff] (3.65,0.5) node {$\alpha > \alpha_n(p_{s})+\alpha_n(\infty)$};
\draw[color=qqqqff] (-2.7,4.6) node {$\alpha > \alpha_n(p_{s})+\alpha_n(\infty)$};
\fill [color=qqqqff] (6,6) circle (1.5pt);
\draw[color=qqqqff] (6.29,6.5) node {$\alpha>0$};
\fill [color=qqqqff] (6,0) circle (1.5pt);
\draw[color=qqqqff] (7.2,4.5) node {$\alpha>\alpha_n(p_{s})$};
\fill [color=qqqqff] (4.5,0) circle (1.5pt);
\draw[color=qqqqff] (7.2,0.5) node {$\alpha>\alpha_n(\infty)$};
\fill [color=qqqqff] (4.5,6) circle (1.5pt);
\draw[color=qqqqff] (4,6.5) node {$\alpha>\alpha_n(p_{s})$};
\fill [color=qqqqff] (0,4.5) circle (1.5pt);
\draw[color=qqqqff] (-1.7,6.2) node {$\alpha>\alpha_n(\infty)$};
\fill [color=qqqqff] (4.5,4.5) circle (1.5pt);
\draw[color=qqqqff] (3.3,4.1) node {$\alpha>2\alpha_n(p_s)$};
\fill [color=qqqqff] (6,4.5) circle (1.5pt);
\fill [color=qqqqff] (6,-0.5) circle (0pt);
\draw[color=qqqqff] (6.0,-0.38) node {$\frac{1}{2}$};
\fill [color=qqqqff] (3,-0.5) circle (0pt);
\draw[color=qqqqff] (4.5,-0.48) node {$\frac{1}{p_{s}}$};
\fill [color=qqqqff] (-0.25,-0.42) circle (0pt);
\draw[color=qqqqff] (-0.25,-0.5) node {$(0,0)$};
\fill [color=qqqqff] (-0.5,3) circle (0pt);
\draw[color=qqqqff] (-0.38,4.5) node {$\frac{1}{p_{s}}$};
\fill [color=qqqqff] (-0.5,6) circle (0pt);
\draw[color=qqqqff] (-0.38,6.0) node {$\frac{1}{2}$};
\fill [color=qqqqff] (8.5,-0.5) circle (0pt);
\draw[color=qqqqff] (8.5,-0.16) node {$\frac{1}{p_1}$};
\fill [color=qqqqff] (-0.5,8.5) circle (0pt);
\draw[color=qqqqff] (-0.58,7.5) node {$\frac{1}{p_2}$};
\draw[color=qqqqff] (2.3,2.1) node {$R_1(p_{s})$};
\draw[color=qqqqff] (5.25,2.1) node {$R_3(p_{s})$};
\draw[color=qqqqff] (2.3,5.25) node {$R_3(p_{s})$};
\draw[color=qqqqff] (5.25,5.25) node {$R_2(p_{s})$};
\end{scriptsize}
\end{tikzpicture}
        \caption{Here $\alpha>\alpha_*(p_1, p_2, n, p_s)$ represents the range of the parameter such that $B_R^{\alpha}$ and $B_*^{\alpha}$ are bounded from $L^{p_1}(\mathbb{R}^n) \times L^{p_2}(\mathbb{R}^n) $ to $ L^p(\mathbb{R}^n)$ (see Theorem \ref{Theorem: Jeong_Lee_Vargas_Bilinear} and Theorem \ref{Theorem: Euclidean maximal bilinear Bochner-Riesz}).}
        \label{Figure: Figure for Bilinear Bochner-Riesz Euclidean}
\end{centering}        
\end{figure}
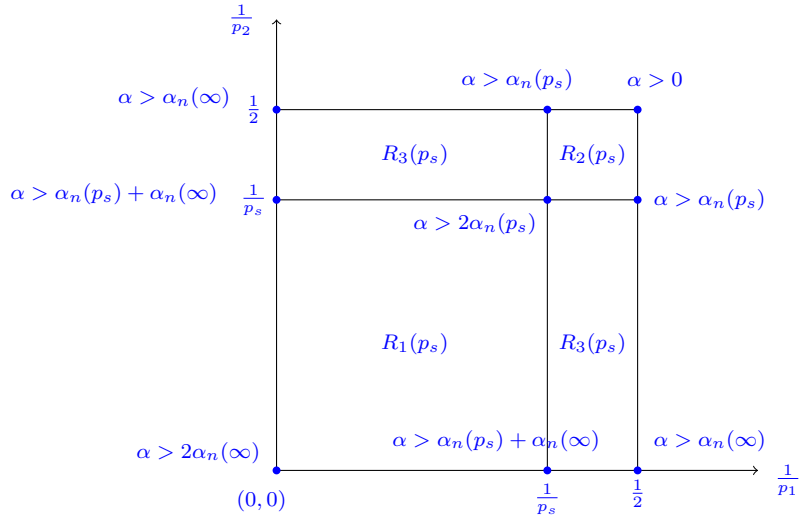

Now we turn our discussion towards the bilinear Bochner-Riesz means on M\'etivier groups. For $f, g \in \mathcal{S}(G)$ and $\alpha\geq 0$, $R>0$, the bilinear Bochner-Riesz means of order $\alpha$ associated with the sub-Laplacian $\mathcal{L}$ on M\'etivier group, denoted by $\mathcal{B}_R^{\alpha}$ and is defined by
\begin{align}
\label{Bilinear Bochner-Riesz menas}
    \mathcal{B}_R^{\alpha}(f,g)(x,u) &= \frac{1}{(2\pi)^{2 d_2}} \int_{\mathfrak{g}_{2,r}^{*}} \int_{\mathfrak{g}_{2,r}^{*}} e^{i \langle \lambda_1 + \lambda_2 , u \rangle} \sum_{\mathbf{k}_1, \mathbf{k}_2 \in \mathbb{N}^\Lambda} \left(1-\frac{\eta_{\mathbf{k}_1}^{\lambda_1}+ \eta_{\mathbf{k}_2}^{\lambda_2}}{R^2} \right)_{+}^{\alpha} \\
    &\nonumber \hspace{1cm} \left[f^{\lambda_1} \times_{\lambda_1} \varphi_{\mathbf{k}_1}^{\mathbf{b}^{\lambda_1}, \mathbf{r}_1}(R_{\lambda_1}^{-1}\cdot) \right](x) \left[g^{\lambda_2} \times_{\lambda_2} \varphi_{\mathbf{k}_2}^{\mathbf{b}^{\lambda_2}, \mathbf{r}_2}(R_{\lambda_2}^{-1}\cdot) \right](x) \,  d\lambda_1 \,  d\lambda_2 .
\end{align}

Analogous to the Euclidean case, it is natural to ask for the optimal range of the parameter $\alpha>0$, which may be possibly expressed in terms of the topological dimension $d=d_1+d_2$ of $G$, such that the operator $\mathcal{B}_R^{\alpha}$ extends to a bounded operator from $L^{p_1}(G) \times L^{p_2}(G)$ to $L^p(G)$ with $(p_1, p_2, p)$ satisfies the H\"older's relation $1/p=1/p_1+1/p_2$ and $1\leq p_1, p_2 \leq \infty$. This problem was recently studied by the author and his collaborators in \cite{Bagchi_Molla_Singh_Bilinear_Metivier} and they obtained the boundedness of $\mathcal{B}_R^{\alpha}$ for $1\leq p_1, p_2 \leq \infty$ with $1/p=1/p_1+1/p_2$. One of the main aim was there to prove the M\'etivier analogue of the result \cite{Bernicot_Grafakos_Song_Yan_Bilinear_Bochner_Riesz_2015} with smoothness parameter given in terms of the topological dimension of M\'etivier groups. In this paper, we significantly improve the result of \cite{Bagchi_Molla_Singh_Bilinear_Metivier} by finding better range of the smoothness parameter $\alpha$ when $p_1, p_2 \in [2,\infty)$. Since here we are only interested in the range $2\leq p_1, p_2 < \infty$, we state the result of \cite{Bagchi_Molla_Singh_Bilinear_Metivier} for that region only.

\begin{theorem}\cite[Theorem 1.2]{Bagchi_Molla_Singh_Bilinear_Metivier}
\label{Theorem: Bilinear Bochner-Riesz Metivier old result}
Let $2\leq p_1, p_2 \leq \infty$ with $1/p = 1/p_1 +1/p_2$. Then $\mathcal{B}_R^{\alpha}$ is bounded from $L^{p_1}(G) \times L^{p_2}(G)$ to $L^{p}(G)$, if $p_1, p_2, p$ and $\alpha> \alpha_d(p_1, p_2)$ satisfy one of the following conditions:
\begin{enumerate}
    \medskip
    \item (Region I) $2 \leq p_1, p_2 \leq \infty$, $1\leq p \leq 2$ and $\alpha_d(p_1, p_2)= (d-1)(1-\frac{1}{p})$.
    \medskip
    \item (Region II) $2 \leq p_1, p_2, p \leq \infty$ and $\alpha_d(p_1, p_2)= \frac{d-1}{2} + d(\frac{1}{2}-\frac{1}{p})$.
\end{enumerate}
\end{theorem}

\begin{figure}[!ht]
\begin{centering}
\definecolor{qqqqff}{rgb}{0,0,1}
\begin{tikzpicture}[line cap=round,line join=round,x=0.8cm,y=0.8cm]
\draw (6,6)-- (0,6);
\draw (6,0)-- (6,6);
\draw (0,6)-- (6,0);
\draw [dashed](0,4.5)-- (6,4.5);
\draw [dashed](4.5,0)-- (4.5,6);
\draw [->] (0,0) -- (8,0);
\draw [->] (0,0) -- (0,7.5);
\begin{scriptsize}
\fill [color=qqqqff] (0,0) circle (1.5pt);
\draw[color=qqqqff] (-1.3,0.3) node {$\alpha>d-\frac{1}{2}$};
\fill [color=qqqqff] (0,6) circle (1.5pt);
\fill [color=qqqqff] (6,6) circle (1.5pt);
\draw[color=qqqqff] (6.29,6.5) node {$\alpha>0$};
\fill [color=qqqqff] (6,0) circle (1.5pt);
\draw[color=qqqqff] (7,4.5) node {$\alpha>\frac{d-1}{d_2+1}$};
\fill [color=qqqqff] (4.5,0) circle (1.5pt);
\draw[color=qqqqff] (7,0.5) node {$\alpha>\frac{d-1}{2}$};
\fill [color=qqqqff] (4.5,6) circle (1.5pt);
\draw[color=qqqqff] (4,6.5) node {$\alpha>\frac{d-1}{d_2+1}$};
\fill [color=qqqqff] (0,4.5) circle (1.5pt);
\draw[color=qqqqff] (-1.5,6) node {$\alpha>\frac{d-1}{2}$};
\fill [color=qqqqff] (4.5,4.5) circle (1.5pt);
\draw[color=qqqqff] (3.3,0.5) node {$\alpha>\frac{d_2(d-1)+(3d-2)}{2(d_2+1)}$};
\draw[color=qqqqff] (-2.6,4.6) node {$\alpha>\frac{d_2(d-1)+(3d-2)}{2(d_2+1)}$};
\fill [color=qqqqff] (6,4.5) circle (1.5pt);
\draw[color=qqqqff] (3.5,4) node {$\alpha>\frac{2(d-1)}{d_2+1}$};
\fill [color=qqqqff] (6,-0.5) circle (0pt);
\draw[color=qqqqff] (6.0,-0.38) node {$\frac{1}{2}$};
\fill [color=qqqqff] (3,-0.5) circle (0pt);
\draw[color=qqqqff] (4.5,-0.48) node {$\frac{1}{p_{d_2}}$};
\fill [color=qqqqff] (-0.25,-0.42) circle (0pt);
\draw[color=qqqqff] (-0.25,-0.5) node {$(0,0)$};
\fill [color=qqqqff] (-0.5,3) circle (0pt);
\draw[color=qqqqff] (-0.38,4.5) node {$\frac{1}{p_{d_2}}$};
\fill [color=qqqqff] (-0.5,6) circle (0pt);
\draw[color=qqqqff] (-0.38,6.0) node {$\frac{1}{2}$};
\fill [color=qqqqff] (8.5,-0.5) circle (0pt);
\draw[color=qqqqff] (8.5,-0.16) node {$\frac{1}{p_1}$};
\fill [color=qqqqff] (-0.5,8.5) circle (0pt);
\draw[color=qqqqff] (-0.58,7.5) node {$\frac{1}{p_2}$};
\draw[color=qqqqff] (5.3,5.1) node {$I$};
\draw[color=qqqqff] (2.3,2.1) node {$II$};
\end{scriptsize}
\end{tikzpicture}
        \caption{Region $I =$ Upper half triangle and Region $II=$ Lower half triangle. Here $\alpha>\alpha_d(p_1, p_2)$ represents that $\mathcal{B}_R^{\alpha}$ is bounded on $L^{p_1}(G) \times L^{p_2}(G) \to L^p(G)$ for $\alpha>\alpha_d(p_1, p_2)$ (see Theorem \ref{Theorem: Bilinear Bochner-Riesz Metivier old result}).}
        \label{Figure: Old Metivier picture}
\end{centering}        
\end{figure}
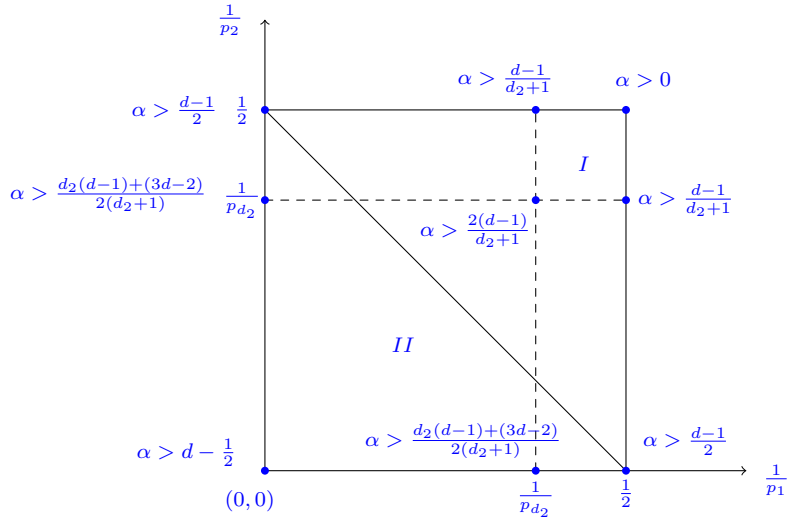

In a similar vain as in the Euclidean settings, here we consider the maximal bilinear Bochner-Riesz means associated with the sub-Laplacian $\mathcal{L}$ on M\'etivier group, defined as
\begin{align*}
    \mathcal{B}^{\alpha}_{*}(f,g)(x,u) &= \sup_{R>0}|\mathcal{B}_R^{\alpha}(f,g)(x,u)| .
\end{align*}
In this paper, one of our other main aim is to study the $L^{p_1}(G) \times L^{p_2}(G)$ to $L^p(G)$ boundedness of $\mathcal{B}^{\alpha}_{*}$ for $1\leq p_1, p_2 \leq \infty$ under the H\"older's relation $1/p=1/p_1+1/p_2$. Although boundedness of $\mathcal{B}_R^{\alpha}$ has been studied recently on M\'etivier groups (Theorem \ref{Theorem: Bilinear Bochner-Riesz Metivier old result}), but the boundedness of associated maximal operators has not been studied yet. In this direction, the following is our main result.

\begin{theorem}
\label{Theorem: Bilinear maximal on Metivier}
Let $2 \leq p_1, p_2 < \infty$ with $1/p=1/p_1 +1/p_2$. Then $\mathcal{B}^{\alpha}_{*}$ is bounded from $L^{p_1}(G) \times L^{p_2}(G)$ to $L^{p}(G)$ whenever $\alpha> \alpha_*(p_1, p_2, d, \mathfrak{p}_G)$.
    
\end{theorem}

As an immediate corollary of the above Theorem \ref{Theorem: Bilinear maximal on Metivier}, we get the boundedness of the bilinear Bochner-Riesz means $\mathcal{B}_R^{\alpha}$ on M\'etivier groups. Here we show that our result improves the smoothness threshold $\alpha_d(p_1, p_2)$ in Theorem \ref{Theorem: Bilinear Bochner-Riesz Metivier old result}. Following is our main result regarding the boundedness of the bilinear Bochner-Riesz means on M\'etivier groups.

\begin{corollary}
\label{Corollary: Bochner-Riesz main theorem Metivier}
Let $2 \leq p_1, p_2 < \infty$ with $1/p=1/p_1 +1/p_2$. Then $\mathcal{B}_R^{\alpha}$ is bounded from $L^{p_1}(G) \times L^{p_2}(G)$ to $L^{p}(G)$ whenever $\alpha> \alpha_*(p_1, p_2, d, \mathfrak{p}_G)$.
    
\end{corollary}

Since the above boundedness result on $\mathcal{B}_R^{\alpha}$ is obtained as a corollary of Theorem \ref{Theorem: Bilinear maximal on Metivier}, when $p_1, p_2 \geq 2$, the range of the smoothness parameter for the boundedness of $\mathcal{B}^{\alpha}_{*}$ (Theorem \ref{Theorem: Bilinear maximal on Metivier}) is exactly same as of the boundedness of $\mathcal{B}_R^{\alpha}$ (Corollary \ref{Corollary: Bochner-Riesz main theorem Metivier}). However, using the ideas of \cite{Jeong_Lee_Vargas_Bilinear_Bochner_Riesz_2018} and the bounds of the discrete square functions obtained in Proposition \ref{Proposition: Discrete square function estimate} one can give a direct proof of Corollary \ref{Corollary: Bochner-Riesz main theorem Metivier}.

\begin{figure}[!ht]
\begin{centering}
\definecolor{qqqqff}{rgb}{0,0,1}
\begin{tikzpicture}[line cap=round,line join=round,x=0.7cm,y=0.7cm]
\draw (6,6)-- (0,6);
\draw (6,0)-- (6,6);
\draw (0,4.5)-- (6,4.5);
\draw (4.5,0)-- (4.5,6);
\draw [->] (6,0) -- (8,0);
\draw [->] (0,6) -- (0,7.5);
\draw [dashed](0,0)-- (6,0);
\draw [dashed](0,0)-- (0,6);
\begin{scriptsize}
\draw [] (0,0) circle (1.5pt);
\draw[color=qqqqff] (-1.3,0.3) node {$\alpha>d-1$};
\draw [] (0,6) circle (1.5pt);
\draw[color=qqqqff] (3.45,0.5) node {$\alpha > \alpha_d(p_{d_2})+\alpha_d(\infty)$};
\draw[color=qqqqff] (-0.4,3.7) node {$\alpha > \alpha_d(p_{d_2})+\alpha_d(\infty)$};
\fill [color=qqqqff] (6,6) circle (1.5pt);
\draw[color=qqqqff] (6.29,6.5) node {$\alpha>0$};
\draw [] (6,0) circle (1.5pt);
\draw[color=qqqqff] (7.5,4.5) node {$\alpha>\alpha_d(p_{d_2})$};
\draw [] (4.5,0) circle (1.5pt);
\draw[color=qqqqff] (7,0.5) node {$\alpha>\frac{d-1}{2}$};
\fill [color=qqqqff] (4.5,6) circle (1.5pt);
\draw[color=qqqqff] (4,6.5) node {$\alpha>\alpha_d(p_{d_2})$};
\draw [] (0,4.5) circle (1.5pt);
\draw[color=qqqqff] (-1.5,5.9) node {$\alpha>\frac{d-1}{2}$};
\fill [color=qqqqff] (4.5,4.5) circle (1.5pt);
\draw[color=qqqqff] (3.8,4.1) node {$\alpha>2\alpha_d(p_{d_2})$};
\fill [color=qqqqff] (6,4.5) circle (1.5pt);
\fill [color=qqqqff] (6,-0.5) circle (0pt);
\draw[color=qqqqff] (6.0,-0.38) node {$\frac{1}{2}$};
\fill [color=qqqqff] (3,-0.5) circle (0pt);
\draw[color=qqqqff] (4.5,-0.48) node {$\frac{1}{p_{d_2}}$};
\fill [color=qqqqff] (-0.25,-0.42) circle (0pt);
\draw[color=qqqqff] (-0.25,-0.5) node {$(0,0)$};
\fill [color=qqqqff] (-0.5,3) circle (0pt);
\draw[color=qqqqff] (-0.38,4.5) node {$\frac{1}{p_{d_2}}$};
\fill [color=qqqqff] (-0.5,6) circle (0pt);
\draw[color=qqqqff] (-0.38,6.0) node {$\frac{1}{2}$};
\fill [color=qqqqff] (8.5,-0.5) circle (0pt);
\draw[color=qqqqff] (8.5,-0.16) node {$\frac{1}{p_1}$};
\fill [color=qqqqff] (-0.5,8.5) circle (0pt);
\draw[color=qqqqff] (-0.58,7.5) node {$\frac{1}{p_2}$};
\draw[color=qqqqff] (2.3,2.1) node {$R_1(p_{d_2})$};
\draw[color=qqqqff] (5.25,2.1) node {$R_3(p_{d_2})$};
\draw[color=qqqqff] (2.3,5.25) node {$R_3(p_{d_2})$};
\draw[color=qqqqff] (5.25,5.25) node {$R_2(p_{d_2})$};
\end{scriptsize}
\end{tikzpicture}
\begin{tikzpicture}[line cap=round,line join=round,x=0.7cm,y=0.7cm]
\fill[color=black,fill=black,fill opacity=0.3] (0,7) -- (5,7) -- (5,0) -- (2.5,2.5) -- (0,6) -- cycle;
\fill[color=red,fill=red,fill opacity=0.3] (0,6) -- (2.5,2.5) -- (5,0) -- (3.5,0.8) -- (0,5) -- cycle;
\draw [dashed](0,6)-- (2.5,2.5);
\draw [dashed](2.5,2.5)-- (5,0);
\draw [dashed](0,5)-- (3.5,0.8);
\draw [dashed](3.5,0.8)-- (5,0);
\draw [->] (0,0) -- (6,0);
\draw [->] (0,0) -- (0,7.5);
\draw [dashed](5,0)-- (5,7);
\draw [dashed](0,7)-- (5,7);
\begin{scriptsize}
\fill [color=qqqqff] (0,0) circle (1.5pt);
\fill [color=qqqqff] (0,6) circle (1.5pt);
\fill [color=qqqqff] (5,0) circle (1.5pt);
\fill [color=qqqqff] (2.5,0) circle (1.5pt);
\fill [color=qqqqff] (0,2.5) circle (1.5pt);
\draw [] (0,5) circle (1.5pt);
\fill [color=qqqqff] (3.5,0.8) circle (1.5pt);
\fill [color=qqqqff] (2.5,2.5) circle (1.5pt);
\fill [color=qqqqff] (0,0.8) circle (1.5pt);
\fill [color=qqqqff] (3.5,0) circle (1.5pt);
\draw[color=qqqqff] (-0.8,6) node {$d-\frac{1}{2}$};
\draw[color=qqqqff] (-0.8,5) node {$d-1$};
\draw[color=qqqqff] (-0.7,2.5) node {$\frac{d-1}{2}$};
\draw[color=qqqqff] (-1.1,0.8) node {$2\alpha_d(p_{d_2})$};
\draw[color=qqqqff] (5.0,-0.4) node {$\frac{1}{2}$};
\draw[color=qqqqff] (2.5,-0.4) node {$\frac{1}{4}$};
\draw[color=qqqqff] (3.5,-0.5) node {$\frac{1}{p_{d_2}}$};
\draw[color=qqqqff] (-0.25,-0.5) node {$(0,0)$};
\draw[color=qqqqff] (6.5,-0.16) node {$\frac{1}{p_0}$};
\draw[color=qqqqff] (-0.58,7.5) node {$\alpha$};
\end{scriptsize}
\end{tikzpicture}
        \caption{Left hand side picture represents that $\mathcal{B}_*^{\alpha}$ and $\mathcal{B}_R^{\alpha}$ are bounded on $L^{p_1}(G) \times L^{p_2}(G) \to L^p(G)$ for $\alpha>\alpha_*(p_1, p_2, d, \mathfrak{p}_G)$ (see Theorem \ref{Theorem: Bilinear maximal on Metivier} and Corollary \ref{Corollary: Bochner-Riesz main theorem Metivier}). Compare this picture with Figure \ref{Figure: Old Metivier picture}. Right hand side picture tells about the boundedness of $\mathcal{B}_R^{\alpha}$ on $L^{p_0}(G) \times L^{p_0}(G) \to L^{p_0/2}(G)$. Black region denotes the previously known result (see Theorem \ref{Theorem: Bilinear Bochner-Riesz Metivier old result}), while red region denote the extended range proved in Corollary \ref{Corollary: Bochner-Riesz main theorem Metivier}.}
        \label{Figure: New Metivier picture}
\end{centering}        
\end{figure}
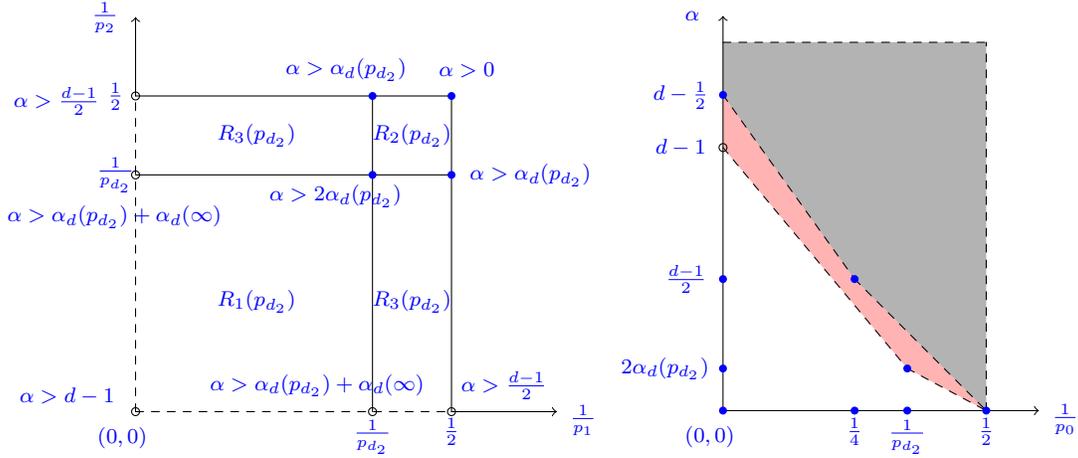

Note that, Corollary \ref{Corollary: Bochner-Riesz main theorem Metivier} provides improved estimate for the range of smoothness parameter $\alpha_d(p_1, p_2)$ over Theorem \ref{Theorem: Bilinear Bochner-Riesz Metivier old result} for the case of M\'etivier and Heisenberg-type groups (compare the range of $\alpha$ between the Figure \ref{Figure: Old Metivier picture} and left side picture of Figure \ref{Figure: New Metivier picture}). This improvement can also seen if one consider $L^{p_0}(G) \times L^{p_0}(G)$ to $L^{p_0/2}(G)$-boundedness of $\mathcal{B}_R^{\alpha}$, see right side picture of Figure \ref{Figure: New Metivier picture}. Also Theorem \ref{Theorem: Bilinear maximal on Metivier} and Corollary \ref{Corollary: Bochner-Riesz main theorem Metivier} can be seen as an analogue of Theorem \ref{Theorem: Euclidean maximal bilinear Bochner-Riesz} and Theorem \ref{Theorem: Jeong_Lee_Vargas_Bilinear} respectively, where the Euclidean dimension $n$ in $\mathbb{R}^n$ is replaced by the topological dimension $d$ of the underlying group $G$ and the number $p_s$ is replaced by $\mathfrak{p}_G$ in case of M\'etivier or Heisenberg-type groups.

\subsection{Bilinear Bochner-Riesz square function on M\'etivier groups}
Analogous to the Stein's square function $\mathfrak{S}^{\alpha}(\mathcal{L})$ \eqref{Stein square function for subLaplacian}, corresponding to the bilinear Bochner-Riesz operator (see \eqref{Definition: Euclidean bilinear Bochner-Riesz means}), for $\alpha\geq 0$, the bilinear Bochner-Riesz square function of order $\alpha$ on $\mathbb{R}^n$ is defined by
\begin{align}
\label{Definition of Bilinear Bochner-Riesz square function}
    \mathscr{G}^{\alpha}(f,g)(x) &= \left( \int_0^{\infty} \left| \frac{\partial}{\partial t} B_t^{\alpha+1}(f,g)(x) \right|^2 t \, dt \right)^{1/2} .
\end{align}
For $p_1, p_2 \geq 2$, the $L^{p_1}(\mathbb{R}^n) \times L^{p_2}(\mathbb{R}^n)$ to $L^p(\mathbb{R}^n)$-boundedness of $\mathscr{G}^{\alpha}$ has been recently studied in \cite{Choudhary_Jotsaroop_Shrivastava_Shuin_Bilinear_square_function}. There they have also shown connection of the bilinear Bochner-Riesz square function with the generalized bilinear Bochner-Riesz means, bilinear fractional Schr\"odinger operators and bilinear radial multipliers. Regarding the boundedness of $\mathscr{G}^{\alpha}$ they obtained the following result.

\begin{theorem}\cite[Theorem 2.2]{Choudhary_Jotsaroop_Shrivastava_Shuin_Bilinear_square_function}
Let $2 \leq p_1, p_2 \leq \infty$ with $1/p=1/p_1 +1/p_2$. Then $\mathscr{G}^{\alpha}$ is bounded from $L^{p_1}(\mathbb{R}^n) \times L^{p_2}(\mathbb{R}^n)$ to $L^{p}(\mathbb{R}^n)$ whenever $\alpha> \alpha_*(p_1, p_2, n, p_s)$.
    
\end{theorem}

Here we also consider bilinear Bochner-Riesz square function on M\'etivier groups. Corresponding to $\mathcal{B}_R^{\alpha}$ \eqref{Bilinear Bochner-Riesz menas}, for $\alpha \geq 0$ the bilinear Bochner-Riesz square function of order $\alpha$ on M\'etivier group is defined by
\begin{align*}
    \mathscr{G}^{\alpha}(\mathcal{L})(f,g)(x,u) &= \left( \int_0^{\infty} \left| \frac{\partial}{\partial t} \mathcal{B}_t^{\alpha+1}(f,g)(x, u) \right|^2 t \, dt \right)^{1/2} .
\end{align*}

The following theorem is our main result regarding the $L^{p_1}(G) \times L^{p_2}(G) \to L^p(G)$ boundedness of the bilinear Bochner-Riesz square function on M\'etivier groups.

\begin{theorem}
\label{Theorem: Bilinear Stein square function estimate}
Let $2 \leq p_1, p_2 < \infty$ with $1/p=1/p_1 +1/p_2$. Then $\mathscr{G}^{\alpha}(\mathcal{L})$ is bounded from $L^{p_1}(G) \times L^{p_2}(G)$ to $L^{p}(G)$ whenever $\alpha> \alpha_*(p_1, p_2, d, \mathfrak{p}_G)$.
    
\end{theorem}

\subsection{Maximal bilinear spectral multipliers on M\'etivier groups}
Similar to \eqref{definition of maximal spectral multipliers}, in this subsection we consider the maximal version of some particular bilinear spectral multipliers. Let $F : \mathbb{R} \to \mathbb{C}$ be a bounded Borel function. Then for $R>0$, we define the following bilinear spectral multiplier given by
\begin{align*}
    & F\left(\frac{\mathcal{L}_{bi}}{R} \right)(f, g)(x,u) = \frac{1}{(2\pi)^{2 d_2}} \int_{\mathfrak{g}_{2,r}^{*}} \int_{\mathfrak{g}_{2,r}^{*}} e^{i \langle \lambda_1 +\lambda_2 , u \rangle} \sum_{\mathbf{k}_1, \mathbf{k}_2 \in \mathbb{N}^{\Lambda}} F\left(\frac{\eta_{\mathbf{k}_1}^{\lambda_1}+ \eta_{\mathbf{k}_2}^{\lambda_2}}{R} \right) \\
    &\nonumber \hspace{3cm} \times \left[f^{\lambda_1} \times_{\lambda_1} \varphi_{\mathbf{k}_1}^{\mathbf{b}^{\lambda_1}, \mathbf{r}_1}(R_{\lambda_1}^{-1}\cdot) \right](x) \left[g^{\lambda_2} \times_{\lambda_2} \varphi_{\mathbf{k}_2}^{\mathbf{b}^{\lambda_2}, \mathbf{r}_2}(R_{\lambda_2}^{-1}\cdot) \right](x) \,  d\lambda_1 \,  d\lambda_2 .
\end{align*}

\medskip
Corresponding to this bilinear spectral multiplier we define the associated maximal bilinear spectral multiplier as
\begin{align}
\label{Bilinear spectral multiplier definition}
    F^{*}(\mathcal{L}_{bi})(f, g)(x,u) &= \sup_{R>0} \left|F\left(\frac{\mathcal{L}_{bi}}{R} \right)(f, g)(x,u) \right| .
\end{align}

Analogous to Theorem \ref{Theorem: Maximal spectral multiplier} in the linear case, here we are interested in the $L^{p_1}(G) \times L^{p_2}(G) \to L^p(G)$ boundedness of \eqref{Bilinear spectral multiplier definition}. In fact, we have the following boundedness result for the above defined maximal bilinear spectral multiplier \eqref{Bilinear spectral multiplier definition}.

\begin{theorem}
\label{Theorem: Bilinear maximal spectral multiplier}
Let $2 \leq p_1, p_2 < \infty$ with $1/p=1/p_1 +1/p_2$ and $F: \mathbb{R} \to \mathbb{C}$ be a bounded Borel function. Then whenever $\alpha> \alpha_*(p_1, p_2, d, \mathfrak{p}_G)+1$ we have
\begin{align*}
    \|F^{*}(\mathcal{L}_{bi})(f, g)\|_{L^p(G)} &\leq C \, \|F\|_{L^2_s(\mathbb{R}^+)} \, \|f\|_{L^{p_1}(G)} \|g\|_{L^{p_2}(G)} .
\end{align*}
    
\end{theorem}

\medskip
Let use make few remarks regarding what we have discussed till now.
\begin{remark}
\label{Remark: Combinind remark}
\begin{enumerate}
    \item It is important to note that in all of the our results about linear and bilinear multipliers and its maximal versions, Theorem \ref{Theorem: Stein square estimate}, Theorem \ref{Theorem: Stein square function for p less that 2 case}, Theorem \ref{Theorem: Analogue of Niedorf theorem}, Corollary \ref{Corollary: Lp boundedness of Bochner-Riesz on mativier}, Theorem \ref{Theorem: Maximal spectral multiplier}, Corollary \ref{Corollary: Boundedness of linear maximal Bochner-Riesz}, Corollary \ref{Corollary: Pointwise convergence of maximal Bochner_Riesz}, Theorem \ref{Theorem: Lp to L2 local smoothing for Metivier}, Theorem \ref{Theorem: Bilinear maximal on Metivier}, Corollary \ref{Corollary: Bochner-Riesz main theorem Metivier}, Theorem \ref{Theorem: Bilinear Stein square function estimate} and Theorem \ref{Theorem: Bilinear maximal spectral multiplier} the smoothness threshold $\alpha_d(p)$ or $\alpha_*(p_1, p_2, d, \mathfrak{p}_G)$ has been expressed in terms of $d$, the topological dimension rather than $Q$, the homogeneous dimension of the underlying group $G$. For stratified Lie groups with step bigger than two, $Q$ is always strictly bigger than $d$. As discussed in the Subsection \ref{Subsection: stein square function on Euclidean}, obtaining multiplier results with smoothness parameter expressed in terms of $d$ has been attracted a lot of attention after the pioneer work of M\"uller-Stein \cite{Muller_Stein_Spectral_Multiplier_Heisenberg_Related_groups_1994} and Hebisch \cite{Hebisch_Spectral_Multiplier_Heisenberg_1993} and subsequently by many authors, see \cite{Cowling_Klima_Sikora_Kohn_Laplacian_On_Sphere_2011}, \cite{Martini_Lie_groups_Polynomial_Growth_2012}, \cite{Martini_Muller_New_Class_Two_Step_Stratified_Groups_2014}, \cite{Bagchi_Molla_Singh_Bilinear_Metivier}, \cite{Bagchi_Molla_Singh_Bilinear_Bochner_Riesz_Grushin}, \cite{Molla_Singh_Bochner_Riesz_Commutators_Grushin} and references therein. One can also ask similar type of problems for other of sub-elliptic operators. In particular for Grushin operators we are working on this problem and it will appear in our upcoming paper.
    \medskip
    \item In all of our results the group $G$ has to be either Heisenberg-type groups or M\'etivier groups but not Heisenberg group. Note that if dimension of the center $\mathfrak{g}_2$ is one, that is, if $d_2 =1$, then we have $\mathfrak{p}_G' = p_{d_2}' = 1$. Since our results are stated with either $1< p \leq \mathfrak{p}_G'$ or $\mathfrak{p}_G \leq
    p< \infty$, therefore $G$ can not be Heisenberg group. This is mainly due to the fact that there is no good restriction theory available on Heisenberg group, the only non-trivial restriction hold on Heisenberg group is $L^1 \to L^{\infty}$, which is because it has one dimensional center, see \cite{Muller_restriction_Heisenberg_group_1990}.
    \medskip
    \item One of the key observation of this paper is the use of strong Hardy-Littlewood maximal functions $\mathcal{M}_{r}^{|\cdot|}$ on M\'etivier groups (see Subsection \ref{Subsection: Variant of Hardy-littlewood}) in the proof of Lemma \ref{Lemma: Weighted L2 estimate for square function}. Due to this crucial lemma, in the proof of $L^p$-boundedness of the local square function with localized frequency $\mathfrak{S}_{\delta, loc}^{\phi}(\mathcal{L})$ (Proposition \ref{Proposition: Local Square function estimate}) we are able to obtain the smoothness threshold in terms of the topological dimension $d$ of the underlying group. This in turns implies the $L^p$-boundedness of the Stein's square function $\mathfrak{S}^{\alpha}(\mathcal{L})$ associated with the Bochner-Riesz means on M\'etivier group (Theorem \ref{Theorem: Stein square estimate}) with again the range of the smoothness parameter $\alpha$ is expressed in terms of the topological dimension $d$. Consequently, as applications we also get all the required results with smoothness parameter achieving in terms of the topological dimension $d$ of $G$.
\end{enumerate}
    
\end{remark}

\medskip

\noindent \textbf{Notational conventions:}
We use the standard notation throughout the paper.
\begin{itemize}
    \item Denote $\mathbb{N}_0 = \{0,1,2, \ldots \}$.
    \item We use the letter $C$ to indicate a positive constant and independent of the main parameters of the estimate, but may differ at each occurrence.
    \item We use the notation $f \lesssim g$ to indicate $f \leq Cg$ for some $C > 0$, and whenever $f \lesssim g\lesssim f$, we shall write $f \sim g$. We sometimes write $f \lesssim_{\epsilon} g$ to indicate that the explicit constant $C$ may depend on the parameter $\epsilon$.
    \item For any ball $ B \subseteq \mathbb{R}^n$, we write $\Bar{B}$, $|B|$ and $B^c$ to denote the closure of the ball $B$, Lebesgue measure of the ball $B$ and complement of the ball $B$ respectively.
    \item For any Lebesgue measurable subset $E$ of $\mathbb{R}^d$, we denote $\chi_E$ to be the characteristic function of the set $E$.
    \item We use the standard notation $L^p_s(\mathbb{R}^n)$ to denote the Sobolev spaces of order $s \in \mathbb{R}$ on $\mathbb{R}^n$ such that $\|f\|_{L^p_s(\mathbb{R}^n)} = \|((1+|\cdot|^2)^{s/2} \hat{f})^{\Check{}}\|_{L^p(\mathbb{R}^n)}$.
    \item We denote the identity element of the group $G$ as $0:=(0,0)$.
    \item For two functions $f,g \in \mathcal{S}(G)$, the space of all Schwartz class function on $G$ (where we identified $G \cong \mathbb{R}^d$) the group convolution of $f$ and $g$ is given by
\begin{align*}
    f*g(x,u) &= \int_G f(x',u') g((x',u')^{-1}(x,u)) \ d(x',u'), \quad \text{for} \quad (x,u) \in G.
\end{align*}
\item For a multiplier function $F$, we always denote $\mathcal{K}_{F(\mathcal{L})}$ to be the convolution kernel of the corresponding spectral multiplier operator $F(\mathcal{L})$.
\end{itemize}

\medskip
\noindent \textbf{Structure of the article:}
Our article is organized as follows:
\begin{itemize}
    \item In Section \ref{Section: Preliminaries} first we carry out some preliminary results on M\'etivier groups \eqref{Subsection: Priliminaries on Metivier groups} and then discuss about $(L^p, L^2)$ restriction-type estimates for the joint functional of $\mathcal{L}$ and $T$ \eqref{Subsection: Lp L2 restriction type estimates}, some variant of the Hardy-Littlewood maximal functions on $G$ \eqref{Subsection: Variant of Hardy-littlewood}, pointwise and $L^2$-kernel estimates on M\'etivier groups \eqref{Subsection: Kernel estimate}, criterion for weak-type $(1,1)$ and $L^p$-boundedness \eqref{Subsection: Littlewood-Paley} and boundedness of bilinear spectral multipliers for some nice class of multipliers \eqref{Subsection: Bilinear spectral multiplier}.
    \medskip
    \item Section \ref{Section: Square function on Metivier groups} contains $L^p$-boundedness of the various square functions on M\'etivier groups and proof of Theorem \ref{Theorem: Stein square estimate} and Theorem \ref{Theorem: Stein square function for p less that 2 case}.
    \begin{itemize}
    \medskip
    \item[$\blacksquare$] In Subsection \ref{Subsection: Local Square function estimate} we prove the $L^p$-boundedness of the local square function with localized frequency \eqref{Definition: Square function} which is one of our main ingredients to prove various results of this paper and also as applications we prove discrete square function estimate (Proposition \ref{Proposition: Discrete square function estimate}) and some maximal operator boundedness (Proposition \ref{Prop: Maximal Lp norm estimates of localized}).
    \medskip
    \item[$\blacksquare$] In Subsection \ref{Subsection Third section}, we prove equivalence of the $L^p$-boundedness of the local and global square function with localized frequency (Proposition \ref{Proposition: Square function estimate}). In addition we also prove some $L^p$-boundedness result of maximal operator associated with some square function of Bochner-Riesz operator (Proposition \ref{Proposition: Maximal square function estimate}) and some square function estimate (Proposition \ref{Proposition: Lp boundedness of square function with bump}).
    \medskip
    \item[$\blacksquare$] In Subsection \ref{Subsection: Stein square function} we prove one of our first main result of this paper, $L^p$ boundedness of the Stein's square function associated with the sub-Laplacian on M\'etiver groups (Theorem \ref{Theorem: Stein square estimate} and Theorem \ref{Theorem: Stein square function for p less that 2 case}).
    \end{itemize}
    \medskip
    \item In Section \ref{Section: Applications of Steins square function}, we present several applications of the $L^p$-boundedness of the Stein's square function and other square functions defined on the previous section on M\'etivier groups.
    \begin{itemize}
    \medskip
    \item[$\blacksquare$] In Subsection \ref{Subsection: Sharp spectral multiplier} we prove a sharp version of Mikhlin-H\"ormander type multiplier theorem associated with the sub-Laplacian on M\'etivier groups (Theorem \ref{Theorem: Analogue of Niedorf theorem}). This result was recently proved in \cite{Niedorf_Metivier_group_2023}, but here we give a different proof then the previous one, which also gives sharp result.
    \medskip
    \item[$\blacksquare$] While in Subsection \ref{Subsection: Boundedness of maximal spectral multiplier} we show that as an application of the $L^p$-boundedness of the Stein's square function we get the $L^p$-boundedness of the maximal spectral multiplier operators (Theorem \ref{Theorem: Maximal spectral multiplier}) on the same range as of Stein's square function. We also prove sharp $L^p$-boundedness result for the maximal Bochner-Riesz operator on M\'etivier groups (Corollary \ref{Corollary: Boundedness of linear maximal Bochner-Riesz}), which also improves the work of \cite[Theorem 1.2]{Horwich_Martini_Pointwise_Heisenberg_type_2021} on M\'etivier groups. Consequently, this result also implies pointwise almost everywhere convergence result of the same (Corollary \ref{Corollary: Pointwise convergence of maximal Bochner_Riesz}), see Subsection \ref{Subsection: Boundedness of maximal spectral multiplier} for details.
    \medskip
    \item[$\blacksquare$] Subsection \ref{Subsection: Local smoothing} concerns about the proof of some regularity estimates for the solution of the fractional Schr\"odinger equations on M\'etivier groups (Theorem \ref{Theorem: Lp to L2 local smoothing for Metivier}).
    \medskip
    \item[$\blacksquare$] While Subsection \ref{Subsection: Boundedness of bilinear maximal Bochner-Riesz means} is devoted to the proof of boundedness of the maximal bilinear Bochner-Riesz operator on M\'etivier groups (Theorem \ref{Theorem: Bilinear maximal on Metivier}). Our result shows that for $p_1, p_2 \geq 2$, boundedness of bilinear Bochner-Riesz means (Corollary \ref{Corollary: Bochner-Riesz main theorem Metivier}) and its associated maximal operators (Theorem \ref{Theorem: Bilinear maximal on Metivier}) on M\'etivier groups holds in the same range and with the same smoothness threshold.
    \medskip
    \item[$\blacksquare$] In Subsection \ref{Subsection: Bilinear Bochner-Riesz square function} we prove the boundedness of the bilinear Bochner-Riesz square functions on M\'etivier groups (Theorem \ref{Theorem: Bilinear Stein square function estimate}), which is bilinear analogue the Stein's square function on M\'etivier groups (see \eqref{Stein square function for subLaplacian}).
    \medskip
    \item[$\blacksquare$] Finally, in Subsection \ref{Subsection: Bilinear maximal spectral multipliers} we prove a maximal bilinear spectral multiplier theorem on M\'etivier groups (Theorem \ref{Theorem: Bilinear maximal spectral multiplier}).
    \end{itemize}
\end{itemize}

\section{Preliminaries and useful estimates on M\'etivier groups}
\label{Section: Preliminaries}
In this section, first we collect some preliminary details and results about M\'etivier groups, and then discuss about some variant of Hardy-Littlewood maximal functions, $L^p$ to $L^2$ restriction type estimate for the joint functional of $\mathcal{L}$ and $T$, several kernel estimates associated to the spectral multiplier of the sub-Laplacian $\mathcal{L}$, some weak-type $(1,1)$ and $L^p$-boundedness criterion and boundedness of bilinear spectral multipliers for some nice class of multipliers. 

\subsection{Preliminaries on M\'etivier groups}
\label{Subsection: Priliminaries on Metivier groups}
Note that in this paper we have assumed $G$ to be a M\'etivier group, which is two-step stratified Lie group and via the exponential coordinates $G$ can be identified with its Lie algebra $\mathfrak{g}= \mathfrak{g}_1 \oplus \mathfrak{g}_2$. The group law is defined by
\begin{align*}
    (x_1,u_1) (x_2,u_2) &:= (x_1+x_2, u_1+u_2+\tfrac{1}{2}[x_1,x_1']), \quad \text{where} \quad x_1, x_2 \in \mathfrak{g}_1, \  u_1, u_2 \in \mathfrak{g}_2 .
\end{align*}

Also recall that we have chosen $\{X_1, \ldots, X_{d_1}, T_1, \ldots, T_{d_2}\}$ in such a way, so that they become an orthonormal basis of $\mathfrak{g}$. Then corresponding to the first-layer left-invariant vector fields, that is $X_1, \ldots, X_{d_1}$, the associated sub-Laplacian $\mathcal{L}$ is defined by 
\begin{align*}
    \mathcal{L} = -(X_1^2 + \cdots + X_{d_1}^2) .
\end{align*}
If we consider $\mathcal{L}$, $-i T_1, \ldots, -i T_{d_2}$, then they form a system of formally self-adjoint, pairwise commuting differential operators, therefore they admit a joint functional calculus.

On $G$, there is family of non-isotropic dilation defined by
\begin{align}
\label{Definition of dilation}
    \delta_R (x,u) &:= (Rx, R^2u) \quad \text{for} \quad (x, u) \in G \quad \text{and} \quad R>0 .
\end{align}
Also for any $(x,u) \in G$, if we define
\begin{align}
\label{Definition of homogeneous norm}
    \|(x,u)\| &:= (|x|^4 + |u|^2)^{1/4} ,
\end{align}
then this become a homogeneous norm with respect to the dilation $\delta_R$ for $R>0$.

Let $\varrho$ denote the sub-Riemannian distance on $G$ and for $R>0$ we denote $B((x,u), R)$ to be the sub-Riemannian ball centered at $(x,u)$ and of radius $R$. The volume of the ball  is given by
\begin{align*}
    |B((x,u),R)| \sim R^Q |B(0,1)|,
\end{align*}
where $|\cdot|$ denote the Lebesgue measure on $\mathbb{R}^{d_1} \times \mathbb{R}^{d_2}$ and $Q=d_1+2d_2$. Thus, the metric measure space $(G, \varrho, |\cdot|)$ is indeed a space of homogeneous type with homogeneous dimension $Q$. We also call $d=d_1+d_2$ to be the topological dimension of $G$.

Now let us mention one important fact about the balls $B(0, R)$ for $R>0$, which will play a very useful role in our proof later. There exists a constant $C>0$ such that
\begin{align}
\label{Decomposition of ball into Euclidean balls}
    B(0, R) \subseteq B^{|\cdot|}(0, C R) \times B^{|\cdot|}(0, C R^2) \subseteq \mathbb{R}^{d_1} \times \mathbb{R}^{d_2} ,
\end{align}
where $B^{|\cdot|}(a, R)$ denotes the ball of radius $R$ and centered at $a$ with respect to corresponding Euclidean distance (see \cite[Lemma 2.1]{Bagchi_Molla_Singh_Bilinear_Metivier}).

\subsection{\texorpdfstring{$L^p \to L^2$}{} restriction type estimates}
\label{Subsection: Lp L2 restriction type estimates}
Let us start with the definition of a discrete norm of function defined on $\mathbb{R}$. For any bounded Borel function $m : \mathbb{R} \to \mathbb{C}$ with $\supp m \subseteq [0,1]$, we define the discrete norm $m$ as
\begin{align*}
    \|m\|_{N, 2}= \left( \frac{1}{N} \sum_{k =1}^{N} \sup_{\lambda \in [\frac{k-1}{N}, \frac{k}{N}]} |m(\lambda)|^2 \right)^{1/2}, \quad \quad N>0.
\end{align*}
This discrete norm of $m$ is related to the $L^{\infty}$, $L^2$ and $L^2$-Sobolev norm via the following estimates (see \cite[eq. (1.7)]{Niedorf_Restriction_Stratified_group_2025}):
\begin{align}
\label{discrete norm is dominated by sup norm}
    \|m\|_{N,2} &\leq C\, \|m\|_{L^{\infty}} ,
\end{align}
and
\begin{align}
\label{equation: cowling-sikora norm relation}
 \|m\|_{L^2}\leq  \|m\|_{N,2} \leq C \left(\|m\|_{L^2} + N^{-s} \|m\|_{L^2_s} \right) \quad \quad \text{for} \quad s>1/2 .
\end{align}

Let $\Theta : \mathbb{R} \to [0,1]$ be a compactly supported even smooth function with $\supp \Theta \subseteq [R^2/2, 2R^2]$ for $R>0$ and satisfies 
\begin{align}
\label{Definition: Cutoff function theta preli}
    \sum_{M \in \mathbb{Z}} \Theta_{M}(\tau) =1 , \quad \text{where} \quad \Theta_{M}(\tau) = \Theta(2^{M} \tau) .
\end{align}
Also, let $F : \mathbb{R} \to \mathbb{C}$ be a bounded Borel function supported in $ [R/8, 8R]$. Then for $M \in \mathbb{Z}$, we define $F_M : \mathbb{R} \times \mathbb{R} \to \mathbb{C}$ by 
\begin{align}
\label{Introducing theta in multiplier preli}
    F_M(\kappa, \tau) = \left\{ \begin{array}{ll}
        F(\sqrt{\kappa}) \Theta( 2^{M} \tau), & \quad \kappa \geq 0 \\
        0, & \quad \text{otherwise} .
    \end{array} \right.
\end{align}
Let us set $T := (-(T_1^2 + \cdots + T_{d_2}^2))^{1/2}$. Then since $\mathcal{L}$, $-i T_1, \ldots, -i T_{d_2}$ admit a joint functional calculus, it follows that $\mathcal{L}$ and $T$ also admit joint functional calculus and hence from \eqref{Definition: Cutoff function theta preli} and \eqref{Introducing theta in multiplier preli} we have the following decomposition:
\begin{align}
\label{Writing F as sum of FM}
    F(\mathcal{\sqrt{L}}) f &= \sum_{M \in \mathbb{Z}} F_{M}(\mathcal{L}, T) f ,
\end{align}
where
\begin{align}
\label{Joint functional of L and T}
    F_{M}(\mathcal{L}, T)f(x,u) &= \frac{1}{(2\pi)^{d_2}} \int_{\mathfrak{g}_{2,r}^{*}} \sum_{\mathbf{k} \in \mathbb{N}^\Lambda} F_M(\eta_{\mathbf{k}}^{\lambda}, |\lambda|) \left[f^{\lambda} \times_{\lambda} \varphi_{\mathbf{k}}^{\mathbf{b}^{\lambda}, \mathbf{r}}(R_{\lambda}^{-1}\cdot) \right](x) \, e^{i \langle \lambda, u \rangle} \, d\lambda ,
\end{align}
with $\eta_{\mathbf{k}}^{\lambda} = \sum_{n=1}^{\Lambda} (2 k_n + r_n) b_n^{\lambda}$ (see \eqref{Definition: general spectral multipler on Metivier}).

Actually, in \eqref{Writing F as sum of FM} the sum over $M \in \mathbb{Z}$ is non-zero only for $M \in [-l_0, \infty)$ for some $l_0 \in \mathbb{N}$, which depends on $J_{\lambda}$ and the inner product on $\mathfrak{g}$. In fact, there exists $l_0 \in \mathbb{N}$ such that (see \cite[p. 14]{Molla_Singh_Commutator_Metivier_Arxiv})
\begin{align}
\label{In the Thete reducing M from Z to l}
    F(\mathcal{\sqrt{L}}) f &= \sum_{M =-l_0}^{\infty} F_{M}(\mathcal{L}, T) f .
\end{align}

Now we are in position to discuss about the $L^p \to L^2$ restriction type estimates for the operator $F_{M}(\mathcal{L}, T)$. This estimates will play an important role later in our proof of Proposition \ref{Proposition: Local Square function estimate}. Recall that $p_{d_2} = \frac{2(d_2+1)}{(d_2-1)}$.

\begin{proposition}
\label{Proposition: Truncated Restriction Estimate strong form}
Let $1 \leq p \leq p_{d_2}'$. Then for $\beta>1/2$ we have
\begin{align}
\label{Metivier group restriction}
    & \| F_M(\mathcal{L}, T) f\|_{L^2(G)} \\
    &\nonumber \leq C R^{Q(1/p - 1/2) } 2^{-M d_2(1/p-1/2)} \left(\|F(R \cdot)\|_{L^2} + 2^{-M \beta \theta_p} \|F(R \cdot)\|_{L^2}^{1-\theta_p}  \|F(R\cdot)\|_{L^2_{\beta}}^{\theta_p} \right) \|f\|_{L^p(G)} ,
\end{align}
where $\theta_p \in [0,1]$  satisfies $1/p = (1- \theta_p) + \theta_p/p_{d_2}'$.

\medskip
In particular, if $G$ is Heisenberg-type group, then for $1 \leq p \leq p_{d_2}'$ we have
\begin{align}
\label{Heisenberg type restriction}
    \| F_M(\mathcal{L}, T) f\|_{L^2(G) } &\leq C R^{Q(1/p - 1/2) } 2^{-M d_2(1/p-1/2)} \,\|F(R \cdot)\|_{L^2} \|f\|_{L^p(G)} .
\end{align}
\end{proposition}

\begin{proof}
From \cite[Proposition 3.1]{Molla_Singh_Commutator_Metivier_Arxiv} for $1 \leq p \leq p_{d_2}'$ we have
\begin{align*}
    \| F_M(\mathcal{L},T)f \|_{L^2(G)} & \leq  C R^{Q  (1/p -1/2)} 2^{-M d_2(1/p - 1/2)}\, \|\delta_{R}F\|_{L^2} ^{1- \theta_p}\,\|\delta_{R}F\|_{2^{M}, 2} ^{\theta_p}\,  \|f\|_{L^p (G)},
\end{align*}
where $\theta_p \in [0,1]$  satisfies $1/p = (1- \theta_p) + \theta_p/p_{d_2}'$.

Now using the fact \eqref{equation: cowling-sikora norm relation}, one get the required estimate \eqref{Metivier group restriction}. On the other hand, the estimate \eqref{Heisenberg type restriction} can be deduced similarly as above with the help of \cite[Theorem 3.2]{Niedorf_p_specific_Heisenberg_group_2024}.
\end{proof}

The following result follows from the above proposition, where instead of $L^2$ or Sobolev norm on the multiplier function $F$, we have $L^{\infty}$-norm of $F$.
\begin{corollary}
\label{Corollary: Truncated restriction equation weak form I}
Let $1 \leq p \leq p_{d_2}'$. Then
\begin{align}
\label{Restriction with L infinity norm}
    \| F_M(\mathcal{L}, T) f\|_{L^2(G) } \leq C R^{Q(1/p - 1/2) } 2^{-M d_2(1/p-1/2)} \|F(R\cdot)\|_{L^{\infty}} \, \|f\|_{L^p(G)} ,
\end{align}
and
\begin{align}
\label{Stein Tomas restriction with infinity norm on multiplier}
    \| F(\sqrt{\mathcal{L}}) f\|_{L^2(G) } \leq C R^{Q(1/p - 1/2) } \|F(R \cdot)\|_{L^{\infty}} \, \|f\|_{L^p(G)} .
\end{align}
\end{corollary}

\begin{proof}
In the proof of \eqref{Metivier group restriction} of Proposition \ref{Proposition: Truncated Restriction Estimate strong form}, instead of using \eqref{equation: cowling-sikora norm relation}, if we use \eqref{discrete norm is dominated by sup norm} and the fact $\|F(R \cdot)\|_{L^2} \leq C \|F(R \cdot)\|_{L^{\infty}}$, then we get the first estimate \eqref{Restriction with L infinity norm}.

The second estimate \eqref{Stein Tomas restriction with infinity norm on multiplier} can be deduced from the fact \eqref{In the Thete reducing M from Z to l} and first estimate \eqref{Restriction with L infinity norm}.
\end{proof}

\subsection{Strong Hardy-Littlewood maximal functions}
\label{Subsection: Variant of Hardy-littlewood}
In this subsection, we define a variant of the Hardy-Littlewood maximal function on $G$. These maximal functions will appear in the Lemma \ref{Lemma: Weighted L2 estimate for square function}, play a very crucial role in the proof of Proposition \ref{Proposition: Local Square function estimate}, to achieve the smoothness parameter in terms of the topological dimension $d$ of $G$.

Define the Hardy-Littlewood maximal function on $G$ by
\begin{align*}
    \mathcal{M}f(x,u) &= \sup_{B(0,R) \ni (x,u)} \frac{1}{|B(0,R)|} \int_{B(0,R)} |f(y,t)| \, d(y,t) .
\end{align*}
Recall that $B^{|\cdot|}(0, R) \subseteq \mathbb{R}^n$ denote the ball of radius $R$ and centered at $0$ with respect to the Euclidean distance. For a function $f \in L^1_{loc}(\mathbb{R}^{d_1} \times \mathbb{R}^{d_2})$, we define strong Hardy-Littlewood maximal function by
\begin{align*}
    \mathcal{M}^{|\cdot|}f(x,u) &= \sup_{B^{|\cdot|}(0,R) \times B^{|\cdot|}(0,S) \ni (x,u)} \frac{1}{|B^{|\cdot|}(0,R)| |B^{|\cdot|}(0,S)|} \int_{B^{|\cdot|}(0,R)} \int_{B^{|\cdot|}(0,S)} |f(y,t)|\, dy\, dt .
\end{align*}
For $r>0$, let us also define the following maximal functions:
\begin{align}
\label{Definition: Maximal with r factor}
    \mathcal{M}_{r}f(x,u) := \mathcal{M}(|f|^r)^{1/r}(x,u) \quad \text{and} \quad \mathcal{M}_{r}^{|\cdot|}f(x,u) := \mathcal{M}^{|\cdot|}(|f|^r)^{1/r}(x,u) .
\end{align}

The following lemma discusses about the boundedness of all these maximal functions.

\begin{lemma}
\label{Lemma: Boundedness of maximal and Euclidean maximal}
Let $1<p\leq \infty$. Then $\mathcal{M}$ and $\mathcal{M}^{|\cdot|}$ are bounded on $L^p(G)$. In particular, if $1\leq r<p \leq \infty$, then $\mathcal{M}_{r}$ and $\mathcal{M}_{r}^{|\cdot|}$ are bounded on $L^p(G)$.
    
\end{lemma}

\begin{proof}
For function $h_1$ and $h_2$ defined on $\mathbb{R}^{d_1}$ and $\mathbb{R}^{d_2}$ respectively, we first define two maximal functions as follows:
\begin{align*}
    \mathcal{M}^{|\cdot|_1}h_1(x) := \sup_{B^{|\cdot|}(0,R) \ni x} \frac{1}{|B^{|\cdot|}(0,R)|} \int_{B^{|\cdot|}(0,R)} |h_{1}(y)| \, dy ,
\end{align*}
and 
\begin{align*}
    \mathcal{M}^{|\cdot|_2}h_2(u) := \sup_{B^{|\cdot|}(0,R) \ni u} \frac{1}{|B^{|\cdot|}(0,R)|} \int_{B^{|\cdot|}(0,R)} |h_{2}(t)| \, dt .
\end{align*}
For each $y \in \mathbb{R}^{d_1}$, let us set $f_y : \mathbb{R}^{d_2} \to \mathbb{C}$ such that $ f_y(t) = f(y,t) $. Then one can easily see that
\begin{align}
\label{Euclidean maximal in terms of two maximal}
    \mathcal{M}^{|\cdot|}f(x,u) &\leq \mathcal{M}^{|\cdot|_1}(\mathcal{M}^{|\cdot|_2} f_{(\cdot)}(u))(x) .
\end{align}
For $1<p\leq \infty$, the $L^p$-boundedness of $\mathcal{M}$ is well known, see \cite{Stein_Harmonic_Analysis_1993}. Since $\mathcal{M}^{|\cdot|_1}$ and $\mathcal{M}^{|\cdot|_2}$ are just the usual Hardy-Littlewood maximal functions on $\mathbb{R}^{d_1}$ and $\mathbb{R}^{d_2}$ respectively, so they are also bounded on $L^p$ for $1<p\leq \infty$, (\cite{Stein_Harmonic_Analysis_1993}). On the other hand, using \eqref{Euclidean maximal in terms of two maximal} and the $L^p$-boundedness of $\mathcal{M}^{|\cdot|_1}$ and $\mathcal{M}^{|\cdot|_2}$, the $L^p$-boundedness of $\mathcal{M}^{|\cdot|}$ is immediate. Now since $1\leq r<p \leq\infty$, the boundedness of $\mathcal{M}_{r}$ and $\mathcal{M}_{r}^{|\cdot|}$ are easy consequences of the boundedness of corresponding $\mathcal{M}$ and $\mathcal{M}^{|\cdot|}$.
\end{proof}

\subsection{Kernel estimates on M\'etivier groups}
\label{Subsection: Kernel estimate}
In this subsection, we discuss various pointwise and $L^2$-kernel estimate of the convolution kernel associated to some spectral multiplier. Recall that, $\mathcal{K}_{F(\mathcal{L})}$ denote the convolution kernel corresponding to the spectral multiplier $F(\mathcal{L})$.

The following proposition sometimes call the \emph{weighted Plancherel estimates with respect to the first-layer weight} for the joint functional calculus of $\mathcal{L}$ and $T$.

\begin{proposition}\cite[Proposition 3.3]{Molla_Singh_Commutator_Metivier_Arxiv}
\label{Proposition : First layer weighted Plancherel}
Let $\mathcal{K}_{F_M(\mathcal{L}, T)}$ denote the convolution kernel of the operator $F_M(\mathcal{L}, T)$ as defined in \eqref{Joint functional of L and T}. Then for any $\alpha \geq 0$, we have 
\begin{align*}
 \int_G \left| |x|^{\alpha} \mathcal{K}_{F_M(\mathcal{L}, T)} (x, u) \right|^2 \, d(x, u) &\leq C\, 2^{M(2\alpha -d_2)}\, R^{Q- 2\alpha} \|F(R \cdot)\|_{L^2}^2.
\end{align*}
\end{proposition}

Now we discuss a weighted Plancherel estimates for the sub-Laplacian on M\'etivier groups \cite{Hebisch_Zienkiewicz_Product_Generalized_Heisenberg_1996}, \cite{Martini_Lie_groups_Polynomial_Growth_2012}. Such estimates play a crucial role in obtaining sharp spectral multiplier results.

\begin{proposition}
\label{Prop: Weighted Plancherel using weight and distance}
If $F: \mathbb{R} \to \mathbb{C}$ is a bounded Borel function supported in $[0, R]$, then for all $\beta \geq 0$, all $\epsilon, R>0$, all $0\leq \alpha <d_2/2$ we have
    \begin{align*}
        \left(\int_G | (1+R\|(x,u)\|)^{\beta} (1+R|x|)^{\alpha} \mathcal{K}_{F(\sqrt{\mathcal{L}})}(x,u)|^2 \, d(x,u) \right)^{1/2} &\leq C R^{Q/2} \|F(R \cdot)\|_{L^2_{\beta+\epsilon}} .
    \end{align*}
\end{proposition}

\begin{proof}
First note that from \cite[Theorem 6.1]{Martini_Sharp_Multiplier_Kohn_Laplacian_2017} for all $\beta \geq 0$, all $\epsilon, R>0$, and all $F: \mathbb{R} \to \mathbb{C}$ with $\supp F \subseteq [0, R]$ we have
\begin{align*}
    \left(\int_G | (1+R\|(x,u)\|)^{\beta} \mathcal{K}_{F(\sqrt{\mathcal{L}})}(x,u)|^2 \, d(x,u) \right)^{1/2} &\leq C R^{Q/2} \|F(R \cdot)\|_{L^{\infty}_{\beta+\epsilon}} .
\end{align*}
Since we have $|x| \leq \| (x,u) \|$, an application of Sobolev embedding and from the above estimate for all $\alpha \geq 0$ we obtain
\begin{align}
\label{Kernel estimate with distance}
    \left(\int_G | (1+R\|(x,u)\|)^{\beta} (1+R|x|)^{\alpha} \mathcal{K}_{F(\sqrt{\mathcal{L}})}(x,u)|^2 \, d(x,u) \right)^{1/2} &\leq C R^{Q/2} \|F(R \cdot)\|_{L^{2}_{\beta+\alpha+\frac{1}{2}+\epsilon}} .
\end{align}
On the other hand from \cite[Proposition 3.2]{Bagchi_Molla_Singh_Bilinear_Metivier} for all $\alpha \in [0,d_2/2)$, and all $F: \mathbb{R} \to \mathbb{C}$ with $\supp F \subseteq [0, R]$ yields 
\begin{align}
\label{kernel estimate with weight}
    \left(\int_G | (1+R|x|)^{\alpha} \mathcal{K}_{F(\sqrt{\mathcal{L}})}(x,u)|^2 \, d(x,u) \right)^{1/2} &\leq C R^{Q/2} \|F(R \cdot)\|_{L^{2}} .
\end{align}
Finally, interpolating \eqref{Kernel estimate with distance} and \eqref{kernel estimate with weight} (see \cite[proof of Lemma 1.2]{Mauceri_Meda_Multipliers_Stratified_groups_1990}) we obtain the required estimate.    
\end{proof}

Before proceed further let us mention a useful lemma which we will use in the upcoming proofs. Since the proof of the following Lemma \ref{Pointwise convolution dominated by Maximal function} is quite standard, so we omit the details here (see also \cite[Lemma 2.4]{Bagchi_Molla_Singh_Bilinear_Bochner_Riesz_Grushin}).

\begin{lemma}
\label{Pointwise convolution dominated by Maximal function}
Suppose $K$ be a function on $G$ which satisfies
\begin{align*}
    |K(x,u)| &\leq A\, R^Q (1+R\|(x,u)\|)^{-N}, 
\end{align*}
for $N>0$ and some constant $A>0$. Then whenever $N>Q$, we have
\begin{align*}
    |f * K(x,u)| &\leq C A\, \mathcal{M}f(x,u) .
\end{align*}
    
\end{lemma}

In the sequel, we also require the pointwise weighted kernel estimate.  Let us set the following $X:=(X_1, \ldots, X_{d_1}, T_1, \ldots, T_{d_2})$.

\begin{lemma}
\label{Lemma: Pointwise kernel estimate for linear kernel}
Let $\Gamma \in \mathbb{N}^{d_1 + d_2}$.  Also let $F : \mathbb{R} \to \mathbb{C}$ be a bounded Borel function supported in $[0,R^2]$. Then for any $\beta \geq 0$ and $\epsilon>0$ we have
\begin{align}
\label{Kernel estimate for linear multiplier}
    (1+ R\|(x,u)\|)^{\beta} |\mathcal{K}_{F(\mathcal{L})}(x,u)| &\leq C R^Q \|F(R^2 \cdot)\|_{L^{\infty}_{\beta+\epsilon}(\mathbb{R})} \\
    \text{and} \label{Pointwise Kernel estimate with weight second} \quad (1+R \|(x,u)\|)^{\beta} |X^{\Gamma} \mathcal{K}_{F(\mathcal{L})}(x,u) | &\leq C\, R^{|\Gamma|}\, R^Q\, \|F(R^2 \cdot)\|_{L^{\infty}_{\beta+\epsilon}(\mathbb{R})} .
\end{align}
Moreover,
\begin{align}
\label{Linear multiplier estimate by maximal function}
    |F(\mathcal{L})f(x,u)| &\leq C \|F(R^2 \cdot)\|_{L^{\infty}_{Q+\epsilon}(\mathbb{R})} \mathcal{M}f(x,u) .
\end{align}
In particular,
\begin{align}
\label{Dominate operator by maximal function}
    \sup_{R>0} \left|\Psi\left(2^j \left(1-\frac{\mathcal{L}}{R^2} \right) \right)f(x,u) \right| &\leq C 2^{j(Q+\epsilon)} \mathcal{M}f(x,u) .
\end{align}
    
\end{lemma}

\begin{proof}
The estimates \eqref{Kernel estimate for linear multiplier} and \eqref{Pointwise Kernel estimate with weight second} was proved in \cite[Proposition 3.4]{Molla_Singh_Commutator_Metivier_Arxiv}. For the estimate of \eqref{Linear multiplier estimate by maximal function} first note that 
\begin{align*}
    F(\mathcal{L})f(x,u) &= f * \mathcal{K}_{F(\mathcal{L})}(x,u) .
\end{align*}
Therefore from \eqref{Kernel estimate for linear multiplier} and applying Lemma \ref{Pointwise convolution dominated by Maximal function} with $\beta>Q$, we get the required estimate \eqref{Linear multiplier estimate by maximal function}. Finally, the estimate \eqref{Dominate operator by maximal function} follows from \eqref{Linear multiplier estimate by maximal function} by taking $F(\eta)= \Psi\left(2^j \left(1-\frac{\eta}{R^2} \right) \right)$.
\end{proof}

Let us mention two integral estimates about the homogeneous norm \eqref{Definition of homogeneous norm} and the first layer weight function $|x|$ for $(x,u) \in G$.

\begin{lemma}\cite[Lemma 2.1]{Molla_Singh_Commutator_Metivier_Arxiv}
\label{lemma: outside distance}
    Let $R, r> 0$. Then for any $s> Q$ we have
    \begin{align*}
        \int_{\|(x, u)\| > r} \frac{d(x,u)}{\big( 1 + R\|(x, u)\| \big)^s} &\leq C R^{-s} r^{-s + Q}.
    \end{align*}
\end{lemma}

\begin{lemma}
\label{Lemma: Integration of distance and weight}
Let $R>0$. Then for all $\alpha+\beta>Q$ and all $0\leq \alpha <d_1$ we have
\begin{align*}
    \int_G \frac{d(x,u)}{(1+R\|(x,u)\|)^{\beta} (1+R|x|)^{\alpha}} &\leq C R^{-Q} .
\end{align*}
    
\end{lemma}

\begin{proof}
Using \eqref{Definition of homogeneous norm} and making change of variable $(x,u) \mapsto \delta_{R^{-1}}(x, u)$ we get
\begin{align*}
    \int_G \frac{d(x,u)}{(1+R\|(x,u)\|)^{\beta} (1+R|x|)^{\alpha}} &\lesssim R^{-Q} \int_G \frac{d(x,u)}{(1+|x|+|u|^{1/2})^{\beta} (1+|x|)^{\alpha}} .
\end{align*}
Now since $\beta>Q-\alpha$ and $d_1-\alpha>0$, we decompose $\beta=\beta_1 + \beta_2$ in such a way that $\beta_1>d_1-\alpha$ and $\beta_2>2 d_2$. Hence
\begin{align*}
    \int_G \frac{d(x,u)}{(1+|x|+|u|^{1/2})^{\beta} (1+|x|)^{\alpha}} &\leq C \int_{\mathbb{R}^{d_1}} \frac{dx}{(1+|x|)^{\beta_1+\alpha}} \int_{\mathbb{R}^{d_2}} \frac{du}{(1+|u|)^{\beta_2/2}} \leq C .
\end{align*}
This completes the proof of the lemma.
\end{proof}

\subsection{Criterion for weak-type \texorpdfstring{$(1,1)$}{} and \texorpdfstring{$L^p$}{}-boundedness}
\label{Subsection: Littlewood-Paley}
In this subsection we discuss about the various criterion for a sublinear operator $T$ to be of weak-type $(1,1)$ or $L^p$-bounded for $1<p<\infty$. For $t>0$, define 
\begin{align}
\label{Definition of approximate identity}
    A_t &= \exp{(-t^2 \mathcal{L})} .
\end{align}
It is known that for all $t>0$, the heat kernel $\mathcal{K}_{\exp{(-t^2\mathcal{L})}}$ satisfies the following estimates (see \cite{Varopoulos_Analysis_Lie_Group_1988}, \cite{Alexopoulos_Spectral_multi_poly_growth_1994}).
\begin{align*}
    |\mathcal{K}_{\exp{(-t^2\mathcal{L})}}(x,u)| & \leq C t^{-Q} \exp{\left\{-c \tfrac{\|(x,u)\|^2}{t^2} \right\}} \quad \text{for} \quad t>0 .
\end{align*}

\begin{lemma}\cite[Lemma 2.1]{Duong_Ouhabaz_Sikora_Plancherel_estimate_2002}
\label{Lemma: L2 norm of the heat kernel}
There exists some $C, c>0$ such that
\begin{align*}
    \int_{G \setminus B((x,u), r)} |\mathcal{K}_{e^{-t \mathcal{L}}}((z,s)^{-1}(x,u))|^2 \, d(z,s) &\leq C t^{-Q/2} e^{-c \, r^2/t} .
\end{align*}
    
\end{lemma}

\begin{proposition}\cite[Proposition 3.1]{Coulhon_Duong_Li_Square_function}, \cite[Theorem 4.5]{Duong_Ouhabaz_Sikora_Plancherel_estimate_2002}
\label{Proposition: Criterion for weak boundedness}
Suppose $\|F\|_{L^{\infty}} \leq C$ and
\begin{align*}
    |F(\mathcal{L})(I-A_t)f(x,u)| &\leq \int_G |K_t((z,s)^{-1}(x,u))| |f(z,s)| \, d(z,s) \quad \text{for all}\ \  t>0 .
\end{align*}
If
\begin{align*}
    \sup_{t>0} \int_{\|(x,u)\| \geq t} |K_{t}(x,u)| \, d(x,u) \leq C ,
\end{align*}
then $F(\mathcal{L})$ is of weak-type $(1,1)$.
    
\end{proposition}

Since $(G, \varrho, |\cdot|)$ is a space of homogeneous type, we have the following results, which will be used later to prove $L^p$-boundedness of certain operators. For a ball $B$ with radius $r_B$, we define
\begin{align*}
    U_1(B) := 4B , \quad \text{and} \quad U_j(B) = 2^{j+1}B \setminus 2^j B \quad j=2,3 \ldots .
\end{align*}
\begin{proposition}\cite[Theorem 2.1]{Pascal_Necessary_Sufficient_Lp_boundedness_2007}
\label{Proposition: Lp boundedness criteria from p less than 2 case}
Let $T$ be a sublinear operator which is bounded on $L^2(G)$. Assume that for $j \geq 2$,
\begin{align*}
    \left(\int_{U_j(B)} |T(I-A_{r_B})f(x,u)|^2 \, d(x,u) \right)^{1/2} &\leq C g(j) \, |2^{j+1}B|^{-1/2} \int_B |f(x,u)| \, d(x,u) .
\end{align*}
If $\sum_j g(j) < \infty$, then $T$ is of weak-type $(1,1)$ and hence bounded on $L^p(G)$ for $1<p<2$.
    
\end{proposition}

\begin{proposition}\cite[Theorem 2.2]{Pascal_Necessary_Sufficient_Lp_boundedness_2007}
\label{Proposition: Lp Boundedness for p bigger than 2 case}
Let $T$ be a sublinear operator which is bounded on $L^2(G)$. Assume
\begin{align}
\label{Assumption 1 in prop for bigger p}
    \left(\frac{1}{|B|} \int_B |T(I-A_{r_B})f(x,u)|^2 \, d(x,u) \right)^{1/2} &\leq C \mathcal{M}(|f|^2)^{1/2}(x,u) ,
\end{align}
and
\begin{align}
\label{Assumption 2 in prop for bigger p}
    \|TA_{r_B}f\|_{L^{\infty}(B)} &\leq C \mathcal{M}(|Tf|^2)^{1/2}(x,u) ,
\end{align}
for all $f \in L^2(G)$ and all $B$. If $2<p<\infty$ and $Tf \in L^p(G)$ for $f \in L^p(G)$, then $T$ is of strong-type $(p,p)$ and the operator norm is bounded by a constant depending only on $(2,2)$ norm, doubling constant, $Q$, $p$ and the constants $C$ in the above two assumptions.
    
\end{proposition}

Before end this subsection, let us record a useful lemma for our later use.
\begin{lemma}\cite[Lemma 2.2]{Chen_Duong_Yan_Stein_Square_function_Homogeneous_2013}
\label{Lemma: Finiteness of Sobolev norm for Bochner-Riesz}
Let $\varphi \in C_c^{\infty}(0, \infty)$ supported on $[1/4,1]$. Suppose $j \in \mathbb{Z}$, $m \in 2\mathbb{N}$ and $\alpha>-1/2$. Then for any $\alpha>s-1/q$ with $2\leq q \leq \infty$, there exists constant $C>0$ such that
\begin{align*}
    \sup_{j \in \mathbb{Z} : j \geq -1} \|\varphi(\eta)(1-2^{-mj} \eta^m)_{+}^{\alpha} \|_{L_s^q(\mathbb{R})} &\leq C .
\end{align*}
    
\end{lemma}

\subsection{Bilinear spectral multipliers on M\'etivier groups}
\label{Subsection: Bilinear spectral multiplier}
Let $m$ be a compactly supported smooth function on $\mathbb{R}^2$. Also, let $\mathcal{L}_1 := \mathcal{L} \otimes I$ and $\mathcal{L}_2 := I \otimes \mathcal{L}$. Then the operators $\mathcal{L}_1$ and $\mathcal{L}_2$ commute strongly (see \cite[Lemma 7.24]{Konrad_Unbounded_Selfadjoint_operator_2012}). Therefore using bivariate spectral theorem (see \cite[Theorem 5.21]{Konrad_Unbounded_Selfadjoint_operator_2012}) one can write
\begin{align}
\label{Bivariate bochner-Riesz multiplier}
    & m(\mathcal{L}_1,\mathcal{L}_2)(f \otimes g)((x,u),(x',u'))
    = \frac{1}{(2\pi)^{2 d_2}} \int_{\mathfrak{g}_{2,r}^{*}} \int_{\mathfrak{g}_{2,r}^{*}} e^{i \langle \lambda_1 , u \rangle} e^{i \langle \lambda_2 , u' \rangle} \sum_{\mathbf{k}_1, \mathbf{k}_2 \in \mathbb{N}^{\Lambda}} m(\eta_{\mathbf{k}_1}^{\lambda_1}, \eta_{\mathbf{k}_2}^{\lambda_2}) \\
    &\nonumber \hspace{3cm} \times \left[f^{\lambda_1} \times_{\lambda_1} \varphi_{\mathbf{k}_1}^{\mathbf{b}^{\lambda_1}, \mathbf{r}_1}(R_{\lambda_1}^{-1}\cdot) \right](x) \left[g^{\lambda_2} \times_{\lambda_2} \varphi_{\mathbf{k}_2}^{\mathbf{b}^{\lambda_2}, \mathbf{r}_2}(R_{\lambda_2}^{-1}\cdot) \right](x') \,  d\lambda_1 \,  d\lambda_2 .
\end{align}

As discussed in \cite[Section 3]{Bagchi_Molla_Singh_Bilinear_Metivier} for $f, g \in \mathcal{S}(G)$ the bilinear spectral multiplier associated to $m$ is defined by
\begin{align}
\label{Definition: Bilinear spectral multiplier}
    \mathcal{B}_m(f, g)(x,u) &:= m(\mathcal{L}_1,\mathcal{L}_2)(f \otimes g)((x,u),(x,u)) .
\end{align}

Regarding the $L^{p_1}(G) \times L^{p_2}(G)$ to $L^p(G)$ boundedness of some nice class of the bilinear spectral multiplier we have the following result.
\begin{lemma}
\label{Lemma: Pointwise kernel estimate for Bj}
Let $j\geq 0$, $R>0$ and $\kappa \in \mathbb{R}$. Suppose $m_{j, \kappa} \in C_c^{\infty}([0,2R^2]^2)$ such that for any non-negative integer $\beta_1$ and $\beta_2$ and large positive integer $N$, it satisfies
\begin{align}
\label{Pointwise assumption on m func}
    |\partial_{\eta_1}^{\beta_1} \partial_{\eta_2}^{\beta_2} m_{j, \kappa}(\eta_1, \eta_2)| &\leq C (1+|\kappa|)^{N} 2^{j(\beta_1+\beta_2)} R^{-2(\beta_1+\beta_2)} ,
\end{align}
for $\beta_1+\beta_2 \leq N$. Then
\begin{align}
\label{Pointwise bilinear estimate in maximal funct}
     |\mathcal{B}_{m_{j, \kappa}}(f,g)(x,u)| &\leq C (1+|\kappa|)^N 2^{j2Q} \mathcal{M}f(x,u) \mathcal{M}g(x,u) ,
\end{align}
and for $p_1, p_2, p \geq 1$ with $1/p_1 +1/p_2 \geq 1/p$, we have 
    \begin{align*}
     \|\mathcal{B}_{m_{j, \kappa}}(f,g)\|_{L^p} &\leq C \, (1+|\kappa|)^N R^{Q(\frac{1}{p_1}+\frac{1}{p_2}-1)} 2^{j Q (2+\frac{1}{p}-\frac{1}{p_1}-\frac{1}{p_2})} \|f\|_{L^{p_1}} \|g\|_{L^{p_2}} ,
\end{align*}
for some constant $C>0$, independent of $j$.
\end{lemma}

\begin{proof}
The idea of the proof is taken from combination of \cite[Lemma 4.3]{Duong_Ouhabaz_Sikora_Plancherel_estimate_2002} and \cite[Lemma 3.3]{Jeong_Lee_Vargas_Bilinear_Bochner_Riesz_2018}. Let us set $F(\eta_{1},\eta_{2}) = \exp(\eta_{1}+\eta_{2}) m_{j, \kappa}(R^2\eta_{1},R^2\eta_{2})$. So that
\begin{align*}
    m_{j, \kappa}(\eta_{1},\eta_{2}) &= F(R^{-2}\eta_{1},R^{-2} \eta_{2}) \exp(R^{-2}(\eta_{1}+\eta_{2})) .
\end{align*}
Using Fourier inversion formula $m_{j, \kappa}$ can be expressed as
\begin{align*}
    m_{j, \kappa}(\eta_{1},\eta_{2}) &= \frac{1}{4 \pi^2}\int_{\mathbb{R}^2} \widehat{F}(\tau_1, \tau_2) \exp((i \tau_1-1) R^{-2} \eta_{1}) \exp((i \tau_2-1) R^{-2} \eta_{2}) \, d\tau_1 \, d\tau_2 .
\end{align*}
Therefore from \eqref{Definition: Bilinear spectral multiplier} for $f, g \in \mathcal{S}(G)$, we obtain
\begin{align}
\label{Bilinear multiplier in terms of kernel}
    & \mathcal{B}_{m_{j, \kappa}}(f,g)(x,u) = m_{j, \kappa}(\mathcal{L}_1,\mathcal{L}_2)(f \otimes g)((x,u),(x,u)) \\
    &\nonumber = \int_G \int_G \mathcal{K}_{m_{j, \kappa}(\mathcal{L}_1,\mathcal{L}_2)}((y,t)^{-1}(x,u),(z,s)^{-1}(x,u)) f(y,t) g(z,s) \ d(y,t) \ d(z,s) ,
\end{align}
where
\begin{align*}
    \mathcal{K}_{m_{j, \kappa}(\mathcal{L}_1,\mathcal{L}_2)}((y,t),(z,s)) &= \frac{1}{4 \pi^2}\int_{\mathbb{R}^2} \widehat{F}(\tau_1, \tau_2) \mathcal{K}_{\exp((i \tau_1-1) R^{-2} \mathcal{L})}(y,t) \mathcal{K}_{\exp((i \tau_2-1) R^{-2} \mathcal{L})}(z,s) d\tau_1 d\tau_2 .
\end{align*}
Hence for $N_1, N_2 \geq 0$, we can see that
\begin{align*}
    & | \mathcal{K}_{m_{j, \kappa}(\mathcal{L}_1,\mathcal{L}_2)}((y,t),(z,s)) | (1+ R 2^{-j} \|(y,t)\|)^{N_1} (1+ R 2^{-j} \|(z,s)\|)^{N_2} \\
    &\leq C \int_{\mathbb{R}^2} |\widehat{F}(\tau_1, \tau_2)| | \mathcal{K}_{\exp((i \tau_1-1) R^{-2} \mathcal{L})}(y,t) \mathcal{K}_{\exp((i \tau_2-1) R^{-2} \mathcal{L})}(z,s)| \\
    &\hspace{5cm} (1+ R 2^{-j} \|(y,t)\|)^{N_1} (1+ R 2^{-j} \|(z,s)\|)^{N_2} d\tau_1 d\tau_2 .
\end{align*}
Note that from \cite[Theorem 7.3]{Ouhabaz_Analysis_heat_equation_domain_2005} we have the following pointwise estimate:
\begin{align*}
    |\mathcal{K}_{\exp((i \tau_1-1) R^{-2} \mathcal{L})}(x,u)| &\leq C R^Q \exp{\left\{-c \tfrac{R^2 \|(x,u)\|^2}{(1+\tau_1^2)} \right\}} .
\end{align*}
Thus, it follows from the above estimate that
\begin{align*}
    & |\mathcal{K}_{\exp((i \tau_1-1) R^{-2} \mathcal{L})}(y,t) \mathcal{K}_{\exp((i \tau_2-1) R^{-2} \mathcal{L})}(z,s)| (1+ R 2^{-j} \|(y,t)\|)^{N_1} (1+ R 2^{-j} \|(z,s)\|)^{N_2} \\
    & \leq C R^{2Q} (1+2^{-j}|\tau_1|)^{N_1} (1+2^{-j}|\tau_2|)^{N_2} .
\end{align*}
Now taking $N_1+N_2+2=N$ and an application of H\"older's inequality implies
\begin{align}
\label{Use of Holders inequality in pointwise kernel estimate}
    &\left| \mathcal{K}_{m_{j, \kappa}(\mathcal{L}_1,\mathcal{L}_2)}((y,t),(z,s)) \right| (1+ R 2^{-j} \|(y,t)\|)^{N_1} (1+ R 2^{-j} \|(z,s)\|)^{N_2} \\
    &\nonumber \leq C R^{2Q} \int_{\mathbb{R}^2} |\widehat{F}(\tau_1, \tau_2)| (1+2^{-j} |\tau_1|)^{N_1} (1+2^{-j} |\tau_2|)^{N_2} \, d\tau_1 \, d\tau_2 \\
    &\nonumber \leq C R^{2Q} 2^{2 j} \left(\int_{\mathbb{R}^2} |\widehat{F}(2^j \tau_1, 2^j \tau_2)|^2 (1+|\tau_1|)^{2 N_1+2}(1+|\tau_2|)^{2 N_2+2} d\tau_1 d\tau_2 \right)^{\frac{1}{2}} \\
    &\nonumber \hspace{9cm} \left(\int_{\mathbb{R}^2} \frac{\ d\tau_1 d\tau_2}{(1+|\tau_1|)^2(1+|\tau_2|)^2} \right)^{\frac{1}{2}} \\
    &\nonumber \leq C R^{2Q} 2^{2 j} \sum_{0\leq\beta_1 \leq N_1+1,\  0\leq\beta_2 \leq N_2+1} \left(\int_{\mathbb{R}^2} |\partial_{\eta_1}^{\beta_1} \partial_{\eta_2}^{\beta_2} G_j(\eta_1, \eta_2)|^2 d\eta_1 d\eta_2 \right)^{\frac{1}{2}} ,
\end{align}
where $\widehat{G}_j(\tau_1, \tau_2) = \widehat{F}(2^j \tau_1, 2^j \tau_2)$.

Using \eqref{Pointwise assumption on m func} we can see
\begin{align*}
    & \sum_{\beta_1 \leq N_1+1,\  \beta_2 \leq N_2+1} \left(\int_{\mathbb{R}^2} |\partial_{\eta_1}^{\beta_1} \partial_{\eta_2}^{\beta_2} G_j(\eta_1, \eta_2)|^2 d\eta_1 d\eta_2 \right)^{\frac{1}{2}} \\
    &= 2^{-2 j} \sum_{\beta_1 \leq N_1+1,\ \beta_2 \leq N_2+1} 2^{-j (\beta_1+ \beta_2)} \left(\int_{\mathbb{R}^2} |(\partial_{\eta_1}^{\beta_1} \partial_{\eta_2}^{\beta_2} F)(2^{-j} \eta_1, 2^{-j} \eta_2)|^2 d\eta_1 d\eta_2 \right)^{\frac{1}{2}} \\
    &\leq C 2^{-2 j} (1+|\kappa|)^{N} \sum_{\beta_1 \leq N_1+1,\ \beta_2 \leq N_2+1} 2^{-j (\beta_1+ \beta_2)} R^{2(\beta_1+\beta_2)} 2^{j (\beta_1+ \beta_2)} R^{-2(\beta_1+\beta_2)} \\
    &\leq C 2^{-2 j} (1+|\kappa|)^{N} .
\end{align*}
Therefor plugging the above estimate into \eqref{Use of Holders inequality in pointwise kernel estimate} we obtain 
\begin{align}
\label{Pointwise kernel estimate}
    \left| \mathcal{K}_{m_{j, \kappa}(\mathcal{L}_1,\mathcal{L}_2)}((y,t),(z,s)) \right| (1+ R 2^{-j} \|(y,t)\|)^{N_1} (1+ R 2^{-j} \|(z,s)\|)^{N_2} &\leq C R^{2Q} (1+|\kappa|)^N .
\end{align}
Let us set
\begin{align*}
    k_1(x,u) = \frac{1}{(1+ R 2^{-j} \|(x,u)\|)^{N_1}} \quad \text{and} \quad k_2(x,u) = \frac{1}{(1+ R 2^{-j} \|(x,u)\|)^{N_2}} .
\end{align*}
Consequently, plugging the estimate \eqref{Pointwise kernel estimate} into \eqref{Bilinear multiplier in terms of kernel} and using Lemma \ref{Pointwise convolution dominated by Maximal function} yields
\begin{align*}
    |\mathcal{B}_{m_{j, \kappa}}(f,g)(x,u)| &\leq C R^{2Q} \, (1+|\kappa|)^N (|f|*k_1)(x,u) \, (|g|*k_2)(x,u) \\
    &\leq C (1+|\kappa|)^N 2^{j2Q} \mathcal{M}f(x,u) \mathcal{M}g(x,u) ,
\end{align*}
provided we choose $N_1, N_2 >Q$. This proves the estimate \eqref{Pointwise bilinear estimate in maximal funct}.

Similarly as in the previous estimate and in addition with H\"older's inequality we obtain
\begin{align}
\label{L infinity estimate for bilinear multiplier}
    |\mathcal{B}_{m_{j, \kappa}}(f,g)(x,u)| &\leq C R^{2Q} \, (1+|\kappa|)^N (|f|*k_1)(x,u) \, (|g|*k_2)(x,u) \\ 
    &\nonumber \leq C R^{2Q} \, (1+|\kappa|)^N \|f\|_{L^{p_1}} \|k_1\|_{L^{p_1'}} \|g\|_{L^{p_2}} \|k_2\|_{L^{p_2'}} \\
    &\nonumber \leq C R^{Q(2-\frac{1}{p_1'}-\frac{1}{p_2'})} \, (1+|\kappa|)^N 2^{j Q(\frac{1}{p_1'}+\frac{1}{p_2'})} \|f\|_{L^{p_1}} \|g\|_{L^{p_2}} ,
\end{align}
provided we choose $N_1> Q/p_1'$ and $ N_2 >Q/p_2'$.

So that for $N_1> Q/p_1'$ and $ N_2 >Q/p_2'$ we have
\begin{align}
\label{Pointwise estimate in addition Holder}
    \|\mathcal{B}_{m_{j, \kappa}}(f,g)\|_{L^{\infty}} &\leq C R^{Q(\frac{1}{p_1}+\frac{1}{p_2})} \, (1+|\kappa|)^N 2^{j Q(2-\frac{1}{p_1}-\frac{1}{p_2})} \|f\|_{L^{p_1}} \|g\|_{L^{p_2}} .
\end{align}
On the other hand, for the $L^1$-estimate, similar to \eqref{L infinity estimate for bilinear multiplier}, using H\"older's inequality and Young's convolution inequality yields
\begin{align}
\label{L1 estimate before interpolation}
    \|\mathcal{B}_{m_{j, \kappa}}(f,g)\|_{L^{1}} &\leq C R^{2Q} \, (1+|\kappa|)^N \left\||f| * k_1 \right\|_{L^{p_1}} \left\||g|* k_2 \right\|_{L^{{p_1}'}} \\
    &\nonumber \leq C R^{2Q} \, (1+|\kappa|)^N \|f\|_{L^{p_1}} \|k_1\|_{L^1} \|g\|_{L^{p_2}} \|k_2\|_{L^{r_2}} \\
    &\nonumber \leq C R^{Q(1-\frac{1}{r_2})} \, (1+|\kappa|)^N 2^{j Q (1+\frac{1}{r_2})} \|f\|_{L^{p_1}} \|g\|_{L^{p_2}} \\
    &\nonumber \leq C R^{Q(\frac{1}{p_1}+\frac{1}{p_2}-1)} \, (1+|\kappa|)^N 2^{j Q (3-\frac{1}{p_1}-\frac{1}{p_2})} \|f\|_{L^{p_1}} \|g\|_{L^{p_2}} ,
\end{align}
where $1+\frac{1}{p_1'}=\frac{1}{p_2} + \frac{1}{r_2}$ with $r_2 \geq 1$ and provided we choose $N_1>Q$, $N_2>Q/r_2$.

Note that since $r_2 \geq 1$ in the last estimate, we always have $\frac{1}{p_1} + \frac{1}{p_2} = 2-\frac{1}{r_2} \geq 1$. Therefore using the bilinear interpolation between \eqref{Pointwise estimate in addition Holder} and \eqref{L1 estimate before interpolation} for $p_1, p_2, p \geq 1$ with $1/p_1 +1/p_2 \geq 1/p$ and choosing $N$ sufficiently large we obtain
\begin{align*}
    \|\mathcal{B}_{m_{j, \kappa}}(f,g)\|_{L^{p}} &\leq C R^{(\frac{1}{p_1}+\frac{1}{p_2}-\frac{1}{p})} \, (1+|\kappa|)^N 2^{j Q (2+\frac{1}{p}-\frac{1}{p_1}-\frac{1}{p_2})} \|f\|_{L^{p_1}} \|g\|_{L^{p_2}} .
\end{align*}
This completes the proof of the Lemma.
\end{proof}

\section{Square functions on M\'etivier groups}
\label{Section: Square function on Metivier groups}
This section is divided into three subsections. In the first subsection, we discuss about the local square functions on M\'etivier groups. In the second subsection, we consider the global version of the local square function on M\'etivier groups and prove that $L^p$-boundedness of two square functions are equivalent. The final subsection is devoted to the proof of Theorem \ref{Theorem: Stein square estimate} and \ref{Theorem: Stein square function for p less that 2 case}.

\subsection{Local square function estimate on M\'etivier groups}
\label{Subsection: Local Square function estimate}
In this subsection, we define a local square function with localized frequency on M\'etivier groups and discuss their corresponding $L^p$-boundedness results and also give some applications. This is one of the main contribution of this paper and will be our key ingredients in the forthcoming proof of Proposition \ref{Proposition: Square function estimate}.

Let $\phi \in C_c^{\infty}([-1,1])$ and $|\phi(t)| \leq 1$ for all $t \in [-1,1]$. For all $\delta \in (0,1]$, we define the local square function with localized frequency by
\begin{align}
\label{Definition: Square function}
    \mathfrak{S}_{\delta, loc}^{\phi}(\mathcal{L})f(x,u) &= \left(\int_{1/2}^{2} \Big| \phi\left(\frac{t-\mathcal{L}}{\delta} \right) f(x,u) \Big|^2 dt \right)^{1/2} ,
\end{align}

The following proposition describes the $L^p$-boundedness of the above square function $\mathfrak{S}_{\delta, loc}^{\phi}(\mathcal{L})$.

\begin{proposition}
\label{Proposition: Local Square function estimate}
Let $\mathfrak{p}_G \leq p < \infty$. Whenever $\alpha>\alpha_d(p)$ we have
\begin{align}
\label{Local Square function estimate for p biggar than infinity}
    \|\mathfrak{S}_{\delta, loc}^{\phi}(\mathcal{L})f\|_{L^p} &\leq C \delta^{\frac{1}{2}-\alpha} \|f\|_{L^p} .
\end{align}
Moreover,
\begin{align}
\label{Local Square function estimate for p=2 case}
    \|\mathfrak{S}_{\delta, loc}^{\phi}(\mathcal{L})f\|_{L^2} &\leq C \delta^{\frac{1}{2}} \|f\|_{L^2} .
\end{align}
    
\end{proposition}

In the following we show that the proof of the above proposition is a consequence of the following lemma. From subsection \ref{Subsection: Variant of Hardy-littlewood} recall that $\mathcal{M}_{r}$ and $\mathcal{M}_{r}^{|\cdot|}$ denotes the maximal functions related to the Hardy-Littlewood maximal functions and strong Hardy-Littlewood maximal functions on M\'etivier groups respectively.

\begin{lemma}
\label{Lemma: Weighted L2 estimate for square function}
Let $1\leq q \leq \mathfrak{p}_G'$ and set $1/r = 2/q-1$. Then for any $\omega \geq 0$, we have
\begin{align*}
    \int_{\mathbb{R}^d} |\mathfrak{S}_{\delta, loc}^{\phi}(\mathcal{L})f(x,u)|^2 \omega(x,u) \, d(x,u) &\leq C \delta^{2-d(\frac{2}{q}-1)-\epsilon} \int_{\mathbb{R}^d} | f(x,u)|^2 \mathcal{M}_{r}^{|\cdot|}\omega(x,u)\, d(x,u) \\
    &\hspace{1cm} + C \delta^{2-d (\frac{2}{q}-1)} \int_{\mathbb{R}^d} |f(x,u)|^2 \mathcal{M}_{r}\omega(x,u) \, d(x,u) .
\end{align*}
\end{lemma}

Assuming the above lemma for the moment, let us complete the proof of Proposition \ref{Proposition: Local Square function estimate}.

\begin{proof}[Proof of Proposition \ref{Proposition: Local Square function estimate}]
Let us first start with the proof of \eqref{Local Square function estimate for p=2 case}. From the spectral theorem for $f \in L^2(G)$ we have
\begin{align}
\label{L2 estimate of the square function}
    \|\mathfrak{S}_{\delta, loc}^{\phi}(\mathcal{L})f\|_{L^2} &= \left(\int_{1/2}^{2} \left\langle \phi^2\left(\frac{t-\mathcal{L}}{\delta} \right) f, f \right\rangle \, dt \right)^{1/2} \\
    &\nonumber = \left\| \left(\int_{1/2}^{2} \phi^2\left(\frac{t-\mathcal{L}}{\delta} \right) \, dt \right)^{1/2} f \right\|_{L^2} = \sup_{\eta} \left\{ \int_{1/2}^{2} \phi^2\left(\frac{t-\eta}{\delta} \right) \, dt \right\}^{1/2} \|f\|_{L^2} .
\end{align}
Since $\supp{\phi} \subseteq [-1, 1]$, we have $\eta-\delta \leq t \leq \eta+ \delta$. Therefore we obtain
\begin{align*}
    \|\mathfrak{S}_{\delta, loc}^{\phi}(\mathcal{L})f\|_{L^2} &\leq C \sup_{\eta} \left( \int_{\eta-\delta}^{\eta+\delta} dt \right)^{1/2} \|f\|_{L^2} \leq C \delta^{1/2} \|f\|_{L^2} .
\end{align*}
This completes the proof for $p=2$ case. Now we move to the proof for $\mathfrak{p}_G \leq p <\infty$.

\medskip
Recall that $\alpha_d(p) = d(\frac{1}{2}-\frac{1}{p}) - \frac{1}{2}$ for $\mathfrak{p}_G \leq p < \infty$. Therefore in order to prove \eqref{Local Square function estimate for p biggar than infinity}, enough to show that, for $1\leq q \leq \mathfrak{p}_G'$ and $2\leq p<q'$ we have
\begin{align}
\label{Reduction of square function to later Sangyuklee}
    \|\mathfrak{S}_{\delta, loc}^{\phi}(\mathcal{L})f\|_{L^p} &\leq C \delta^{1-d(\frac{1}{q}-\frac{1}{2})-\epsilon} \|f\|_{L^p} ,
\end{align}
for some $\epsilon>0$.

The proof of \eqref{Local Square function estimate for p biggar than infinity} then follows from \eqref{Reduction of square function to later Sangyuklee} with an interpolation between the two instances $q=1$ and $q=\mathfrak{p}_G'$. Therefore in the following we prove \eqref{Reduction of square function to later Sangyuklee}.

\medskip
Since for $1\leq q \leq \mathfrak{p}_G'$, we have $d(\frac{1}{q}-\frac{1}{2}) - \frac{1}{2}-\epsilon>0$ for sufficiently small $\epsilon>0$, hence the estimate \eqref{Reduction of square function to later Sangyuklee} at $p=2$ follows from \eqref{Local Square function estimate for p=2 case}. Therefore, it remains to prove \eqref{Reduction of square function to later Sangyuklee} for $2<p<q'$. In this case we show that proof of \eqref{Reduction of square function to later Sangyuklee} is a consequence of the Lemma \ref{Lemma: Weighted L2 estimate for square function}.

\medskip
We take $\omega \in L^s$ such that $\|\omega\|_{L^s} \leq 1$ with $1/s+2/p=1$. Since we also have $1/r = 2/q-1$, we can easily see that $r <s$. Therefore using duality we write
\begin{align}
\label{Use of duality in square function}
    \|\mathfrak{S}_{\delta, loc}^{\phi}(\mathcal{L})f\|_{L^p} &= \| |\mathfrak{S}_{\delta, loc}^{\phi}(\mathcal{L})f|^2 \|_{L^{p/2}}^{1/2} = \left\{ \sup_{\substack{\omega \in L^s \\ \|\omega\|_{L^s} \leq 1}} \left| \int_{\mathbb{R}^d} |\mathfrak{S}_{\delta, loc}^{\phi}(\mathcal{L})f(x,u)|^2 \omega(x,u) \, d(x,u) \right|\right\}^{1/2} .
\end{align}
Now using Lemma \ref{Lemma: Weighted L2 estimate for square function}, H\"older's inequality, Lemma \ref{Lemma: Boundedness of maximal and Euclidean maximal} and the fact $r < s$, we obtain
\begin{align*}
    \int_{\mathbb{R}^d} |\mathfrak{S}_{\delta, loc}^{\phi}(\mathcal{L})f(x,u)|^2 \omega(x,u) \, d(x,u) &\leq C \delta^{2\{1-d(\frac{1}{q}-\frac{1}{2})\}-\epsilon} \||f|^2\|_{L^{p/2}} \|\mathcal{M}_{r}^{|\cdot|}\omega\|_{L^s} \\
    &\hspace{2cm} + C \delta^{2\{1-d(\frac{1}{q}-\frac{1}{2})\}} \||f|^2\|_{L^{p/2}} \|\mathcal{M}_{r}\omega\|_{L^s} \\
    &\leq C \delta^{2\{1-d(\frac{1}{q}-\frac{1}{2})\}-2\epsilon} \|f\|_{L^p}^2 \|\omega\|_{L^s} .
\end{align*}
Therefore plugging the above estimate into \eqref{Use of duality in square function} we get
\begin{align*}
    \|\mathfrak{S}_{\delta, loc}^{\phi}(\mathcal{L})f\|_{L^p} &\leq C \delta^{1-d(\frac{1}{q}-\frac{1}{2})-\epsilon} \|f\|_{L^p} .
\end{align*}
This completes the proof of \eqref{Reduction of square function to later Sangyuklee} and hence the proof of Proposition \ref{Proposition: Local Square function estimate} is also completed.    
\end{proof}

Now we proceed to prove Lemma \ref{Lemma: Weighted L2 estimate for square function}.
\begin{proof}[Proof of Lemma \ref{Lemma: Weighted L2 estimate for square function}]
First note that
\begin{align}
\label{Dominate local by slightly different}
    |\mathfrak{S}_{\delta, loc}^{\phi}(\mathcal{L})f(x,u)|^2 &\leq 2\sqrt{2}  \int_{1/\sqrt{2}}^{\sqrt{2}} \Big| \phi\left(\delta_t^{-1} \left(1- \frac{\mathcal{L}}{t^2} \right) \right) f(x,u) \Big|^2 \, dt ,
\end{align}
where $\delta_t^{-1}= \delta^{-1} t^2 \sim \delta^{-1}$, since $t \in [1/\sqrt{2}, \sqrt{2}]$.

Let us set $\phi_{\delta_t}(s) = \phi(\delta_t^{-1}(1-s^2))$. Since $\phi_{\delta_t}$ is an even function, by Fourier inversion formula we can write
\begin{align*}
    \phi_{\delta_t}(s) &= \frac{1}{2\pi} \int_{\mathbb{R}} \widehat{\phi_{\delta_t}}(u) \cos(s u)\, du .
\end{align*}

For $\delta \in (0, 1]$, choose an integer $j_0$ such that $2^{-j_0-1} \leq \delta < 2^{-j_0}$. Also let $\sigma$ be a bump function on $\mathbb{R}$, such that it is $1$ on $[-1,1]$ and supported on $[-2,2]$. With the help of $\sigma$, for any integer $j \geq j_0$, we define
\begin{align}
\label{Definition of zeta}
    \vartheta_j(s) = \left\{\begin{array}{ll}
        \sigma(2^{-j_0} s) & \text{if} \quad j=j_0  \\
        \sigma(2^{-j}s) - \sigma(2^{-j+1}s) & \text{if} \quad j>j_0 . 
    \end{array} \right.
\end{align}
Then $\vartheta_j$ become the partition of unity, that is
\begin{align}
\label{Partition of unity see Stein}
    \sum_{j \geq j_0} \vartheta_j(s) &= 1 \quad \text{for all} \quad s>0 .
\end{align}
If we define
\begin{align}
\label{Decomposition in the Fourier transform side}
    \phi_{\delta_t, j}(s) &= \frac{1}{2\pi} \int_{\mathbb{R}} \vartheta_j(u) \widehat{\phi_{\delta_t}}(u) \cos(s u)\, du ,
\end{align}
then from \eqref{Partition of unity see Stein}, for $s>0$ we immediately obtain
\begin{align}
\label{Partition of unity for further cutoff}
    \phi_{\delta_t}(s) &= \sum_{j \geq j_0} \phi_{\delta_t, j}(s) .
\end{align}
At this point let us record a estimate for $\phi_{\delta_t, j}$, which will be very useful later in our proof. In fact we have the following bound
\begin{align}
\label{Pointwise bound for phi delta j function}
    |\phi_{\delta_t, j}(s)| &\leq \left\{\begin{array}{ll}
        C_N \, 2^{(j_0-j) N} & \text{if} \quad s \in [1/4, 8] \\
        C_N \, 2^{j-j_0} (1+2^j |s-1|)^{-N} & \text{elsewhere} ,
    \end{array} \right.
\end{align}
for all $N \geq 0$ and $j \geq j_0$ (see \cite[p. 18]{Christ_almost_Everywhere_Bochner_Riesz_1985}, \cite[p. 23]{Chen_Duong_He_Lee_Yan_Bochner_Hermite_2021}).

Then from \eqref{Decomposition in the Fourier transform side} and \eqref{Partition of unity for further cutoff}, using spectral theorem for $f \in L^2(G) \cap L^p(G)$ we write
\begin{align}
\label{Decomposition of phi delta for j bigger j not}
    \phi\left(\delta_t^{-1} \left(1-\frac{\mathcal{L}}{t^2} \right) \right)f &= \sum_{j \geq j_0} \phi_{\delta_t, j}(\mathcal{\sqrt{L}}/t )f ,
\end{align}
where
\begin{align}
\label{Writting phi in terms of wave operator}
    \phi_{\delta_t, j}(\mathcal{\sqrt{L}}/t )f &= \frac{1}{2\pi} \int_{\mathbb{R}} \vartheta_j(u) \widehat{\phi_{\delta_t}}(u) \cos(\mathcal{\sqrt{L}} u/t)f\, du .
\end{align}
Therefore in view of the above decomposition \eqref{Decomposition of phi delta for j bigger j not}, from \eqref{Dominate local by slightly different} for any $\omega \geq 0$, we get
\begin{align}
\label{Final estimate after further decomposition cutoff}
    \int_{\mathbb{R}^d} |\mathfrak{S}_{\delta, loc}^{\phi}(\mathcal{L})f(x,u)|^2 \omega(x,u) \, d(x,u) & \leq C \left[ \sum_{j \geq j_0} \left\{ \int_{1/\sqrt{2}}^{\sqrt{2}} \left\langle \Big| \phi_{\delta_t, j} (\mathcal{\sqrt{L}}/t ) f \Big|^2, \omega \right\rangle \, dt \right\}^{\frac{1}{2}} \right]^2 .
\end{align}
First note that from \eqref{Definition of zeta} for $j \geq j_0$, $\supp{\vartheta_j} \subseteq [-2^{j+1}, 2^{j+1}]$. Let $\mathcal{K}_{\phi_{\delta_t,j}(\sqrt{\mathcal{L}}/t)}$ denote the convolution kernel of the operator $\phi_{\delta_t,j}(\sqrt{\mathcal{L}}/t)$. Then for $(x,u) \in G$ we write
\begin{align*}
    \mathcal{K}_{\phi_{\delta_t,j}(\sqrt{\mathcal{L}}/t)}((y,t)^{-1}(x,u)) =: \widetilde{\mathcal{K}_{\phi_{\delta_t,j}}}(\sqrt{\mathcal{L}}/t)((x,u),(y,t)) .
\end{align*}
Since $\cos(\mathcal{\sqrt{L}})$ satisfies the finite speed propagation property, from \eqref{Writting phi in terms of wave operator} we can see
\begin{align}
\label{Support of phi using finite speed}
    \supp{\widetilde{\mathcal{K}_{\phi_{\delta_t,j}}}(\sqrt{\mathcal{L}}/t)} \subseteq \mathcal{D}_{\varrho_j} := \{((x,u), (y,t)) : \varrho((x,u), (y,t)) \leq 2^{j+2} \} ,
\end{align}
Then similarly as in \cite[subsec. 5.3]{Bagchi_Molla_Singh_Bilinear_Metivier} we choose a sequence $\{(x_m, u_m)\}_{m\in \mathbb{N}}$ such that for $m \neq l$, $\varrho((x_m,u_m), (x_l,u_l)) > 2^{j+2}/10$ and $\sup_{(x,u) \in \mathbb{R}^d} \inf_{m} \varrho((x,u), (x_m,u_m)) \leq 2^{j+2}/10$. Let us define
\begin{align*}
    S_m &= \Bar{B}\left((x_m,u_m), \tfrac{2^{j+2}}{10}\right) \setminus \bigcup_{l<m} \Bar{B}\left((x_l,u_l), \tfrac{2^{j+2}}{10}\right) .
\end{align*}
We can easily see that whenever $m \neq l$, we have $B\left((x_m,u_m), \frac{2^{j+2}}{20}\right) \cap B\left((x_l,u_l), \frac{2^{j+2}}{20}\right) \neq \emptyset$. So that an application of doubling property of balls gives the following bounded overlapping property 
\begin{align}
\label{Bounded overlapping property}
    \sup_{m} \# \{l : \varrho((x_m,u_m),(x_l,u_l)) \leq 2 \cdot 2^{j+2}\} \leq C .
\end{align}   
From \eqref{Support of phi using finite speed}, we also obtain
\begin{align*}
    \mathcal{D}_{\varrho_j} &\subseteq \bigcup_{l,m: \varrho((x_m,u_m),(x_l,u_l)) < 2 \cdot 2^{j+2}} S_l \times S_m .
\end{align*}
Consequently, we write
\begin{align*}
    \phi_{\delta_t, j} (\mathcal{\sqrt{L}}/t )f &= \sum_{l, m: \varrho((x_m,u_m),(x_l,u_l)) < 2 \cdot 2^{j+2}} \chi_{S_l} \phi_{\delta_t, j}(\mathcal{\sqrt{L}}/t)\chi_{S_m} f .
\end{align*}
It is easy to see that if $\varrho((x_m,u_m),(x_l,u_l)) < 2 \cdot 2^{j+2}$, then $S_l \subset B((x_m,u_m), 3 \cdot 2^{j+2})$. We denote $B_m := B((x_m,u_m), 3 \cdot 2^{j+2})$. Therefore using the fact that the sets $S_l$ are disjoint, an application of bounded overlapping property \eqref{Bounded overlapping property} and $S_l \subset B_m$ yields
\begin{align}
\label{Application of bounded overlapp and other}
    \left\langle \Big| \phi_{\delta_t, j} (\mathcal{\sqrt{L}}/t ) f \Big|^2, \omega \right\rangle &= \sum_{l} \left\langle \Big| \sum_{m: \varrho((x_m,u_m),(x_l,u_l)) < 2 \cdot 2^{j+2}} \chi_{S_l} \phi_{\delta_t, j} (\mathcal{\sqrt{L}}/t )\chi_{S_m} f \Big|^2, \omega \right\rangle \\
    &\nonumber \leq C \sum_{l} \sum_{m: \varrho((x_m,u_m),(x_l,u_l)) < 2 \cdot 2^{j+2}} \left\langle \Big| \chi_{S_l} \phi_{\delta_t, j}(\mathcal{\sqrt{L}}/t )\chi_{S_m} f \Big|^2, \omega \right\rangle \\
    &\nonumber \leq C \sum_{m} \left\langle \Big| \chi_{B_m} \phi_{\delta_t, j} (\mathcal{\sqrt{L}}/t )\chi_{S_m} f \Big|^2, \omega \right\rangle .
\end{align}
Note that the function $\phi_{\delta_t,j}$ is not compactly supported, therefore we decompose it further. Choose an even bump function $\zeta$ such that it is $1$ on $(-2,2)$ and supported on $(-4,4)$. For any integer $\ell \geq 0$, if we define
\begin{align}
\label{Definition of psi l delta cutoff}
    \psi_{\ell, \delta}(s) = \left\{\begin{array}{ll}
        \zeta(\delta^{-1} (1-s)) & \text{if} \quad \ell=0  \\
        \zeta(2^{-\ell} \delta^{-1} (1-s)) - \zeta(2^{-\ell+1} \delta^{-1} (1-s)) & \text{if} \quad \ell \geq 1 ,
    \end{array} \right.
\end{align}
then it satisfies
\begin{align*}
    \sum_{\ell=0}^{\infty} \psi_{\ell, \delta}(s) &= 1 .
\end{align*}
Consequently, for any $g \in L^2(G) \cap L^p(G)$ we write
\begin{align*}
    \phi_{\delta_t, j} (\mathcal{\sqrt{L}}/t ) g &= \sum_{\ell=0}^{\infty} (\psi_{\ell, \delta} \phi_{\delta_t, j}) (\mathcal{\sqrt{L}}/t ) g .
\end{align*}
Hence, from \eqref{Application of bounded overlapp and other} we obtain
\begin{align}
\label{Decomposition of ell}
    \left\langle \Big| \phi_{\delta_t, j} (\mathcal{\sqrt{L}}/t ) f \Big|^2, \omega \right\rangle &\leq C \sum_{m} \left\langle \Big|\sum_{\ell=0}^{j_0} \chi_{B_m} (\psi_{\ell, \delta} \phi_{\delta_t, j}) (\mathcal{\sqrt{L}}/t ) \chi_{S_m} f \Big|^2, \omega \right\rangle \\
    &\nonumber \hspace{1cm} + C \sum_{m} \left\langle \Big|\sum_{\ell=j_0+1}^{\infty} \chi_{B_m} (\psi_{\ell, \delta} \phi_{\delta_t, j}) (\mathcal{\sqrt{L}}/t ) \chi_{S_m} f \Big|^2, \omega \right\rangle ,
\end{align}
where $j_0$ be the integer chosen earlier just before \eqref{Definition of zeta}. Note that \eqref{Definition of psi l delta cutoff}, implies $\supp{\psi_{\ell, \delta}} \subseteq [1-2^{\ell+2} \delta, 1+2^{\ell+2} \delta]$. We set
\begin{align}
\label{Definition of s/t in suffix t case}
    (\psi_{\ell, \delta} \phi_{\delta_t, j})_t(s) := (\psi_{\ell, \delta} \phi_{\delta_t, j})(s/t) .
\end{align}
So that support of $(\psi_{\ell, \delta} \phi_{\delta_t, j})_t$ is contained in $[t(1-2^{\ell+2} \delta), t(1+2^{\ell+2} \delta)]$. Let $\Theta$ be the bump function as in \eqref{Definition: Cutoff function theta preli} with $R=t(1+2^{\ell+2} \delta)$. Then for $M \in \mathbb{Z}$, we define the function $(\psi_{\ell, \delta} \phi_{\delta_t, j})_{t, M} : \mathbb{R} \times \mathbb{R} \to \mathbb{C}$ by 
\begin{align}
\label{Introducing theta in multiplier}
    (\psi_{\ell, \delta} \phi_{\delta_t, j})_{t, M}(\eta, \tau) := \left\{\begin{array}{ll}
       (\psi_{\ell, \delta} \phi_{\delta_t, j})_{t}(\sqrt{\eta}) \Theta( 2^{M} \tau),  & \quad \text{if} \quad \eta \geq 0 \\
        0 & \quad \text{elsewhere} .
    \end{array} \right.
\end{align}
Then similarly as in \eqref{In the Thete reducing M from Z to l} we can write
\begin{align*}
    (\psi_{\ell, \delta} \phi_{\delta_t, j})_t (\mathcal{\sqrt{L}}) &= \sum_{M =-l_0}^{\infty} (\psi_{\ell, \delta} \phi_{\delta_t, j})_{t, M}(\mathcal{L}, T) = \left(\sum_{M =-l_0}^{j} + \sum_{M=j+1}^{\infty} \right) (\psi_{\ell, \delta} \phi_{\delta_t, j})_{t, M}(\mathcal{L}, T) .
\end{align*}
Using the above decomposition and from \eqref{Decomposition of ell} we obtain
\begin{align}
\label{Spliting inner product in the sum of M}
    \left\langle \Big| \phi_{\delta_t, j} (\mathcal{\sqrt{L}}/t ) f \Big|^2, \omega \right\rangle & \leq C \sum_{m} \left\langle \Big| \sum_{M =-l_0}^{j} \sum_{\ell=0}^{j_0} \chi_{B_m} (\psi_{\ell, \delta} \phi_{\delta_t, j})_{t, M}(\mathcal{L}, T) \chi_{S_m} f \Big|^2, \omega \right\rangle \\
    &\nonumber \hspace{.5cm} + C \sum_{m} \left\langle \Big| \sum_{M=j+1}^{\infty} \sum_{\ell=0}^{j_0} \chi_{B_m} (\psi_{\ell, \delta} \phi_{\delta_t, j})_{t, M}(\mathcal{L}, T) \chi_{S_m} f \Big|^2, \omega \right\rangle \\
    &\nonumber \hspace{1cm} + C \sum_{m} \left\langle \Big|\sum_{\ell=j_0+1}^{\infty} \chi_{B_m} (\psi_{\ell, \delta} \phi_{\delta_t, j}) (\mathcal{\sqrt{L}}/t ) \chi_{S_m} f \Big|^2, \omega \right\rangle \\
    &\nonumber =: E_1(t) + E_2(t) + E_3(t) .
\end{align}
Let us first start with the estimate of $E_1(t)$. Using H\"older's inequality we get
\begin{align}
\label{First inequality in the estimate of E1 case}
    E_1(t) &\leq C (j+l_0+1) \sum_{m} \sum_{M =-l_0}^{j} \left\langle \Big|\sum_{\ell=0}^{j_0} \chi_{B_m} (\psi_{\ell, \delta} \phi_{\delta_t, j})_{t, M}(\mathcal{L}, T) \chi_{S_m} f \Big|^2, \omega \right\rangle .
\end{align}
Recall that $B_m := B((x_m,u_m), 12 \cdot 2^{j})$. For each $m \in \mathbb{N}$, let us denote $B_{m,0} := (x_m, u_m)^{-1} B_m$ and $S_{m,0} := (x_m, u_m)^{-1} S_{m}$. Then using the translation invariance property, we see that
\begin{align}
\label{use of translation invariance in inner product}
    \left\langle \Big|\sum_{\ell=0}^{j_0} \chi_{B_m} (\psi_{\ell, \delta} \phi_{\delta_t, j})_{t, M}(\mathcal{L}, T) \chi_{S_m} f \Big|^2, \omega \right\rangle &= \left\langle \Big|\sum_{\ell=0}^{j_0} \chi_{B_{m,0}} (\psi_{\ell, \delta} \phi_{\delta_t, j})_{t, M}(\mathcal{L}, T) \chi_{S_{m,0}} f \Big|^2, \omega \right\rangle .
\end{align}
Also note that $S_{m,0} \subseteq B_{m,0}= B(0, 12 \cdot 2^{j})$. From \eqref{Decomposition of ball into Euclidean balls} there exists $C>0$ such that
\begin{align*}
    B(0, 12 \cdot 2^{j}) \subseteq B^{|\cdot|}(0, C 12 \cdot 2^{j}) \times B^{|\cdot|}(0, C 144 \cdot 2^{2j}) .
\end{align*}
Note in the estimate of $E_1$, we always have $-l_0 \leq M \leq j$. Accordingly, for each $M \in \{-l_0, \ldots, j\}$, we decompose $B^{|\cdot|}(0, C 12 \cdot 2^{j}) \times B^{|\cdot|}(0, C 144 \cdot 2^{2j})$ with respect to the first layer into $N_M$ number of disjoint sets $S_{m,0,n}^{M}$ such that
\begin{align}
\label{Expression: Decomposition of ball into smaller balls}
    B_{m,0} &= \bigcup_{n=1}^{N_{M}} S_{m,0,n}^{M} ,
\end{align}
with the property 
\begin{align}
\label{Property of the smaller balls}
    S_{m,0,n}^{M} \subseteq B^{|\cdot|}\left(x_{m,0, n}^{M}, C 12 \cdot 2^{M}\right) \times B^{|\cdot|}\left(0, C 144 \cdot 2^{2j}\right)
\end{align}
and whenever $n \neq n'$,  $|x_{m,0,n}^{M}-x_{m,0,n'}^{M}|> C 12 \cdot 2^{M}/2 $ holds. Furthermore, the number of subsets $N_{M}$ in this decomposition is bounded by constant times $2^{(j-M)d_1}$. For each $1\leq n \leq N_{M}$ and $\gamma>0$, we also define
\begin{align}
\label{Defination of dilated balls}
    \widetilde{B}_{m,0,n}^{M} &:= B^{|\cdot|}\left(x_{m,0, n}^{M}, C 12 \cdot 2^{M} 2^{\gamma j+1} \right) \times B^{|\cdot|}\left(0, C 144 \cdot 2^{2j}\right) .
\end{align}
Let $N_{\gamma}$ denote the number of overlapping balls $\widetilde{B}_{m,0,n}^{M}$ for $1\leq n \leq N_{M}$. Then this can be estimated as
\begin{align}
\label{Number of overlapping balls}
    N_{\gamma} \leq C 2^{C\gamma j} .
\end{align}
With the aid of the above decomposition \eqref{Expression: Decomposition of ball into smaller balls}, we express $\chi_{S_{m,0}} f$ as follows: 
\begin{align}
\label{Decomposition of f and g}
    \chi_{S_{m,0}} f = \sum_{n=1}^{N_{M}} \chi_{S_{m,0,n}^{M}} f .
\end{align}
The above expression leads us to the following decomposition.
\begin{align*}
    & \chi_{B_{m,0}}(x,u) (\psi_{\ell, \delta} \phi_{\delta_t, j})_{t, M}(\mathcal{L}, T) \chi_{S_{m,0}} f(x,u) \\
    &\nonumber = \sum_{n=1}^{N_M} \chi_{B_{m,0}}(x,u) \chi_{\widetilde{B}_{m,0,n}^{M}}(x,u) (\psi_{\ell, \delta} \phi_{\delta_t, j})_{t, M}(\mathcal{L}, T) \chi_{S_{m,0,n}^M} f(x,u) \\
    &\nonumber \hspace{2cm} + \sum_{n=1}^{N_M} \chi_{B_{m,0}}(x,u) (1-\chi_{\widetilde{B}_{m,0,n}^{M}})(x,u) (\psi_{\ell, \delta} \phi_{\delta_t, j})_{t, M}(\mathcal{L}, T) \chi_{S_{m,0,n}^M} f(x,u) .
\end{align*}
Hence with the help of above decomposition, using the fact \eqref{Number of overlapping balls} and from \eqref{use of translation invariance in inner product} we obtain
\begin{align}
\label{Decomposition the expression into inner and outer part}
    & \left\langle \Big|\sum_{\ell=0}^{j_0} \chi_{B_m} (\psi_{\ell, \delta} \phi_{\delta_t, j})_{t, M}(\mathcal{L}, T) \chi_{S_m} f \Big|^2, \omega \right\rangle \\
    &\nonumber \leq C 2^{C\gamma j} \sum_{n=1}^{N_M} \left\langle \Big|\sum_{\ell=0}^{j_0} \chi_{\widetilde{B}_{m,0,n}^{M}} (\psi_{\ell, \delta} \phi_{\delta_t, j})_{t, M}(\mathcal{L}, T) \chi_{S_{m,0,n}^M} f \Big|^2, \chi_{\widetilde{B}_{m,0,n}^{M}} \omega \right\rangle \\
    &\nonumber \hspace{2cm} + \left\langle \Big|\sum_{\ell=0}^{j_0} \sum_{n=1}^{N_M} \chi_{B_{m,0}} (1-\chi_{\widetilde{B}_{m,0,n}^{M}}) (\psi_{\ell, \delta} \phi_{\delta_t, j})_{t, M}(\mathcal{L}, T) \chi_{S_{m,0,n}^M} f \Big|^2, \chi_{B_{m,0}} \omega \right\rangle .
\end{align}
Therefore plugging the estimate \eqref{Decomposition the expression into inner and outer part} into \eqref{First inequality in the estimate of E1 case} and using H\"older's inequality (since $\frac{1}{r} + \frac{2}{q'}= 1$) implies
\begin{align}
\label{After splitting final estimate of the first term}
    E_1(t) & \leq C j 2^{C\gamma j} \sum_{m} \sum_{M =-l_0}^{j} \sum_{n=1}^{N_M} \left\langle \Big|\sum_{\ell=0}^{j_0} \chi_{\widetilde{B}_{m,0,n}^{M}} (\psi_{\ell, \delta} \phi_{\delta_t, j})_{t, M}(\mathcal{L}, T) \chi_{S_{m,0,n}^M} f \Big|^2, \chi_{\widetilde{B}_{m,0,n}^{M}} \omega \right\rangle \\
    &\nonumber \hspace{0.5cm} + C j \sum_{m} \sum_{M =-l_0}^{j} \left\langle \Big|\sum_{n=1}^{N_M} \sum_{\ell=0}^{j_0} \chi_{B_{m,0}} (1-\chi_{\widetilde{B}_{m,0,n}^{M}}) (\psi_{\ell, \delta} \phi_{\delta_t, j})_{t, M}(\mathcal{L}, T) \chi_{S_{m,0,n}^M} f \Big|^2, \chi_{B_{m,0}} \omega \right\rangle \\
    &\nonumber \leq C j 2^{C\gamma j} \sum_{m} \sum_{M =-l_0}^{j} \sum_{n=1}^{N_M} \| \chi_{\widetilde{B}_{m,0,n}^{M}} \omega\|_{L^{r}} \left\|\sum_{\ell=0}^{j_0} \chi_{\widetilde{B}_{m,0,n}^{M}} (\psi_{\ell, \delta} \phi_{\delta_t, j})_{t, M}(\mathcal{L}, T) \chi_{S_{m,0,n}^M} f \right\|_{L^{q'}}^2 \\
    &\nonumber \hspace{0.5cm} + C j \sum_{m} \sum_{M =-l_0}^{j} \| \chi_{B_{m,0}} \omega\|_{L^{r}} \left\|\sum_{n=1}^{N_M} \sum_{\ell=0}^{j_0} \chi_{B_{m,0}} (1-\chi_{\widetilde{B}_{m,0,n}^{M}}) (\psi_{\ell, \delta} \phi_{\delta_t, j})_{t, M}(\mathcal{L}, T) \chi_{S_{m,0,n}^M} f \right\|_{L^{q'}}^2 .
\end{align}
At this point we stop with our estimate of $E_1(t)$ and let us continue with the other two estimates of $E_2(t)$ and $E_3(t)$.

For the estimate of $E_2(t)$, recall that we have $1/r +2/q' =1$. Hence, applying H\"older's inequality provides
\begin{align}
\label{After splitting estimate of the second term}
    E_2(t) & = C \sum_{m} \left\langle \Big| \sum_{M=j+1}^{\infty} \sum_{\ell=0}^{j_0} \chi_{B_m} (\psi_{\ell, \delta} \phi_{\delta_t, j})_{t, M}(\mathcal{L}, T) \chi_{S_m} f \Big|^2, \chi_{B_m} \omega \right\rangle \\
    &\nonumber \leq C \sum_{m} \|\chi_{B_m} \omega\|_{L^{r}} \left\|\sum_{M=j+1}^{\infty} \sum_{\ell=0}^{j_0} \chi_{B_m} (\psi_{\ell, \delta} \phi_{\delta_t, j})_{t, M}(\mathcal{L}, T) \chi_{S_m} f \right\|_{L^{q'}}^2 .
\end{align}

For the estimate of $E_3(t)$, similarly as earlier using H\"older's inequality with the exponent $r$ and $q'/2$ yields
\begin{align}
\label{After splitting final estimate of the third term}
    E_3(t) &= C \sum_{m} \left\langle \Big|\sum_{\ell=j_0+1}^{\infty} \chi_{B_m} (\psi_{\ell, \delta} \phi_{\delta_t, j}) (\mathcal{\sqrt{L}}/t ) \chi_{S_m} f \Big|^2, \chi_{B_m} \omega \right\rangle \\
    &\nonumber \leq C \sum_{m} \|\chi_{B_m} \omega\|_{L^{r}} \left\| \sum_{\ell=j_0+1}^{\infty} \chi_{B_m} (\psi_{\ell, \delta} \phi_{\delta_t, j}) (\mathcal{\sqrt{L}}/t )\chi_{S_m} f \right\|_{L^{q'}}^2 .
\end{align}

Therefore, combining all the estimates of $E_1(t)$ \eqref{After splitting final estimate of the first term}, $E_2(t) $ \eqref{After splitting estimate of the second term}, $E_3(t)$ \eqref{After splitting final estimate of the third term} and plugging them into the estimate \eqref{Spliting inner product in the sum of M} we obtain
\begin{align*}
    & \left\langle \Big| \phi_{\delta_t, j} (\mathcal{\sqrt{L}}/t ) f \Big|^2, \omega \right\rangle \\
    &\nonumber \leq C j 2^{C\gamma j} \sum_{m} \sum_{M =-l_0}^{j} \sum_{n=1}^{N_M} \| \chi_{\widetilde{B}_{m,0,n}^{M}} \omega\|_{L^{r}} \left\{\sum_{\ell=0}^{j_0} \left\| \chi_{\widetilde{B}_{m,0,n}^{M}} (\psi_{\ell, \delta} \phi_{\delta_t, j})_{t, M}(\mathcal{L}, T) \chi_{S_{m,0,n}^M} f \right\|_{L^{q'}} \right\}^2 \\
    &\nonumber + C j \sum_{m} \sum_{M =-l_0}^{j} \| \chi_{B_{m,0}} \omega\|_{L^{r}} \left\{\sum_{n=1}^{N_M} \sum_{\ell=0}^{j_0} \left\| \chi_{B_{m,0}} (1-\chi_{\widetilde{B}_{m,0,n}^{M}}) (\psi_{\ell, \delta} \phi_{\delta_t, j})_{t, M}(\mathcal{L}, T) \chi_{S_{m,0,n}^M} f \right\|_{L^{q'}} \right\}^2 \\
    &\nonumber \hspace{2cm} + C \sum_{m} \|\chi_{B_m} \omega\|_{L^{r}} \left\{\sum_{M=j+1}^{\infty} \sum_{\ell=0}^{j_0} \left\| \chi_{B_m} (\psi_{\ell, \delta} \phi_{\delta_t, j})_{t, M}(\mathcal{L}, T) \chi_{S_m} f \right\|_{L^{q'}} \right\}^2 \\
    &\hspace{3cm} + C \sum_{m} \|\chi_{B_m} \omega\|_{L^{r}} \left\{\sum_{\ell=j_0+1}^{\infty} \left\| \chi_{B_m} (\psi_{\ell, \delta} \phi_{\delta_t, j}) (\mathcal{\sqrt{L}}/t )\chi_{S_m} f \right\|_{L^{q'}} \right\}^2 .
\end{align*}

With the help of the above estimate, using the fact $\sqrt{a+b+c+d} \leq \sqrt{a}+\sqrt{b}+\sqrt{c}+\sqrt{d}$ for any $a,b,c,d>0$ and applying Minkowski's inequality yields
\begin{align}
\label{Final estimate after finite spped cutoff}
    &\left( \int_{1/\sqrt{2}}^{\sqrt{2}} \left\langle \Big| \phi_{\delta_t, j} (\mathcal{\sqrt{L}}/t ) f \Big|^2, \omega \right\rangle \, dt \right)^{\frac{1}{2}} \\
    &\nonumber \leq C j^{\frac{1}{2}} 2^{C\gamma j} \sum_{\ell=0}^{j_0} \left[ \int_{\frac{1}{\sqrt{2}}}^{\sqrt{2}} \sum_{M =-l_0}^{j} \sum_{n=1}^{N_M} \sum_{m} \| \chi_{\widetilde{B}_{m,0,n}^{M}} \omega\|_{L^{r}} \left\| \chi_{\widetilde{B}_{m,0,n}^{M}} (\psi_{\ell, \delta} \phi_{\delta_t, j})_{t, M}(\mathcal{L}, T) \chi_{S_{m,0,n}^M} f \right\|_{L^{q'}}^2 dt \right]^{\frac{1}{2}} \\
    &\nonumber + C j^{\frac{1}{2}} \sum_{\ell=0}^{j_0} \left[ \int_{\frac{1}{\sqrt{2}}}^{\sqrt{2}} \sum_{m} \| \chi_{B_{m,0}} \omega\|_{L^{r}} \right. \\
    &\nonumber \hspace{3cm} \left. \times \sum_{M =-l_0}^{j} \left\{\sum_{n=1}^{N_M} \left\| \chi_{B_{m,0}} (1-\chi_{\widetilde{B}_{m,0,n}^{M}}) (\psi_{\ell, \delta} \phi_{\delta_t, j})_{t, M}(\mathcal{L}, T) \chi_{S_{m,0,n}^M} f \right\|_{L^{q'}} \right\}^2 \, dt \right]^{\frac{1}{2}} \\
    &\nonumber \hspace{1cm} + C \sum_{M=j+1}^{\infty} \sum_{\ell=0}^{j_0} \left[ \int_{\frac{1}{\sqrt{2}}}^{\sqrt{2}} \sum_{m} \|\chi_{B_m} \omega\|_{L^{r}} \left\| \chi_{B_m} (\psi_{\ell, \delta} \phi_{\delta_t, j})_{t, M}(\mathcal{L}, T) \chi_{S_m} f \right\|_{L^{q'}}^2 \, dt \right]^{\frac{1}{2}} \\
    &\nonumber \hspace{2cm} + C \sum_{\ell=j_0+1}^{\infty} \left[ \int_{\frac{1}{\sqrt{2}}}^{\sqrt{2}} \sum_{m} \|\chi_{B_m} \omega\|_{L^{r}} \left\| \chi_{B_m} (\psi_{\ell, \delta} \phi_{\delta_t, j}) (\mathcal{\sqrt{L}}/t )\chi_{S_m} f \right\|_{L^{q'}}^2 \, dt \right]^{\frac{1}{2}} \\
    &\nonumber =: I_{1}(j) + I_{2}(j) + I_{3}(j) + I_{4}(j) .
\end{align}

In the following we estimate each of $I_{1}(j)$, $I_{2}(j)$, $I_{3}(j)$ and $I_{4}(j)$ seperately.

\medskip
\noindent \textbf{Estimate of \texorpdfstring{$I_{1}(j)$}{}:}
Note that $[\frac{1}{\sqrt{2}}, \sqrt{2}] = \frac{1}{\sqrt{2}}[1,2]$. For $\nu= 0, 1, \ldots, \nu_0= \lfloor 2/\delta \rfloor +1$, let us set 
\begin{align*}
    I_{\nu} &= \tfrac{1}{\sqrt{2}}[1 +\nu \delta, 1 +(\nu+1) \delta] .
\end{align*}
Then we can cover the interval $[1/\sqrt{2}, \sqrt{2}]$ as follows
\begin{align*}
    \left[\frac{1}{\sqrt{2}}, \sqrt{2} \right] \subseteq \bigcup_{\nu=0}^{\nu_0} I_{\nu} .
\end{align*}
Now take a bump function $\Omega$ supported in $(-1,1)$ such that $\sum_{\nu' \in \mathbb{Z}} \Omega(s+\nu') =1$ for all $s \in \mathbb{R}$. In particular taking $s$ to be of the form $\frac{1-s \sqrt{2}}{\delta}$ we get
\begin{align}
\label{Introducing the cutoff omega}
    1 &= \sum_{\nu' \in \mathbb{Z}} \Omega\left(\nu'+ \frac{1-s \sqrt{2}}{\delta} \right) =: \sum_{\nu' \in \mathbb{Z}} \Omega_{\nu'}(s) .
\end{align}
Recall that from \eqref{Introducing theta in multiplier}, $(\psi_{\ell, \delta} \phi_{\delta_t, j})_{t, M}(\eta, \tau) = (\psi_{\ell, \delta} \phi_{\delta_t, j})(\sqrt{\eta}/t) \Theta( 2^{M} \tau)$. Then we write
\begin{align*}
    (\psi_{\ell, \delta} \phi_{\delta_t, j})(s/t) &= \sum_{\nu' \in \mathbb{Z}} (\psi_{\ell, \delta} \phi_{\delta_t, j})(s/t) \Omega_{\nu'}(s) .
\end{align*}
For $t \in I_{\nu}$, The above summand of $\nu'$ is non-zero only if the term $\psi_{\ell, \delta}(s/t) \Omega_{\nu'}(s) \neq 0$, which only happen when $\nu-2^{\ell+6} \leq \nu' \leq \nu + 2^{\ell+ 6}$. Using this fact we get
\begin{align}
\label{Decomposition of operator using Omega}
    (\psi_{\ell, \delta} \phi_{\delta_t, j})_{t, M}(\mathcal{L}, T) &= \sum_{\nu' = \nu- 2^{\ell+6}}^{\nu+ 2^{\ell+6}} (\psi_{\ell, \delta} \phi_{\delta_t, j})_{t, M}(\mathcal{L}, T) \Omega_{\nu'}(\mathcal{\sqrt{L}}) .
\end{align}
In view of the above decomposition we obtain
\begin{align}
\label{Estimate of I1j before two parts}
    I_{1}(j) &\leq C j^{1/2} 2^{C \gamma j} \sum_{\ell=0}^{j_0} \Bigg[ \sum_{M =-l_0}^{j} \sum_{n=1}^{N_M} \sum_{m} \| \chi_{\widetilde{B}_{m,0,n}^{M}} \omega\|_{L^{r}} \sum_{\nu=0}^{\nu_0} \int_{I_{\nu}} \Bigg(\sum_{\nu' = \nu- 2^{\ell+6}}^{\nu+ 2^{\ell+6}} \\
    &\nonumber \hspace{3cm} \left\| \chi_{\widetilde{B}_{m,0,n}^{M}} (\psi_{\ell, \delta} \phi_{\delta_t, j})_{t, M}(\mathcal{L}, T) \Omega_{\nu'}(\mathcal{\sqrt{L}}) (\chi_{S_{m,0,n}^M} f) \right\|_{L^{q'}} \Bigg)^2 \, dt \Bigg]^{\frac{1}{2}} .
\end{align}
Now for the remaining estimate of $I_1(j)$ we divide the proof into two parts, one for M\'etivier groups and the other for Heisenberg-type groups.

\medskip

\noindent \textbf{M\'etivier groups:}
Let us first estimate the $L^2 \to L^{q'}$ operator norm of the following. From \eqref{Definition of s/t in suffix t case} recall that support of $(\psi_{\ell, \delta} \phi_{\delta_t, j})_t$ is contained in $[t(1-2^{\ell+2} \delta), t(1+2^{\ell+2} \delta)]$. For $\beta>1/2$, applying \eqref{Metivier group restriction} of Proposition \ref{Proposition: Truncated Restriction Estimate strong form} yields
\begin{align}
\label{Use of truncated restriction type estimate}
    & \| \chi_{\widetilde{B}_{m,0,n}^{M}} (\psi_{\ell, \delta} \phi_{\delta_t, j})_{t, M}(\mathcal{L}, T) \|_{L^2 \to L^{q'}} \\
    &\nonumber = \|(\psi_{\ell, \delta} \phi_{\delta_t, j})_{t, M}(\mathcal{L}, T) \chi_{\widetilde{B}_{m,0,n}^{M}} \|_{L^{q} \to L^2} \\
    &\nonumber \leq C (t(1+2^{\ell+2} \delta))^{Q(\frac{1}{q}-\frac{1}{2})} 2^{-M d_2 (\frac{1}{q}-\frac{1}{2})} \Big(\|(\psi_{\ell, \delta} \phi_{\delta_t, j})_t(t(1+2^{\ell+2} \delta) \cdot)\|_{L^2} \\
    &\hspace{2cm}\nonumber + 2^{-M \beta \theta_q} \|(\psi_{\ell, \delta} \phi_{\delta_t, j})_t(t(1+2^{\ell+2} \delta) \cdot)\|_{L^2}^{1-\theta_q} \|(\psi_{\ell, \delta} \phi_{\delta_t, j})_t(t(1+2^{\ell+2} \delta) \cdot)\|_{L^2_{\beta}}^{\theta_q} \Big) \\
    &\nonumber \leq C 2^{-M d_2 (\frac{1}{q}-\frac{1}{2})} \Big(\|(\psi_{\ell, \delta} \phi_{\delta_t, j})((1+2^{\ell+2} \delta) \cdot)\|_{L^2} \\
    &\nonumber \hspace{2cm} + 2^{-M \beta \theta_q} \|(\psi_{\ell, \delta} \phi_{\delta_t, j})((1+2^{\ell+2} \delta) \cdot)\|_{L^2}^{1-\theta_q} \|(\psi_{\ell, \delta} \phi_{\delta_t, j})((1+2^{\ell+2} \delta) \cdot)\|_{L^2_{\beta}}^{\theta_q} \Big) .
\end{align}
where in the last line we have used the facts $t \sim 1$ and since $0\leq \ell \leq j_0$, so that $2^{\ell} \delta \leq 1$.

From \eqref{Definition of psi l delta cutoff}, we can see that $\psi_{\ell, \delta}(s) = 0$ if $s \in (1-2^{\ell} \delta, 1+2^{\ell} \delta)$. So that from \eqref{Pointwise bound for phi delta j function}, for any $\beta \geq 0$, we have the following estimate,
\begin{align}
\label{Sobolev norm of psi phi}
    \|(\psi_{\ell, \delta} \phi_{\delta_t, j})((1+2^{\ell+2} \delta) \cdot)\|_{L^2_{\beta}} &\leq C \left\{\begin{array}{ll}
        \delta^{1/2} 2^{(j_0-j)N} 2^{j \beta}, & \quad \ell =0   \\
        \delta^{1/2} 2^{\ell/2} 2^{j-j_0} (2^{j+\ell} \delta)^{-N-1} 2^{j \beta} , & \quad 1 \leq \ell \leq j_0 \\
        \delta^{1/2} 2^{(j_0-j)N} 2^{-\ell N} 2^{j \beta} , & \quad 0 \leq \ell \leq j_0 . \\
    \end{array} \right.
\end{align}
Therefore using the above estimate we get
\begin{align}
\label{Estimate of L2 and L2 sobolev norm}
    & \|(\psi_{\ell, \delta} \phi_{\delta_t, j})((1+2^{\ell+2} \delta) \cdot)\|_{L^2} + 2^{-M \beta \theta_q} \|(\psi_{\ell, \delta} \phi_{\delta_t, j})((1+2^{\ell+2} \delta) \cdot)\|_{L^2}^{1-\theta_q} \\
    &\nonumber \hspace{8cm} \times \|(\psi_{\ell, \delta} \phi_{\delta_t, j})((1+2^{\ell+2} \delta) \cdot)\|_{L^2_{\beta}}^{\theta_q} \\
    &\nonumber \leq C \delta^{1/2} 2^{(j_0-j)N} 2^{-\ell N} + C 2^{-M \beta \theta_q} (\delta^{1/2} 2^{(j_0-j)N} 2^{-\ell N})^{1-\theta_q} (\delta^{1/2} 2^{(j_0-j)N} 2^{-\ell N} 2^{j \beta})^{\theta_q} \\
    &\nonumber \leq C \delta^{1/2} 2^{(j_0-j)N} 2^{-\ell N} \max\{1, 2^{(j-M) \beta \theta_q}\} .
\end{align}
Consequently plugging the estimate \eqref{Estimate of L2 and L2 sobolev norm} into \eqref{Use of truncated restriction type estimate} we obtain
\begin{align}
\label{Operator norm estimate of truncated operator}
    \| \chi_{\widetilde{B}_{m,0,n}^{M}} (\psi_{\ell, \delta} \phi_{\delta_t, j})_{t, M}(\mathcal{L}, T) \|_{L^2 \to L^{q'}} & \leq C \delta^{1/2} 2^{(j_0-j)N} 2^{-\ell N} \max\{1, 2^{(j-M) \beta \theta_q}\} 2^{-M d_2 (\frac{1}{q}-\frac{1}{2})} .
\end{align}
Note that from \eqref{Defination of dilated balls}, we have $|\widetilde{B}_{m,0,n}^{M}| \lesssim 2^{(M d_1 +C \gamma j + 2j d_2)}$ and also recall that $\frac{1}{2r} = \frac{1}{q}-\frac{1}{2}$. Since $\beta>1/2$, we write $\beta=1/2+\varepsilon$ for $\varepsilon>0$. Therefore from \eqref{Operator norm estimate of truncated operator} and using $0\leq M \leq j$ yields
\begin{align*}
    & \left\| \chi_{\widetilde{B}_{m,0,n}^{M}} (\psi_{\ell, \delta} \phi_{\delta_t, j})_{t, M}(\mathcal{L}, T) \Omega_{\nu'}(\mathcal{\sqrt{L}}) (\chi_{S_{m,0,n}^M} f) \right\|_{L^{q'}} \\
    &\leq C \delta^{1/2} 2^{(j_0-j)N} 2^{-\ell N} 2^{(j-M) (\frac{1}{2}+\varepsilon) \theta_q} |\widetilde{B}_{m,0,n}^{M}|^{\frac{1}{2}-\frac{1}{q}} 2^{(M d_1 +C \gamma j + 2j d_2)(\frac{1}{q}-\frac{1}{2})} \\
    &\hspace{8cm} 2^{-M d_2 (\frac{1}{q}-\frac{1}{2})} \|\Omega_{\nu'}(\mathcal{\sqrt{L}}) (\chi_{S_{m,0,n}^M} f ) \|_{L^{2}} \\
    &\leq C \delta^{1/2} 2^{(j_0-j)N} 2^{-\ell N} |\widetilde{B}_{m,0,n}^{M}|^{\frac{1}{2}-\frac{1}{q}} 2^{C \gamma j/(2r)} 2^{j d(\frac{1}{q}-\frac{1}{2})} 2^{(M-j)(d_1- d_2-r \theta_q -2 r \varepsilon \theta_q)/(2r)} \\
    &\hspace{9cm} \|\Omega_{\nu'}(\mathcal{\sqrt{L}}) (\chi_{S_{m,0,n}^M} f ) \|_{L^{2}} .
\end{align*}
Now putting the above estimate into \eqref{Estimate of I1j before two parts} we get
\begin{align}
\label{Estimate of I1j after putting restriction estimate}
    &I_{1}(j) \leq C j^{1/2} 2^{C \gamma j} \delta^{1/2} 2^{(j_0-j)N} 2^{jd(\frac{1}{q}-\frac{1}{2})} \sum_{\ell=0}^{j_0} 2^{-\ell N} \left[\sum_{M =-l_0}^{j} 2^{2(M-j)(d_1- d_2-r \theta_q -2 r \varepsilon \theta_q)/(2r)} \right. \\
    &\nonumber \left. \sum_{m} \sum_{n=1}^{N_M} |\widetilde{B}_{m,0,n}^{M}|^{1-\frac{2}{q}} \| \chi_{\widetilde{B}_{m,0,n}^{M}} \omega\|_{L^{r}} \sum_{\nu=0}^{\nu_0} \int_{I_{\nu}} \left(\sum_{\nu' = \nu- 2^{\ell+6}}^{\nu+ 2^{\ell+6}} \|\Omega_{\nu'}(\mathcal{\sqrt{L}}) (\chi_{S_{m,0,n}^M} f) \|_{L^{2}} \right)^2 \, dt \right]^{\frac{1}{2}} .
\end{align}
Note that, since $1/r=2/q-1$, we have
\begin{align}
\label{Estimate of weight in terms of maximal function}
    \|\chi_{\widetilde{B}_{m,0,n}^{M}} \omega\|_{L^{r}} &= |\widetilde{B}_{m,0,n}^{M}|^{1/r} \left( \frac{1}{|\widetilde{B}_{m,0,n}^{M}|} \int_{\widetilde{B}_{m,0,n}^{M}} \omega^{r} \right)^{1/r} \\
    &\nonumber \leq |\widetilde{B}_{m,0,n}^{M}|^{\frac{2}{q}-1} \inf_{\widetilde{B}_{m,0,n}^{M} \ni (x,u)} \mathcal{M}_{r}^{|\cdot|}\omega(x,u) .
\end{align}
Applying Cauchy-Schwartz inequality, using the fact $|I_{\nu}| \leq C \delta$ and \eqref{Introducing the cutoff omega} yields
\begin{align}
\label{Calculation for Omega to sum nu}
    & \sum_{\nu=0}^{\nu_0} \int_{I_{\nu}} \left(\sum_{\nu' = \nu- 2^{\ell+6}}^{\nu+ 2^{\ell+6}} \|\Omega_{\nu'}(\mathcal{\sqrt{L}}) (\chi_{S_{m,0,n}^M} f) \|_{L^{2}} \right)^2 \, dt \\
    &\nonumber \leq C 2^{\ell} \sum_{\nu \in \mathbb{Z}} \sum_{\nu' = \nu- 2^{\ell+6}}^{\nu+ 2^{\ell+6}} \|\Omega_{\nu'}(\mathcal{\sqrt{L}}) (\chi_{S_{m,0,n}^M} f ) \|_{L^{2}}^2 \left( \int_{I_{\nu}} dt \right) \leq C 2^{2\ell} \delta \|\chi_{S_{m,0,n}^M} f \|_{L^{2}}^2 .
\end{align}
Plugging the estimates \eqref{Estimate of weight in terms of maximal function} and \eqref{Calculation for Omega to sum nu} into the estimate \eqref{Estimate of I1j after putting restriction estimate} and choosing $\epsilon_1>2 r \varepsilon \theta_q>0$ provides
\begin{align}
\label{Final estimate of I1j case}
    I_1(j) &\leq C \delta j^{1/2} 2^{C \gamma j} 2^{(j_0-j)N} 2^{jd(\frac{1}{q}-\frac{1}{2})} \sum_{\ell=0}^{j_0} 2^{-\ell (N-1)} \left[\sum_{M =-l_0}^{j} 2^{2(M-j)(d_1- d_2-r \theta_q -2 r \varepsilon \theta_q+\epsilon_1)/(2r)} \right. \\
    &\nonumber \hspace{5cm} \left. 2^{(j-M)\epsilon_1/r} \sum_{m} \sum_{n=1}^{N_M} \int_{S_{m,0,n}^M} | f(x,u)|^2 \mathcal{M}_{r}^{|\cdot|}\omega(x,u)\, d(x,u) \right]^{\frac{1}{2}} \\
    &\nonumber \leq C \delta j^{1/2} 2^{C \gamma j+j \epsilon_1/2r} 2^{(j_0-j)N} 2^{jd(\frac{1}{q}-\frac{1}{2})} \sum_{\ell=0}^{j_0} 2^{-\ell (N-1)} \left[\sum_{M =-l_0}^{j} 2^{2(M-j)(d_1- d_2-r \theta_q -2 r \varepsilon \theta_q+\epsilon_1)/(2r)} \right. \\
    &\nonumber \hspace{6cm} \left. \int_{\mathbb{R}^d} | f(x,u)|^2 \mathcal{M}_{r}^{|\cdot|}\omega(x,u)\, d(x,u) \right]^{\frac{1}{2}} \\
    &\nonumber \leq C \delta 2^{C \gamma j + j \epsilon_2} 2^{(j_0-j)N} 2^{jd(\frac{1}{q}-\frac{1}{2})} \left( \int_{\mathbb{R}^d} | f(x,u)|^2 \mathcal{M}_{r}^{|\cdot|}\omega(x,u)\, d(x,u) \right)^{\frac{1}{2}} ,
\end{align}
where $\epsilon_2 >0$, which depends on $\epsilon_1$ and provided we choose $N>1$ as well as $d_1-d_2-r \theta_q \geq 0$.

Note that for $\theta_q = p_{d_2}(1-1/q)$ and $\frac{1}{2r} = \frac{1}{q}-\frac{1}{2}$ for $1\leq q \leq 2$, the condition $d_1-d_2-r \theta_q \geq 0$ is equivalent to
\begin{align*}
    1\leq q \leq \frac{p_{d_2} + 2(d_1- d_2)}{p_{d_2}+ d_1 - d_2}:=\mathscr{P}_{d_1, d_2}.
\end{align*}
From \eqref{Reduction of square function to later Sangyuklee}, recall that we have $1\leq q \leq \mathfrak{p}_G'$, where $\mathfrak{p}_G$ is defined as in \eqref{Definition of mathfrak pG}. Since for $(d_1, d_2) \notin \{(4,3), (8,6), (8, 7)\}$, we have $d_1 > 3d_2/2$ (see \cite[Proposition 2.1]{Bagchi_Molla_Singh_Bilinear_Metivier}), which in turn implies that $ p_{d_2}' \leq \mathscr{P}_{d_1, d_2}$. Therefore for $1\leq q \leq p_{d_2}'$, the condition $d_1-d_2-r \theta_q \geq 0$ always hold.

On the other hand, for $(d_1,d_2)=(4,3), (8,6)$ and $(8,7)$ we have $\mathscr{P}_{d_1,d_2}=6/5$, $17/12$ and $14/11$ respectively, which is same as $P_{d_1, d_2}'$. Therefore the condition $d_1-d_2-r \theta_q \geq 0$ also hold for these points.

\medskip

\noindent \textbf{Heisenberg-type groups:}
In case of Heisenberg-type groups, one has follow the same calculation as we just did for M\'etivier group case with $\theta_q =0$ and instead of \eqref{Metivier group restriction} of Proposition \ref{Proposition: Truncated Restriction Estimate strong form} one have to apply \eqref{Heisenberg type restriction} of Proposition \ref{Proposition: Truncated Restriction Estimate strong form}. Therefore, following the same calculation as in M\'etivier group case we get the required estimate for Heisenberg-type groups.

\medskip
\noindent \textbf{Estimate of \texorpdfstring{$I_{2}(j)$}{}:}
Note that we have
\begin{align}
\label{Rewriting of I2 estimate}
    I_{2}(j) &= C j^{1/2} \sum_{\ell=0}^{j_0} \left[ \sum_{m} \| \chi_{B_{m,0}} \omega\|_{L^{r}} \sum_{M =-l_0}^{j} \int_{1/\sqrt{2}}^{\sqrt{2}} \right. \\
    &\nonumber \hspace{1cm} \left. \left\{\sum_{n=1}^{N_M} \left\| \chi_{B_{m,0}} (1-\chi_{\widetilde{B}_{m,0,n}^{M}}) (\psi_{\ell, \delta} \phi_{\delta_t, j})_{t, M}(\mathcal{L}, T) (\chi_{S_{m,0,n}^M} f) \right\|_{L^{q'}} \right\}^2 \, dt \right]^{\frac{1}{2}} .
\end{align}
Let us start with finding the $L^2 \to L^{q'}$ operator norm of the following, which again by duality
\begin{align}
\label{Use of duality in finiding operator norm}
    \| \chi_{B_{m,0}} (1-\chi_{\widetilde{B}_{m,0,n}^{M}}) &(\psi_{\ell, \delta} \phi_{\delta_t, j})_{t, M}(\mathcal{L}, T) \chi_{S_{m,0,n}^M} \|_{L^2 \to L^{q'}} \\
    &\nonumber = \|\chi_{S_{m,0,n}^M} (\psi_{\ell, \delta} \phi_{\delta_t, j})_{t, M}(\mathcal{L}, T) \chi_{B_{m,0}} (1-\chi_{\widetilde{B}_{m,0,n}^{M}}) \|_{L^{q} \to L^2} .
\end{align}
In order to estimate the above, we first calculate the following $L^1 \to L^2$ operator norm:
\begin{align*}
    \|\chi_{S_{m,0,n}^M} (\psi_{\ell, \delta} \phi_{\delta_t, j})_{t, M}(\mathcal{L}, T) \chi_{B_{m,0}} (1-\chi_{\widetilde{B}_{m,0,n}^{M}}) \|_{L^{1} \to L^2} .
\end{align*}
An application of Minkowski's integral inequality implies
\begin{align}
\label{Use of Minkowski inequality in I2 case}
    & \|\chi_{S_{m,0,n}^M} (\psi_{\ell, \delta} \phi_{\delta_t, j})_{t, M}(\mathcal{L}, T) \chi_{B_{m,0}} (1-\chi_{\widetilde{B}_{m,0,n}^{M}}) f\|_{L^2} \\
    &\nonumber \leq \int_{\mathbb{R}^d} |f(y,t)| \Bigg( \int_{\mathbb{R}^d} \chi_{B_{m,0}}(y,t) (1-\chi_{\widetilde{B}_{m,0,n}^{M}})(y,t) \chi_{S_{m,0,n}^M}(x,u) \\
    &\nonumber \hspace{4cm} \times |\mathcal{K}_{(\psi_{\ell, \delta} \phi_{\delta_t, j})_{t, M}(\mathcal{L}, T)}((y,t)^{-1}(x,u))|^2 \, d(x,u) \Bigg)^{1/2}\, d(y,t) .
\end{align}
Note that if $(x,u) \in \supp{\chi_{S_{m,0,n}^M}}$ and $(y,t) \in \supp{\chi_{B_{m,0}} (1-\chi_{\widetilde{B}_{m,0,n}^{M}})}$, then
\begin{align*}
     |x-x_{m,0,n}^{M}| \leq C 2^{M} \quad \text{and} \quad |y-x_{m,0,n}^{M}| \geq C 2^{\gamma j+1} 2^{M} ,
\end{align*}
which in turn again implies $|x-y| \geq C 2^{\gamma j} 2^{M}$.

Therefore, using the above observation along with the translation invariance of the Haar measure and Proposition \ref{Proposition : First layer weighted Plancherel} for $N >Q/2$ we see that,
\begin{align*}
    & \Big( \int_{G} |\chi_{B_{m,0}}(y,t) (1-\chi_{\widetilde{B}_{m,0,n}^{M}})(y,t) \chi_{S_{m,0,n}^M}(x,u) |\mathcal{K}_{(\psi_{\ell, \delta} \phi_{\delta_t, j})_{t, M}(\mathcal{L}, T)}((y,t)^{-1}(x,u))|^2 \ d(x,u) \Big)^{1/2} \\
    &\nonumber \leq C (2^{\gamma j} 2^{M})^{-N} \Big( \int_{G} ||x-y|^{N} \mathcal{K}_{(\psi_{\ell, \delta} \phi_{\delta_t, j})_{t, M}(\mathcal{L}, T)}(x-y, u-t-\tfrac{1}{2}[y,x])|^2 \ d(x,u) \Big)^{1/2} \\
    &\nonumber \leq C 2^{-\gamma j N} 2^{-M N} \Big( \int_{G} ||x|^{N} \mathcal{K}_{(\psi_{\ell, \delta} \phi_{\delta_t, j})_{t, M}(\mathcal{L}, T)}(x, u)|^2 \ d(x,u) \Big)^{1/2} \\
    &\nonumber \leq C 2^{-\gamma j N} 2^{-M N} 2^{M(N-d_2/2)} (t(1+2^{\ell+2} \delta))^{Q/2-N} \|(\psi_{\ell, \delta} \phi_{\delta_t, j})_t(t(1+2^{\ell+2} \delta) \cdot)\|_{L^2} \\
    &\nonumber \leq C 2^{-\gamma j N} 2^{-M d_2/2} \|(\psi_{\ell, \delta} \phi_{\delta_t, j})((1+2^{\ell+2} \delta) \cdot)\|_{L^2} ,
\end{align*}
where we have used $(1+2^{\ell+2} \delta) \geq 1$ and $t \sim 1$.

Recall that $B_{m,0} := B(0, 12 \cdot 2^{j})$. Hence putting the above estimate into \eqref{Use of Minkowski inequality in I2 case}, using \eqref{Sobolev norm of psi phi} with $\beta=0$ and H\"older's inequality we get
\begin{align*}
    & \|\chi_{S_{m,0,n}^M} (\psi_{\ell, \delta} \phi_{\delta_t, j})_{t, M}(\mathcal{L}, T) \chi_{B_{m,0}} (1-\chi_{\widetilde{B}_{m,0,n}^{M}}) f\|_{L^2} \\
    &\leq \delta^{1/2} 2^{(j_0-j)N} 2^{-\ell N} 2^{-\gamma j N} 2^{-M d_2/2} \|\chi_{B_{m,0}} (1-\chi_{\widetilde{B}_{m,0,n}^{M}}) f\|_{L^1} \\
    &\leq \delta^{1/2} 2^{(j_0-j)N} 2^{-\ell N} 2^{-\gamma j N} 2^{-M d_2/2} 2^{jQ(1-\frac{1}{q})} \|\chi_{B_{m,0}} (1-\chi_{\widetilde{B}_{m,0,n}^{M}}) f\|_{L^{q}} .
\end{align*}
Therefore in view of the above estimate and from \eqref{Use of duality in finiding operator norm}, we can see that
\begin{align}
\label{L2 to Lq prime norm for I2 case}
    & \left\| \chi_{B_{m,0}} (1-\chi_{\widetilde{B}_{m,0,n}^{M}}) (\psi_{\ell, \delta} \phi_{\delta_t, j})_{t, M}(\mathcal{L}, T) (\chi_{S_{m,0,n}^M} f)\right\|_{L^{q'}} \\
    &\nonumber \leq C \delta^{1/2} 2^{(j_0-j)N} 2^{-\ell N} 2^{-\gamma j N} 2^{-M d_2/2} 2^{jQ(1-\frac{1}{q})} \|\chi_{S_{m,0,n}^M} f\|_{L^2} .
\end{align}
On the other hand, since $1/r=2/q-1$, we have
\begin{align}
\label{Estimate of weight in terms of original maximal function}
    \|\chi_{B_{m,0}} \omega\|_{L^{r}} &= |B_{m,0}|^{1/r} \left( \frac{1}{|B_{m,0}|} \int_{B_{m,0}} \omega^{r} \right)^{1/r} \leq |B_{m,0}|^{\frac{2}{q}-1} \inf_{B_{m,0} \ni (x,u)} \mathcal{M}_{r}\omega(x,u) .
\end{align}
Hence plugging the estimates \eqref{L2 to Lq prime norm for I2 case}, \eqref{Estimate of weight in terms of original maximal function} into \eqref{Rewriting of I2 estimate}, using H\"older's inequality and the fact $N_M \lesssim 2^{j d_1}$ (see just below of \eqref{Property of the smaller balls}) we obtain
\begin{align}
\label{Final estimate of I2j case}
    I_{2}(j) &\leq C j^{1/2} \delta^{1/2} 2^{(j_0-j)N} 2^{-\gamma j N} 2^{jQ(1-\frac{1}{q})} \sum_{\ell=0}^{j_0} 2^{-\ell N} \left[ \sum_{m} |B_{m,0}|^{\frac{2}{q}-1} \right. \\
    &\nonumber \hspace{2cm} \left. \inf_{B_{m,0} \ni (x,u)} \mathcal{M}_{r}\omega(x,u) \sum_{M =-l_0}^{j} 2^{-M d_2} \left(\sum_{n=1}^{N_M} \|\chi_{S_{m,0,n}^M} f\|_{L^2} \right)^2 \int_{1/\sqrt{2}}^{\sqrt{2}} dt \right]^{\frac{1}{2}} \\
    &\nonumber \leq C j^{1/2} \delta^{1/2} 2^{(j_0-j)N} 2^{-\gamma j N} 2^{jQ(1-\frac{1}{q})} 2^{jQ(\frac{1}{q}-\frac{1}{2})} 2^{j d_1/2} \\
    &\nonumber \hspace{4cm} \left[ \sum_{M =-l_0}^{j} 2^{-M d_2} \sum_{m} \int_{B_{m,0}} | f(x,u)|^2 \mathcal{M}_{r}\omega(x,u) \, d(x,u) \right]^{\frac{1}{2}} \\
    &\nonumber \leq C \delta 2^{(j_0-j)(N+1/2)} 2^{-\gamma j N} 2^{j(Q/2+d_1/2+1/2+\epsilon_3)} \left( \int_{\mathbb{R}^d} | f(x,u)|^2 \mathcal{M}_{r}\omega(x,u)\, d(x,u) \right)^{\frac{1}{2}} ,
\end{align}
where $\epsilon_3>0$ and we have used the fact $\delta^{-1} \sim 2^{j_0}$ and we choose $N>0$.

\medskip
\noindent \textbf{Estimate of \texorpdfstring{$I_{3}(j)$}{}:}
Similarly as in the estimate of $I_1(j)$, using \eqref{Decomposition of operator using Omega} we have
\begin{align}
\label{First estimate of I1jk equation cutoff}
    I_{3}(j) &\leq C \sum_{M=j+1}^{\infty} \sum_{\ell=0}^{j_0} \left[ \sum_{m} \|\chi_{B_m} \omega\|_{L^{r}} \sum_{\nu=0}^{\nu_0} \int_{I_{\nu}} \Bigg(\sum_{\nu' = \nu- 2^{\ell+6}}^{\nu+ 2^{\ell+6}} \right. \\
    &\nonumber \hspace{3cm} \left. \left\| \chi_{B_m} (\psi_{\ell, \delta} \phi_{\delta_t, j})_{t, M}(\mathcal{L}, T) \Omega_{\nu'}(\mathcal{\sqrt{L}}) \left(\chi_{S_m} f \right) \right\|_{L^{p'}} \Bigg)^2 \, dt \right]^{\frac{1}{2}} .
\end{align}
Recall that $|B_m| \lesssim 2^{j Q}$. Then with the help of the estimate \eqref{Operator norm estimate of truncated operator} for $\beta>1/2$ we obtain
\begin{align}
\label{Operator norm estimate for I3 case}
    & \| \chi_{B_m} (\psi_{\ell, \delta} \phi_{\delta_t, j})_{t, M}(\mathcal{L}, T) \Omega_{\nu'}(\mathcal{\sqrt{L}}) \left(\chi_{S_m} f \right) \|_{L^{q'}} \\
    &\nonumber \leq C \delta^{1/2} 2^{(j_0-j)N} 2^{-\ell N} \max\{1, 2^{(j-M) \beta \theta_q}\} |B_m|^{\frac{1}{2}-\frac{1}{q}} 2^{jQ(\frac{1}{q}-\frac{1}{2})} 2^{-M d_2 (\frac{1}{q}-\frac{1}{2})} \|\Omega_{\nu'}(\mathcal{\sqrt{L}}) \left(\chi_{S_m} f \right) \|_{L^{2}} .
\end{align}
Also note that similar to \eqref{Estimate of weight in terms of original maximal function} here also
\begin{align}
\label{Creating maximal function in I3 case}
    \|\chi_{B_{m}} \omega\|_{L^{r}} &\leq |B_{m}|^{\frac{2}{q}-1} \inf_{B_{m} \ni (x,u)} \mathcal{M}_{r}\omega(x,u) .
\end{align}
Plugging the above two estimates \eqref{Operator norm estimate for I3 case} and \eqref{Creating maximal function in I3 case} into \eqref{First estimate of I1jk equation cutoff} yields
\begin{align}
\label{Estimate of I1 j k after putting pointwise bound cutoff}
    I_{3}(j) &\leq C \delta^{1/2} 2^{(j_0-j)N} 2^{jQ(\frac{1}{q}-\frac{1}{2})} \sum_{M=j+1}^{\infty} 2^{-M d_2 (\frac{1}{q}-\frac{1}{2})} \sum_{\ell=0}^{j_0} 2^{-\ell N} \left[ \sum_{m} \right. \\
    &\nonumber \hspace{1cm} \left. \inf_{B_{m} \ni (x,u)} \mathcal{M}_{r}\omega(x,u) \sum_{\nu=0}^{\nu_0} \int_{I_{\nu}} \left(\sum_{\nu' = \nu- 2^{\ell+6}}^{\nu+ 2^{\ell+6}} \|\Omega_{\nu'}(\mathcal{\sqrt{L}}) \left(\chi_{S_m} f \right) \|_{L^{2}} \right)^2 \, dt \right]^{\frac{1}{2}} .
\end{align}
Similarly as in \eqref{Calculation for Omega to sum nu} here we have
\begin{align}
\label{Estimate after changing nu and nu prime cutoff}
    \sum_{\nu=0}^{\nu_0} \int_{I_{\nu}} \left(\sum_{\nu' = \nu- 2^{\ell+6}}^{\nu+ 2^{\ell+6}} \|\Omega_{\nu'}(\mathcal{\sqrt{L}}) \left(\chi_{S_m} f \right) \|_{L^{2}} \right)^2 \, dt  &\leq C 2^{2\ell} \delta \|\chi_{S_m} f \|_{L^{2}}^2 .
\end{align}
Now putting the above estimate \eqref{Estimate after changing nu and nu prime cutoff} into \eqref{Estimate of I1 j k after putting pointwise bound cutoff} gives
\begin{align}
\label{Final estimate of I3j case}
    I_{3}(j) &\leq C \delta 2^{(j_0-j)N} 2^{jQ(\frac{1}{q}-\frac{1}{2})} \sum_{M=j+1}^{\infty} 2^{-M d_2 (\frac{1}{q}-\frac{1}{2})} \sum_{\ell=0}^{j_0} 2^{-\ell (N-1)} \left[ \sum_{m} \int_{S_m} |f(x,u)|^2 \mathcal{M}_{r}\omega(x,u) \, d(x,u) \right]^{\frac{1}{2}} \\
    &\nonumber \leq C \delta 2^{(j_0-j)N} 2^{j d (\frac{1}{q}-\frac{1}{2})} \left( \int_{\mathbb{R}^d} |f(x,u)|^2 \mathcal{M}_{r}\omega(x,u) \, d(x,u) \right)^{\frac{1}{2}} ,
\end{align}
by choosing $N>1$.

\medskip
\noindent \textbf{Estimate of \texorpdfstring{$I_{4}(j)$}{}:}
In this case we have
\begin{align*}
    I_4(j) &\leq C \sum_{\ell=j_0+1}^{\infty} \left( \int_{1/\sqrt{2}}^{\sqrt{2}} \sum_{m} \|\chi_{B_m} \omega\|_{L^{r}} \left\| \chi_{B_m} (\psi_{\ell, \delta} \phi_{\delta_t, j}) (\mathcal{\sqrt{L}}/t )\chi_{S_m} f \right\|_{L^{q'}}^2 \, dt \right)^{\frac{1}{2}} .
\end{align*}
Note that $|B_m| \lesssim 2^{j Q}$. Using the fact \eqref{Definition of s/t in suffix t case} and Corollary \eqref{Corollary: Truncated restriction equation weak form I} we have
\begin{align}
\label{Operator norm for I4 case}
    & \| \chi_{B_m} (\psi_{\ell, \delta} \phi_{\delta_t, j})(\mathcal{\sqrt{L}}/t ) \|_{L^2 \to L^{q'}} = \|(\psi_{\ell, \delta} \phi_{\delta_t, j})_t (\mathcal{\sqrt{L}}) \chi_{B_m} \|_{L^{q} \to L^2} \\
    &\nonumber \leq C ( t(1+2^{\ell+2} \delta))^{Q(\frac{1}{q}-\frac{1}{2})} \|(\psi_{\ell, \delta} \phi_{\delta_t, j})_t(t(1+2^{\ell+2} \delta) \cdot)\|_{L^{\infty}} \\
    &\nonumber \leq C |B_m|^{\frac{1}{2}-\frac{1}{q}} (2^{j} (1+2^{\ell+2} \delta))^{Q(\frac{1}{q}-\frac{1}{2})} \|(\psi_{\ell, \delta} \phi_{\delta_t, j})((1+2^{\ell+2} \delta) \cdot)\|_{L^{\infty}} .
\end{align}
With the help of \eqref{Pointwise bound for phi delta j function}, for $\ell \geq j_0 +1$ one can see that
\begin{align}
\label{L infinity norm of psi phi}
    \|(\psi_{\ell, \delta} \phi_{\delta_t, j})((1+2^{\ell+2} \delta) \cdot)\|_{L^{\infty}} &\leq C 2^{j-j_0} (2^{j+\ell} \delta)^{-N} ,
\end{align}
for any $N \geq 0$.

Consequently, putting the above estimate \eqref{L infinity norm of psi phi} into \eqref{Operator norm for I4 case} implies
\begin{align}
\label{L2 to Lq prime norm estimate in I4 case}
    & \| \chi_{B_m} (\psi_{\ell, \delta} \phi_{\delta_t, j}) (\mathcal{\sqrt{L}}/t )\chi_{S_m} f \|_{L^{q'}} \\
    &\nonumber \leq C 2^{j-j_0} (2^{j+\ell} \delta)^{-N} |B_m|^{\frac{1}{2}-\frac{1}{q}} (2^{j} (1+2^{\ell+2} \delta))^{Q(\frac{1}{q}-\frac{1}{2})} \|\chi_{S_m} f\|_{L^2} .
\end{align}
Since $\ell \geq j_0 +1$ and $2^{-j_0 -1} \leq \delta < 2^{-j_0}$, so that $2^{\ell} \delta >1$. Therefore using the above estimate \eqref{L2 to Lq prime norm estimate in I4 case} and \eqref{Creating maximal function in I3 case}, we obtain
\begin{align}
\label{Final estimate of I4j case}
    I_4(j) &\leq C \sum_{\ell = j_0 + 1}^{\infty} 2^{j-j_0} (2^{j+\ell} \delta)^{Q(\frac{1}{q}-\frac{1}{2})-N} \Biggl(\sum_{m} \int_{S_m} |f(x,u)|^2 \mathcal{M}_{r}\omega(x,u) \, d(x,u) \int_{1/\sqrt{2}}^{\sqrt{2}} dt \Biggr)^{\frac{1}{2}} \\
    &\nonumber \leq C \delta 2^{j[Q(\frac{1}{q}-\frac{1}{2})-N+1]} \sum_{2^{\ell} \delta > 1} (2^{\ell} \delta)^{Q(\frac{1}{q}-\frac{1}{2})-N} \left( \int_{\mathbb{R}^d} |f(x,u)|^2 \mathcal{M}_{r}\omega(x,u) \, d(x,u) \right)^{\frac{1}{2}} \\
    &\nonumber \leq C \delta 2^{j[Q(\frac{1}{q}-\frac{1}{2})-N+1]} \left( \int_{\mathbb{R}^d} |f(x,u)|^2 \mathcal{M}_{r}\omega(x,u) \, d(x,u) \right)^{\frac{1}{2}} ,
\end{align}
provided we choose $N>Q(\frac{1}{q}-\frac{1}{2})$.

Finally, combining all the estimates $I_1(j)$ \eqref{Final estimate of I1j case}, $I_2(j)$ \eqref{Final estimate of I2j case}, $I_3(j)$ \eqref{Final estimate of I3j case}, $I_4(j)$ \eqref{Final estimate of I4j case} and putting them into \eqref{Final estimate after finite spped cutoff} and then again plugging \eqref{Final estimate after finite spped cutoff} into the estimate \eqref{Final estimate after further decomposition cutoff} we get
\begin{align*}
    & \left(\int_{\mathbb{R}^d} |\mathfrak{S}_{\delta, loc}^{\phi}(\mathcal{L})f(x,u)|^2 \omega(x,u) \, d(x,u) \right)^{1/2} \\
    &\leq C \delta 2^{j_0 N} \sum_{j \geq j_0} 2^{-j\{N-d(\frac{1}{q}-\frac{1}{2})-C \gamma -\epsilon_2\}} \left( \int_{\mathbb{R}^d} | f(x,u)|^2 \mathcal{M}_{r}^{|\cdot|}\omega(x,u)\, d(x,u) \right)^{\frac{1}{2}} \\
    &+ C \Big[ \delta 2^{j_0 (N+1/2)} \sum_{j \geq j_0} 2^{-j\{(N+1/2)+\gamma N-(Q/2+d_1/2+1/2+\epsilon_3)\}} + \delta 2^{j_0 N} \sum_{j \geq j_0} 2^{-j\{N-d (\frac{1}{q}-\frac{1}{2})\}} \\
    &\hspace{4cm} + \delta \sum_{j \geq j_0} 2^{-j\{N-Q(\frac{1}{q}-\frac{1}{2})-1\}} \Big] \left( \int_{\mathbb{R}^d} |f(x,u)|^2 \mathcal{M}_{r}\omega(x,u) \, d(x,u) \right)^{\frac{1}{2}} \\
    &\leq C \delta 2^{j_0\{d(\frac{1}{q}+\frac{1}{2})+C \gamma+\epsilon_2\}} \left( \int_{\mathbb{R}^d} | f(x,u)|^2 \mathcal{M}_{r}^{|\cdot|}\omega(x,u)\, d(x,u) \right)^{\frac{1}{2}} + C \delta \Big[ 2^{-j_0\{\gamma N-(Q/2+2d_1+1/2+\epsilon_3)\}} \\
    &\hspace{3cm} + 2^{j_0\{d (\frac{1}{q}-\frac{1}{2})\}} + 2^{-j_0\{N-Q(\frac{1}{q}-\frac{1}{2})-1\}} \Big] \left( \int_{\mathbb{R}^d} |f(x,u)|^2 \mathcal{M}_{r}\omega(x,u) \, d(x,u) \right)^{\frac{1}{2}} \\
    &\leq C \delta^{1-d(\frac{1}{q}-\frac{1}{2})-\epsilon/2} \left( \int_{\mathbb{R}^d} | f(x,u)|^2 \mathcal{M}_{r}^{|\cdot|}\omega(x,u)\, d(x,u) \right)^{\frac{1}{2}} \\
    &\hspace{4cm} + C \delta^{1-d (\frac{1}{q}-\frac{1}{2})} \left( \int_{\mathbb{R}^d} |f(x,u)|^2 \mathcal{M}_{r}\omega(x,u) \, d(x,u) \right)^{\frac{1}{2}} ,
\end{align*}
where $\epsilon/2=C \gamma+\epsilon_2>0$ is very small, which we get by choosing $\gamma, \epsilon_2, \epsilon_3>0$ sufficiently small and choosing $N$ sufficiently large. We have also used $\delta^{-1} \sim 2^{j_0}$.

Hence we obtain
\begin{align*}
    \int_{\mathbb{R}^d} |\mathfrak{S}_{\delta, loc}^{\phi}(\mathcal{L})f(x,u)|^2 \omega(x,u) \, d(x,u) &\leq C \delta^{2-d(\frac{2}{q}-1)-\epsilon} \int_{\mathbb{R}^d} | f(x,u)|^2 \mathcal{M}_{r}^{|\cdot|}\omega(x,u)\, d(x,u) \\
    &\hspace{1.5cm} + C \delta^{2-d (\frac{2}{q}-1)} \int_{\mathbb{R}^d} |f(x,u)|^2 \mathcal{M}_{r}\omega(x,u) \, d(x,u) .
\end{align*}
This completes the proof of the Lemma \ref{Lemma: Weighted L2 estimate for square function}.    
\end{proof}

\medskip

For the proof of Theorem \ref{Theorem: Lp to L2 local smoothing for Metivier} we need to consider slightly different square function then the one defined in \eqref{Definition: Square function}. Let $\zeta : \mathbb{R}^2 \to \mathbb{R}$ be a function which is non-vanishing and smooth on the set
\begin{align*}
    \{(\eta, t) : \tfrac{1}{2} \leq \eta \leq 4 , \tfrac{1}{2} < t < \tfrac{5}{2} \} .
\end{align*}
For $\phi$ as in \eqref{Definition: Square function}, set
\begin{align*}
    \widetilde{\phi}\left(\frac{t-\eta}{\delta} \right) = \phi\left(\zeta(\eta, t) \frac{t-\eta}{\delta} \right) .
\end{align*}
Then we define
\begin{align*}
    \widetilde{\phi}\left(\frac{t-\mathcal{L}}{\delta} \right)f(x,u) &= \frac{1}{(2\pi)^{d_2}} \int_{\mathfrak{g}_{2,r}^{*}} \sum_{\mathbf{k} \in \mathbb{N}^\Lambda} \phi\left(\zeta(\eta_{\mathbf{k}}^{\lambda}, t) \frac{t-\eta_{\mathbf{k}}^{\lambda}}{\delta} \right) \left[f^{\lambda} \times_{\lambda} \varphi_{\mathbf{k}}^{\mathbf{b}^{\lambda}, \mathbf{r}}(R_{\lambda}^{-1}\cdot) \right](x) \, e^{i \langle \lambda, u \rangle} \, d\lambda .
\end{align*}

Now corresponding to the $\widetilde{\phi}$, we consider the local square function with localized frequency denoted by $\mathfrak{S}_{\delta, loc}^{\widetilde{\phi}}(\mathcal{L})$ and defined as in \eqref{Definition: Square function} with $\phi$ replaced with $\widetilde{\phi}$. Then we have the following result.

\begin{proposition}
\label{Proposition: Generalized Square function estimate}
Let $\mathfrak{p}_G \leq p < \infty$. Whenever $\alpha>\alpha_d(p)$ we have
\begin{align*}
    \|\mathfrak{S}_{\delta, loc}^{\widetilde{\phi}}(\mathcal{L})f\|_{L^p} &\leq C \delta^{\frac{1}{2}-\alpha} \|f\|_{L^p} .
\end{align*}
Moreover,
\begin{align*}
    \|\mathfrak{S}_{\delta, loc}^{\widetilde{\phi}}(\mathcal{L})f\|_{L^2} &\leq C \delta^{\frac{1}{2}} \|f\|_{L^2} .
\end{align*}
    
\end{proposition}

\begin{proof}
Since $\zeta$ is smooth and non-vanishing on the set where $\eta \sim 1$ and $t \sim 1$. Therefore the same proof as of Proposition \ref{Proposition: Local Square function estimate} will go through with obvious modification.
\end{proof}

As an application of the $L^p$-boundedness of the local square function with localized frequency $\mathfrak{S}_{\delta, loc}^{\phi}(\mathcal{L})$ we obtain the following result, which will be useful later in the proof of Theorem \ref{Theorem: Bilinear maximal on Metivier}.
\begin{proposition}
\label{Prop: Maximal Lp norm estimates of localized}
Whenever $\mathfrak{p}_G \leq p<\infty$ and $\alpha>\alpha_d(p)$ we have
\begin{align}
\label{Maximal Lp estimate for p bigger than pg}
    \left\|\sup_{t>0} \left|\phi\left(\delta^{-1} \left(1-\frac{\mathcal{L}}{t^2} \right) \right) f \right| \right\|_{L^p} &\leq C \delta^{-\alpha} \|f\|_{L^p} ,
\end{align}   
and
\begin{align}
\label{Maximal L2 estimate}
    \left\|\sup_{t>0} \left|\phi\left(\delta^{-1} \left(1-\frac{\mathcal{L}}{t^2} \right) \right) f \right| \right\|_{L^2} &\leq C \|f\|_{L^2} .
\end{align}
If we take $\phi^{\beta}(t) = t^{\beta} \phi(t)$ for $0 \leq \beta \leq N$, then the same conclusion \eqref{Maximal Lp estimate for p bigger than pg} and \eqref{Maximal L2 estimate} hold with $\phi$ replaced by $\phi^{\beta}$.

In particular, if $\psi^k(s) = s^{-k} \psi(s)$ for $k>0$ and $\psi \in C_c^{\infty}([1/2,2])$, then for $\mathfrak{p}_G \leq p<\infty$ or $p=2$ and $\alpha>\alpha_d(p)$ we have
\begin{align}
\label{Maximal Lp estimate with negative power}
    \left\|\sup_{t>0} \left|\psi^k\left(\delta^{-1} \left(1-\frac{\mathcal{L}}{t^2} \right) \right) f \right| \right\|_{L^p} &\leq C \, 2^{N+k} k^{N+1} \delta^{-\alpha} \|f\|_{L^p} ,
\end{align}
where $N$ is some fixed positive number.
\end{proposition}

\begin{proof}
From \cite[Lemma 1, p. 499]{Stein_Harmonic_Analysis_1993}, we have
\begin{align*}
    \sup_{1\leq t \leq 2} |F(t)| &\leq \delta^{-1/2} \left(\int_{1}^2 |F(t)|^2 \, dt \right)^{1/2} + \delta^{1/2} \left(\int_{1}^2 |F'(t)|^2 \, dt \right)^{1/2} .
\end{align*}
Therefore
\begin{align*}
    &\sup_{t>0} \left|\phi\left(\delta^{-1} \left(1-\frac{\mathcal{L}}{t^2} \right) \right) f(x,u) \right| = \sup_{k \in \mathbb{Z}} \sup_{1\leq t \leq 2} \left|\phi\left(\delta^{-1} \left(1-\frac{\mathcal{L}}{(2^k t)^2} \right) \right) f(x,u) \right| \\
    &\leq \sup_{k \in \mathbb{Z}} \Bigg[\delta^{-1/2} \left(\int_1^{2} \Big| \phi\left(\delta^{-1} \left(1-\frac{\mathcal{L}}{(2^k t)^2} \right) \right) f(x,u) \Big|^2 \, dt \right)^{1/2} \\
    &\hspace{5cm} + \delta^{1/2} \left(\int_1^{2} \Big| \frac{\partial}{\partial t}\phi\left(\delta^{-1} \left(1-\frac{\mathcal{L}}{(2^k t)^2} \right) \right) f(x,u) \Big|^2 \, dt \right)^{1/2} \Bigg] .
\end{align*}
Note that for $1\leq t \leq 2$, the factor $\delta \frac{\partial}{\partial t}\phi\left(\delta^{-1} \left(1-\frac{\mathcal{L}}{(2^k t)^2} \right) \right)$ satisfies same quantitative estimate as $\phi\left(\delta^{-1} \left(1-\frac{\mathcal{L}}{(2^k t)^2} \right) \right)$. Hence, using Proposition \ref{Proposition: Local Square function estimate} we get the required estimate \eqref{Maximal Lp estimate for p bigger than pg} and \eqref{Maximal L2 estimate}. 

Since $\phi \in C_c^{\infty}$ and $0\leq \beta\leq N$, estimate for $\phi^{\beta}$ follows easily by repeating the same argument as of \eqref{Maximal Lp estimate for p bigger than pg} and \eqref{Maximal L2 estimate} with $\phi$ replaced by $\phi^{\beta}$ with obvious modification.

On the other hand, to estimate \eqref{Maximal Lp estimate with negative power}, similarly as above, applying Proposition \ref{Proposition: Local Square function estimate} for $\mathfrak{p}_G \leq p<\infty$ or $p=2$ and $\alpha>\alpha_d(p)$ we can see that
\begin{align}
\label{Calculation for maximal estimate with negative}
    \left\|\sup_{t>0} \left|\psi^k\left(\delta^{-1} \left(1-\frac{\mathcal{L}}{t^2} \right) \right) f \right| \right\|_{L^p} &\leq C \, \sup_{s \in [1/2,2], 0\leq m\leq N} \left|\frac{d^m \psi^k}{ds^m}(s) \right| \delta^{-\alpha} \|f\|_{L^p} ,
\end{align}
where $N$ is some fixed positive number.

Since $\psi \in C_c^{\infty}([1/2,2])$, it satisfies
\begin{align}
\label{Calculation for psi with negetive factor outside}
    \sup_{s \in [1/2,2], 0\leq m\leq N} \left|\frac{d^m \psi^k}{ds^m}(s) \right| &\leq C 2^{N+k} k^{N+1} .
\end{align}
Now plugging the above estimate \eqref{Calculation for psi with negetive factor outside} into \eqref{Calculation for maximal estimate with negative} we get \eqref{Maximal Lp estimate with negative power}.
\end{proof}

\medskip

Now we consider the discrete version of the local square function with localized frequency \eqref{Definition: Square function}. Let $\phi$ be as in \eqref{Definition: Square function}. Corresponding to this $\phi$ and for every $\delta \in (0,1/4]$, we define the discrete square function by
\begin{align*}
    \mathfrak{D}_{\delta}^{\phi}(\mathcal{L})f(x,u) &= \left(\sum_{\nu \in \delta \mathbb{Z} \cap [0,2]} \Big|\phi\left(\frac{\nu-\mathcal{L}}{\delta} \right)f(x,u) \Big|^2 \right)^{1/2} .
\end{align*}
As an another application of Proposition \ref{Proposition: Local Square function estimate} we get the $L^p$-boundedness of the discrete square function $\mathfrak{D}_{\delta}^{\phi}(\mathcal{L})$. In fact we the following result.

\begin{proposition}
\label{Proposition: Discrete square function estimate}
Let $\mathfrak{p}_G \leq p <\infty$. Whenever $\alpha>\alpha_d(p)$ we have
\begin{align}
\label{Lp boundedness of square function for bigger p}
    \|\mathfrak{D}_{\delta}^{\phi}(\mathcal{L})f\|_{L^p} &\leq C \delta^{-\alpha} \|f\|_{L^p} .
\end{align}
Moreover,
\begin{align}
\label{L2 boundedness of discrete square function}
    \|\mathfrak{D}_{\delta}^{\phi}(\mathcal{L})f\|_{L^2} &\leq C \|f\|_{L^2} .
\end{align}    
\end{proposition}

\begin{proof}
Estimate of \eqref{L2 boundedness of discrete square function} is easy. Indeed, using the spectral decomposition of $\mathcal{L}$ (see \eqref{Definition: general spectral multipler on Metivier}), orthogonality and the support of $\phi_{\delta, \nu}$ we get the estimate \eqref{L2 boundedness of discrete square function}.

Hence it remains to estimate \eqref{Lp boundedness of square function for bigger p}. First we divide the interval $[0,2]$ as $[0,2] = [0, 4\delta] \cup [4\delta, 2]$ and decompose $[4\delta, 2]$ as follows
\begin{align}
\label{Decomposition of interval into two and further}
    [4\delta, 2] & = \bigcup_{\ell=-1}^{\ell_0} I_{\ell} ,
\end{align}
where $I_{\ell}= [2^{-\ell-1}, 2^{-\ell}] \cap [4\delta, 2]$ and $\ell_0+1$ is the smallest non-negative integer such that $[2^{-\ell_0-1}, 2^{-\ell_0}] \cap [4\delta, 2] = \emptyset$.

We set $\phi_{\delta,\nu}(\eta)=\phi(\delta^{-1}(\nu-\eta))$. Therefore in view of the above decomposition we have
\begin{align}
\label{Decompositing the interval 01 into dyadically}
    \mathfrak{D}_{\delta}^{\phi}(\mathcal{L})f(x,u) &\leq \left( \sum_{\nu \in \delta \mathbb{Z} \cap [0, 4\delta]} |\phi_{\delta, \nu}(\mathcal{L})f(x,u)|^2 \right)^{1/2} + \sum_{\ell=-1}^{\ell_0} \left( \sum_{\nu \in \delta \mathbb{Z} \cap I_{\ell}} |\phi_{\delta, \nu}(\mathcal{L})f(x,u)|^2 \right)^{1/2} .
\end{align}
Let us first estimate the second term in the right hand side of the above expression. For the case $\ell=0$, similarly as in \cite[Lemma 2.3]{Jeong_Lee_Vargas_Bilinear_Bochner_Riesz_2018} we obtain
\begin{align*}
    & \left(\sum_{\nu \in \delta \mathbb{Z} \cap I_0} \left|\phi\left(\frac{\nu-\mathcal{L}}{\delta} \right)f(x,u) \right|^2 \right)^{1/2} \\
    &\leq \delta^{-\frac{1}{2}} \left[ \left( \int_{1/2}^{2} \left| \phi\left(\frac{t-\mathcal{L}}{\delta} \right)f(x,u) \right|^2 \, dt \right)^{1/2} + \left( \int_{1/2}^{2} \left| \phi'\left(\frac{t-\mathcal{L}}{\delta} \right)f(x,u) \right|^2 \, dt \right)^{1/2} \right] .
\end{align*}
Since $\phi$ is $C_c^{\infty}$ and so does $\phi'$. Therefore applying Proposition \ref{Proposition: Local Square function estimate} for $\phi$ and $\phi'$, whenever $\mathfrak{p}_G \leq p<\infty$ and $\alpha>\alpha_d(p)$,
\begin{align}
 \label{Estimate of square function for k zero}
    \left\|\left(\sum_{\nu \in \delta \mathbb{Z} \cap I_0} \left|\phi\left(\frac{\nu-\mathcal{L}}{\delta} \right)f \right|^2 \right)^{1/2} \right\|_{L^{p}} &\leq C \delta^{-\frac{1}{2}} \delta^{\frac{1}{2}-\alpha} \|f\|_{L^{p}} = C \delta^{-\alpha} \|f\|_{L^{p}} .
\end{align}
Now we estimate the second term in the right hand side of \eqref{Decompositing the interval 01 into dyadically} for $\ell=-1,1, \ldots, \ell_0$. Let us calculate the following. Using the spectral decomposition of $\mathcal{L}$ (see \eqref{Definition: general spectral multipler on Metivier}), making change of variable $(z,s) \mapsto (2^{-\ell/2}z, 2^{-\ell}s)$ and $\lambda \mapsto 2^{\ell} \lambda$ we have
\begin{align}
\label{Dilation for Phij, beta}
    & \phi_{2^{\ell} \delta, 2^{\ell}\nu}(\mathcal{L}) (\delta_{2^{\ell/2}} f)(\delta_{2^{-\ell/2}}(x,u)) \\
    &\nonumber = \frac{1}{(2\pi)^{d_2}} \int_{\mathfrak{g}_{2,r}^{*}} \int_{\mathbb{R}^{d_1}} \int_{\mathbb{R}^{d_2}} \sum_{\mathbf{k} \in \mathbb{N}^\Lambda} \phi(\delta^{-1} 2^{-\ell}(2^{\ell} \nu-\eta_{\mathbf{k}}^{\lambda})) f(2^{\ell/2} z,2^{\ell}s) \\
    &\nonumber \hspace{4cm} \varphi_{\mathbf{k}}^{\mathbf{b}^{\lambda}, \mathbf{r}}(R_{\lambda}^{-1}(2^{-\ell/2} x-z)) e^{-i\langle \lambda, s \rangle} e^{\frac{i}{2}\omega_{\lambda}(2^{-\ell/2} x, z)} \ e^{i \langle \lambda, 2^{-\ell} u \rangle} \, ds \, dz \, d\lambda \\
    &\nonumber = \frac{2^{-\ell(Q/2-d_2)}}{(2\pi)^{d_2}} \int_{\mathfrak{g}_{2,r}^{*}} \int_{\mathbb{R}^{d_1}} \int_{\mathbb{R}^{d_2}} \sum_{\mathbf{k} \in \mathbb{N}^\Lambda} \phi(\delta^{-1}(\nu-\eta_{\mathbf{k}}^{\lambda})) f(z,s) \\
    &\nonumber \hspace{6cm} 2^{\ell d_1/2} \varphi_{\mathbf{k}}^{\mathbf{b}^{\lambda}, \mathbf{r}}(R_{\lambda}^{-1}(x- z)) e^{-i\langle \lambda, s \rangle} e^{\frac{i}{2}\omega_{\lambda}(x, z)} \, e^{i \langle \lambda, u \rangle} \, ds \, dz \, d\lambda \\
    &\nonumber= \phi_{\delta, \nu}(\mathcal{L}) f(x,u) ,
\end{align}
where we have used the following facts: Since the functions $\lambda \mapsto b_n^{\lambda}$ are homogeneous of degree $1$ and the functions $\lambda \mapsto R_{\lambda}$ are homogeneous of degree $0$ respectively, hence we get (see \cite[eq. (3.7), (3.8)]{Molla_Singh_Commutator_Metivier_Arxiv})
\begin{align*}
    \eta_{\mathbf{k}}^{2^{\ell} \lambda} = 2^{\ell} \eta_{\mathbf{k}}^{\lambda}, \quad \text{and} \quad \varphi_{\mathbf{k}}^{\mathbf{b}^{2^{\ell} \lambda}, \mathbf{r}}(R_{2^{\ell} \lambda}^{-1} 2^{-\ell/2} x) = 2^{\ell d_1/2} \varphi_{\mathbf{k}}^{\mathbf{b}^{\lambda}, \mathbf{r}}(R_{\lambda}^{-1} x) ,
\end{align*}
and using the fact \cite[eq. (3.4) of Proposition 3.4]{Niedorf_p_specific_Heisenberg_group_2024} we also have
\begin{align*}
    \omega_{2^{\ell} \lambda}(2^{-\ell/2} x, 2^{-\ell/2} z) = \omega_{\lambda}(x, z) .
\end{align*}
Therefore using \eqref{Dilation for Phij, beta} and \eqref{Estimate of square function for k zero} for $\mathfrak{p}_G \leq p<\infty$, $\alpha>\alpha_d(p)$ and since $2^k \delta \leq 1/4$ for $k \geq -1$, we obtain
\begin{align}
\label{Estimate for the other ks}
    \left\|\left( \sum_{\nu \in \delta \mathbb{Z} \cap I_{\ell}} |\phi_{\delta, \nu}(\mathcal{L})f|^2 \right)^{1/2} \right\|_{L^{p}} &= 2^{\frac{\ell Q}{2p}} \left\|\left( \sum_{\nu \in 2^{\ell} \delta \mathbb{Z} \cap I_0} |\phi_{2^{\ell} \delta, \nu}(\mathcal{L}) (\delta_{2^{\ell/2}} f)(x,u)|^2 \right)^{1/2} \right\|_{L^{p}} \\
    &\nonumber\leq C \delta^{-\alpha} 2^{-\ell \alpha} \|f\|_{L^{p}} .
\end{align}
Since $\mathfrak{p}_G \leq p<\infty$, we have $\alpha>\alpha_d(p)>0$. So that using the above estimate and \eqref{Estimate of square function for k zero} we get
\begin{align}
\label{Estimte of square function when k is there}
    \sum_{\ell=-1}^{\ell_0} \left\| \left( \sum_{\nu \in \delta \mathbb{Z} \cap I_{\ell}} |\phi_{\delta, \nu}(\mathcal{L})f|^2 \right)^{1/2} \right\|_{L^{p}} &\leq C \sum_{\ell=-1}^{\ell_0} \delta^{-\alpha} 2^{-\ell \alpha} \|f\|_{L^{p}} \leq C \delta^{-\alpha} \|f\|_{L^{p}} .
\end{align}
Hence it remains to estimate the first factor in the right hand side of \eqref{Decompositing the interval 01 into dyadically}. Since in this case $0\leq \nu \leq 4\delta$, so that $\supp \phi_{\delta, \nu} \subseteq [0, 4\delta]$. Therefore for any $N>0$, Lemma \eqref{Lemma: Pointwise kernel estimate for linear kernel} yields
\begin{align*}
    |\mathcal{K}_{\phi_{\delta, \nu}}(x,u)| &\leq C \delta^{Q/2} (1+\delta^{1/2} \|(x,u)\|)^{-N} .
\end{align*}
The above estimate immediately tells that whenever $0\leq \nu \leq 4\delta$, choosing $N>Q$ we have $\|\mathcal{K}_{\phi_{\delta, \nu}}\|_{L^1} \leq C$. Then an application of Young's convolution inequality gives
\begin{align*}
    \|\phi_{\delta, \nu}(\mathcal{L})f\|_{L^{p}} \leq C \|f\|_{L^{p}} .
\end{align*}
Since the number of terms such that $\nu \in \delta \mathbb{Z} \cap [0, 4\delta]$ is bounded, we can conclude that
\begin{align}
\label{Estimate of square function when nu is near zero}
    \left\|\Bigg( \sum_{\nu \in \delta \mathbb{Z} \cap [0, 4\delta]} |\phi_{\delta, \nu}(\mathcal{L})f|^2 \Bigg)^{1/2} \right\|_{L^{p}} &\leq C \|f\|_{L^{p}} .
\end{align}
Combining the estimates \eqref{Estimte of square function when k is there} and \eqref{Estimate of square function when nu is near zero} completes the proof for the case $\mathfrak{p}_G \leq p<\infty$.
\end{proof}

\subsection{Square function estimates on M\'etivier groups}
\label{Subsection Third section}

In this subsection, we discuss about the $L^p$-boundedness of the global version of the local square function with localized frequency $\mathfrak{S}_{\delta, loc}^{\phi}(\mathcal{L})$ (Proposition \ref{Proposition: Square function estimate}) as well as other square function estimate (Proposition \ref{Proposition: Lp boundedness of square function with bump}) and some estimate of the maximal operator related to square function of Bochner-Riesz operator (Proposition \ref{Proposition: Maximal square function estimate}).

We define the square function (global) with localized frequency by
\begin{align*}
    \mathfrak{S}_{\delta}^{\phi}(\mathcal{L})f(x,u) &= \left(\int_0^{\infty} \Big| \phi\left(\delta^{-1} \left(1-\frac{\mathcal{L}}{t^2} \right) \right) f(x,u) \Big|^2 \frac{dt}{t} \right)^{1/2} .
\end{align*}

In the following we show that the $L^p$-boundedness of $\mathfrak{S}_{\delta, loc}^{\phi}(\mathcal{L})$ and $\mathfrak{S}_{\delta}^{\phi}(\mathcal{L})$ are equivalent. In fact, note that enough to prove that Proposition \ref{Proposition: Local Square function estimate} implies the following result.

\begin{proposition}
\label{Proposition: Square function estimate}
Let $\mathfrak{p}_G \leq p < \infty$. Whenever $\alpha>\alpha_d(p)$ we have
\begin{align}
\label{Square function estimate for p biggar than infinity}
    \|\mathfrak{S}_{\delta}^{\phi}(\mathcal{L})f\|_{L^p} &\leq C \delta^{\frac{1}{2}-\alpha} \|f\|_{L^p} .
\end{align}
Moreover,
\begin{align}
\label{Square function estimate for p=2 case}
    \|\mathfrak{S}_{\delta}^{\phi}(\mathcal{L})f\|_{L^2} &\leq C \delta^{\frac{1}{2}} \|f\|_{L^2} .
\end{align}
If we take $\phi^{\beta}(t) = t^{\beta} \phi(t)$ for $0 \leq \beta \leq N$, then the same conclusion \eqref{Square function estimate for p biggar than infinity} and \eqref{Square function estimate for p=2 case} hold with $\phi$ replaced by $\phi^{\beta}$.

In particular, if $\psi^k(s) = s^{-k} \psi(s)$ for $k>0$ and $\psi \in C_c^{\infty}([1/2,2])$, then for $\mathfrak{p}_G \leq p<\infty$ or $p=2$ and $\alpha>\alpha_d(p)$ we have
\begin{align}
\label{Square function Lp estimate with negative power}
    \left\|\mathfrak{S}_{\delta}^{\psi^k}(\mathcal{L})f \right\|_{L^p} &\leq C \, 2^{N+k} k^{N+1} \delta^{\frac{1}{2}-\alpha} \|f\|_{L^p} ,
\end{align}
where $N$ is some fixed positive number.
    
\end{proposition}

\begin{proof}
The proof follows from adapting the ideas of \cite[Proposition 4.2]{Guo_Roos_Yung_Sharp_Variation_2020} into our setup. The proof goes exactly in the same line, only difference is that in order to prove eq. (7.9) of  \cite{Guo_Roos_Yung_Sharp_Variation_2020} in our case, one has to use Lemma \ref{Lemma: Pointwise kernel estimate for linear kernel} and Lemma \ref{lemma: outside distance}.
\end{proof}

\medskip
Next we consider the following maximal operator
\begin{align*}
    \mathcal{M}^{\alpha}(\mathcal{L})f(x,u) &= \sup_{R>0} \left(\frac{1}{R} \int_{0}^R |S_t^{\alpha}(\mathcal{L})f(x,u)|^2 \, dt \right)^{1/2} .
\end{align*}

As a application of Theorem \ref{Theorem: Stein square estimate}, we have the following $L^p$-boundedness result for the operator $\mathcal{M}^{\alpha}(\mathcal{L})$, which will be very useful in the proof of Theorem \ref{Theorem: Bilinear maximal on Metivier} and Theorem \ref{Theorem: Bilinear Stein square function estimate}.

\begin{proposition}
\label{Proposition: Maximal square function estimate}
Let $\mathfrak{p}_G \leq p< \infty$ or $p=2$. Then whenever $\alpha>\alpha_d(p)-1/2$ we have
\begin{align*}
    \|\mathcal{M}^{\alpha}(\mathcal{L})f\|_{L^p} &\leq C \|f\|_{L^p} .
\end{align*}
    
\end{proposition}

\begin{proof}
Let us write
\begin{align*}
    S_t^{\alpha}(\mathcal{L})f &= \sum_{k=1}^{N} (S_t^{\alpha+k-1}(\mathcal{L})f-S_t^{\alpha+k}(\mathcal{L})f) + S_t^{\alpha+N}(\mathcal{L})f ,
\end{align*}
where $N>0$ will be chosen later.

Hence we obtain
\begin{align}
\label{Creating difference technique in square function}
    \left(\frac{1}{R} \int_0^{R} |S_t^{\alpha}(\mathcal{L})f(x,u)|^2 \, dt \right)^{1/2} &\leq \sum_{k=1}^N \left(\int_0^{\infty} |S_t^{\alpha+k-1}(\mathcal{L})f(x,u)-S_t^{\alpha+k}(\mathcal{L})f(x,u)|^2 \, \frac{dt}{t} \right)^{1/2} \\
    &\nonumber \hspace{3.5cm} + \left(\frac{1}{R}\int_0^{R} |S_t^{\alpha+N}(\mathcal{L})f(x,u)|^2 \, dt \right)^{1/2} .
\end{align}
First note that Stein's square function can also be written as
\begin{align}
\label{Equivalent expression of Stein square function}
    \mathfrak{S}^{\alpha}(\mathcal{L})f(x, u) &= 2\alpha \left( \int_0^{\infty} \left| S_t^{\alpha}(\mathcal{L})f(x,u) - S_t^{\alpha-1}(\mathcal{L})f(x,u) \right|^2 \, \frac{dt}{t} \right)^{1/2} .
\end{align}
Therefore from \eqref{Creating difference technique in square function} we see that
\begin{align*}
    \mathcal{M}^{\alpha}(\mathcal{L})f(x,u) &\leq \sum_{k=1}^N  \frac{1}{2(\alpha+k)} \mathfrak{S}^{\alpha+k}(\mathcal{L})f(x,u) + S_{*}^{\alpha+N}(\mathcal{L})f(x,u) .
\end{align*}
Choose $N>0$ large such that $\alpha+N>\frac{Q-1}{2}$, then $S_{*}^{\alpha+N}$ is bounded on $L^p(G)$ for $2\leq p\leq \infty$ (see \cite[Corollary 2.8]{Mauceri_Meda_Multipliers_Stratified_groups_1990}). On the other hand, from Theorem \ref{Theorem: Stein square estimate} we have $\mathfrak{S}^{\alpha+k}(\mathcal{L})$ is bounded on $L^p(G)$ for $\mathfrak{p}_G \leq p<\infty$ or $p=2$ if $\alpha+k>\alpha_d(p)+1/2$, that is, $\alpha>\alpha_d(p)-1/2$ for all $k \geq 1$. Combining both these results we get the required estimate.
\end{proof}

Let $\varphi \in C_c^{\infty}(\mathbb{R})$ such that it is supported on $[1,2]$. Then we end this subsection by considering the following square function associated to $\varphi$, which is defined by
\begin{align}
\label{Square function corresponding to bump function}
    \mathcal{G}_{\varphi}(\mathcal{L})f(x,u) &= \left(\int_{0}^{\infty} \left|\varphi\left(R^{-1}\mathcal{L}\right)f(x,u)\right|^2 \frac{dR}{R} \right)^{1/2} .
\end{align}

Regarding the $L^p$-boundedness of $\mathcal{G}_{\varphi}(\mathcal{L})$ we have the following result, which will play an important role in the proof of Theorem \ref{Theorem: Analogue of Niedorf theorem}.

\begin{proposition}
\label{Proposition: Lp boundedness of square function with bump}
Let $\varphi \in C_c^{\infty}(\mathbb{R})$ with $\supp{\varphi} \subseteq [1,2]$. Then for $1<p<\infty$,
\begin{align}
\label{Equivalence of other square function for phi}
    C^{-1} \|f\|_{L^p} \leq \|\mathcal{G}_{\varphi}(\mathcal{L})f\|_{L^p} \leq C \|f\|_{L^p} .
\end{align}
Moreover, $\mathcal{G}_{\varphi}(\mathcal{L})$ is of weak-type $(1,1)$.
    
\end{proposition}

\begin{proof}
The proof of the above proposition can be easily completed using the similar ideas as of \cite[Theorem 1.1]{Chen_Duong_Yan_Stein_Square_function_Homogeneous_2013}. There it was proved for Stein's square function on space of homogeneous type. Since $\varphi \in C_c^{\infty}(\mathbb{R})$ one can easily adapt the proof of \cite[Theorem 1.1]{Chen_Duong_Yan_Stein_Square_function_Homogeneous_2013} to our setup. In fact, for $1<p<2$ one have to use Proposition \ref{Prop: Weighted Plancherel using weight and distance} and Proposition \ref{Proposition: Lp boundedness criteria from p less than 2 case}; while for $2<p<\infty$ one have to use Proposition \ref{Proposition: Lp Boundedness for p bigger than 2 case} and Lemma \ref{Lemma: L2 norm of the heat kernel}.
\end{proof}

\subsection{Proof of Theorem \ref{Theorem: Stein square estimate} and Theorem \ref{Theorem: Stein square function for p less that 2 case}: Stein's square function on M\'etivier groups}
\label{Subsection: Stein square function}
This subsection is devoted to the proof of Theorem \ref{Theorem: Stein square estimate} and Theorem \ref{Theorem: Stein square function for p less that 2 case}. The idea of the proof are as follows: Boundedness of $\mathfrak{S}^{\alpha}(\mathcal{L})$ at $p=2$ follows from orthogonality (\emph{Plancherel theorem}), and for $\mathfrak{p}_G \leq p<\infty$ can be deduced from the $L^p$-boundedness of the square function $\mathfrak{S}_{\delta}^{\phi}(\mathcal{L})$. While for the range $1<p<2$, first we prove weak-type $(1,1)$ boundedness of $\mathfrak{S}^{\alpha}(\mathcal{L})$ and then interpolating with the case $p=2$ we get the required estimate.

\begin{proof}[Proof of Theorem \ref{Theorem: Stein square estimate}]

Let $\Psi \in C_c^{\infty}(1/2,2)$ such that for any $s>0$ it satisfies
\begin{align*}
    s^{\alpha-1} = \sum_{j \in \mathbb{Z}} 2^{-j (\alpha-1)} \Psi(2^j s) .
\end{align*}
Therefore, for $\eta>0$ we have
\begin{align*}
    \frac{\eta}{t^2} \left(1-\frac{\eta}{t^2} \right)_{+}^{\alpha-1} &= \sum_{j=0}^{\infty} 2^{-j (\alpha-1)} \Psi\left(2^j\left(1-\frac{\eta}{t^2} \right) \right) + \sum_{j=0}^{\infty} 2^{-j \alpha} \left(2^j\left(1-\frac{\eta}{t^2} \right) \right) \Psi\left(2^j\left(1-\frac{\eta}{t^2} \right) \right) .
\end{align*}
Recall we defined $\Psi^1(s) = s \Psi(s)$. Note that $t \frac{\partial}{\partial t} \left(1-\frac{\eta}{t^2} \right)_{+}^{\alpha} = 2 \alpha \frac{\eta}{t^2} \left(1-\frac{\eta}{t^2} \right)_{+}^{\alpha-1}$ and hence we get
\begin{align*}
    \mathfrak{S}^{\alpha}(\mathcal{L})f(x, u) &\leq 2\alpha \sum_{j=0}^{\infty} 2^{-j (\alpha-1)} \left( \int_0^{\infty} \left|\Psi\left(2^j \left(1-\frac{\mathcal{L}}{t^2} \right)\right) \right|^2 \, \frac{dt}{t} \right)^{1/2} \\
    &\hspace{2cm} + 2\alpha \sum_{j=0}^{\infty} 2^{-j \alpha} \left( \int_0^{\infty} \left|\Psi^1 \left(2^j \left(1-\frac{\mathcal{L}}{t^2} \right)\right) \right|^2 \, \frac{dt}{t} \right)^{1/2} .
\end{align*}
Now taking $L^p$-norm in both side of the above estimate and applying Proposition \ref{Proposition: Square function estimate} for $\mathfrak{p}_G \leq p<\infty$ we obtain
\begin{align*}
    \|\mathfrak{S}^{\alpha}(\mathcal{L})f\|_{L^p} &\leq C \sum_{j=0}^{\infty} 2^{-j (\alpha-1)} 2^{-j(\frac{1}{2}-\alpha_d(p)-\epsilon)} \|f\|_{L^p} \leq C \|f\|_{L^p} ,
\end{align*}
where $\epsilon>0$ and since $\alpha> d(\frac{1}{2}-\frac{1}{p})$.

Now it remains to prove Theorem \ref{Theorem: Stein square estimate} for $p=2$. If we set $\phi_t^{\alpha}(s) = 2\alpha \frac{s}{t^2}(1-\frac{s}{t^2})_{+}^{\alpha-1}$, then orthogonality yields
\begin{align*}
    \|\mathfrak{S}^{\alpha}(\mathcal{L})f\|_{L^2}^2 &= \int_0^{\infty} \int_{\mathfrak{g}_{2,r}^{*}} \sum_{\mathbf{k} \in \mathbb{N}^\Lambda} |\phi_t^{\alpha}(\eta_{\mathbf{k}}^{\lambda})|^2 \|f^{\lambda} \times_{\lambda} \varphi_{\mathbf{k}}^{\mathbf{b}^{\lambda}, \mathbf{r}}(R_{\lambda}^{-1}\cdot)\|_{L^2}^2 \, d\lambda \, \frac{dt}{t} .
\end{align*}
Note, from \cite[p. 278]{Stein_Weiss_Fourier_Analysis_book} whenever $\alpha>\frac{1}{2}$ we have
\begin{align*}
    \int_{0}^{\infty} |\phi_t^{\alpha}(\eta_{\mathbf{k}}^{\lambda})|^2 \, \frac{dt}{t} &= 4 \alpha^2 \int_{\sqrt{\eta_{\mathbf{k}}^{\lambda}}}^{\infty} \left(\tfrac{\sqrt{\eta_{\mathbf{k}}^{\lambda}}}{t} \right)^4 \left(1-\left(\tfrac{\sqrt{\eta_{\mathbf{k}}^{\lambda}}}{t} \right)^2 \right)^{2\alpha-2} \, \frac{dt}{t} \\
    &= 4\alpha^2 [2(2\alpha-1)(2\alpha+1)]^{-1} ,
\end{align*}
Therefore with the help of the above estimate whenever $\alpha>\frac{1}{2}$ we obtain
\begin{align*}
    \|\mathfrak{S}^{\alpha}(\mathcal{L})f\|_{L^2}^2 &\leq C \int_{\mathfrak{g}_{2,r}^{*}} \sum_{\mathbf{k} \in \mathbb{N}^\Lambda} \|f^{\lambda} \times_{\lambda} \varphi_{\mathbf{k}}^{\mathbf{b}^{\lambda}, \mathbf{r}}(R_{\lambda}^{-1}\cdot)\|_{L^2}^2 \, d\lambda \leq C \|f\|_{L^2}^2 .
\end{align*}
This completes the proof of Theorem \ref{Theorem: Stein square estimate}.
\end{proof}

\medskip
\begin{proof}[Proof of Theorem \ref{Theorem: Stein square function for p less that 2 case}]
First note that to prove Theorem \ref{Theorem: Stein square function for p less that 2 case}, enough to show whenever $\alpha>(d+1)/2$, the operator $\mathfrak{S}^{\alpha}(\mathcal{L})$ is of weak-type $(1,1)$. Since the $L^p$-boundedness of $\mathfrak{S}^{\alpha}(\mathcal{L})$ for $1<p<2$ follows from the interpolation of Theorem \ref{Theorem: Stein square estimate} with $p=2$ case and the weak-type $(1,1)$ boundedness. Moreover, in order to prove weak-type $(1,1)$ boundedness for $\mathfrak{S}^{\alpha}(\mathcal{L})$, we apply Proposition \ref{Proposition: Criterion for weak boundedness} with $F(\mathcal{L})=\mathfrak{S}^{\alpha}(\mathcal{L})$.

From \eqref{Definition of approximate identity} recall that $A_t = \exp{(-t^2 \mathcal{L})}$. Let us write
\begin{align}
\label{Expression of square function with approximate identity}
    \mathfrak{S}^{\alpha}(\mathcal{L})(I-A_t)f(x,u) &= \left( \int_0^{\infty} | F^{\alpha}_{R,t}(\sqrt{\mathcal{L}})f(x, u) |^2 \, \frac{dR}{R} \right)^{1/2} ,
\end{align}
where
\begin{align}
\label{Definition of F R t alpha}
    F^{\alpha}_{R,t}(\eta) &= R \frac{\partial}{\partial R} S^{\alpha}_R(\eta^2) (1-e^{-t^2 \eta^2}) .
\end{align}
Let $\varphi \in C_c^{\infty}(0, \infty)$ compactly supported smooth function supported on $[1/4,1]$ such that
\begin{align*}
    \sum_{j \in \mathbb{Z}} \varphi(2^j \eta) = 1 \quad \quad \text{for any} \ \ \eta>0 .
\end{align*}
Since $\supp{F^{\alpha}_{R,t}} \subseteq [0,R]$ and $\supp{\varphi} \subseteq [1/4,1]$, using the above partition of unity we obtain
\begin{align}
\label{Use of partition of unity in F R alpha}
    F^{\alpha}_{R,t}(\sqrt{\mathcal{L}})f &= \sum_{j=-1}^{\infty} F^{\alpha}_{R,t, j}(\sqrt{\mathcal{L}})f ,
\end{align}
where $F^{\alpha}_{R,t, j}(\eta)= F^{\alpha}_{R,t}(\eta) \varphi(2^j \eta/R)$.

Therefore plugging \eqref{Use of partition of unity in F R alpha} into \eqref{Expression of square function with approximate identity} and applying Minkowski's integral inequality yields
\begin{align*}
    \mathfrak{S}^{\alpha}(\mathcal{L})(I-A_t)f(x,u)  &\leq \sum_{j=-1}^{\infty} \int_G \left(\int_{0}^{\infty} |\mathcal{K}_{F^{\alpha}_{R,t,j}(\sqrt{\mathcal{L}})}((y,t)^{-1}(x,u))|^2 \, \frac{dR}{R} \right)^{1/2} |f(y,t)|\, d(y,t) \\
    &=: \int_G |K_t^{\alpha}((y,t)^{-1}(x,u))| |f(y,t)|\, d(y,t) ,
\end{align*}
where 
\begin{align*}
    |K_t^{\alpha}(x,u)| &:= \sum_{j=-1}^{\infty} \left(\int_{0}^{\infty} |\mathcal{K}_{F^{\alpha}_{R,t,j}(\sqrt{\mathcal{L}})}(x,u)|^2 \, \frac{dR}{R} \right)^{1/2} .
\end{align*}
Hence, in view of Proposition \ref{Proposition: Criterion for weak boundedness} enough to prove that, whenever $\alpha>(d+1)/2$ we have
\begin{align}
\label{To show for weak 1 1 proof}
    \sup_{t>0} \int_{\|(x,u)\|\geq t} |K_t^{\alpha}(x,u)| \, d(x,u) &\leq C .
\end{align}
In order to estimate the above, we decompose it further as follows:
\begin{align}
\label{Dominate Kt alpha by Ijk sum}
    \int_{\|(x,u)\|\geq t} |K_t^{\alpha}(x,u)| \, d(x,u) &\leq \sum_{j=-1}^{\infty} \sum_{k\in \mathbb{Z}} \int_{\|(x,u)\|\geq t} \left(\int_{2^k}^{2^{k+1}} |\mathcal{K}_{F^{\alpha}_{R,t,j}(\sqrt{\mathcal{L}})}(x,u)|^2 \, \frac{dR}{R} \right)^{1/2} \, d(x,u) \\
    &\nonumber =: \sum_{j=-1}^{\infty} \sum_{k\in \mathbb{Z}} I_{j,k} .
\end{align}
Let us start with estimating $I_{j,k}$ for $j\geq -2$ and $k \in \mathbb{Z}$. Suppose $\beta, \sigma, \gamma>0$. An application of H\"older's inequality yields
\begin{align}
\label{Estimate of Ijk part}
    I_{j,k} &\leq \left(\int_{2^k}^{2^{k+1}} \int_{G} |(1+2^{-j}R\|(x,u)\|)^{\beta+\sigma} (1+2^{-j}R|x|)^{\gamma} \mathcal{K}_{F^{\alpha}_{R,t,j}(\sqrt{\mathcal{L}})}(x,u)|^2 \, d(x,u) \, \frac{dR}{R} \right)^{1/2} \\
    &\nonumber \hspace{4cm} \times (1+2^{k-j} t)^{-\beta} \left(\int_{G} \frac{d(x,u)}{(1+2^{k-j} \|(x,u)\|)^{2\sigma} (1+2^{k-j} |x|)^{2\gamma}} \right)^{1/2} .
\end{align}
Note that since $\supp{F^{\alpha}_{R,t,j}} \subseteq [2^{-j}R/4,2^{-j}R]$, from Proposition \ref{Prop: Weighted Plancherel using weight and distance} for all $\beta, \sigma, \varepsilon>0$ and $0\leq \gamma<d_2/2$ we have
\begin{align}
\label{Application of weighted plancherel}
    & \int_{G} |(1+2^{-j}R\|(x,u)\|)^{\beta+\sigma} (1+2^{-j}R|x|)^{\gamma} \mathcal{K}_{F^{\alpha}_{R,t,j}(\sqrt{\mathcal{L}})}(x,u)|^2 \, d(x,u) \\
    &\nonumber \leq C 2^{-jQ} R^Q \|F^{\alpha}_{R,t,j}(2^{-j}R \cdot)\|_{L^2_{\beta+\sigma+\varepsilon}}^2 \\
    &\nonumber \leq C t^{-Q} \max\{1,(2^{-j} Rt)^Q\} \|F^{\alpha}_{R,t,j}(2^{-j}R \cdot)\|_{L^2_{\beta+\sigma+\varepsilon}}^2 .
\end{align}
If $m \in N$ such that $m>\beta+\sigma+\varepsilon$, then for any $N \geq 0$ from \eqref{Definition of F R t alpha} we obtain
\begin{align}
\label{Computation of the sobolev norm}
    \|F^{\alpha}_{R,t,j}(2^{-j}R \cdot)\|_{L^2_{\beta+\sigma+\varepsilon}} &\leq C \|(2^{-j} \eta)^2 \varphi(\eta) (1-2^{-2j} \eta^2)_{+}^{\alpha-1}\|_{L^2_{\beta+\sigma+\varepsilon}} \|(1-e^{-(2^{-j}Rt)^2 \eta^2})\|_{C^{m}([1/4,1])} \\
    &\nonumber \leq C 2^{-2j} \min\{1, (2^{-j}Rt)^{2N}\} ,
\end{align}
where in the last line we have used Lemma \ref{Lemma: Finiteness of Sobolev norm for Bochner-Riesz} for $\alpha>\beta+\sigma+\epsilon+1/2$.

On the other hand from Lemma \ref{Lemma: Integration of distance and weight} for $\sigma+\gamma>Q/2$ and $0\leq \gamma<d_1/2$ yields
\begin{align}
\label{Distance weight integral}
    \int_{G} \frac{d(x,u)}{(1+2^{k-j} \|(x,u)\|)^{2\sigma} (1+2^{k-j} |x|)^{2\gamma}} &\leq C 2^{(-k+j) Q} .
\end{align}
Note that, since $G$ is M\'etivier group, we always have $d_1>d_2$ (see \cite[p. 12]{Bagchi_Molla_Singh_Bilinear_Metivier}). So that $\sigma+\gamma>Q/2$ and $0\leq \gamma<d_2/2$ implies $\sigma>d/2$. Since we have $\alpha>(d+1)/2$, by choosing $\beta, \epsilon>0$ sufficiently small we can make sure that $\alpha> \beta+\sigma+\epsilon+1/2$.

Hence plugging the estimates \eqref{Application of weighted plancherel}, \eqref{Computation of the sobolev norm} and \eqref{Distance weight integral} into \eqref{Estimate of Ijk part} for $\alpha>(d+1)/2$ we obtain
\begin{align*}
    I_{j,k} &\leq C \Bigg(\int_{2^k}^{2^{k+1}} t^{-Q} \max\{1,(2^{-j} Rt)^Q\} 2^{-4j} \min\{1, (2^{-j}Rt)^{4N}\} \, \frac{dR}{R} \Bigg)^{1/2} (2^{k-j} t)^{-\beta} 2^{(-k+j) Q/2} \\
    &\leq C \Bigg(\int_{2^k}^{2^{k+1}} \max\{1,(2^{-j} Rt)^Q\} 2^{-4j} \min\{1, (2^{-j}Rt)^{4N}\} (2^{k-j}t)^{-2\beta-Q} \, \frac{dR}{R} \Bigg)^{1/2} \\
    &\leq C 2^{-2j} \max\{1,(2^{k} s_j)^{Q/2}\} \min\{1, (2^{k}s_j)^{2N}\} (2^{k}s_j)^{-\beta-Q/2} ,
\end{align*}
where $s_j=2^{-j}t$.

Therefore with the help of above estimate and from \eqref{Dominate Kt alpha by Ijk sum} whenever $\alpha>(d+1)/2$ we have
\begin{align*}
    \int_{\|(x,u)\|\geq t} |K_t^{\alpha}(x,u)| \, d(x,u) &= C \sum_{j=-1}^{\infty} 2^{-2j} \Big( \sum_{k: 2^k s_j \leq 1} (2^{k}s_j)^{2N-\beta-Q/2} + \sum_{k: 2^k s_j>1} (2^{k}s_j)^{-\beta} \Big) \\
    &\leq C \sum_{j=-1}^{\infty} 2^{-2j} \leq C ,
\end{align*}
provided we choose $N$ sufficiently large, that is, $2N>\beta+Q/2$.

This completes the proof of the estimate \eqref{To show for weak 1 1 proof}.
\end{proof}


\section{Applications of Stein's square function on M\'etivier groups}
\label{Section: Applications of Steins square function}
This section is devoted to the proof of Theorem \ref{Theorem: Analogue of Niedorf theorem}, \ref{Theorem: Maximal spectral multiplier}, \ref{Theorem: Lp to L2 local smoothing for Metivier}, \ref{Theorem: Bilinear maximal on Metivier}, \ref{Theorem: Bilinear Stein square function estimate}, \ref{Theorem: Bilinear maximal spectral multiplier}, which we obtain using the boundedness of the various square functions defined on M\'etivier groups and Theorem \ref{Theorem: Stein square estimate}, \ref{Theorem: Stein square function for p less that 2 case} already proved in Section \ref{Section: Square function on Metivier groups}.

\subsection{Proof of Theorem \ref{Theorem: Analogue of Niedorf theorem}: Sharp spectral multipliers on M\'etivier groups}
\label{Subsection: Sharp spectral multiplier}
This subsection is devoted to the proof of Theorem \ref{Theorem: Analogue of Niedorf theorem}. First note that by duality enough to prove Theorem \ref{Theorem: Analogue of Niedorf theorem} only for $\mathfrak{p}_G \leq p<\infty$. Let $\varphi \in C_c^{\infty}(\mathbb{R})$ such that it is supported on $[1,2]$. Let us set $\sigma(\eta) = \eta \varphi(\eta)$ and then from \eqref{Square function corresponding to bump function} recall that
\begin{align*}
    \mathcal{G}_{\sigma}(\mathcal{L})f(x,u) &= \left(\int_{0}^{\infty} \left|\sigma\left(R^{-1}\mathcal{L}\right)f(x,u)\right|^2 \frac{dR}{R} \right)^{1/2} ,
\end{align*}
Hence from Proposition \ref{Proposition: Lp boundedness of square function with bump} we have
\begin{align}
\label{spectral mult to localised square function}
    \|F(\mathcal{L})f\|_{L^p} &\leq C \|\mathcal{G}_{\sigma}(\mathcal{L})(F(\mathcal{L})f)\|_{L^p} .
\end{align}
Now we make the following claim: for $\alpha>1/2$
\begin{align}
\label{claim: Localised square to Stein square function}
    \|\mathcal{G}_{\sigma}(\mathcal{L})(F(\mathcal{L})f)\|_{L^p} &\leq C \|F\|_{L^{2}_{\alpha, sloc}} \|\mathfrak{S}^{\alpha}(\mathcal{L})f\|_{L^p} .
\end{align}
Note that if we assume the above claim \eqref{claim: Localised square to Stein square function} for the moment, then combining \eqref{spectral mult to localised square function}, \eqref{claim: Localised square to Stein square function} and Theorem \ref{Theorem: Stein square estimate} we immediately get Theorem \ref{Theorem: Analogue of Niedorf theorem}.

Hence it remains to prove \eqref{claim: Localised square to Stein square function}. First we recall the well known Riemann-Liouville formula: (see \cite{Carbery_Maximal_Radial_Fourier_Mult_1985}) Let $h \in L^2_{\alpha}(\mathbb{R})$ for $\alpha>1/2$ and is supported on $(-\infty, a]$ for some $a>0$, then we have
\begin{align}
\label{Riemann-Liouville formula}
    h(\eta) &= C \int_{\eta}^{\infty} (s-\eta)^{\alpha-1} \frac{d^{\alpha}h}{ds^{\alpha}}(s) \, ds,  \quad \quad \quad a.e.
\end{align}
where $\left(\frac{d^{\alpha}h}{ds^{\alpha}}\right)^{\widehat{}}(\xi) = (-i \xi)^{\alpha} \widehat{h}(\xi)$.

Therefore applying \eqref{Riemann-Liouville formula} to $\sigma(R^{-1}\eta)=\varphi(R^{-1}\eta) F(\eta)$ with $R>0$, for $\alpha>1/2$ we get
\begin{align}
\label{Equality between sigma and Lamda}
    \sigma(R^{-1}\eta) F(\eta) &= C R^{-1} \int_{R}^{2R} \frac{\eta}{s} \left(1-\frac{\eta}{s} \right)_{+}^{\alpha-1} s^{\alpha} \frac{d^{\alpha}}{ds^{\alpha}}[\varphi(R^{-1} \cdot) F(\cdot)](s) \, ds =: C \Lambda_R(\eta) .
\end{align}
In view of the above expression we see that
\begin{align}
\label{Square function with sigma in terms of Gamma}
    \mathcal{G}_{\sigma}(\mathcal{L})(F(\mathcal{L})f)(x,u) &= C \left(\int_{0}^{\infty} \left| \Lambda_R(\mathcal{L})f(x,u)\right|^2 \frac{dR}{R} \right)^{1/2} ,
\end{align}
where
\begin{align}
\label{Expression of captial Lamda}
    \Lambda_R(\mathcal{L})f(x,u) &= R^{-1} \int_{R}^{2R} \frac{\mathcal{L}}{s} \left(1-\frac{\mathcal{L}}{s} \right)_{+}^{\alpha-1}f(x,u) s^{\alpha} \frac{d^{\alpha}}{ds^{\alpha}}[\varphi(R^{-1} \cdot) F(\cdot)](s) \, ds .
\end{align}
From \eqref{Expression of captial Lamda} applying Cauchy-Schwartz inequality yields
\begin{align*}
    & |\Lambda_R(\mathcal{L})f(x,u)| \\
    &\leq \left( \frac{1}{R} \int_{R}^{2R} \left|\frac{\mathcal{L}}{s} \left(1-\frac{\mathcal{L}}{s} \right)_{+}^{\alpha-1}f(x,u) \right|^2 \, ds \right)^{1/2} \left( \frac{1}{R} \int_{R}^{2R} \left| s^{\alpha} \frac{d^{\alpha}}{ds^{\alpha}}[\varphi(R^{-1} \cdot) F(\cdot)](s) \right|^2 \, ds \right)^{1/2} .
\end{align*}
Using $\left(\frac{d^{\alpha}h}{ds^{\alpha}}\right)^{\widehat{}}(\xi) = (-i \xi)^{\alpha} \widehat{h}(\xi)$, one can see that
\begin{align*}
    \left( \frac{1}{R} \int_{R}^{2R} \left| s^{\alpha} \frac{d^{\alpha}}{ds^{\alpha}}[\varphi(R^{-1} \cdot) F(\cdot)](s) \right|^2 \, ds \right)^{1/2} &\sim \|\varphi F(R\cdot)\|_{L^2_{\alpha}(\mathbb{R})} .
\end{align*}
Hence with the help of the above estimate we obtain
\begin{align}
\label{Square function with Gamma r part}
    \left(\int_{0}^{\infty} \left|\Lambda_R(\mathcal{L})f(x,u)\right|^2 \frac{dR}{R} \right)^{1/2} & \leq C \sup_{R>0} \|\varphi F(R\cdot)\|_{L^2_{\alpha}(\mathbb{R})} \left(\int_{0}^{\infty} \left|\frac{\mathcal{L}}{s} \left(1-\frac{\mathcal{L}}{s} \right)_{+}^{\alpha-1}f(x,u) \right|^2 \, \frac{ds}{s} \right)^{1/2} \\
    &\nonumber \leq C \sup_{R>0} \|\varphi F(R\cdot)\|_{L^2_{\alpha}(\mathbb{R})} \, \mathfrak{S}^{\alpha}(\mathcal{L})f(x,u) .
\end{align}
Therefore plugging the above estimate \eqref{Square function with Gamma r part} into \eqref{Square function with sigma in terms of Gamma} yields
\begin{align*}
    \|\mathcal{G}_{\sigma}(\mathcal{L})(F(\mathcal{L})f)\|_{L^p} &\leq C \|F\|_{L^{2}_{\alpha, sloc}} \|\mathfrak{S}^{\alpha}(\mathcal{L})f\|_{L^p} ,
\end{align*}
which completes the proof of the claim \eqref{claim: Localised square to Stein square function}.


\subsection{Proof of Theorem \ref{Theorem: Maximal spectral multiplier}: Maximal spectral multipliers on M\'etivier groups}
\label{Subsection: Boundedness of maximal spectral multiplier}
In this subsection, we prove Theorem \ref{Theorem: Maximal spectral multiplier} as an application of the $L^p$-boundedness of $\mathfrak{S}^{\alpha}(\mathcal{L})$.

Take $F \in C_c^{\infty}(0, \infty)$ and applying \eqref{Riemann-Liouville formula} with $h(\eta) = \frac{F(\eta)}{\eta}$ we get
\begin{align}
\label{Main identity to use maximal multiplier result}
    F\left(\frac{\eta}{R} \right) &= C \int_{0}^{\infty} \frac{\eta}{Rs} \left(1-\frac{\eta}{Rs} \right)_{+}^{\alpha-1} s^{\alpha} \frac{d^{\alpha}}{ds^{\alpha}}\left(\frac{F(s)}{s} \right) \, ds .
\end{align}
Therefore we obtain
\begin{align*}
    F\left(\frac{\mathcal{L}}{R} \right)f(x,u) &= C \int_{0}^{\infty} \frac{\mathcal{L}}{Rs} \left(1-\frac{\mathcal{L}}{Rs} \right)_{+}^{\alpha-1}f(x,u) \, s^{\alpha+1} \frac{d^{\alpha}}{ds^{\alpha}}\left(\frac{F(s)}{s} \right) \, \frac{ds}{s} .
\end{align*}
Applying Cauchy-Schwartz inequality yields
\begin{align}
\label{Application of Holder in maximal spectral}
    & \left|F\left(\frac{\mathcal{L}}{R} \right)f(x,u)\right| \\
    &\nonumber \leq C \left(\int_{0}^{\infty} \left|\frac{\mathcal{L}}{Rs} \left(1-\frac{\mathcal{L}}{Rs} \right)_{+}^{\alpha-1}f(x,u) \right|^2 \frac{ds}{s} \right)^{1/2} \left( \int_0^{\infty} \left|s^{\alpha+1} \frac{d^{\alpha}}{ds^{\alpha}}\left(\frac{F(s)}{s} \right) \right|^2 \, \frac{ds}{s} \right)^{1/2} .
\end{align}
Therefore, from \eqref{Application of Holder in maximal spectral} we obtain
\begin{align}
\label{After application of Holder final estimate}
    F^{*}(\mathcal{L})f(x,u) &\leq C \|F\|_{L^2_{\alpha}(\mathbb{R}^+)} \mathfrak{S}^{\alpha}(\mathcal{L})f(x,u) .
\end{align}
Finally, taking the $L^p$-norm on both sides of the above estimate and applying Theorem \ref{Theorem: Stein square estimate} and Theorem \ref{Theorem: Stein square function for p less that 2 case} we obtain Theorem \ref{Theorem: Maximal spectral multiplier}.


\subsection{Proof of Theorem \ref{Theorem: Lp to L2 local smoothing for Metivier}: Regularity estimate for the solution of fractional Schr\"odinger equation on M\'etivier groups}
\label{Subsection: Local smoothing}
In this subsection we give the proof of Theorem \ref{Theorem: Lp to L2 local smoothing for Metivier}. First note that it is enough to prove the Theorem \ref{Theorem: Lp to L2 local smoothing for Metivier} for large $k$. Recall that $\chi \in C_c^{\infty}(\frac{5}{8}, \frac{15}{8})$ and non-negative such that $\sum_{k \in \mathbb{Z}} \chi(2^{-k} \eta) = 1$ for $\eta>0$. Let $\phi$ be a real valued even function which is supported on $[-1/4,1/4]$ such that it $\widehat{\phi}(s)>C>0$ for all $s \in I$ and some constant $C>0$. Then we have
\begin{align}
\label{Lp norm before finite slicing}
    \left\|\left( \int_I |e^{i s \mathcal{L}^{a/2}} \chi(2^{-k} \sqrt{\mathcal{L}}) f|^2 \, ds \right)^{1/2} \right\|_{L^p} &\leq C \left\|\left( \int_{\mathbb{R}} |\widehat{\phi}(s) e^{i s \mathcal{L}^{a/2}} \chi(2^{-k} \sqrt{\mathcal{L}}) f|^2 \, ds \right)^{1/2} \right\|_{L^p} .
\end{align}
Applying the Plancherel theorem in $s$-variable, using the definition of Fourier transform and using the support condition of $\phi$ and $\chi$ for each $(x,u) \in G$ we can write
\begin{align}
\label{Use of Plancherel in local smoothing}
    \left( \int_{\mathbb{R}} |\widehat{\phi}(s) e^{i s \mathcal{L}^{a/2}} \chi(2^{-k} \sqrt{\mathcal{L}}) f(x,u)|^2 \, ds \right)^{1/2} & = \left( \int_{2^{(k-3)a}}^{2^{(k+3)a}} |\phi(\mathcal{L}^{a/2}-\tau) \chi(2^{-k} \sqrt{\mathcal{L}})f(x,u) |^2 \, d\tau \right)^{\frac{1}{2}} .
\end{align}
Again splitting the above integral into finitely many pieces enough to consider the the case $\tau \in [(2^{k}R)^{a}, (2^{k+1} R)^{a}]$ for $R \sim 1$. Set $\nu= (2^k R)^{-a}$. Then making the change of variable $\tau = \nu^{-1} r^{a/2}$ and using the fact (see \eqref{Dilation for Phij, beta})
\begin{align*}
    \phi(\mathcal{L}^{a/2}-\nu^{-1} r^{a/2}) \chi(2^{-k} \sqrt{\mathcal{L}})f(\delta_{\nu^{1/a}}(x, u)) &= \phi(\nu^{-1}(\mathcal{L}^{a/2}- r^{a/2})) \chi(R \sqrt{\mathcal{L}})(\delta_{\nu^{1/a}}f)(x, u) .
\end{align*}
we see that
\begin{align}
\label{making change of variable in tau}
    & \int_{\nu^{-1}}^{\nu^{-1} 2^{a}} |\phi(\mathcal{L}^{a/2}-\tau) \chi(2^{-k} \sqrt{\mathcal{L}})f(x,u) |^2 \, d\tau \\
    &\nonumber \sim \int_{1}^{4} |\phi(\nu^{-1}(\mathcal{L}^{a/2}- r^{a/2})) \chi(R \sqrt{\mathcal{L}})(\delta_{\nu^{1/a}}f)(\delta_{\nu^{-1/a}}(x, u))|^2 \, \nu^{-1} \, dr ,
\end{align}
Since $\phi$ is even, let us write
\begin{align*}
    \phi(\nu^{-1}(\eta^{a/2}- r^{a/2})) &= \phi\left(\zeta(\eta, r) \frac{r-\eta}{\nu} \right) \quad \quad \text{where} \quad \zeta(\eta,r) = \frac{\eta^{a/2}- r^{a/2}}{\eta- r} .
\end{align*}
Note that the function $\zeta$ is non-vanishing and smooth on $(0, \infty)^2$. Let us set $\widetilde{\phi}\left(\frac{r-\eta}{\nu} \right) = \phi\left(\zeta(\eta, r) \frac{r-\eta}{\nu} \right)$. Therefore we have
\begin{align}
\label{After change of variable from R to normal}
    & \int_{\nu^{-1}}^{\nu^{-1} 2^{a}} |\phi(\mathcal{L}^{a/2}-\tau) \chi(2^{-k} \sqrt{\mathcal{L}})f(x,u) |^2 \, d\tau \\
    &\nonumber \sim \nu^{-1} \int_{1}^{4} \Big|\widetilde{\phi}\left(\frac{r-\mathcal{L}}{\nu} \right) \chi(R \sqrt{\mathcal{L}})(\delta_{\nu^{1/a}}f)(\delta_{\nu^{-1/a}}(x, u)) \Big|^2 \, dr ,
\end{align}
where
\begin{align*}
    \widetilde{\phi}\left(\frac{r-\mathcal{L}}{\nu} \right)g(x,u) &= \frac{1}{(2\pi)^{d_2}} \int_{\mathfrak{g}_{2,r}^{*}} \sum_{\mathbf{k} \in \mathbb{N}^\Lambda} \phi\left(\zeta(\eta_{\mathbf{k}}^{\lambda}, r) \frac{r-\eta_{\mathbf{k}}^{\lambda}}{\nu} \right) \left[g^{\lambda} \times_{\lambda} \varphi_{\mathbf{k}}^{\mathbf{b}^{\lambda}, \mathbf{r}}(R_{\lambda}^{-1}\cdot) \right](x) \, e^{i \langle \lambda, u \rangle} \, d\lambda .
\end{align*}
Therefore, combining \eqref{After change of variable from R to normal}, \eqref{Use of Plancherel in local smoothing} and from \eqref{Lp norm before finite slicing} we obtain
\begin{align}
\label{Reduction to phi tilde}
    \left\|\left( \int_I |e^{i s \mathcal{L}^{a/2}} \chi(2^{-k} \sqrt{\mathcal{L}}) f|^2 \, ds \right)^{\frac{1}{2}} \right\|_{L^p} &\lesssim \nu^{-\frac{1}{2}} \nu^{\frac{Q}{a}} \left\|\left( \int_{1}^{4} \Big|\widetilde{\phi}\left(\frac{r-\mathcal{L}}{\nu} \right) \chi(R \sqrt{\mathcal{L}})(\delta_{\nu^{\frac{1}{a}}}f) \Big|^2 \, dr \right)^{\frac{1}{2}} \right\|_{L^p}.
\end{align}
Since $\zeta$ is non-vanishing and smooth on $(0, \infty)^2$ and on the support of $\chi$ we have $\eta \sim 1$, applying Proposition \ref{Proposition: Generalized Square function estimate} for $\mathfrak{p}_G \leq p<\infty$ and $\alpha>\alpha_d(p)$ we obtain
\begin{align}
\label{Use of phi tilde proposition}
    \left\|\left( \int_{1}^{4} \Big|\widetilde{\phi}\left(\frac{r-\mathcal{L}}{\nu} \right) \chi(R \sqrt{\mathcal{L}})(\delta_{\nu^{1/a}}f) \Big|^2 \, dr \right)^{1/2} \right\|_{L^p} &\leq C \nu^{\frac{1}{2}-\alpha} \|\chi(R \sqrt{\mathcal{L}})(\delta_{\nu^{1/a}}f)\|_{L^p} \\
    &\nonumber \leq C \nu^{\frac{1}{2}-\alpha} \nu^{-Q/a} \|f\|_{L^p} ,
\end{align}
where in the last inequality we have used $\chi \in C_c^{\infty}(\frac{5}{8}, \frac{15}{8})$ and the fact $R \sim 1$.

Note that $\nu \sim 2^{-k a}$. Therefore combining the above two estimates \eqref{Reduction to phi tilde} and \eqref{Use of phi tilde proposition} for $\mathfrak{p}_G \leq p<\infty$ and $\alpha>\alpha_d(p)$ yields
\begin{align*}
    \left\|\left( \int_I |e^{i s \mathcal{L}^{a/2}} \chi(2^{-k} \sqrt{\mathcal{L}}) f|^2 \, ds \right)^{1/2} \right\|_{L^p} &\lesssim 2^{k a \alpha} \|f\|_{L^p} .
\end{align*}
This completes the proof of the Theorem \ref{Theorem: Lp to L2 local smoothing for Metivier}.


\subsection{Proof of Theorem \ref{Theorem: Bilinear maximal on Metivier}: Maximal bilinear Bochner-Riesz mean on M\'etivier groups}
\label{Subsection: Boundedness of bilinear maximal Bochner-Riesz means}
This subsection is devoted to the proof of Theorem \ref{Theorem: Bilinear maximal on Metivier}. In order to prove the $L^{p_1}(G) \times L^{p_2}(G)$ to $L^p(G)$ boundedness of $\mathcal{B}_*^{\alpha}$, we use the ideas from \cite{Jotsaroop_Shrivastava_Maximal_Bochner_Riesz_2022}, therefore here will be brief.

Let $\Psi$ be a bump function supported in $[1/2,2]$ such that for any $t>0$ we have
\begin{align}
\label{Definition of partition of unity for maximal}
    1 &= \sum_{j\geq 2} \Psi(2^{j}(1-t)) + \Psi_0(t) ,
\end{align}
where $\Psi_0$ is again a bump function supported in $[-3/4,3/4]$.

If we take $t=\frac{\eta_1}{R^2}$, then we can write
\begin{align}
\label{Decomposition of the bilinear Bochner-Riesz multiplier}
    \left(1-\tfrac{\eta_1+\eta_2}{R^2} \right)_{+}^{\alpha} &= \sum_{j\geq 2} \Psi \left(2^j\left(1-\tfrac{\eta_1}{R^2} \right) \right) \left(1-\tfrac{\eta_1}{R^2} \right)_{+}^{\alpha} \left(1-\tfrac{\eta_2}{R^2}\left(1-\tfrac{\eta_1}{R^2} \right)^{-1} \right)_{+}^{\alpha} + \Psi_0\left(\tfrac{\eta_1}{R^2} \right) \left(1-\tfrac{\eta_1+\eta_2}{R^2} \right)_{+}^{\alpha} \\
    &\nonumber =: \sum_{j\geq 2} m_{j,R}^{\alpha}(\eta_1, \eta_2) + m_{0,R}^{\alpha}(\eta_1, \eta_2) .
\end{align}
Therefore we have
\begin{align*}
    \mathcal{B}_R^{\alpha}(f,g)(x,u) &= \mathcal{B}_{0,R}^{\alpha}(f,g)(x,u) + \sum_{j \geq 2} \mathcal{B}_{j,R}^{\alpha}(f,g)(x,u) ,
\end{align*}
where $\mathcal{B}_{0,R}^{\alpha}$ and $\mathcal{B}_{j,R}^{\alpha}$ denote the bilinear multipliers corresponding to the multipliers $m_{0,R}^{\alpha}$ and $m_{j,R}^{\alpha}$.

Consequently we get
\begin{align*}
    \mathcal{B}_*^{\alpha}(f,g) &\leq \mathcal{B}_{0,*}^{\alpha}(f,g)(x,u) + \sum_{j \geq 2} \mathcal{B}_{j,*}^{\alpha}(f,g)(x,u) ,
\end{align*}
where $\mathcal{B}_{j,*}^{\alpha}(f,g)(x,u) = \sup_{R>0}|\mathcal{B}_{j,R}^{\alpha}(f,g)(x,u)|$ and $\mathcal{B}_{0,*}^{\alpha}$ is also defined similarly.

\medskip
\noindent \textbf{Estimate of $\mathcal{B}_{j,*}^{\alpha}$:}
\label{Estimate of BJ*}
It is enough to show that for $p_i=2$ or $\mathfrak{p}_G \leq p_i <\infty$ with $i=1,2$, whenever $\alpha> \alpha_*(p_1, p_2, d, \mathfrak{p}_G)$ we have
\begin{align}
\label{Final estimate of B alpha j case}
    \|\mathcal{B}_{j,*}^{\alpha}(f,g)\|_{L^p} &\leq C 2^{-j \epsilon'} \|f\|_{L^{p_1}} \|g\|_{L^{p_2}} ,
\end{align}
for some $\epsilon'>0$.

Set $\phi_R(\eta_1)= \left(1-\frac{\eta_1}{R^2} \right)_{+}$. Recall that from \cite[pp. 278-279]{Stein_Weiss_Fourier_Analysis_book} we have
\begin{align}
\label{Stein_Weiss expression}
    \left(1-\frac{\eta}{R^2} \right)^{\alpha} &= C_{\alpha, \beta} R^{-2\alpha} \int_{\sqrt{\eta}}^{R} (R^2-t^2)^{\alpha-\beta-1} t^{2\beta+1} \left(1-\frac{\eta}{t^2} \right)^{\beta} \, dt ,\quad \quad \text{for} \quad \eta>0 ,
\end{align}
where $C_{\alpha, \beta} = \frac{2\Gamma(\alpha+1)}{\Gamma(\beta+1) \Gamma(\alpha-\beta)}$.

From this we get
\begin{align}
\label{Use of Stein-Weiss equality}
    \left(1-\frac{\eta_2}{R^2 \phi_R(\eta_1)} \right)_{+}^{\alpha} &= C R^{-2\alpha} \phi_R(\eta_1)^{-\alpha} \int_{0}^{R \sqrt{\phi_R(\eta_1)}} (R^2 \phi_R(\eta_1)-t^2)_{+}^{\beta-1} t^{2 \delta+1} \left(1-\frac{\eta_2}{t^2} \right)_{+}^{\delta} \, dt ,
\end{align}
where $\beta>1/2$, $\delta>-1/2$ and $\alpha=\beta+\delta$.

Hence from the expression of $m_{j,R}^{\alpha}$ (see \eqref{Decomposition of the bilinear Bochner-Riesz multiplier}), we obtain
\begin{align}
\label{Writing multiplier as product of two}
    m_{j,R}^{\alpha}(\eta_1, \eta_2) &= C R^{-2\alpha} \Psi \left(2^j\left(1-\frac{\eta_1}{R^2} \right) \right) \int_{0}^{R_j} (R^2 \phi_R(\eta_1)-t^2)_{+}^{\beta-1} t^{2 \delta+1} \left(1-\frac{\eta_2}{t^2} \right)_{+}^{\delta} \, dt ,
\end{align}
where $R_j:= R \sqrt{2^{-j+1}}$, which we get by using the support condition of $\Psi$.

Therefore the above expression \eqref{Writing multiplier as product of two} of $m_{j,R}^{\alpha}$ yields 
\begin{align}
\label{Writing bilinear operator as product of two}
    \mathcal{B}_{j,R}^{\alpha}(f,g)(x,u) &= C R^{-2\alpha} \int_0^{R_j} T_{j, \beta}^{R,t}(\mathcal{L})f(x,u) \, S_{t}^{\delta}(\mathcal{L})g(x,u) t^{2 \delta+1} \, dt ,
\end{align}
where
\begin{align}
\label{Definition of the T operator in maximal}
    T_{j, \beta}^{R,t}(\mathcal{L})f(x,u) &= \frac{1}{(2\pi)^{d_2}} \int_{\mathfrak{g}_{2,r}^{*}} e^{i \langle \lambda_1 , u \rangle} \sum_{\mathbf{k}_1 \in \mathbb{N}^\Lambda} \Psi \Big(2^j\Big(1-\frac{\eta_{\mathbf{k}_1}^{\lambda_1}}{R^2} \Big) \Big) (R^2 \phi_R(\eta_{\mathbf{k}_1}^{\lambda_1})-t^2)_{+}^{\beta-1} \\
    &\nonumber \hspace{5cm} \left[f^{\lambda_1} \times_{\lambda_1} \varphi_{\mathbf{k}_1}^{\mathbf{b}^{\lambda_1}, \mathbf{r}_1}(R_{\lambda_1}^{-1}\cdot) \right](x) \,  d\lambda_1 ,
\end{align}
and $S_{t}^{\delta}(\mathcal{L})$ denote the Bochner-Riesz means of order $\delta$.

Therefore, similarly as in \cite[p. 16]{Jotsaroop_Shrivastava_Maximal_Bochner_Riesz_2022} here we obtain
\begin{align}
\label{Decomposition bilinear operator into two parts}
    & |\mathcal{B}_{j,*}^{\alpha}(f,g)(x,u)| = \sup_{R>0} |\mathcal{B}_{j,R}^{\alpha}(f,g)(x,u)| \\
    &\nonumber \leq C 2^{-\frac{j}{4}} \sup_{R>0} \left(\int_0^{\sqrt{2^{-j+1}}} |S_{j, \beta}^{R, t}(\mathcal{L}) f(x,u)\, t^{2\delta+1}|^2 \, dt \right)^{\frac{1}{2}} \sup_{R>0} \left( \frac{1}{R_j} \int_0^{R_j} |S_{t}^{\delta}(\mathcal{L}) g(x,u)|^2 \, dt \right)^{\frac{1}{2}} ,
\end{align}
where
\begin{align}
\label{Expression of S j beta  t part}
    S_{j, \beta}^{R, t}(\mathcal{L}) f(x,u) &= \frac{1}{(2\pi)^{d_2}} \int_{\mathfrak{g}_{2,r}^{*}} e^{i \langle \lambda_1 , u \rangle} \sum_{\mathbf{k}_1 \in \mathbb{N}^\Lambda} \Psi \Big(2^j\Big(1-\frac{\eta_{\mathbf{k}_1}^{\lambda_1}}{R^2} \Big) \Big) \Big(1-\frac{\eta_{\mathbf{k}_1}^{\lambda_1}}{R^2}-t^2 \Big)_{+}^{\beta-1} \\
    &\nonumber \hspace{6cm} \left[f^{\lambda_1} \times_{\lambda_1} \varphi_{\mathbf{k}_1}^{\mathbf{b}^{\lambda_1}, \mathbf{r}_1}(R_{\lambda_1}^{-1}\cdot) \right](x) \,  d\lambda_1 .
\end{align}

\medskip
Now we claim that the following proposition hold.
\begin{proposition}
\label{Proposition: Main proposition of S in full range}
Let $\mathfrak{p}_G \leq p<\infty$ or $p=2$. Then whenever $\beta>\alpha_d(p)+1/2$ we have
\begin{align*}
    \left\|\sup_{R>0} \left(\int_0^{\sqrt{2^{-j+1}}} |S_{j, \beta}^{R, t}(\mathcal{L}) f(\cdot) \, t^{2\delta+1}|^2 \, dt \right)^{1/2} \right\|_{L^p} &\leq C 2^{j(\alpha_d(p)-\alpha+1/4+\epsilon)} \|f\|_{L^p} ,
\end{align*}
where $\alpha=\beta+\delta$ and $\epsilon>0$.
    
\end{proposition}

Assuming the claim for the moment, that is, assuming Proposition \ref{Proposition: Main proposition of S in full range} to be true, and with the help of Proposition \ref{Proposition: Maximal square function estimate} from \eqref{Decomposition bilinear operator into two parts} applying H\"older's inequality we obtain
\begin{align*}
    & \|\mathcal{B}_{j,*}^{\alpha}(f,g)\|_{L^p} \\
    &\nonumber \leq C 2^{-\frac{j}{4}} \left\|\sup_{R>0} \left(\int_0^{\sqrt{2^{-j+1}}} |S_{j, \beta}^{R, t}(\mathcal{L})f\, t^{2\delta+1}|^2 \, dt \right)^{\frac{1}{2}} \right\|_{L^{p_1}} \left\|\sup_{R>0}\left( \frac{1}{R_j} \int_0^{R_j} |S_{t}^{\delta}(\mathcal{L})g|^2 \, dt \right)^{\frac{1}{2}} \right\|_{L^{p_2}} \\
    &\leq C 2^{-\frac{j}{4}} 2^{j(\alpha_d(p_1)-\alpha+\frac{1}{4}+\epsilon)} \|f\|_{L^{p_1}} \|g\|_{L^{p_2}} ,
\end{align*}
where $\beta>\alpha_d(p_1)+1/2$ for $p_1=2$ or $\mathfrak{p}_G \leq p_1<\infty$ and $\delta>\alpha_d(p_2)-1/2$ for $p_2=2$ or $\mathfrak{p}_G \leq p_2<\infty$.

\medskip
Now since $\alpha=\beta+\delta$, for $p_i=2$ or $\mathfrak{p}_G \leq p_i <\infty$ with $i=1,2$, whenever $\alpha>\alpha_d(p_1)+\alpha_d(p_2) = \alpha_*(p_1, p_2, d, \mathfrak{p}_G)$, the above estimate yields \eqref{Final estimate of B alpha j case}.

\medskip
Let us now start proving the above claim, that is, Proposition \ref{Proposition: Main proposition of S in full range}.

\begin{proof}[Proof of Proposition \ref{Proposition: Main proposition of S in full range}]
Since the proof of the Proposition \ref{Proposition: Main proposition of S in full range} is similar to that of Theorem 5.1 in \cite{Jotsaroop_Shrivastava_Maximal_Bochner_Riesz_2022}, here we will be brief.

Let us set $\beta=\gamma+1$. Suppose $t^2 \in [0, 2^{-j-1})$, then we see that on the support of $ \Psi \left(2^j\left(1-\frac{\eta_1}{R^2} \right) \right)$, we have $1-\frac{\eta_1}{R^2}-t^2 \geq 0$. Hence we can write
\begin{align*}
    \Psi \left(2^j\left(1-\frac{\eta_1}{R^2} \right) \right) \left(1-\frac{\eta_1}{R^2}-t^2 \right)_{+}^{\gamma} &= \Psi \left(2^j\left(1-\frac{\eta_1}{R^2} \right) \right) \left(1-\frac{\eta_1}{R^2}-t^2 \right)^{\gamma} .
\end{align*}
Now for $0< \epsilon_0 << 1$, we divide the interval $[0, \sqrt{2^{-j+1}}]$ into two parts as $[0, \sqrt{2^{-j-1-\epsilon_0}}]$ and $[\sqrt{2^{-j-1-\epsilon_0}}, \sqrt{2^{-j+1}}]$.

\medskip

\noindent \textbf{Case I: $t \in [0, \sqrt{2^{-j-1-\epsilon_0}}]$.}

Estimates in this case is similar to that of \cite[Case I, proof of Theorem 5.1]{Jotsaroop_Shrivastava_Maximal_Bochner_Riesz_2022}. Instead of estimate (15) in \cite{Jotsaroop_Shrivastava_Maximal_Bochner_Riesz_2022}, here one have to use Proposition \ref{Prop: Maximal Lp norm estimates of localized}. Consequently, adapting the proof of \cite[Case I, proof of Theorem 5.1]{Jotsaroop_Shrivastava_Maximal_Bochner_Riesz_2022} in our setup for $\mathfrak{p}_G \leq p<\infty$ or $p=2$ we obtain
\begin{align}
\label{Application of Minkowski in gamma less 0 case}
    \left\|\sup_{R>0} \left(\int_0^{\sqrt{2^{-j-1-\epsilon_0}}} |S_{j, \gamma+1}^{R, t}(\mathcal{L})f\, t^{2\delta+1}|^2 \, dt \right)^{\frac{1}{2}} \right\|_{L^p} 
    &\lesssim 2^{j(\alpha_d(p)-\alpha+1/4+\epsilon)} \|f\|_{L^p} .
\end{align}

\medskip

\noindent \textbf{Case II: $t \in [\sqrt{2^{-j-1-\epsilon_0}}, \sqrt{2^{-j+1}}]$.}
\label{Case II for maximal Bochner-Riesz}

First note that in this case $t \sim 2^{-j/2}$. Then analogous to \cite[p. 21]{Jotsaroop_Shrivastava_Maximal_Bochner_Riesz_2022} here we see that
\begin{align}
\label{Taking out delta from S calculation}
    \left(\int_{\sqrt{2^{-j-1-\epsilon_0}}}^{\sqrt{2^{-j+1}}} |S_{j, \gamma+1}^{R, t}(\mathcal{L})f(x,u)\, t^{2\delta+1}|^2 \, dt \right)^{\frac{1}{2}} &\sim 2^{-j \delta} 2^{-j/4} \left(\int_{\sqrt{1-2^{-j+1}}}^{\sqrt{1-2^{-j-1-\epsilon_0}}} |\widetilde{S}_{j, \gamma+1}^{R, s}(\mathcal{L})f(x,u)|^2 \, ds \right)^{\frac{1}{2}} .
\end{align}
where
\begin{align}
\label{Definition of S j gamma tilde R s operator}
    & \widetilde{S}_{j, \gamma+1}^{R, s}(\mathcal{L})f(x,u) \\
    &\nonumber = \frac{1}{(2\pi)^{d_2}} \int_{\mathfrak{g}_{2,r}^{*}} e^{i \langle \lambda , u \rangle} \sum_{\mathbf{k} \in \mathbb{N}^\Lambda} \Psi \Big(2^j\Big(1-\tfrac{\eta_{\mathbf{k}}^{\lambda}}{R^2} \Big) \Big) \Big(1-\tfrac{\eta_{\mathbf{k}}^{\lambda}}{s^2 R^2}\Big)_{+}^{\gamma} \left[f^{\lambda} \times_{\lambda} \varphi_{\mathbf{k}}^{\mathbf{b}^{\lambda}, \mathbf{r}}(R_{\lambda}^{-1}\cdot) \right](x) \,  d\lambda .
\end{align}

\medskip
Let us set $s_1=\sqrt{1-2^{-j+1}}$ and $s_2=\sqrt{1-2^{-j-1-\epsilon_0}}$. Let $M>100$ be a large number. Then we claim the following two propositions: Proposition \ref{Proposition: Estimate of S tilde when j less than M case} and Proposition \ref{Prop: Estimate of S j tilde for j large}.

\begin{proposition}
\label{Proposition: Estimate of S tilde when j less than M case}
Let $\mathfrak{p}_G \leq p<\infty$ or $p=2$. Whenever $\gamma>\alpha_d(p)-1/2$ and $0<\epsilon<1$ we have
\begin{align*}
    \left\|\sup_{R>0} \left(\int_{s_1}^{s_2} |\widetilde{S}_{j, \gamma+1}^{R, s}(\mathcal{L})f|^2 \, ds \right)^{1/2} \right\|_{L^p} &\leq C(M, \gamma, \epsilon) \|f\|_{L^p} ,
\end{align*}
for all $2\leq j \leq M$.
    
\end{proposition}

\begin{proposition}
\label{Prop: Estimate of S j tilde for j large}
Let $\mathfrak{p}_G \leq p<\infty$ or $p=2$ and $\gamma>\alpha_d(p)-1/2$. Then for all $j \geq M$ and $0<\epsilon_1<1$ we have
\begin{align*}
    \left\|\sup_{R>0} \left(\int_{s_1}^{s_2} |\widetilde{S}_{j, \gamma+1}^{R, s}(\mathcal{L}) f|^2 \, ds \right)^{1/2} \right\|_{L^p} &\leq C 2^{-j \gamma} 2^{j(\alpha_d(p)-1/2)} 2^{\epsilon_1 j} \|f\|_{L^p} .
\end{align*}
    
\end{proposition}

\medskip
Assuming the above two propositions to be true for the moment, we complete the proof of the Proposition \ref{Proposition: Main proposition of S in full range}. Combining Proposition \ref{Proposition: Estimate of S tilde when j less than M case} and Proposition \ref{Prop: Estimate of S j tilde for j large} and using \eqref{Taking out delta from S calculation} for $\beta>\alpha_d(p)+1/2$ (since $\beta=\gamma+1$) we obtain
\begin{align}
\label{Final estimate for CaseII}
    \left\|\sup_{R>0} \left(\int_{s_1}^{s_2} |S_{j, \gamma+1}^{R, t}(\mathcal{L})f(\cdot)\, t^{2\delta+1}|^2 \, dt \right)^{1/2} \right\|_{L^p} &\leq C 2^{-j \delta} 2^{-j/4} \left\|\sup_{R>0} \left(\int_{s_1}^{s_2} |\widetilde{S}_{j, \gamma+1}^{R, s}(\mathcal{L})f|^2 \, ds \right)^{1/2} \right\|_{L^p} \\
    &\nonumber \leq C 2^{-j \delta} 2^{-j/4} 2^{-j \gamma} 2^{j(\alpha_d(p)-1/2)} 2^{j \epsilon_1} \|f\|_{L^p} \\
    &\nonumber \leq C 2^{-j \delta} 2^{-j/4} 2^{-j (\alpha-\delta-1)} 2^{j(\alpha_d(p)-1/2)} 2^{j \epsilon_1} \|f\|_{L^p} \\
    &\nonumber \leq C 2^{j(\alpha_d(p)-\alpha+1/4+\epsilon_2)} \|f\|_{L^p}
\end{align}
where $\alpha=\beta+\delta=\gamma+1+\delta$ and $\epsilon_1, \epsilon_2>0$.

This completes the proof of Case II, upon assuming Proposition \ref{Proposition: Estimate of S tilde when j less than M case} and Proposition \ref{Prop: Estimate of S j tilde for j large}. Now combining the estimates for Case I and Case II, the proof of the Proposition \ref{Proposition: Main proposition of S in full range} is completed.
\end{proof}

\begin{proof}[Proof of Proposition \ref{Proposition: Estimate of S tilde when j less than M case}]
\renewcommand{\qedsymbol}{}

First note that from \eqref{Definition of S j gamma tilde R s operator}, the operator $\widetilde{S}_{j, \gamma+1}^{R, s}(\mathcal{L})$ can be written as
\begin{align*}
    \widetilde{S}_{j, \gamma+1}^{R, s}(\mathcal{L})f(x,u) &= \Psi \left(2^j\left(1-\frac{\mathcal{L}}{R^2} \right) \right) (S_{sR}^{\gamma}(\mathcal{L})f)(x,u) ,
\end{align*}
where $S_{sR}^{\gamma}(\mathcal{L})$ denote the Bochner-Riesz means of order $\gamma$ associated the sub-Laplacian $\mathcal{L}$ on M\'etivier group, as defined in \eqref{Definition: Bochner-Riesz means}.

Making the change of variable $s \to R^{-1} s$ and applying \eqref{Dominate operator by maximal function} of Lemma \ref{Lemma: Pointwise kernel estimate for linear kernel} we obtain
\begin{align*}
    \left\|\sup_{R>0} \left(\int_{s_1}^{s_2} |\widetilde{S}_{j, \gamma+1}^{R, s}(\mathcal{L})f|^2 \, ds \right)^{1/2} \right\|_{L^p} &= \left\|\sup_{R>0} \left( \int_{R s_1}^{R s_2} |\Psi \left(2^j\left(1-\frac{\mathcal{L}}{R^2} \right) \right) (S_{s}^{\gamma}(\mathcal{L})f)|^2 \, \frac{ds}{R} \right)^{1/2} \right\|_{L^p} \\
    &\leq C 2^{j(Q+\epsilon)} \left\|\left(\int_{0}^{\infty} |\mathcal{M} S_{s}^{\gamma}(\mathcal{L})f|^2 \, \frac{ds}{s} \right)^{1/2} \right\|_{L^p} ,
\end{align*}
where in the last inequality we have used $s \sim R$ when $s \in [R s_1, R s_2]$.

Note that from Theorem \ref{Theorem: Stein square estimate} for $\mathfrak{p}_G \leq p <\infty$ or $p=2$ whenever $\gamma>\alpha_{d}(p)-\frac{1}{2}$ we obtain
\begin{align}
\label{Stein squre function like estimate}
    \left\| \left(\int_0^{\infty} |S_{R}^{\gamma}(\mathcal{L}) f|^2 \, \frac{dR}{R} \right)^{1/2} \right\|_{L^{p_2}} &\leq C \|f\|_{L^{p}} .
\end{align}
Therefore applying vector-valued boundedness of the Hardy-Littlewood maximal function \cite[Theorem 1.2]{Grafakos_Liu_Yang_Vector_valued__space_homogeneous} and the estimate \eqref{Stein squre function like estimate} for $\mathfrak{p}_G \leq p <\infty$ or $p=2$ whenever $\gamma>\alpha_{d}(p)-\frac{1}{2}$ yields
\begin{align*}
    \left\|\sup_{R>0} \left(\int_{s_1}^{s_2} |\widetilde{S}_{j, \gamma+1}^{R, s}(\mathcal{L})f|^2 \, ds \right)^{1/2} \right\|_{L^p} &\leq C 2^{j(Q+\epsilon)} \left\|\left(\int_{0}^{\infty} |S_{s}^{\gamma}(\mathcal{L})f|^2 \, \frac{ds}{s} \right)^{1/2} \right\|_{L^p} \leq C \|f\|_{L^p} ,
\end{align*}
where in the last inequality we have used the fact $2\leq j \leq M$.

This completes the proof of Proposition \ref{Proposition: Estimate of S tilde when j less than M case}.
\end{proof}

\begin{proof}[Proof of Proposition \ref{Prop: Estimate of S j tilde for j large}]
\renewcommand{\qedsymbol}{}
Let $\widetilde{\Psi}$ be a bump function supported in $[1/2,2]$ such that
\begin{align}
\label{Introducting Psi tilde function for S tilde}
    \left(1-\frac{\eta}{s^2 R^2} \right)_{+}^{\gamma} &= \sum_{k \geq 2} 2^{-k \gamma} \widetilde{\Psi}\left(2^{k}\left(1-\frac{\eta}{s^2 R^2} \right) \right) + \widetilde{\Psi}_0 \left(\frac{\eta}{s^2 R^2} \right) ,
\end{align}
where $\widetilde{\Psi}_0$ is again a bump function supported in $[0,3/4]$.

Therefore from \eqref{Definition of S j gamma tilde R s operator} we obtain
\begin{align}
\label{Writing S tilde as sum of I and II part}
    & \widetilde{S}_{j, \gamma+1}^{R, s}f(x,u) \\
    &\nonumber = \frac{1}{(2\pi)^{d_2}} \int_{\mathfrak{g}_{2,r}^{*}} e^{i \langle \lambda , u \rangle} \sum_{\mathbf{k} \in \mathbb{N}^\Lambda} \sum_{k \geq j-2} 2^{-k \gamma} \widetilde{\Psi}\left(2^{k}\left(1-\frac{\eta_{\mathbf{k}}^{\lambda}}{s^2 R^2} \right) \right) \Psi \left(2^j\left(1-\frac{\eta_{\mathbf{k}}^{\lambda}}{R^2} \right) \right) \\
    &\nonumber \hspace{8cm} \left[f^{\lambda} \times_{\lambda} \varphi_{\mathbf{k}}^{\mathbf{b}^{\lambda}, \mathbf{r}}(R_{\lambda}^{-1}\cdot) \right](x) \,  d\lambda \\
    &\nonumber + \frac{1}{(2\pi)^{d_2}} \int_{\mathfrak{g}_{2,r}^{*}} e^{i \langle \lambda , u \rangle} \sum_{\mathbf{k} \in \mathbb{N}^\Lambda} \widetilde{\Psi}_0 \left(\frac{\eta_{\mathbf{k}}^{\lambda}}{s^2 R^2} \right) \Psi \left(2^j\left(1-\frac{\eta_{\mathbf{k}}^{\lambda}}{R^2} \right) \right) \left[f^{\lambda} \times_{\lambda} \varphi_{\mathbf{k}}^{\mathbf{b}^{\lambda}, \mathbf{r}}(R_{\lambda}^{-1}\cdot) \right](x) \,  d\lambda \\
    &\nonumber =: I + II ,
\end{align}
where in $I$ the sum over $k \geq 2$ reduces to $k \geq j-2$ due to the support condition of the function $\widetilde{\Psi}\left(2^{k}\left(1-\frac{\eta}{s^2 R^2} \right) \right) \Psi \left(2^j\left(1-\frac{\eta}{R^2} \right) \right)$.

First note that, because of the support consideration of $\widetilde{\Psi}_0 \left(\frac{\eta}{s^2 R^2} \right) \Psi \left(2^j\left(1-\frac{\eta}{R^2} \right) \right)$ and using the fact $j\geq M>100$, one can see that $II=0$.

Now to estimate $I$, let us choose another bump function $\Phi$ supported in $[-1,1]$ such that
\begin{align*}
    \sum_{m \in \mathbb{Z}} \Phi(t-m) &= 1, \quad t \in \mathbb{R} .
\end{align*}
Let us set $d_{k,m}^{\epsilon} = 2^k (2^{-k+1}-2^{-k(1+\epsilon)}m)$ and $\eta_{k, m}^{R s} = (\frac{\eta}{s^2 R^2}-1+2^{-k+1}-2^{-(1+\epsilon)k} m)$. Then proceeding similarly as in \cite[p. 22-24]{Jotsaroop_Shrivastava_Maximal_Bochner_Riesz_2022} from \eqref{Writing S tilde as sum of I and II part} we obtain
\begin{align}
\label{Writing S tilde as sum of U and P tilde}
    \widetilde{S}_{j, \gamma+1}^{R, s}f(x,u) & = \sum_{k \geq j-2} 2^{-k \gamma} \sum_{0\leq m \leq [\nu^{-1}]+1} \sum_{\beta=0}^{N-1} \frac{(-1)^{\beta}}{\beta !} \frac{(2 \pi i 2^{-\epsilon k})^{\beta}}{2 \pi i} \frac{\partial^{\beta} \widetilde{\Psi}}{\partial \eta^{\beta}}(d_{k,m}^{\epsilon}) U_{R s, R, \Phi^{\beta}}^{j, k, m}(\mathcal{L})f(x,u) \\
    &\nonumber \hspace{6cm} + \sum_{k \geq j-2} 2^{-k \gamma} \sum_{0\leq m \leq [\nu^{-1}]+1} \widetilde{P}_{R s, R, \Phi}^{j, k, m}(\mathcal{L}) f(x,u) ,
\end{align}
where for $\Phi^{\beta}(\eta) = \eta^{\beta} \Phi(\eta)$ we have
\begin{align}
\label{Expression of URs}
    U_{R s, R, \Phi^{\beta}}^{j, k, m}(\mathcal{L})f(x,u) &= \frac{1}{(2\pi)^{d_2}} \int_{\mathfrak{g}_{2,r}^{*}} e^{i \langle \lambda , u \rangle} \sum_{\mathbf{k} \in \mathbb{N}^\Lambda} \Phi^{\beta} \left(2^{(1+\epsilon)k} (\eta_{\mathbf{k}}^{\lambda})_{k, m}^{R s} \right) \Psi \left(2^j\left(1-\frac{\eta_{\mathbf{k}}^{\lambda}}{R^2} \right) \right) \\
    &\nonumber \hspace{6cm} \left[f^{\lambda} \times_{\lambda} \varphi_{\mathbf{k}}^{\mathbf{b}^{\lambda}, \mathbf{r}}(R_{\lambda}^{-1}\cdot) \right](x) \,  d\lambda ,
\end{align}
and
\begin{align}
\label{Definition of P tilde RsPhi}
    \widetilde{P}_{R s, R, \Phi}^{j, k, m}(\mathcal{L}) f(x,u) &= \int_{\mathbb{R}} \widehat{\widetilde{\Psi}}(\kappa) e^{2 \pi i d_{k,m}^{\epsilon} \kappa} P_{R s, R, \Phi}^{j, k, \kappa, m}(\mathcal{L}) f(x,u) \, d\kappa ,
\end{align}
with
\begin{align}
\label{Definition of PRsPhi}
    P_{R s, R, \Phi}^{j, k, \kappa, m}(\mathcal{L}) f(x,u) &= \frac{1}{(2\pi)^{d_2}} \int_{\mathfrak{g}_{2,r}^{*}} e^{i \langle \lambda , u \rangle} \sum_{\mathbf{k} \in \mathbb{N}^\Lambda} \Phi \left(2^{(1+\epsilon)k} (\eta_{\mathbf{k}}^{\lambda})_{k, m}^{R s} \right) \Psi \left(2^j\left(1-\frac{\eta_{\mathbf{k}}^{\lambda}}{R^2} \right) \right) \\
    &\nonumber \hspace{5cm} \mathcal{E}(2^k (\eta_{\mathbf{k}}^{\lambda})_{k,m}^{R s} \kappa) \left[f^{\lambda} \times_{\lambda} \varphi_{\mathbf{k}}^{\mathbf{b}^{\lambda}, \mathbf{r}}(R_{\lambda}^{-1}\cdot) \right](x) \,  d\lambda ,
\end{align}
and for $0\leq \beta \leq N$ it satisfies
\begin{align*}
    |\mathcal{E}^{(\beta)}(t)| &\leq C_k |t|^{N-\beta} .
\end{align*}
Note that for $m \leq [\nu^{-1}]+1$ and $k \geq 2$, the quantity $d_{k,m}^{\epsilon}$ is uniformly bounded in $k$ and $m$, and $\widetilde{\Psi} \in C_c^{\infty}([1/2,2])$, we see that
\begin{align}
\label{Uniform bound for Psi tilde}
    \sup_{0\leq \beta \leq N-1} \left|\frac{\partial^{\beta} \widetilde{\Psi}}{\partial \eta^{\beta}}(d_{k,m}^{\epsilon}) \right| \leq C ,
\end{align}
where $C$ is independent of $k$ and $m$.

\noindent\textbf{Estimate of $U_{R s, R, \Phi_{\beta}}^{j, k, m}(\mathcal{L})$:}
In order to estimate $U_{R s, R, \Phi^{\beta}}^{j, k, m}(\mathcal{L})$, again we proceed as in \cite[eq. (30), p. 25]{Jotsaroop_Shrivastava_Maximal_Bochner_Riesz_2022} and from \eqref{Expression of URs} we arrive at the following
\begin{align}
\label{Writing U in terms of V and X term}
    U_{R s, R, \Phi^{\beta}}^{j, k, m}(\mathcal{L})f(x,u) &= \sum_{\ell =0}^{N-1} \frac{(-1)^{\ell}}{\ell !} s^{2\ell} (2 \pi i)^{\ell-1} 2^{-(k-j)(1+\epsilon)\ell} 2^{-j \epsilon \ell} \frac{\partial^{\ell} \Psi}{\partial \eta^{\ell}}(2^j (1-a-s^2 2^{-k(1+\epsilon)}m)) \\
    &\nonumber\hspace{4cm} \times V_{R s, \Phi^{\beta+\ell}}^{k, m}(\mathcal{L})f(x,u) + X_{N, Rs, s}^{\Phi^{\beta}, k}(\mathcal{L}) f(x,u) ,
\end{align}
where
\begin{align*}
    V_{R s, \Phi^{\beta+\ell}}^{k, m}(\mathcal{L})f(x,u) &= \frac{1}{(2\pi)^{d_2}} \int_{\mathfrak{g}_{2,r}^{*}} e^{i \langle \lambda , u \rangle} \sum_{\mathbf{k} \in \mathbb{N}^\Lambda} \Phi^{\beta+\ell} \left(2^{(1+\epsilon)k} (\eta_{\mathbf{k}}^{\lambda})_{k, m}^{R s} \right) \left[f^{\lambda} \times_{\lambda} \varphi_{\mathbf{k}}^{\mathbf{b}^{\lambda}, \mathbf{r}}(R_{\lambda}^{-1}\cdot) \right](x) \,  d\lambda ,
\end{align*}
and
\begin{align}
\label{Expression of X in terms Q}
    X_{N, Rs, s}^{\Phi^{\beta}, k}(\mathcal{L}) f(x,u) &= \int_{\mathbb{R}} \widehat{\Psi}(\kappa) e^{2 \pi i 2^j (1-a-s^2 2^{-k(1+\epsilon)}m)\kappa} Q_{N, R s, s}^{\Phi^{\beta}, k, \kappa}(\mathcal{L})f(x,u) \, d\kappa ,
\end{align}
with
\begin{align}
\label{Definition of QNRs}
    Q_{N, R s, s}^{\Phi^{\beta}, k, \kappa}(\mathcal{L})f(x,u) &= \frac{1}{(2\pi)^{d_2}} \int_{\mathfrak{g}_{2,r}^{*}} e^{i \langle \lambda , u \rangle} \sum_{\mathbf{k} \in \mathbb{N}^\Lambda} \Phi^{\beta} \left(2^{(1+\epsilon)k} (\eta_{\mathbf{k}}^{\lambda})_{k, m}^{R s} \right) \mathcal{E}(2^j s^2 (\eta_{\mathbf{k}}^{\lambda})_{k,m}^{R s} \kappa) \\
    &\nonumber \hspace{6cm} \left[f^{\lambda} \times_{\lambda} \varphi_{\mathbf{k}}^{\mathbf{b}^{\lambda}, \mathbf{r}}(R_{\lambda}^{-1}\cdot) \right](x) \,  d\lambda .
\end{align}
Note that since $k \geq j-2$ that is, $j-k \leq 2$ and $\psi \in C_c^{\infty}([1/2,2])$ therefore one can see that
\begin{align}
\label{Derivative of Psi is uniformly bounded}
    \sup_{0\leq \ell \leq N-1} \left|\frac{\partial^{\ell} \Psi}{\partial \eta^{\ell}}(2^j (1-a-s^2 2^{-k(1+\epsilon)}m)) \right| \leq C ,
\end{align}
where $C$ is independent of $j, k$, $s \in [s_1, s_2]$ and $m \leq [\nu^{-1}]+1$.

\begin{lemma}
\label{Lemma: Estimate of V term}
Let $p$ and $\alpha_d(p)$ as in Proposition \ref{Prop: Estimate of S j tilde for j large}. Then the following hold
\begin{align*}
    \left\|\sup_{R>0} \left(\int_{s_1}^{s_2} |V_{R s, \Phi^{\beta+\ell}}^{k, m}(\mathcal{L}) f|^2 \, ds \right)^{1/2} \right\|_{L^p} &\lesssim_{\epsilon'} C 2^{\epsilon'(1+\epsilon)k} 2^{(1+\epsilon)(\alpha_d(p)-1/2)k} \|f\|_{L^p} ,
\end{align*}
for any $\epsilon' < \epsilon$.
    
\end{lemma}

\begin{proof}
\renewcommand{\qedsymbol}{}
Proceeding analogously as in \cite[Lemma 5.5]{Jotsaroop_Shrivastava_Maximal_Bochner_Riesz_2022} the quantity in the lft hand side inside $L^p$-norm $\sup_{R>0} \left(\int_{s_1}^{s_2} |V_{R s, \Phi^{\beta+\ell}}^{k, m}(\mathcal{L}) f(x,u)|^2 \, ds \right)^{1/2}$ can be dominated by
\begin{align}
\label{dominate V by Stein square function}
    \left( \int_{0}^{\infty} \left|\Phi^{\beta+\ell} \left(-\delta^{-1} \left(1-\frac{\mathcal{L}}{s^2} \right) \right) f(x,u) \right|^2 \, \frac{ds}{s} \right)^{1/2} .
\end{align}
Therefore using the $L^p$-boundedness of \eqref{dominate V by Stein square function} (Proposition \ref{Proposition: Square function estimate}) we obtain at the required conclusion.
\end{proof}

\begin{lemma}
\label{Lemma: Estimate of X term}
For $1< p \leq \infty$, the following holds
\begin{align*}
    \Big\| \sup_{R>0, s \in [s_1, s_2]} |X_{N, Rs, s}^{\Phi^{\beta}, k}(\mathcal{L}) f| \Big\|_{L^p} \lesssim 2^{(1+\epsilon)(\alpha_d(p)-1/2)k} \|f\|_{L^p} .
\end{align*}
    
\end{lemma}

\begin{proof}
\renewcommand{\qedsymbol}{}
In order to prove the required estimate, from \eqref{Expression of X in terms Q} we see that enough to prove the following: for all $N \in \mathbb{N}$ we have
\begin{align}
\label{Claim for QNRs}
    \Big\| \sup_{R>0, s \in [s_1, s_2]} |Q_{N, R s, s}^{\Phi^{\beta}, k, \kappa}(\mathcal{L})f| \Big\|_{L^p} \lesssim (1+|\kappa|)^{N} 2^{-j \epsilon N} 2^{-(k-j)(1+\epsilon)N} 2^{k(1+\epsilon)(Q+2)} \|f\|_{L^p} ,
\end{align}
for some $\epsilon'>0$ and $1<p \leq \infty$.

From \eqref{Definition of QNRs} let us write
\begin{align}
\label{Writing QNRs as kernel expression}
    &Q_{N, R s, s}^{\Phi^{\beta}, k, \kappa}(\mathcal{L})f(x,u) = 2^{-j \epsilon N} 2^{-(k-j)(1+\epsilon)N} \widetilde{Q}_{N, R s, s}^{\Phi^{\beta}, k, \kappa}(\mathcal{L})f(x,u) \\
    &\nonumber = 2^{-j \epsilon N} 2^{-(k-j)(1+\epsilon)N} \int_{G} \mathcal{K}_{\widetilde{Q}_{N, R s, s}^{\Phi^{\beta}, k, \kappa}(\mathcal{L})}((y,t)^{-1}(x,u)) f(y,t) \, d(y,t) ,
\end{align}
where 
\begin{align*}
    \widetilde{Q}_{N, R s, s}^{\Phi^{\beta}, k, \kappa}(\eta) &= 2^{j \epsilon N} 2^{(k-j)(1+\epsilon)N} \Phi^{\beta} \left(2^{(1+\epsilon)k} \eta_{k, m}^{R s} \right) \mathcal{E}(2^j s^2 \eta_{k,m}^{R s} \kappa) .
\end{align*}
Since $s \sim 1$ and $m \leq [\nu^{-1}]+1$, we see that $\widetilde{Q}_{N, R s, s}^{\Phi^{\beta}, k, \kappa}$ is supported in $[0, C R^2]$. Therefore from Lemma \ref{Lemma: Pointwise kernel estimate for linear kernel} we get
\begin{align}
\label{Kernel estimate of QNRs}
    |\mathcal{K}_{\widetilde{Q}_{N, R s, s}^{\Phi^{\beta}, k, \kappa}(\mathcal{L})}(x,u)| (1+ R \|(x,u)\|)^{Q+1} &\leq C R^Q \|\widetilde{Q}_{N, R s, s}^{\Phi^{\beta}, k, \kappa}(R^2 \cdot)\|_{L^2_{Q+2}(\mathbb{R})} .
\end{align}
Set
\begin{align*}
    \widetilde{Q}_{N, s}^{\Phi^{\beta}, k, \kappa}(\eta) = \widetilde{Q}_{N, R s, s}^{\Phi^{\beta}, k, \kappa}(R^2 \eta) &= 2^{j \epsilon N} 2^{(k-j)(1+\epsilon)N} \Phi^{\beta} \left(2^{(1+\epsilon)k} \eta_{k, m}^{s} \right) \mathcal{E}(2^j s^2 \eta_{k,m}^{s} \kappa) ,
\end{align*}
and note $\Phi^{\beta}$ is non-zero for $|\eta_{k,m}^{s}| \leq 2^{-(1+\epsilon)k}$. Then since $s \sim 1$, for any $\gamma \in \mathbb{N}$ we see that
\begin{align*}
    |\partial_{\eta}^{\gamma} \widetilde{Q}_{N, s}^{\Phi^{\beta}, k, \kappa}(\eta)| &\leq C (1+|\kappa|)^N 2^{(1+\epsilon)k \gamma} .
\end{align*}
Therefore from \eqref{Kernel estimate of QNRs} we have
\begin{align}
\label{Final estimate of kernel of Q tilde}
    |\mathcal{K}_{\widetilde{Q}_{N, R s, s}^{\Phi^{\beta}, k, \kappa}(\mathcal{L})}(x,u)| (1+ R \|(x,u)\|)^{Q+1} &\leq C R^Q (1+|\kappa|)^N 2^{(1+\epsilon) (Q+2)k} .
\end{align}
Hence with the help of the above estimate from \eqref{Writing QNRs as kernel expression} we obtain
\begin{align*}
    & |Q_{N, R s, s}^{\Phi^{\beta}, k, \kappa}(\mathcal{L})f(x,u)| \\
    &= 2^{-j \epsilon N} 2^{-(k-j)(1+\epsilon)N} R^Q (1+|\kappa|)^N 2^{k(1+\epsilon) (Q+2)} \int_{G} \frac{|f(y,t)|}{(1+ R |(y,t)^{-1}(x,u)|)^{\gamma}} \, d(y,t) \\
    &= 2^{-j \epsilon N} 2^{-(k-j)(1+\epsilon)N} (1+|\kappa|)^N 2^{k(1+\epsilon) (Q+2)}\mathcal{M}f(x,u) .
\end{align*}
From this the required claim \eqref{Claim for QNRs} follows immediately.
\end{proof}

Now combining the Lemma \ref{Lemma: Estimate of V term} and Lemma \ref{Lemma: Estimate of X term} and using the fact $s \sim 1$ and \eqref{Derivative of Psi is uniformly bounded} from \eqref{Writing U in terms of V and X term} we get
\begin{align}
\label{Estimate of U Rs phi beta part}
     \left\|\sup_{R>0} \left(\int_{s_1}^{s_2} |U_{R s, R, \Phi^{\beta}}^{j, k, m}(\mathcal{L})f|^2 \, ds \right)^{1/2} \right\|_{L^p} &\leq C 2^{\epsilon'(1+\epsilon)k} 2^{(1+\epsilon)(\alpha_d(p)-1/2)k} \|f\|_{L^p} .
\end{align}

\noindent\textbf{Estimate of $\widetilde{P}_{R s, R, \Phi}^{j, k, m}(\mathcal{L})$:}
Similarly as we have obtained \eqref{Writing U in terms of V and X term}, here from \eqref{Definition of PRsPhi} we can also write
\begin{align*}
    & P_{R s, R, \Phi}^{j, k, \kappa, m}(\mathcal{L}) f(x,u) \\
    &= \sum_{\ell =0}^{N-1} \frac{(-1)^{\ell}}{\ell !} s^{2\ell} (2 \pi i)^{\ell-1} 2^{-(k-j)(1+\epsilon)\ell} 2^{-j \epsilon \ell} \frac{\partial^{\ell} \Psi}{\partial \eta^{\ell}}(2^j (1-a-s^2 2^{-k(1+\epsilon)}m)) D_{R s, \Phi^{\ell}}^{k, m, \kappa}(\mathcal{L})f(x,u) \\
    &\hspace{4cm} + \int_{\mathbb{R}} \widehat{\Psi}(\kappa') e^{2 \pi i 2^j (1-a-s^2 2^{-k(1+\epsilon)}m)\kappa'} T_{N, R s, s}^{\Phi, k, \kappa, \kappa'}(\mathcal{L})f(x,u) \, d\kappa' ,
\end{align*}
where
\begin{align}
\label{Definition of DRs Phi}
    D_{R s, \Phi^{\ell}}^{k, m, \kappa}(\mathcal{L})f(x,u) &= \frac{1}{(2\pi)^{d_2}} \int_{\mathfrak{g}_{2,r}^{*}} e^{i \langle \lambda , u \rangle} \sum_{\mathbf{k} \in \mathbb{N}^\Lambda} \Phi^{\ell} \left(2^{(1+\epsilon)k} (\eta_{\mathbf{k}}^{\lambda})_{k, m}^{R s} \right) \mathcal{E}(2^k (\eta_{\mathbf{k}}^{\lambda})_{k,m}^{R s} \kappa) \\
    &\nonumber \hspace{7cm} \left[f^{\lambda} \times_{\lambda} \varphi_{\mathbf{k}}^{\mathbf{b}^{\lambda}, \mathbf{r}}(R_{\lambda}^{-1}\cdot) \right](x) \,  d\lambda ,
\end{align}
and
\begin{align}
\label{Definition of TNRs Phi}
    T_{N, R s, s}^{\Phi, k, \kappa, \kappa'}(\mathcal{L})f(x,u) &= \frac{1}{(2\pi)^{d_2}} \int_{\mathfrak{g}_{2,r}^{*}} e^{i \langle \lambda , u \rangle} \sum_{\mathbf{k} \in \mathbb{N}^\Lambda} \Phi \left(2^{(1+\epsilon)k} (\eta_{\mathbf{k}}^{\lambda})_{k, m}^{R s} \right) \mathcal{E}(2^k (\eta_{\mathbf{k}}^{\lambda})_{k,m}^{R s} \kappa) \mathcal{E}(2^j s^2 (\eta_{\mathbf{k}}^{\lambda})_{k,m}^{R s} \kappa') \\
    &\nonumber \hspace{7cm} \left[f^{\lambda} \times_{\lambda} \varphi_{\mathbf{k}}^{\mathbf{b}^{\lambda}, \mathbf{r}}(R_{\lambda}^{-1}\cdot) \right](x) \,  d\lambda .
\end{align}
Consequently equation \eqref{Definition of P tilde RsPhi} gives
\begin{align*}
    \widetilde{P}_{R s, R, \Phi}^{j, k, m}(\mathcal{L}) f(x,u) &= \sum_{\ell =0}^{N-1} \frac{(-1)^{\ell}}{\ell !} s^{2\ell} (2 \pi i)^{\ell-1} 2^{-(k-j)(1+\epsilon)\ell} 2^{-j \epsilon \ell} \frac{\partial^{\ell} \Psi}{\partial \eta^{\ell}}(2^j (1-a-s^2 2^{-k(1+\epsilon)}m)) \\
    &\hspace{6cm} \int_{\mathbb{R}} \widehat{\widetilde{\Psi}}(\kappa) e^{2 \pi i d_{k,m}^{\epsilon} \kappa} D_{R s, \Phi_{\ell}}^{k, m, \kappa}(\mathcal{L})f(x,u) \, d\kappa \\
    & + \int_{\mathbb{R}} \int_{\mathbb{R}} \widehat{\widetilde{\Psi}}(\kappa) \widehat{\Psi}(\kappa') e^{2 \pi i d_{k,m}^{\epsilon} \kappa} e^{2 \pi i 2^j (1-a-s^2 2^{-k(1+\epsilon)}m)\kappa'} T_{N, R s, s}^{\Phi, k, \kappa, \kappa'}(\mathcal{L})f(x,u) \, d\kappa \, d\kappa' \\
    &=: \sum_{\ell =0}^{N-1} \frac{(-1)^{\ell}}{\ell !} s^{2\ell} (2 \pi i)^{\ell-1} 2^{-(k-j)(1+\epsilon)\ell} 2^{-j \epsilon \ell} \frac{\partial^{\ell} \Psi}{\partial \eta^{\ell}}(2^j (1-a-s^2 2^{-k(1+\epsilon)}m)) \\
    &\hspace{6cm} H_{R s, \Phi^{\ell}}^{k, m}(\mathcal{L})f(x,u) + I_{N, Rs, s}^{\Phi, k}(\mathcal{L})f(x,u) .
\end{align*}

Therefore we have to estimate $H_{R s, \Phi^{\ell}}^{k, m}(\mathcal{L})$ and $I_{N, Rs, s}^{\Phi, k}(\mathcal{L})$. This two can be estimated similarly as Lemma \ref{Lemma: Estimate of X term}. In fact, since $\widetilde{\Psi} \in \mathcal{S}(\mathbb{R})$ we have
\begin{align}
\label{Dominate H by D}
    \Big\|\sup_{R>0, s \in [s_1, s_2]} |H_{R s, \Phi^{\ell}}^{k, m}(\mathcal{L})f| \Big\|_{L^p} &\leq \int_{\mathbb{R}} |\widehat{\widetilde{\Psi}}(\kappa)| 2^{-k \epsilon N} (1+|\kappa|)^N 2^{k(1+\epsilon) (Q+2)} \|\mathcal{M}f\|_{L^p} \, d\kappa \\
    &\nonumber \lesssim 2^{-k \epsilon'} \|f\|_{L^p} ,
\end{align}
for some $\epsilon'>0$ by choosing $N>0$ sufficiently large.

Similarly, we also have
\begin{align}
\label{Estimate of I in the P tilde part}
    \Big\|\sup_{R>0, s \in [s_1, s_2]} |I_{N, Rs, s}^{\Phi, k}(\mathcal{L})f| \Big\|_{L^p} & \leq C \int_{\mathbb{R}} \int_{\mathbb{R}} |\widehat{\widetilde{\Psi}}(\kappa)| |\widehat{\Psi}(\kappa')| 2^{-j \epsilon_1} 2^{-k \epsilon_2} (1+|\kappa|+ |\kappa'|)^N \|\mathcal{M}f\|_{L^p} \, d\kappa \, d\kappa' \\
    &\nonumber \leq C 2^{-j \epsilon_1} 2^{-k \epsilon_2} \|f\|_{L^p} .
\end{align}
Therefore with the help of the estimates \eqref{Dominate H by D}, \eqref{Estimate of I in the P tilde part} and the fact \eqref{Derivative of Psi is uniformly bounded} we obtain
\begin{align}
\label{Final estimate of P tilde Rs phi}
    & \left\|\sup_{R>0} \left(\int_{s_1}^{s_2} |\widetilde{P}_{R s, R, \Phi}^{j, k, m}(\mathcal{L}) f|^2 \, ds \right)^{1/2} \right\|_{L^p} \\
    &\nonumber \leq C \sum_{\ell =0}^{N-1} \frac{(-1)^{\ell}}{\ell !} (2 \pi i)^{\ell-1} 2^{-(k-j)(1+\epsilon)\ell} 2^{-j \epsilon \ell} \Big\|\sup_{R>0, s \in [s_1, s_2]} |H_{R s, \Phi^{\ell}}^{k, m}(\mathcal{L})f| \Big\|_{L^p} \\
    &\nonumber \hspace{8cm} + \Big\|\sup_{R>0, s \in [s_1, s_2]} |I_{N, Rs, s}^{\Phi, k}(\mathcal{L})f| \Big\|_{L^p} \\
    &\nonumber \leq C_{\ell} 2^{-k \epsilon''} \|f\|_{L^p} ,
\end{align}
for some $\epsilon''>0$.

Finally combining the estimates \eqref{Estimate of U Rs phi beta part} and \eqref{Final estimate of P tilde Rs phi} from the expression \eqref{Writing S tilde as sum of U and P tilde} and using the fact \eqref{Uniform bound for Psi tilde} and $\nu=\frac{2}{3}2^{-k \epsilon}$ we get
\begin{align*}
    & \left\|\sup_{R>0} \left(\int_{s_1}^{s_2} |\widetilde{S}_{j, \gamma+1}^{R, s}f|^2 \, ds \right)^{1/2} \right\|_{L^p} \\
    &\nonumber \leq C \sum_{k \geq j-2} 2^{-k \gamma} \sum_{0\leq m \leq [\nu^{-1}]+1} \sum_{\beta=0}^{N-1} \frac{(-1)^{\beta}}{\beta !} \frac{(2 \pi i 2^{-\epsilon k})^{\beta}}{2 \pi i} \frac{\partial^{\beta} \widetilde{\Psi}}{\partial \eta^{\beta}}(d_{k,m}^{\epsilon})  \left\|\sup_{R>0} \left(\int_{s_1}^{s_2} |U_{R s, R, \Phi^{\beta}}^{j, k, m}(\mathcal{L})f|^2 \, ds \right)^{1/2} \right\|_{L^p} \\
    &\nonumber \hspace{5cm} + \sum_{k \geq j-2} 2^{-k \gamma} \sum_{0\leq m \leq [\nu^{-1}]+1} \left\|\sup_{R>0} \left(\int_{s_1}^{s_2} |\widetilde{P}_{R s, R, \Phi}^{j, k, m}(\mathcal{L}) f|^2 \, ds \right)^{1/2} \right\|_{L^p} \\
    &\nonumber \leq C \sum_{k \geq j-2} 2^{-k \gamma} \sum_{0\leq m \leq [\nu^{-1}]+1} \sum_{\beta=0}^{N-1} \frac{(-1)^{\beta}}{\beta !} \frac{(2 \pi i 2^{-\epsilon k})^{\beta}}{2 \pi i} \frac{\partial^{\beta} \widetilde{\Psi}}{\partial \eta^{\beta}}(d_{k,m}^{\epsilon})  2^{\epsilon'(1+\epsilon)k} 2^{(1+\epsilon)(\alpha_d(p)-1/2)k} \|f\|_{L^p} \\
    &\nonumber \hspace{6cm} + \sum_{k \geq j-2} 2^{-k \gamma} \sum_{0\leq m \leq [\nu^{-1}]+1} 2^{-k \epsilon''} \|f\|_{L^p} \\
    &\leq C \sum_{k \geq j-2} 2^{-k \gamma} 2^{k \epsilon} 2^{\epsilon'(1+\epsilon)k} 2^{(1+\epsilon)(\alpha_d(p)-1/2)k} \|f\|_{L^p} \\
    &\leq C 2^{-j \gamma} 2^{j(\alpha_d(p)-1/2)} 2^{j \epsilon_1} \|f\|_{L^p} ,
\end{align*}
for some $\epsilon_1>0$, by choosing $\epsilon>0$ sufficiently small.

This completes the proof of the Proposition \ref{Prop: Estimate of S j tilde for j large}.
\end{proof}

\medskip
\noindent \textbf{Estimate of $\mathcal{B}_{0,*}^{\alpha}$:}
\label{Estimate of Bo part for maximal}
Estimate of $\mathcal{B}_{0,*}^{\alpha}$ can be obtained adapting the similar proof as in subsection 5.3 of \cite{Jotsaroop_Shrivastava_Maximal_Bochner_Riesz_2022}, with obvious modification. One have to use Lemma \ref{Lemma: Boundedness of maximal and Euclidean maximal} and ideas from the Estimate of $\mathcal{B}_{j,*}^{\alpha}$. Hence we omit the details here.


\subsection{Proof of Theorem \ref{Theorem: Bilinear Stein square function estimate}: Bilinear Bochner-Riesz square function on M\'etivier groups}
\label{Subsection: Bilinear Bochner-Riesz square function}
In this subsection, we give the proof of Theorem \ref{Theorem: Bilinear Stein square function estimate}. The idea of the proof is similar to the proof of the boundedness of maximal bilinear Bochner-Riesz operator already discussed in Subsection \ref{Subsection: Boundedness of bilinear maximal Bochner-Riesz means}.

Note that
\begin{align*}
    R \frac{\partial}{\partial R} \left(1-\frac{\eta_1 +\eta_2}{R^2} \right)_{+}^{\alpha+1} = 2 (\alpha+1) \frac{\eta_1 +\eta_2}{R^2} \left(1-\frac{\eta_1 +\eta_2}{R^2} \right)_{+}^{\alpha} =: \mathscr{G}^{\alpha}_{R}(\eta_1, \eta_2) .
\end{align*}
Therefore from \eqref{Definition of Bilinear Bochner-Riesz square function} Stein's square function can be written as the bilinear spectral multiplier corresponding to the bilinear multiplier $\mathscr{G}^{\alpha}_R(\eta_1, \eta_2)$ (see \eqref{Definition: Bilinear spectral multiplier}), that is,
\begin{align*}
    \mathscr{G}^{\alpha}(\mathcal{L})(f,g)(x, u) &= \left(\int_0^{\infty} |\mathcal{B}_{\mathscr{G}^{\alpha}_R}(f,g)(x,u)|^2 \, \frac{dR}{R} \right)^{1/2} \\
    &:= \left(\int_{0}^{\infty} |\mathscr{G}^{\alpha}_R(\mathcal{L}_1, \mathcal{L}_2)(f \otimes g)((x,u), (x,u))|^2 \, \frac{dR}{R} \right)^{1/2} .
\end{align*}
The idea of the proof is almost similar to Theorem \ref{Theorem: Bilinear maximal on Metivier}, hence here we will be brief.

Using the partition of unity \eqref{Definition of partition of unity for maximal} we write
\begin{align*}
    \mathscr{G}^{\alpha}_R(\eta_1, \eta_2) &= \sum_{j\geq 2} \Psi \left(2^j\left(1-\tfrac{\eta_1}{R^2} \right) \right) \tfrac{\eta_1+\eta_2}{R^2} \left(1-\tfrac{\eta_1}{R^2} \right)_{+}^{\alpha} \left(1-\tfrac{\eta_2}{R^2}\left(1-\tfrac{\eta_1}{R^2} \right)^{-1} \right)_{+}^{\alpha} \\
    & \hspace{6cm} + \Psi_0\left(\tfrac{\eta_1}{R^2} \right) \tfrac{\eta_1+\eta_2}{R^2} \left(1-\tfrac{\eta_1+\eta_2}{R^2} \right)_{+}^{\alpha} \\
    &\nonumber =: \sum_{j\geq 2} \mathscr{G}^{\alpha}_{j. R}(\eta_1, \eta_2) + \mathscr{G}^{\alpha}_{0, R}(\eta_1, \eta_2) .
\end{align*}
Consequently, we can obtain
\begin{align*}
    \mathscr{G}^{\alpha}(\mathcal{L})(f,g)(x, u) &\leq \mathscr{G}^{\alpha}_0(\mathcal{L})(f,g)(x, u) + \sum_{j\geq 2} \mathscr{G}^{\alpha}_j(\mathcal{L})(f,g)(x, u) ,
\end{align*}
where 
\begin{align*}
    \mathscr{G}^{\alpha}_j(\mathcal{L})(f,g)(x, u) &= \left(\int_0^{\infty} |\mathcal{B}_{\mathscr{G}^{\alpha}_{j,R}}(f,g)(x,u)|^2 \, \frac{dR}{R} \right)^{1/2} ,
\end{align*}
and $\mathcal{B}_{\mathscr{G}^{\alpha}_{j,R}}$ denote the bilinear spectral multiplier corresponding to the multiplier $\mathscr{G}^{\alpha}_{j,R}$ for $j=0, 2, 3, \ldots$.

Let us start with the estimate of $\mathscr{G}^{\alpha}_j(\mathcal{L})$ for $j \geq 2$. Similarly as in \eqref{Writing bilinear operator as product of two} here we have
\begin{align*}
    \mathcal{B}_{\mathscr{G}^{\alpha}_{j,R}}(f,g)(x,u) &= C R^{-2\alpha} \int_0^{R_j} \mathscr{T}_{j, \beta}^{R,t}(\mathcal{L})f(x,u) \, S_{t}^{\delta}(\mathcal{L})g(x,u) t^{2 \delta+1} \, dt \\
    &\hspace{3cm} + C R^{-2\alpha} \int_0^{R_j} T_{j, \beta}^{R,t}(\mathcal{L})f(x,u) \mathscr{S}_{t}^{\delta}(\mathcal{L})g(x,u) t^{2 \delta+1} \, dt \\
    &=: E_{j, R} + F_{j, R} ,
\end{align*}
where $T_{j, \beta}^{R,t}(\mathcal{L})$ is as defined in \eqref{Definition of the T operator in maximal} and $S_{t}^{\delta}(\mathcal{L})$ denote the Bochner-Riesz means of order $\delta$; while
\begin{align*}
    \mathscr{T}_{j, \beta}^{R,t}(\mathcal{L})f(x,u) &= \frac{1}{(2\pi)^{d_2}} \int_{\mathfrak{g}_{2,r}^{*}} e^{i \langle \lambda_1 , u \rangle} \sum_{\mathbf{k}_1 \in \mathbb{N}^\Lambda} \Psi \Big(2^j\Big(1-\frac{\eta_{\mathbf{k}_1}^{\lambda_1}}{R^2} \Big) \Big) \frac{\eta_{\mathbf{k}_1}^{\lambda_1}}{R^2} (R^2 \phi_R(\eta_{\mathbf{k}_1}^{\lambda_1})-t^2)_{+}^{\beta-1} \\
    & \hspace{5cm} \left[f^{\lambda_1} \times_{\lambda_1} \varphi_{\mathbf{k}_1}^{\mathbf{b}^{\lambda_1}, \mathbf{r}_1}(R_{\lambda_1}^{-1}\cdot) \right](x) \,  d\lambda_1 ,
\end{align*}
and
\begin{align*}
    \mathscr{S}_{t}^{\delta}(\mathcal{L})g(x,u) &= \frac{1}{(2\pi)^{d_2}} \int_{\mathfrak{g}_{2,r}^{*}} e^{i \langle \lambda_2 , u \rangle} \sum_{\mathbf{k}_2 \in \mathbb{N}^\Lambda} \Big(1-\frac{\eta_{\mathbf{k}_2}^{\lambda_2}}{t^2} \Big)_{+}^{\delta}  \frac{\eta_{\mathbf{k}_2}^{\lambda_2}}{R^2}\left[f^{\lambda_2} \times_{\lambda_2} \varphi_{\mathbf{k}_2}^{\mathbf{b}^{\lambda_2}, \mathbf{r}_2}(R_{\lambda_2}^{-1}\cdot) \right](x) \,  d\lambda_2 ,
\end{align*}

\noindent\textbf{Estimate of $F_{j, R}$:}
For the estimate of $F_{j, R}$, similarly as in \eqref{Decomposition bilinear operator into two parts} we obtain
\begin{align}
\label{Application of Cauchy-Schwartz and change}
    & \left\|\left(\int_0^{\infty} |F_{j, R}|^2 \, \frac{dR}{R} \right)^{1/2} \right\|_{L^p} \\
    &\nonumber \leq \left\|\sup_{R>0} \left(\int_0^{\sqrt{2^{-j+1}}} |S_{j, \beta}^{R, t}(\mathcal{L}) f \, t^{2\delta+1}|^2 \, dt \right)^{1/2} \right\|_{L^{p_1}} \left\|\left[\int_0^{\infty} \left(\int_0^{\sqrt{2^{-j+1}}} |\mathscr{S}_{R t}^{\delta}(\mathcal{L}) g|^2 \, dt \right) \frac{dR}{R} \right]^{1/2} \right\|_{L^{p_2}} ,
\end{align}
where $S_{j, \beta}^{R, t}(\mathcal{L})$ is as defined in \eqref{Expression of S j beta  t part}.

Note that estimate of first factor of the right hand side of the above inequality is already obtain in Proposition \ref{Proposition: Main proposition of S in full range}. Hence, it is enough to estimate the other factor. For this we have the following estimate: First note similarly as in \eqref{Stein squre function like estimate}, using Theorem \ref{Theorem: Stein square estimate} for $\mathfrak{p}_G \leq p_2 <\infty$ or $p_2=2$ whenever $\delta>\alpha_{d}(p_2)-\frac{1}{2}$ we obtain
\begin{align}
\label{Stein squre function like estimate inside bilinear square}
    \left\| \left(\int_0^{\infty} |\mathscr{S}_{R}^{\delta}(\mathcal{L}) g|^2 \, \frac{dR}{R} \right)^{1/2} \right\|_{L^{p_2}} &\leq C \|g\|_{L^{p_2}} .
\end{align}
Now making the change of variable $R \mapsto R t^{-1}$ and then using \eqref{Stein squre function like estimate inside bilinear square} with $\mathfrak{p}_G \leq p_2 <\infty$ or $p_2=2$ and whenever $\delta>\alpha_{d}(p_2)-\frac{1}{2}$ yields
\begin{align*}
    \left\|\left[\int_0^{\infty} \left(\int_0^{\sqrt{2^{-j+1}}} |\mathscr{S}_{R t}^{\delta}(\mathcal{L}) g|^2 \, dt \right) \frac{dR}{R} \right]^{1/2} \right\|_{L^{p_2}} &\leq C 2^{- 5j/4} \left\| \left(\int_0^{\infty} |\mathscr{S}_{R}^{\delta}(\mathcal{L}) g|^2 \, \frac{dR}{R} \right)^{1/2} \right\|_{L^{p_2}} \\
    &\leq C 2^{- 5j/4} \|g\|_{L^{p_2}} .
\end{align*}
Hence combining Proposition \ref{Proposition: Main proposition of S in full range} and the above estimate for $\mathfrak{p}_G \leq p_1, p_2<\infty$ or $p=2$ and whenever $\beta>\alpha_d(p_1)+\frac{1}{2}$, $\delta>\alpha_{d}(p_2)-\frac{1}{2}$ provides
\begin{align}
\label{Final estimate for F}
    \left\|\left(\int_0^{\infty} |F_{j, R}|^2 \, \frac{dR}{R} \right)^{1/2} \right\|_{L^p} &\leq C 2^{j(\alpha_d(p)-\alpha+1/4+\epsilon-5/4)} \|f\|_{L^{p_1}} \|g\|_{L^{p_2}} \leq C 2^{-j \varepsilon} \|f\|_{L^{p_1}} \|g\|_{L^{p_2}} ,
\end{align}
for some $\varepsilon>0$. Since $\alpha=\beta+\delta$, for $p_i=2$ or $\mathfrak{p}_G \leq p_i <\infty$ with $i=1,2$, the above estimate hold whenever $\alpha>\alpha_d(p_1)+\alpha_d(p_2) = \alpha_*(p_1, p_2, d, \mathfrak{p}_G)$.

\medskip
\noindent\textbf{Estimate of $E_{j, R}$:}
Similarly as in \eqref{Application of Cauchy-Schwartz and change}, in this case we obtain
\begin{align}
\label{Cauchy-Schwartz in E part}
    \left\|\left(\int_0^{\infty} |E_{j, R}|^2 \, \frac{dR}{R} \right)^{1/2} \right\|_{L^p} & \leq C 2^{-\frac{j}{4}} \left\|\left[\int_0^{\infty} \left(\int_0^{\sqrt{2^{-j+1}}} |\widetilde{\mathscr{T}}_{j, \beta}^{R, t}(\mathcal{L}) f \, t^{2\delta+1}|^2 \, dt \right) \frac{dR}{R} \right]^{1/2} \right\|_{L^{p_1}} \\
    &\nonumber \hspace{3cm} \times \left\| \sup_{R>0} \left(\frac{1}{R_j}\int_0^{R_j} |S_{t}^{\delta}(\mathcal{L}) g|^2 \, dt \right)^{1/2} \right\|_{L^{p_2}} ,
\end{align}
where
\begin{align}
\label{Expression of widetilde mathsec S j beta  t part}
    \widetilde{\mathscr{T}}_{j, \beta}^{R, t}(\mathcal{L}) f(x,u) &= \frac{1}{(2\pi)^{d_2}} \int_{\mathfrak{g}_{2,r}^{*}} e^{i \langle \lambda_1 , u \rangle} \sum_{\mathbf{k}_1 \in \mathbb{N}^\Lambda} \Psi \Big(2^j\Big(1-\frac{\eta_{\mathbf{k}_1}^{\lambda_1}}{R^2} \Big) \Big) \frac{\eta_{\mathbf{k}_1}^{\lambda_1}}{R^2} \Big(1-\frac{\eta_{\mathbf{k}_1}^{\lambda_1}}{R^2}-t^2 \Big)_{+}^{\beta-1} \\
    &\nonumber \hspace{6cm} \left[f^{\lambda_1} \times_{\lambda_1} \varphi_{\mathbf{k}_1}^{\mathbf{b}^{\lambda_1}, \mathbf{r}_1}(R_{\lambda_1}^{-1}\cdot) \right](x) \,  d\lambda_1 .
\end{align}
Estimate of the second factor in the right hand side of the above inequality follows from Proposition \ref{Proposition: Maximal square function estimate}. Therefore it remains to estimate only the first factor. Note that application of Minkowski's integral inequality gives
\begin{align}
\label{Application of Minkowski for the E part}
    & \left\|\left[\int_0^{\infty} \left(\int_0^{\sqrt{2^{-j+1}}} |\widetilde{\mathscr{T}}_{j, \beta}^{R, t}(\mathcal{L}) f \, t^{2\delta+1}|^2 \, dt \right) \frac{dR}{R} \right]^{1/2} \right\|_{L^{p_1}} \\
    &\nonumber \leq \left(\int_0^{\sqrt{2^{-j+1}}} \left\| \left( \int_0^{\infty} |\widetilde{\mathscr{T}}_{j, \beta}^{R, t}(\mathcal{L}) f|^2 \, \frac{dR}{R} \right)^{1/2} \right\|_{L^{p_1}}^2 t^{4\delta+2} \, dt \right)^{1/2} .
\end{align}

Now we claim that the following proposition hold.
\begin{proposition}
\label{Proposition: Estimate of mathscr S in full range}
Let $\mathfrak{p}_G \leq p<\infty$ or $p=2$ and $\epsilon>0$. Then whenever $\beta>\alpha_d(p)+1/2$ we have
\begin{align}
\label{To prove in mathscr T}
    \left\| \left( \int_0^{\infty} |\widetilde{\mathscr{T}}_{j, \beta}^{R, t}(\mathcal{L}) f|^2 \, \frac{dR}{R} \right)^{1/2} \right\|_{L^{p}} &\leq C 2^{j(\alpha_d(p)-\beta+1/2+\epsilon)} \|f\|_{L^p} ,
\end{align}
where the constant $C$ is independent of $t \in [0, \sqrt{2^{-j+1}}]$.
    
\end{proposition}

Assuming the proposition for the moment, let us complete the estimate of $\mathscr{G}^{\alpha}_j(\mathcal{L})$ for $j \geq 2$. With the help of Proposition \ref{Proposition: Estimate of mathscr S in full range} from \eqref{Application of Minkowski for the E part} whenever $\beta>\alpha_d(p_1)+1/2$ with $\alpha=\beta+\delta$ for $p_1=2$ or $\mathfrak{p}_G \leq p_1 <\infty$ we obtain
\begin{align}
\label{Using proposition obtained estimate for scr}
    \left\|\left[\int_0^{\infty} \left(\int_0^{\sqrt{2^{-j+1}}} |\widetilde{\mathscr{T}}_{j, \beta}^{R, t}(\mathcal{L}) f \, t^{2\delta+1}|^2 \, dt \right) \frac{dR}{R} \right]^{1/2} \right\|_{L^{p_1}} &\leq C 2^{j(\alpha_d(p)-\beta+1/2+\epsilon)} 2^{-j(4\delta+3)/4} \|f\|_{L^{p_1}} \\
    &\nonumber \leq C 2^{j(\alpha_d(p)-\alpha-1/4+\epsilon)} \|f\|_{L^{p_1}} .
\end{align}
Therefore combining the estimate \eqref{Using proposition obtained estimate for scr} and Proposition \ref{Proposition: Maximal square function estimate}, for $\beta>\alpha_d(p_1)+1/2$, $\delta>\alpha_d(p_2)-1/2$ and $p_i=2$ or $\mathfrak{p}_G \leq p_i <\infty$ with $i=1,2$ from \eqref{Cauchy-Schwartz in E part} we get
\begin{align}
\label{Final estimate for E}
    \left\|\left(\int_0^{\infty} |E_{j, R}|^2 \, \frac{dR}{R} \right)^{1/2} \right\|_{L^p} &\leq C 2^{-j/4} 2^{j(\alpha_d(p)-\alpha-1/4+\epsilon)} \|f\|_{L^{p_1}} \|g\|_{L^{p_2}} \\
    &\nonumber \leq C 2^{-j \varepsilon} \|f\|_{L^{p_1}} \|g\|_{L^{p_2}} ,
\end{align}
for some $\varepsilon>0$ and since $\alpha=\beta+\delta$ and $\alpha>\alpha_d(p_1)+\alpha_d(p_2) = \alpha_*(p_1, p_2, d, \mathfrak{p}_G)$.

Now combining the estimates of $E_{j, R}$ \eqref{Final estimate for E} and $F_{j,R}$ \eqref{Final estimate for F} we get the required estimate for $\mathscr{G}^{\alpha}_j(\mathcal{L})$ with $j \geq 2$.

\begin{proof}[Proof of Proposition \ref{Proposition: Estimate of mathscr S in full range}]
\renewcommand{\qedsymbol}{}
The idea of proof of this proposition is similar to the Proposition \ref{Proposition: Main proposition of S in full range}. Set $\beta=\gamma+1$ and divide the interval of $t$ into two parts as $[0, \sqrt{2^{-j-1-\epsilon_0}}]$ and $[\sqrt{2^{-j-1-\epsilon_0}}, \sqrt{2^{-j+1}}]$.

\medskip
\noindent \textbf{Case I: $t \in [0, \sqrt{2^{-j-1-\epsilon_0}}]$.}
Estimate of this case can be concluded using the ideas from Case I of Proposition \ref{Proposition: Main proposition of S in full range} and \cite[Case I of Theorem 4.2]{Choudhary_Jotsaroop_Shrivastava_Shuin_Bilinear_square_function}. Here we have to just use \eqref{Square function Lp estimate with negative power} of Proposition \ref{Proposition: Square function estimate}. Hence we omit the details here.

\medskip
\noindent \textbf{Case II: $t \in [\sqrt{2^{-j-1-\epsilon_0}}, \sqrt{2^{-j+1}}]$.}
Similarly as in \cite[Case II of  Theorem 4.2]{Choudhary_Jotsaroop_Shrivastava_Shuin_Bilinear_square_function}, in this case it is
enough to consider the following operator
\begin{align*}
    \mathscr{U}_{j, \beta}^{R, s}(\mathcal{L}) f(x,u) &= \frac{1}{(2\pi)^{d_2}} \int_{\mathfrak{g}_{2,r}^{*}} e^{i \langle \lambda_1 , u \rangle} \sum_{\mathbf{k}_1 \in \mathbb{N}^\Lambda} \Psi \Big(2^j\Big(1-\frac{\eta_{\mathbf{k}_1}^{\lambda_1}}{R^2} \Big) \Big) \Big(1-\frac{\eta_{\mathbf{k}_1}^{\lambda_1}}{s^2 R^2} \Big)_{+}^{\gamma} \frac{\eta_{\mathbf{k}_1}^{\lambda_1}}{s^2 R^2} \\
    &\nonumber \hspace{6cm} \left[f^{\lambda_1} \times_{\lambda_1} \varphi_{\mathbf{k}_1}^{\mathbf{b}^{\lambda_1}, \mathbf{r}_1}(R_{\lambda_1}^{-1}\cdot) \right](x) \,  d\lambda_1 ,
\end{align*}
where $s \in [s_1, s_2]:= [\sqrt{1-2^{-j+1}}, \sqrt{1-2^{-j-1-\epsilon_0}}]$, and prove the following: whenever $\gamma>\alpha_d(p)-1/2$ and $\mathfrak{p}_G \leq p<\infty$ or $p=2$ we have
\begin{align}
\label{To prove in mathscr U}
    \left\| \left( \int_0^{\infty} |\mathscr{U}_{j, \gamma+1}^{R, s}(\mathcal{L}) f|^2 \, \frac{dR}{R} \right)^{1/2} \right\|_{L^{p}} &\leq C 2^{j(\alpha_d(p)-\gamma-1/2+\epsilon)} \|f\|_{L^p} ,
\end{align}
where the constant $C$ is independent of $s$.

Let $M>100$ be a large number. Again, in order to get \eqref{To prove in mathscr U}, it is enough to prove the following two propositions:

\begin{proposition}
\label{Proposition: Estimate of mathscr U when j less than M case}
Let $\mathfrak{p}_G \leq p<\infty$ or $p=2$. Whenever $\gamma>\alpha_d(p)-1/2$ and $0<\epsilon<1$ we have
\begin{align*}
    \left\|\left( \int_0^{\infty} |\mathscr{U}_{j, \gamma+1}^{R, s}(\mathcal{L}) f|^2 \, \frac{dR}{R} \right)^{1/2} \right\|_{L^p} &\leq C(M, \gamma, \epsilon) \|f\|_{L^p} ,
\end{align*}
for all $2\leq j \leq M$.
    
\end{proposition}

\begin{proposition}
\label{Prop: Estimate of math scr U for j large}
Let $\mathfrak{p}_G \leq p<\infty$ or $p=2$ and $\gamma>\alpha_d(p)-1/2$. Then for all $j \geq M$ and $0<\epsilon_1<1$ we have
\begin{align*}
    \left\|\left( \int_0^{\infty} |\mathscr{U}_{j, \gamma+1}^{R, s}(\mathcal{L}) f|^2 \, \frac{dR}{R} \right)^{1/2} \right\|_{L^p} &\leq C 2^{-j \gamma} 2^{j(\alpha_d(p)-1/2)} 2^{\epsilon_1 j} \|f\|_{L^p} .
\end{align*}
    
\end{proposition}

Therefore it only remains to prove the above two propositions. Note that proof of Proposition \ref{Proposition: Estimate of mathscr U when j less than M case} can be completed using the same line of argument as of Proposition \ref{Proposition: Estimate of S tilde when j less than M case}, while the proof of Proposition \ref{Prop: Estimate of math scr U for j large} can be obtained analogously as of Proposition \ref{Prop: Estimate of S j tilde for j large} with appropriate modification. 

Note that similarly as in \eqref{Introducting Psi tilde function for S tilde} using the partition of unity here in this case we have to write $\left(1-\frac{\eta}{s^2 R^2} \right)_{+}^{\gamma} \frac{\eta}{s^2 R^2}$ in terms of $\widetilde{\Psi}$ function. But since $\frac{\eta}{s^2 R^2} \sim 1$, on the support of the function $\widetilde{\Psi}\left(2^k\left(1-\frac{\eta}{s^2 R^2} \right)_{+}^{\gamma} \right)$ for $k \geq 2$, hence we can ignore the factor $\frac{\eta}{s^2 R^2}$. Therefore the corresponding operators are basically exactly same. Let us mention that instead of Lemma \ref{Lemma: Estimate of X term}, here in this case one needs the following analogue:
\begin{lemma}
For $1< p < \infty$, the following holds
\begin{align*}
    \left\| \left( \int_0^{\infty} |X_{N, Rs, s}^{\Phi^{\beta}, k}(\mathcal{L}) f|^2 \, \frac{dR}{R} \right)^{1/2} \right\|_{L^p} \lesssim 2^{(1+\epsilon)(\alpha_d(p)-1/2)k} \|f\|_{L^p} .
\end{align*}
    
\end{lemma}

\begin{proof}
\renewcommand{\qedsymbol}{}
Similarly as in \eqref{Claim for QNRs}, enough to estimate the following:
\begin{align}
\label{To prove for square function for Q}
    \left\| \left( \int_0^{\infty} |Q_{N, R s, s}^{\Phi^{\beta}, k, \kappa}(\mathcal{L})f|^2 \, \frac{dR}{R} \right)^{1/2} \right\|_{L^p} \lesssim (1+|\kappa|)^{N} 2^{-j \epsilon N} 2^{-(k-j)(1+\epsilon)N} 2^{k(1+\epsilon)(Q+3)} \|f\|_{L^p} .
\end{align}
From \eqref{Writing QNRs as kernel expression} we have
\begin{align*}
    Q_{N, R s, s}^{\Phi^{\beta}, k, \kappa}(\mathcal{L})f(x,u) &= 2^{-j \epsilon N} 2^{-(k-j)(1+\epsilon)N} \int_{G} \mathcal{K}_{\widetilde{Q}_{N, R s, s}^{\Phi^{\beta}, k, \kappa}(\mathcal{L})}((y,t)^{-1}(x,u)) f(y,t) \, d(y,t) ,
\end{align*}
Analogous to \eqref{Final estimate of kernel of Q tilde}, using \eqref{Pointwise Kernel estimate with weight second} of Lemma \ref{Lemma: Pointwise kernel estimate for linear kernel} for any $\beta>0$ we obtain
\begin{align*}
    |X\mathcal{K}_{\widetilde{Q}_{N, R s, s}^{\Phi^{\beta}, k, \kappa}(\mathcal{L})}(x,u)| (1+ R \|(x,u)\|)^{\beta} &\leq C R^{Q+1} (1+|\kappa|)^N 2^{(1+\epsilon) (\beta+1)k} ,
\end{align*}
where $X=(X_1, \ldots, X_{d_1}, T_1, \ldots, T_{d_2})$.

Then choosing $\beta=Q+2$ in the previous estimate for $(x,u) \neq 0$, yields
\begin{align}
\label{Condition for vector valued Calderon-Zygmund}
    \left( \int_0^{\infty} |X\mathcal{K}_{\widetilde{Q}_{N, R s, s}^{\Phi^{\beta}, k, \kappa}(\mathcal{L})}(x,u)|^2 \, \frac{dR}{R} \right)^{1/2} & \leq C (1+|\kappa|)^N 2^{k(1+\epsilon) (Q+3)} \|(x,u)\|^{-(Q+1)} .
\end{align}
Now in view of the above estimate \eqref{Condition for vector valued Calderon-Zygmund}, one can apply vector-valued Calder\'on-Zygmund theory on space of homogeneous type (see \cite{Ruiz_Torrea_Vector-Valued_Calderon_Zygmund}) to obtain the required estimate \eqref{To prove for square function for Q}.
\end{proof}
This completes the proof of $\mathscr{G}^{\alpha}_j(\mathcal{L})$ for $j \geq 2$. Now the estimate for $\mathscr{G}^{\alpha}_0(\mathcal{L})$ can be handled similarly as in estimate of $\mathcal{B}_{0,*}^{\alpha}$ in Theorem \ref{Theorem: Bilinear maximal on Metivier}, we omit the details here.

Therefore the proof Theorem \ref{Theorem: Bilinear Stein square function estimate} is completed.
\end{proof}


\subsection{Proof of Theorem \ref{Theorem: Bilinear maximal spectral multiplier}: Maximal bilinear spectral multipliers on M\'etivier groups}
\label{Subsection: Bilinear maximal spectral multipliers}
This subsection concerns with the proof of Theorem \ref{Theorem: Bilinear maximal spectral multiplier}. The main idea of the proof is similar to the boundedness of the linear maximal spectral multipliers, see Subsection \ref{Subsection: Boundedness of maximal spectral multiplier}.

From \eqref{Main identity to use maximal multiplier result} we have
\begin{align*}
    F\left(\frac{\eta_1 + \eta_2}{R} \right) &= C \int_{0}^{\infty} \frac{\eta_1 +\eta_2}{Rs} \left(1-\frac{\eta_1 +\eta_2}{Rs} \right)_{+}^{\alpha-1} s^{\alpha} \frac{d^{\alpha}}{ds^{\alpha}}\left(\frac{F(s)}{s} \right) \, ds .
\end{align*}
Therefore we obtain
\begin{align*}
    & F\left(\frac{\mathcal{L}_1 + \mathcal{L}_2}{R} \right)(f \otimes g)((x,u), (x,u)) \\
    &= C \int_{0}^{\infty} \frac{\mathcal{L}_1 +\mathcal{L}_2}{Rs} \left(1-\frac{\mathcal{L}_1 +\mathcal{L}_2}{Rs} \right)_{+}^{\alpha-1}(f \otimes g)((x,u), (x,u)) \, s^{\alpha+1} \frac{d^{\alpha}}{ds^{\alpha}}\left(\frac{F(s)}{s} \right) \, \frac{ds}{s} .
\end{align*}
Note that
\begin{align*}
    F\left(\frac{\mathcal{L}_{bi}}{R} \right)(f, g)(x,u) &= F\left(\frac{\mathcal{L}_1 + \mathcal{L}_2}{R} \right)(f \otimes g)((x,u), (x,u)) .
\end{align*}
Hence similarly as in \eqref{After application of Holder final estimate} applying H\"older's inequality we get
\begin{align*}
    F^{*}(\mathcal{L}_{bi})(f,g)(x,u) &\leq C \|F\|_{L^2_{\alpha}(\mathbb{R}^+)} \, \mathscr{G}^{\alpha-1}(\mathcal{L})(f, g)(x,u) .
\end{align*}
Now taking $L^p$-norm on both side of the above inequality and applying Theorem \ref{Theorem: Bilinear Stein square function estimate} for $2 \leq p_1, p_2 < \infty$ and $\alpha> \alpha_*(p_1, p_2, d, \mathfrak{p}_G)+1$ yields
\begin{align*}
    \|F^{*}(\mathcal{L}_{bi})(f,g)\|_{L^p} &\leq C \|F\|_{L^2_{\alpha}(\mathbb{R}^+)} \|f\|_{L^{p_1}} \|g\|_{L^{p_2}} .
\end{align*}
This completes the proof of Theorem \ref{Theorem: Bilinear maximal spectral multiplier}.

\section*{Acknowledgments}
A substantial part of this research was carried out when the author was a Prime Minister's Research Fellow (PMRF) and research associated at Indian Institute of Science Education and Research Kolkata. Some part of this research, the author was supported from institute post-doctoral fellowship, Indian Institute of Science Education and Research Mohali. We would like to thank Sayan Bagchi for his careful reading of the manuscript and numerous useful suggestions regarding the organization of the paper as well as discussions.


\newcommand{\etalchar}[1]{$^{#1}$}
\providecommand{\bysame}{\leavevmode\hbox to3em{\hrulefill}\thinspace}
\providecommand{\MR}{\relax\ifhmode\unskip\space\fi MR }
\providecommand{\MRhref}[2]{%
  \href{http://www.ams.org/mathscinet-getitem?mr=#1}{#2}
}
\providecommand{\href}[2]{#2}

\end{document}